\documentclass[letterpaper,11pt,oneside,reqno]{amsart}
\usepackage[T2A]{fontenc}%
\usepackage[utf8]{inputenc}%
\usepackage[english]{babel}%
\usepackage{amsmath,amssymb,amsthm,amsfonts}%
\usepackage{hyperref}%
\usepackage{enumerate}%
\usepackage{array}%
\usepackage{graphicx}
\usepackage[mathscr]{euscript}
\usepackage{color,tikz}
\usetikzlibrary{shapes}
\usepackage[DIV15]{typearea}
\usepackage[width=.9\textwidth]{caption}
\allowdisplaybreaks%
\numberwithin{equation}{section}%
\usepackage{upgreek}

\newcommand{\Z}{\mathbb{Z}}
\renewcommand{\C}{\mathbb{C}}
\newcommand{\R}{\mathbb{R}}
\DeclareMathOperator{\E}{\mathbb{E}}
\DeclareMathOperator{\sgn}{\mathrm{sgn}}
\DeclareMathOperator{\Span}{\mathrm{span}}
\DeclareMathOperator*{\Res}{\mathsf{Res}}
\renewcommand{\i}{\mathbf{i}}

\newcommand{\al}{\alpha}
\newcommand{\la}{\lambda}

\newcommand{\be}{\beta}

\newcommand{\conj}{\mathsf{c}}
\newcommand{\F}{\mathsf{F}}
\newcommand{\G}{\mathsf{G}}

\newcommand{\sign}[1]{\mathsf{Sign}_{#1}}
\newcommand{\signp}[1]{\mathsf{Sign}_{#1}^{\scriptscriptstyle+}}
\newcommand{\signpe}{\mathsf{Sign}^{\scriptscriptstyle+}}

\newcommand{\md}{\,|\,{}}
\newcommand{\Sym}{\mathfrak{S}}

\let\oldphi\phi \let\phi\varphi \let\varphi\oldphi

\newcommand{\Lmatr}{\mathsf{L}}
\newcommand{\LJ}[1]{\mathsf{L}^{{\scriptscriptstyle(#1)}}}
\newcommand{\WJ}[1]{w^{{\scriptscriptstyle(#1)}}}
\newcommand{\WTJ}[1]{{\widetilde w}^{{\scriptscriptstyle(#1)}}}
\newcommand{\ybspec}{u}			
\newcommand{\ip}{\upxi} 			
\newcommand{\ipb}{{\Xi}} 			
\newcommand{\ipbb}{{\overline{\ipb}}} 			
\newcommand{\pow}{\upvarphi}			
\newcommand{\AY}{\mathsf{A}}
\newcommand{\BY}{\mathsf{B}}
\newcommand{\CY}{\mathsf{C}}
\newcommand{\DY}{\mathsf{D}}
\newcommand{\ay}{\mathsf{a}}	
\newcommand{\dy}{\mathsf{d}}
\newcommand{\Prob}{\mathsf{Prob}}

\newcommand{\TY}{\mathsf{T}}

\newcommand{\bv}{\mathsf{e}}

\newcommand{\adm}[2]{({#1},{#2})\in\mathsf{Adm}_{\ipb,\SPB}}
\newcommand{\admi}[2]{({#1},{#2})\in\mathsf{Adm}_{\ipbb,\SPB}}

\newcommand{\aind}{{\upalpha}}
\newcommand{\bind}{{\upbeta}}

\newcommand{\IS}{\mathcal{I}}
\newcommand{\JS}{\mathcal{J}}
\newcommand{\KS}{\mathcal{K}}
\newcommand{\ks}{\mathsf{k}}
\newcommand{\KSB}{{\mathcal{K}^{c}}}
\newcommand{\JSB}{{\mathcal{J}^{c}}}
\newcommand{\LS}{\mathcal{L}}

\newcommand{\sh}{\tau}
\newcommand{\SP}{{\mathsf{s}}}
\newcommand{\SPB}{\mathsf{S}}
\newcommand{\SPBB}{{\overline{\mathsf{S}}}}

\newcommand{\PI}[1]{{0^{\scriptscriptstyle{#1}}}}
\newcommand{\RHO}{{\boldsymbol\varrho}}

\newcommand{\MM}{\mathscr{M}}
\newcommand{\UU}{\mathbf{u}}
\newcommand{\VV}{\mathbf{v}}

\newcommand{\ZZZ}{\mathbf{z}}
\newcommand{\equalsignforQLa}{\circ}
\newcommand{\plusQLa}{{\scriptscriptstyle+}}
\newcommand{\minusQLa}{{\scriptscriptstyle-}}
\newcommand{\Qe}{\mathscr{Q}^{\equalsignforQLa}}
\newcommand{\Qp}{\mathscr{Q}^{\plusQLa}}
\newcommand{\Lae}{\Lambda^{\equalsignforQLa}}
\newcommand{\Lam}{\Lambda^{\minusQLa}}

\newcommand{\tnu}{{\tilde\nu}}
\newcommand{\EF}[2]{\Psi^{#1}_{#2}}

\newcommand{\Xe}{\mathscr{X}^{\equalsignforQLa}}
\newcommand{\Xp}{\mathscr{X}^{\plusQLa}}
\newcommand{\Xii}{\mathscr{X}^{\infty}}
\newcommand{\Xeh}{\mathscr{X}_{\textnormal{$q$-Hahn}}^{\equalsignforQLa}}
\newcommand{\Xph}{\mathscr{X}_{\textnormal{$q$-Hahn}}^{\plusQLa}}
\newcommand{\Xih}{\mathscr{X}_{\textnormal{$q$-Hahn}}^{{\infty}}}

\newcommand{\smu}{\upmu}
\newcommand{\snu}{\upnu}

\newcommand{\gap}{\mathrm{gap}}

\newcommand{\mm}{\upvartheta} 
\newcommand{\corr}[3]{\mathcal{Q}_{#1}(#2)} 
\newcommand{\corre}[2]{\mathcal{Q}_{#1}} 
\newcommand{\II}{\mathsf{Int}}

\newcommand{\HT}{\mathfrak{h}}

\newcommand{\expancoeff}[1]{r_{#1}}

\newcommand{\inv}{\mathsf{inv}}

\newcommand{\funspat}[1]{\mathcal{W}^{#1}}
\newcommand{\funspec}[1]{\mathcal{C}^{#1}}

\newcommand{\Pli}[1]{\mathscr{T}^{(#1)}}
\newcommand{\Pltrans}[1]{\mathscr{F}}
\newcommand{\Pltransi}[1]{\mathscr{J}}

\newcommand{\contq}[2]{{\boldsymbol\gamma}^{\scriptscriptstyle+}_{#2}[#1]}
\newcommand{\contqe}[1]{{\boldsymbol\gamma}^{\scriptscriptstyle+}_{#1}}
\newcommand{\contqi}[2]{{\boldsymbol\gamma}^{\scriptscriptstyle-}_{#2}[#1]}
\newcommand{\contqie}[1]{{\boldsymbol\gamma}^{\scriptscriptstyle-}_{#1}}
\newcommand{\contn}[1]{{\boldsymbol\gamma}[#1]}
\newcommand{\contnz}[2]{{\boldsymbol\gamma}[#1|{#2}]}
\newcommand{\contqiz}[2]{{\boldsymbol\gamma}^{\scriptscriptstyle-}_{#2}[#1|{#2}]}

\def\epsx{.12}
\newcommand{\vertexoo}[1]{\scalebox{#1}{\begin{tikzpicture}
[scale=1, very thick]
\node (i1) at (0,-1) {$g$};
\node (j1) at (-1,0) {$0$};
\node (i2) at (0,1) {$g$};
\node (j2) at (1,0) {$0$};
\draw[densely dotted] (j1) -- (j2);
\foreach \shi in {(0,0), (\epsx,0), (-\epsx,0)}
{\begin{scope}[shift=\shi]
\node (shi1) at (0,-1) {\phantom{$g$}};
\node (shi2) at (0,1) {\phantom{$g$}};
\draw[->] (shi1) -- (shi2);
\draw[->] (shi1) --++ (0,.5);
\end{scope}}
\end{tikzpicture}}}
\newcommand{\vertexol}[1]{\scalebox{#1}{\begin{tikzpicture}
[scale=1, very thick]
\node (i1) at (0,-1) {$g$};
\node (j1) at (-1,0) {$0$};
\node (i2) at (0,1) {$g-1$};
\node (j2) at (1,0) {$1$};
\foreach \shi in {(0,0), (-\epsx,0)}
{\begin{scope}[shift=\shi]
\node (shi1) at (0,-1) {\phantom{$g$}};
\node (shi2) at (0,1) {\phantom{$g$}};
\draw[->] (shi1) -- (shi2);
\draw[->] (shi1) --++ (0,.5);
\end{scope}}
\foreach \shi in {(\epsx,0)}
{\begin{scope}[shift=\shi]
\node (shi1) at (0,-1) {\phantom{$g$}};
\node (shi2) at (0,1) {\phantom{$g$}};
\draw[->] (shi1) -- (0,0) -- (j2);
\draw[->] (shi1) --++ (0,.5);
\end{scope}}
\draw[densely dotted] (j1) -- (j2);
\end{tikzpicture}}}

\newcommand{\vertexoll}[1]{\scalebox{#1}{\begin{tikzpicture}
[scale=1, very thick]
\node (i1) at (0,-1) {$g+1$};
\node (j1) at (-1,0) {$0$};
\node (i2) at (0,1) {$g$};
\node (j2) at (1,0) {$1$};
\foreach \shi in {(-1.5*\epsx,0), (.5*\epsx,0),(-.5*\epsx,0)}
{\begin{scope}[shift=\shi]
\node (shi1) at (0,-1) {\phantom{$g$}};
\node (shi2) at (0,1) {\phantom{$g$}};
\draw[->] (shi1) -- (shi2);
\draw[->] (shi1) --++ (0,.5);
\end{scope}}
\foreach \shi in {(1.5*\epsx,0)}
{\begin{scope}[shift=\shi]
\node (shi1) at (0,-1) {\phantom{$g$}};
\node (shi2) at (0,1) {\phantom{$g$}};
\draw[->] (shi1) -- (0,0) -- (j2);
\draw[->] (shi1) --++ (0,.5);
\end{scope}}
\draw[densely dotted] (j1) -- (j2);
\end{tikzpicture}}}

\newcommand{\vertexll}[1]{\scalebox{#1}{\begin{tikzpicture}
[scale=1, very thick]
\node (i1) at (0,-1) {$g$};
\node (i1shm1) at (-\epsx,-1) {\phantom{$g$}};
\node (i1sh1) at (\epsx,-1) {\phantom{$g$}};
\node (j1) at (-1,0) {$1$};
\node (i2) at (0,1) {$g$};
\node (i2shm1) at (-\epsx,1) {\phantom{$g$}};
\node (i2sh1) at (\epsx,1) {\phantom{$g$}};
\node (j2) at (1,0) {$1$};
\draw[densely dotted] (j1) -- (j2);
\draw[densely dotted] (i1) -- (i2);
\draw[->] (j1) -- (-\epsx*1.5,0)--++(.5*\epsx,.5*\epsx) -- (i2shm1);
\draw[->] (i1shm1) -- (-\epsx,-\epsx) -- (0,\epsx) -- (i2);
\draw[->] (i1) -- (0,-\epsx) -- (\epsx,\epsx) -- (i2sh1);
\draw[->] (i1sh1) -- (\epsx,-.5*\epsx)--++(.5*\epsx,.5*\epsx) -- (j2);
\draw[->] (j1) --++ (.5,0);
\draw[->] (i1) --++ (0,.5);
\draw[->] (i1sh1) --++ (0,.5);
\draw[->] (i1shm1) --++ (0,.5);
\end{tikzpicture}}}
\newcommand{\vertexlo}[1]{\scalebox{#1}{\begin{tikzpicture}
[scale=1, very thick]
\node (i1) at (0,-1) {$g$};
\node (j1) at (-1,0) {$1$};
\node (i2) at (0,1) {$g+1$};
\node (i2shm2) at (-3/2*\epsx,1) {\phantom{$g$}};
\node (j2) at (1,0) {$0$};
\draw[densely dotted] (j1) -- (j2);
\draw[->] (j1) -- (-3/2*\epsx,0) -- (i2shm2);
\foreach \shi in {(1/2*\epsx,0), (3/2*\epsx,0), (-1/2*\epsx,0)}
{\begin{scope}[shift=\shi]
\node (shi1) at (0,-1) {\phantom{$g$}};
\node (shi2) at (0,1) {\phantom{$g$}};
\draw[->] (shi1) -- (shi2);
\draw[->] (shi1) --++ (0,.5);
\end{scope}}
\draw[->] (j1) --++ (.5,0);
\end{tikzpicture}}}

\newcommand{\emptyvertex}[1]{\raisebox{-1.5pt}{\scalebox{#1}{\begin{tikzpicture}
	[scale=1,very thick]
	\draw[dotted] (0,0)--++(0,.5);
	\draw[dotted] (-.25,.25)--++(.5,0);
\end{tikzpicture}}}}

\newtheorem{proposition}{Proposition}[section]
\newtheorem{lemma}[proposition]{Lemma}
\newtheorem{corollary}[proposition]{Corollary}
\newtheorem{theorem}[proposition]{Theorem}
\newtheorem*{theorem*}{Theorem}
\theoremstyle{definition}
\newtheorem{definition}[proposition]{Definition}
\newtheorem{remark}[proposition]{Remark}
\newtheorem*{example}{Example}

\begin{document}

\title[Higher spin six vertex model]
{Higher spin six vertex model\\ and symmetric rational functions}
\author[A. Borodin]{Alexei Borodin}
\address{A. Borodin, Department of Mathematics, 
Massachusetts Institute of Technology,
77 Massachusetts ave.,
Cambridge, MA 02139, USA,
\newline{}and Institute for Information Transmission Problems, Bolshoy Karetny per. 19, Moscow, 127994, Russia}
\email{borodin@math.mit.edu}

\author[L. Petrov]{Leonid Petrov}
\address{L. Petrov, Department of Mathematics, University of Virginia, 
141 Cabell Drive, Kerchof Hall,
P.O. Box 400137,
Charlottesville, VA 22904, USA,
\newline{}and Institute for Information Transmission Problems, Bolshoy Karetny per. 19, Moscow, 127994, Russia}
\email{lenia.petrov@gmail.com}
\date{}
\begin{abstract}
	We consider a fully inhomogeneous stochastic higher spin six vertex model in a quadrant. For this model we derive concise integral representations for multi-point $q$-moments of the height function and for the $q$-correlation functions. At least in the case of the step initial condition, our formulas degenerate in appropriate limits to many known formulas of such type for integrable probabilistic systems in the (1+1)d KPZ universality class, including the stochastic six vertex model, ASEP, various $q$-TASEPs, and associated zero range processes. 

	Our arguments are largely based on properties of a family of symmetric rational functions that can be defined as partition functions of the inhomogeneous higher spin six vertex model for suitable domains. In the homogeneous case, such functions were previously studied in \cite{Borodin2014vertex}; they also generalize classical Hall--Littlewood and Schur polynomials. 
	A key role is played by Cauchy-like summation identities for these functions, which are obtained as a direct corollary of the Yang--Baxter equation for the higher spin six vertex model. 
\end{abstract}
\maketitle

\setcounter{tocdepth}{1}
\tableofcontents
\setcounter{tocdepth}{3}

\section{Introduction} 
\label{sec:introduction}

\subsection{Preface}
The last two decades have seen remarkable progress in understanding the so-called KPZ universality class in (1+1) dimensions. This is a rather broad and somewhat vaguely defined class of probabilistic systems describing random interface growth, named after a seminal physics paper of Kardar--Parisi--Zhang of 1986 \cite{KPZ1986}. A key conjectural property of the systems in this class is that the large time fluctuations of the interfaces should be the same for all of them. 
See Corwin \cite{CorwinKPZ} for an extensive survey. 

While proving such a universality principle remains largely out of reach, by now many concrete systems have been found, for which the needed asymptotics was actually computed (the universality principle appears to hold so far). 

The first wave of these solved systems started in late 1990's with the papers of Johansson \cite{Johansson1999} and Baik--Deift--Johansson \cite{baik1999distribution}, and the key to their solvability, or \emph{integrability}, was in (highly non-obvious) reductions to what physicists would call \emph{free-fermion models} --- probabilistic systems, many of whose observables are expressed in terms of determinants and Pfaffians. Another domain where free-fermion models are extremely important is Random Matrix Theory. Perhaps not surprisingly, the large time fluctuations of the (1+1)d KPZ models are very similar to those arising in (largest eigenvalues of) random matrices with real spectra. 

The second wave of integrable (1+1)d KPZ systems started in late 2000's. The reasons for their solvability are harder to see, but one way or another they can be traced to quantum integrable systems. For example, looking at the earlier papers of the second wave we see that: (a) The pioneering work of Tracy--Widom \cite{TW_ASEP1}, \cite{TW_ASEP2}, \cite{TW_ASEP4} on the Asymmetric Simple Exclusion Process (ASEP) was based on the famous idea of Bethe \cite{Bethe1931} of looking for eigenfunctions of a quantum many-body system in the form of superposition of those for noninteracting bodies (coordinate Bethe ansatz); (b) The work of O'Connell \cite{Oconnell2009_Toda} and Borodin--Corwin \cite{BorodinCorwin2011Macdonald} on semi-discrete Brownian polymers utilized properties of eigenfunctions of the Macdonald--Ruijsenaars quantum integrable system --- the celebrated Macdonald polynomials and their degenerations; (c) The physics papers of Dotsenko \cite{Dotsenko} and Calabrese--Le Doussal--Rosso \cite{Calabrese_LeDoussal_Rosso}, and a later work of Borodin--Corwin--Sasamoto \cite{BorodinCorwinSasamoto2012} used a duality trick to show that certain observables of infinite-dimensional models solve finite-dimensional quantum many-body systems that are, in their turn, solvable by the coordinate Bethe ansatz; etc.

In fact, most currently known integrable (1+1)d KPZ models of the second wave come from \emph{one and the same} quantum integrable system: Corwin--Petrov \cite{CorwinPetrov2015} recently showed that they can be realized as suitable limits of what they called a \emph{stochastic higher spin vertex model}, see the introduction to their paper for a description of degenerations.\footnote{The integrable (1+1)d KPZ models that have not yet been shown to arise as limits of the stochastic vertex models are various versions of PushTASEP, cf. \cite{CorwinPetrov2013}, \cite{MatveevPetrov2014}. This appears to be simply an oversight as those models are diagonalized by the same wavefunctions, which means that needed reductions should also exist.}  They used duality and coordinate Bethe ansatz to show the integrability of their model; the Bethe ansatz part relied on previous works of Borodin--Corwin--Petrov--Sasamoto \cite{BorodinCorwinPetrovSasamoto2013}, \cite{BCPS2014}, and Borodin \cite{Borodin2014vertex}. 

The main subject of the present paper is the \emph{inhomogeneous} higher spin six vertex model in infinite volume; one homogeneous version of it is the model of \cite{CorwinPetrov2015}. Our main result is an integral representation for certain multi-point $q$-moments of this model. Such formulas are well known to be a source of meaningful asymptotic results, but we leave asymptotic questions outside of the scope of this paper. Several of those are forthcoming, and they will be presented as separate publications. 

At least for the simplest \emph{step initial condition}, our formula for $q$-moments degenerates, in appropriate limits, to most known formulas of this type, which includes all the models mentioned above. It also generalizes those quite a bit. 

We introduce two (infinite) sets of inhomogeneities (space inhomogeneities and spins). Together with another set of time inhomogeneities (spectral parameters) that were already present in \cite{CorwinPetrov2015}, they supply a large amount of freedom in the model, which can be used to induce unusual phase transitions (that work is also forthcoming).  Remarkably, all the parameters enter the expressions for $q$-moments in a rather benign way, still allowing those to be used for the purpose of accessing the asymptotics. In a way, this is similar to parameters of the Schur and Macdonald processes, cf. Okounkov--Reshetikhin \cite{okounkov2003correlation} and Borodin--Corwin \cite{BorodinCorwin2011Macdonald}, but at the moment this is no more than a vague comparison. 

Our methods are also new. The core of our proofs consists of the so-called (skew) Cauchy identities for certain rational symmetric functions. They also depend on two families of parameters (the third family --- spectral parameters --- are their arguments). Homogeneous versions of these functions were introduced in \cite{Borodin2014vertex}. For special parameter values, these functions turn into (skew) Hall--Littlewood and Schur symmetric functions, and the name ``Cauchy identities'' is borrowed from the theory of those, cf. Macdonald \cite{Macdonald1995}. 

Our symmetric rational functions can be defined as partition functions of the higher spin six vertex model for domains with special boundary conditions. Following \cite{Borodin2014vertex}, we use the \emph{Yang--Baxter equation}, or rather its infinite-volume limit, to derive the Cauchy identities. A similar approach to Cauchy-like identities for the Hall--Littlewood polynomials was also realized by Wheeler--Zinn--Justin in \cite{wheeler2015refined}.

Remarkably, the Cauchy identities themselves are essentially sufficient to define our probabilistic models, show their connection to KPZ interfaces, prove orthogonality and completeness of our symmetric rational functions in appropriate functional spaces, and evaluate averages of a large family of observables with respect to our measures. The last bit can be derived from comparing Cauchy identities with different sets of parameters. 

While the Cauchy identities also played an important role in the theory of Schur and Macdonald processes, they were never the main heroes there. In the present work they really take the central stage. Given their direct relation to the Yang--Baxter equation, one could thus say that the \emph{integrability of the (1+1)d KPZ models takes its origin in the Yang--Baxter integrability of the six vertex model.} 

Our present approach circumvents the duality trick that has been so powerful in treating the integrable (1+1)d KPZ models (including that of \cite{CorwinPetrov2015}). We do explain how the duality can be discovered from our results, but we do not prove or rely on it. Unfortunately, for the moment the use and success of the duality approach remains somewhat mysterious and ad hoc; the form of the duality functional needs to be guessed from previously known examples (some of which have better explanations, cf. Sch\"utz \cite{schutz1997dualityASEP}, Borodin--Corwin \cite{BorodinCorwin2013discrete}). We hope that the path that we present here is more straightforward, and that it can be used to shed further light on the existence of nontrivial dualities.  

Let us now describe our results in more detail. 

\subsection{The inhomogeneous model in a quadrant}\label{sec:intro-model}
Consider an ensemble $\mathcal P$ of infinite oriented up-right paths drawn in the first quadrant $\Z_{\ge 1}^2$ of the square lattice, with all the paths starting from a left-to-right arrow entering at each of the points $\{(1,m):m\in\Z_{\ge 1}\}$ on the left boundary (no path enters through the bottom boundary). Assume that no two paths share any horizontal piece (but common vertices and vertical pieces are allowed). See Fig.~\ref{fig:intro}. 

\begin{figure}[htb]
	\begin{tikzpicture}
		[scale=.7,thick]
		\def\d{.1}
		\foreach \xxx in {1,...,6}
		{
		\draw[dotted, opacity=.4] (\xxx-1,5.5)--++(0,-5);
		\node[below] at (\xxx-1,.5) {$\xxx$};
		}
		\foreach \xxx in {1,2,3,4,5}
		{
		\draw[dotted, opacity=.4] (0,\xxx)--++(5.5,0);
		\node[left] at (-1,\xxx) {$\xxx$};
		\draw[->, line width=1.7pt] (-1,\xxx)--++(.5,0);
		}
		\draw[->, line width=1.7pt] (-1,5)--++(1-3*\d,0)--++(\d,\d)--++(0,1-\d);
		\draw[->, line width=1.7pt] (-1,4)--++(1-2*\d,0)--++(\d,\d)--++(0,1-2*\d)--++(\d,2*\d)--++(0,1-\d);
		\draw[->, line width=1.7pt] (-1,3)--++(1-\d,0)--++(\d,\d)--++(0,1-2*\d)--++(\d,2*\d)--++(0,1-2*\d)--++(\d,2*\d)--++(0,1-\d);
		\draw[->, line width=1.7pt] (-1,2)--++(1,0)--++(0,1-\d)--++(\d,\d)--++(1-\d,0)
		--++(0,1)--++(2-\d,0)--++(\d,\d)--++(0,2-\d);
		\draw[->, line width=1.7pt] 
		(-1,1)--++(3,0)--++(0,2)--++(1,0)--++(0,1-\d)--++(\d,\d)--++(1-\d,0)
		--++(0,2);
		\draw[densely dashed] (-.5,4.5)--++(4,-4) node[above,anchor=west,yshift=16,xshift=-9] {$x+y=5$};
	\end{tikzpicture}
	\caption{A path collection $\mathcal P$.}
	\label{fig:intro}
\end{figure}
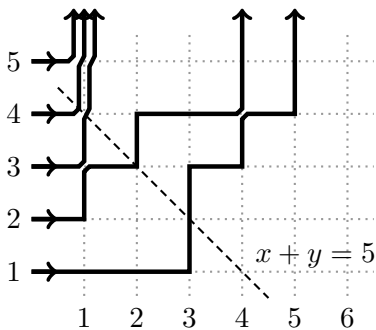
Define a probability measure on the set of such path ensembles in the following Markovian way. For any $n\ge 2$, assume that we already have a probability distribution on the intersections $\mathcal P_n$ of $\mathcal P$ with the triangle $T_n=\{(x,y)\in \Z_{\ge 1}^2: x+y\le n\}$. We are going to increase $n$ by 1. For each point $(x,y)$ on the upper boundary of $T_n$, i.e., for $x+y=n$, every $\mathcal P_n$ supplies us with two inputs: (1) The number of paths that enter $(x,y)$ from the bottom --- denote it by $i_1\in\Z_{\ge 0}$; (2) The number of paths $j_1\in\{0,1\}$ that enter $(x,y)$ from the left. Now choose, independently for all $(x,y)$ on the upper boundary of $T_n$, the number of paths $i_2$ that leave $(x,y)$ in the upward direction, and the number of paths $j_2$ that leave $(x,y)$ in the rightward direction, using the probability distribution with weights
of the transitions $(i_1,j_1)\to (i_2,j_2)$ given by 
(throughout the text $\mathbf{1}_{A}$ stands for the indicator function of the event $A$)
\begin{align}
	\label{intro-weights}
	\begin{array}{rclrcl}
		\Prob((i_1,0)\to (i_2,0))=&\dfrac{1-q^{i_1} \SP_x \ip_x u_y}{1-\SP_x \ip_x u_y}\,\mathbf 1_{i_1=i_2},\\
		\rule{0pt}{22pt}
		\Prob((i_1,0)\to (i_2,1))=&\dfrac{(q^{i_1}-1)\SP_x \ip_x u_y}{1-\SP_x \ip_x u_y}\,\mathbf 1_{i_1=i_2+1},\\
		\rule{0pt}{22pt}
		\Prob((i_1,1)\to (i_2,1))=&\dfrac{q^{i_1} \SP_x^{2}-\SP_x \ip_x u_y}{1-\SP_x \ip_x u_y}\,\,\mathbf 1_{i_1=i_2},\\
		\rule{0pt}{22pt}
		\Prob((i_1,1)\to (i_2,0))=&\dfrac{1-q^{i_1} \SP_x^{2}}{1-\SP_x \ip_x u_y}\,\mathbf 1_{i_1=i_2-1}.
	\end{array}
\end{align}
Assuming that all above expressions are nonnegative, which happens e.g. if $q\in (0,1)$, $\ip_x,u_y> 0$, $\SP_x\in (-1,0)$ for all $x,y\in \Z_{\ge 1}$, this procedure defines a probability measure on the set of all $\mathcal P$'s because we always have $\sum_{i_2,j_2} \Prob((i_1,j_1)\to (i_2,j_2))=1$, and $\Prob((i_1,j_1)\to (i_2,j_2))$ vanishes unless $i_1+j_1=i_2+j_2$. 

The right-hand sides in \eqref{intro-weights} are closely related to matrix elements of the R-matrix for $U_q(\widehat{\mathfrak{sl}_2})$, with one representation being an arbitrary Verma module and the other one being tautological. 

Let us briefly discuss the parameters of the model. The parameter $q$ is fixed throughout the paper. Two sets of parameters $\{\ip_x\}$ and $\{u_y\}$ play symmetric roles in \eqref{intro-weights}, and they can indeed be mapped to each other by a suitably defined transposition of the quadrant. 
In our exposition they will play different roles though, with the $u_y$'s considered \emph{spectral parameters} and the $\ip_x$'s considered \emph{spatial inhomogeneities}.

The parameters $\{\SP_x\}$ are related to \emph{spins}. If $\SP_x^2=q^{-I}$ for a positive integer $I$, then we have $\Prob((I,1)\to (I+1,0))=0$, which means that no more than $I$ paths can share the same vertical piece located at the horizontal coordinate $x$. This corresponds to replacing the arbitrary Verma module in the R-matrix with its $(I+1)$-dimensional irreducible quotient. The spin $\frac12$ situation $\SP_x\equiv q^{-\frac 12}$ gives rise to the stochastic six vertex introduced over 20 years ago by Gwa--Spohn \cite{GwaSpohn1992} (see \cite{BCG6V} for its detailed treatment).

The spin parameters $\{\SP_x\}$ are related to columns, and there are no similar row parameters: recall that no two paths can share the same horizontal piece. This restriction can be repaired using the procedure of \emph{fusion} that goes back to \cite{KulishReshSkl1981yang}. In plain words, fusion means grouping the spectral parameters $\{u_y\}$ into subsequences of the form $\{u,qu,\dots,q^{J-1}u\}$, and collapsing the corresponding $J$ rows onto a single one. Here the positive integer $J$ plays the same role as $I$ in the previous paragraph. The reason we did not use the second set of spin parameters in \eqref{intro-weights} is that the transition probabilities then become rather cumbersome, and one needs to specialize other parameters to achieve simpler expressions. A detailed exposition of the fusion is contained in \S\ref{sec:stochastic_weights_and_fusion} below. 

Let us also note that there are several substantially different possibilities of making the weights
\eqref{intro-weights} nonnegative; 
some of those we consider in detail. 
Since our techniques are algebraic, our results actually apply 
to any generic parameter values, with typically only minor 
modifications needed in case of some denominators vanishing. 

\subsection{The main result} Encode each path ensemble $\mathcal P$ by a \emph{height function}
$\mathfrak h:\Z_{\ge 1}^{2} \to \Z_{\ge 0}$ which assigns to each vertex $(x,y)$ the number $\mathfrak h(x,y)$ of paths in $\mathcal P$ that pass through or to the right of this vertex. 

\begin{theorem*}[Theorem \ref{thm:multi_moments} in the text]Assume that $q\in (0,1)$, $\ip_x,u_y>0$, $\SP_x\in (-1,0)$ for all $x,y\in \Z_{\ge 1}$, $u_i\ne qu_j$ for any $i,j\ge 1$, and
$$
\inf_{i\ge 1}(\ip_i^{-1}|\SP_i|)>q\cdot \sup_{i\ge 1}(\ip_i^{-1}|\SP_i|),\qquad
\inf_{i\ge 1}(\ip_i^{-1}|\SP_i|^{-1})>\sup_{i\ge 1}(\ip_i^{-1}|\SP_i|).
$$
Then for any integers $x_1\ge \ldots\ge x_{\ell}\ge1$ and $y\ge 1$,
\begin{multline}\label{intro-multi_moments}
		\E \prod_{i=1}^{\ell}q^{\HT(x_i,y)}=
		q^{\frac{\ell(\ell-1)}2}
		\oint\limits_{\contnz{\bar\UU}1}\frac{d w_1}{2\pi\i}
		\ldots
		\oint\limits_{\contnz{\bar\UU}\ell}\frac{d w_\ell}{2\pi\i}
		\prod_{1\le \aind<\bind\le \ell}\frac{w_\aind-w_\bind}{w_\aind-qw_\bind}
		\\\times
		\prod_{i=1}^{\ell}\bigg(
		w_i^{-1}
		\prod_{j=1}^{x_i-1}
		\frac{\ip_j-\SP_jw_i}{\ip_j-\SP_j^{-1}w_i}
		\prod_{j=1}^{y}\frac{1-qu_jw_i}{1-u_jw_i}
		\bigg),
	\end{multline}
where the expectation is taken with respect to the probability measure defined in \S\ref{sec:intro-model} above, and the integration contours are described in Definitions \ref{def:cont_gat_U} and \ref{def:circular_contours_around_0} and pictured in Fig.~\ref{fig:big_Gamma_contours} below.
\end{theorem*}

Let us emphasize that the inequalities on the parameters here are exceedingly restrictive; the statement can be analytically continued with suitable modifications of the contours and the integrand. Examples of such analytic continuation can be found in \S\ref{sec:degenerations_of_moment_formulas}, where they are used to degenerate the above result to various $q$-versions of the T(otally)ASEP. 

We also prove integral formulas similar to \eqref{intro-multi_moments} for another set of observables of our model that we call \emph{$q$-correlation functions}. The two are related, but in a rather nontrivial way, and one set of formulas does not immediately imply the other. 

While at the moment averages \eqref{intro-multi_moments} seem more useful for asymptotic analysis (and that is the reason we list them as our main result), it is entirely possible that the $q$-correlation functions will become useful for other asymptotic regimes. The definition and the expressions for the $q$-correlation functions can be found in \S\ref{sec:observables_of_interacting_particle_systems} below. 

\subsection{Symmetric rational functions} One consequence of the Yang--Baxter integrability of our model is that one can explicitly compute the distribution of intersection points of the paths in $\mathcal P$ with any horizontal line. More exactly, let $X_1\ge \ldots\ge X_n\ge 1$ be the $x$-coordinates of the points where our paths intersect the line $y=const$ with $n<const<n+1$; there are exactly $n$ of those, counting the multiplicities. Then
\begin{equation}\label{intro-section_probabilities}
\Prob\{X_1=\nu_1+1,\dots,X_n=\nu_n+1\}=\prod_{i=1}^n\prod_{j=1}^{\nu_i}(-\SP_j) \cdot\prod_{k\ge 0} \frac{(\SP_{k+1};q)_{n_k}}{(q;q)_{n_k}}\cdot \F_\nu\bigl(u_1,\dots,u_n\md \{\ip_x\}_{x\ge 1}, \{\SP_x\}_{x\ge 1}\bigr),
\end{equation}
where $(a;q)_{m}=(1-a)(1-aq)\ldots(1-aq^{m-1})$ are the $q$-Pochhammer symbols, 
$n_k$ is the multiplicity of $k$ in the sequence $\nu=(\nu_1\ge\ldots\ge\nu_n)=0^{n_0}1^{n_1}\cdots$, and for any $\mu=(\mu_1\ge\dots\ge\mu_M\ge 0)$ we define
\begin{align}\label{intro-F_symm_formula}
		\F_\mu\bigl(u_1,\ldots,u_M\md \{\ip_x\}_{x\ge 0},\{\SP_x\}_{x\ge 0}\bigr)=
		\sum_{\sigma\in\Sym_M}
		\sigma\Bigg(
		\prod_{1\le \aind<\bind\le M}\frac{u_\aind-qu_\bind}{u_\aind-u_\bind}
		\prod_{i=1}^{M}
		\pow_{\mu_i}\bigl(u_i\md\{\ip_x\}_{x\ge 0},\{\SP_x\}_{x\ge 0}\bigr)\Bigg)
\end{align}
with $\displaystyle\pow_k\bigl(u\md \{\ip_x\}_{x\ge 0},\{\SP_x\}_{x\ge 0}\bigr)=
\frac{1-q}{1-\SP_k\ip_k u}\prod\nolimits_{j=0}^{k-1}\frac{\ip_j u-\SP_j}{1-\SP_j\ip_ju}$, $k\ge0$.
Here $\Sym_M$ is the symmetric group of degree $M$, and its elements $\sigma$ permute the variables $\{u_i\}_{i=1}^M$ in the right-hand side of \eqref{intro-F_symm_formula}. Note that in the definition \eqref{intro-F_symm_formula} we shifted the lower limit of the index $x$ in $\{\ip_x\}$ and $\{\SP_x\}$ from 1 to 0 to conform with the rest of the paper. 

The symmetric rational functions $\{\F_\mu\}$ play a central role in our work (we view spectral variables $\{u_y\}$ as their arguments, and $q$, $\{\ip_x\}$, $\{\SP_x\}$ as parameters). The right-hand side of \eqref{intro-F_symm_formula} can be viewed as a coordinate Bethe ansatz expression for the eigenfunctions of the transfer-matrix of the higher spin six vertex model. Note that one would need to additionally impose Bethe equations on the $u$'s for periodic in the $x$-direction boundary conditions.

The probabilistic interpretation of $\F_\nu$ given above is equivalent to saying that $\F_\nu$ is the partition function for ensembles of $n$ up-right lattice paths that enter the semi-infinite strip $\Z_{\ge 0}\times\{1,\dots,n\}$ at the left boundary at $(0,1),\dots,(0,n)$ and exit at the top of the strip at locations $(\nu_1,n),\dots,(\nu_n,n)$. The weight of such an ensemble is equal to the product of weights over all vertices of the strip. The vertex weights for $\F_\nu$ itself are slightly modified right-hand sides of \eqref{intro-weights} given by \eqref{weights} below. Allowing some paths to enter at the bottom boundary gives a definition of the skew functions $\F_{\mu/\la}$; removing the paths entering from the left gives a definition of the skew functions $\G_{\mu/\la}$ and non-skew $\G_\mu:=\G_{\mu/0^M}$, cf. Fig.~\ref{fig:paths_FG} below. A symmetrization formula for $\G_\mu$ which is similar to \eqref{intro-F_symm_formula} is given in Theorem \ref{thm:symmetrization} below. 

Homogeneous versions of the functions $\F$ and $\G$ were introduced in \cite{Borodin2014vertex}. As explained there, further degenerations turn them into skew and non-skew Hall--Littlewood and Schur symmetric polynomials. 

\subsection{Cauchy identities} A basic fact about functions $\F$ and $\G$ that we heavily 
use is the following skew Cauchy identity.
Let $u,v\in\C$ satisfy
\begin{align*}
	\lim_{L\to+\infty}\prod_{j=0}^{L}
	\left|\frac{\ip_j u-\SP_j}{1-\SP_j\ip_j u}
	\cdot
	\frac{\ip_j^{-1}v-\SP_j}{1-\SP_j\ip_j^{-1}v}\right|=0.
\end{align*}
Then for any nonincreasing integer sequences 
$\la$ and $\nu$ as above, 
with notation $\ipb=\{\ip_x\}_{x\ge 0}$,  $\ipbb=\{\ip_x^{-1}\}_{x\ge 0}$, and 
$\SPB=\{\SP_x\}_{x\ge 0}$,  we have
\begin{align}\label{intro-skew_Cauchy_good}
	\sum_{\kappa}
	\frac{\conj_{\SPB}(\kappa)}{\conj_{\SPB}(\la)}\G_{\kappa/\la}(v\md\ipbb,\SPB)
	\F_{\kappa/\nu}(u\md\ipb,\SPB)
	= \frac{1-quv}{1-uv}
	\sum_{\mu}
	\F_{\la/\mu}(u\md\ipb,\SPB)
	\frac{\conj_{\SPB}(\nu)}{\conj_{\SPB}(\mu)}\G_{\nu/\mu}(v\md\ipbb,\SPB)
	,
\end{align}
where $\conj_{\SPB}(\al)=\displaystyle
\prod\nolimits_{i\ge 0}\frac{(\SP_i^2;q)_{a_i}}{(q;q)_{a_i}}$ for $\al=0^{a_0}1^{a_1}\cdots$.
This identity is a direct consequence of the Yang--Baxter equation for the R-matrix of the higher spin six vertex model. It involves only two spectral parameters $u$ and $v$ and corresponds
to permuting two single-row transfer matrices. 
Identity \eqref{intro-skew_Cauchy_good} can be immediately iterated to include any finite number of $u$'s and $v$'s, and also to involve non-skew functions (by setting $\nu$ to $\varnothing$ and/or $\lambda$ to $0^L$).

We put different versions of Cauchy identities to multiple uses:

\begin{enumerate}[\bf1.]
	\item The fact that probabilities \eqref{intro-section_probabilities} add up to 1 is a limiting instance of a Cauchy identity. Thus, we can think of the weights of the probability measures we are interested in as of (normalized) terms in a Cauchy identity. Such an interpretation (for other Cauchy identities) lies at the basis of the theory of Schur and Macdonald measures and processes \cite{okounkov2001infinite}, \cite{okounkov2003correlation}, \cite{BorodinCorwin2011Macdonald}. 

	\item Markov chains that connect measures of the form \eqref{intro-section_probabilities} with different values of $n$ are instances of skew Cauchy identities. Such an interpretation was also previously used in the Schur/Macdonald setting, cf. 
	\cite{BorFerr2008DF}, \cite{Borodin2010Schur}, \cite{BorodinCorwin2011Macdonald}.

	\item Comparing two Cauchy identities which differ by adding a few extra variables 
	leads to the average of an observable with respect to the measure whose weights are given by the terms of the other identity. This fact by itself is a triviality, but we show that it can be used to extract nontrivial consequences. To our knowledge,
	such use of Cauchy identities is new. 

	\item In extracting those consequences, a key role is played by a Plancherel theory for the functions 
	$\{\F_\mu\}$, and Cauchy identities 
	can be employed to establish certain orthogonality properties of the $\F_\mu$'s.
	This link goes back to \cite{Borodin2014vertex}.
\end{enumerate} 

\subsection{Plancherel theory}
Let us give more details regarding 
the Plancherel theory for the $\F_\mu$'s.
The (obvious from \eqref{intro-F_symm_formula}) shift property
\begin{align*}
		\F_{\mu+r^{M}}\bigl(u_1,\ldots,u_M\md \{\xi_x\}_{x\ge 0},\{\SPB_x\}_{x\ge 0}\bigr)=
		\bigg(\prod_{i=1}^{M}\prod_{j=0}^{r-1}\frac{\ip_ju_i-\SP_j}
		{1-\SP_j\ip_j u_i}\bigg)		
		\F_{\mu}\bigl(u_1,\ldots,u_M\md\{\xi_x\}_{x\ge r},\{\SPB_x\}_{x\ge r}\bigr)
\end{align*}
allows to extend the definition of the $\F_\mu$'s to arbitrary $\mu=(\mu_1\ge\dots\ge\mu_M)\in\Z^M$.
It turns out that these extended $\F_\mu$'s
form a nice Fourier-like basis in the space of
functions $\{f(\mu)\}$. 
More exactly, we prove that two maps $\mathscr{F}$ and $\mathscr J$ 
which map functions $f(\mu)$ to symmetric rational functions and backwards defined via
\begin{align*}
	(\Pltrans n f)(u_1,\ldots,u_n)&=\sum_{\la=(\la_1\ge \dots\ge\la_n)\in\Z^n}f(\la)\,\F_\la(u_1,\ldots,u_n\md\ipb,\SPB),
	\\
	(\Pltransi n R)(\la_1,\ldots,\la_n)&=
	\frac{\conj_{\SPB}(\la)}{(1-q)^{n}}
	\oint\limits_{\contq {\ipbb\SPB}1}\frac{d u_1}{2\pi\i}
	\ldots
	\oint\limits_{\contq {\ipbb\SPB}n}\frac{d u_n}{2\pi\i}
	\prod_{1\le \aind<\bind\le n}\frac{u_\aind-u_\bind}{u_\aind-qu_\bind}
	\frac{R(u_1,\ldots,u_n)}{u_1 \ldots u_n}
	\prod_{i=1}^{n}\pow_{\la_{i}}(u_{i}^{-1}\md\ipbb,\SPB),
\end{align*}
(with integration contours as in Definition \ref{def:orthogonality_contours} and Fig.~\ref{fig:contours} below), are inverses to each other on appropriately defined functional spaces. 
In the homogeneous case, those were the main results of Borodin--Corwin--Petrov--Sasamoto \cite{BCPS2014}, see also \cite{BorodinCorwinPetrovSasamoto2013}. In the case of the inhomogeneous $q$-Boson model which is slightly lower in the hierarchy than the model we consider here, a very recent work of Wang--Waugh \cite{Wang2015inhomqTASEP} shows that for appropriate limits of the above transforms one has $\mathscr{J}\circ\mathscr{F}=\mathrm{Id}$. 

The two identities $\mathscr{F}\circ\mathscr{J}=\mathrm{Id}$ and $\mathscr{J}\circ\mathscr{F}=\mathrm{Id}$ can be equivalently stated in terms of orthogonality relations for the basis $\{\F_\mu\}$. The latter relation corresponds to the \emph{spatial} orthogonality, where the product of two $\F$'s integrated over their common arguments results in a delta-function in their indices. The former relation corresponds to the \emph{spectral} orthogonality, where the product of two $\F$'s summed over their common index results in a delta-function in the arguments. 
Our proofs for both types of orthogonality use ideas from existing results in the homogeneous case, that of \cite{BCPS2014} for the spatial one and of \cite{Borodin2014vertex} for the spectral one (this is where a Cauchy identity is heavily used).

The orthogonality relations play an important role in simplifying the form of the observables extracted from comparisons of Cauchy identities, which eventually results in concise integral formulas for the observables in \eqref{intro-multi_moments} and for the $q$-correlation functions.

\subsection{Organization of the paper} In \S\ref{sec:vertex_weights} we define the higher spin six vertex model in the language which is used throughout the paper. 
The Yang--Baxter integrability of our model is discussed in \S\ref{sec:yang_baxter_relation}. 
In \S\ref{sec:symmetric_rational_functions} we take the infinite volume limit of the Yang--Baxter equation, introduce functions $\F$ and $\G$, and derive Cauchy identities and symmetrization type formulas for them. 
Fusion --- a procedure of collapsing several horizontal rows with suitable spectral parameters onto a single one
--- is discussed in \S\ref{sec:stochastic_weights_and_fusion}. 
In \S\ref{sec:markov_kernels_and_stochastic_dynamics} we use skew Cauchy identities to define various Markov dynamics for our model, and also show how known integrable (1+1)d KPZ models can be obtained from those.  
In \S\ref{sec:orthogonality_relations} we prove the Plancherel isomorphisms (equivalently, two types of orthogonality relations for the $\F_\mu$'s).
In \S\ref{sec:observables_of_interacting_particle_systems} we derive integral representations for the $q$-correlation functions. 
In \S\ref{sec:_q_moments_of_the_height_function_of_interacting_particle_systems} we prove our main result --- the integral formula \eqref{intro-multi_moments} for the $q$-moments of the height function. 
The final \S\ref{sec:degenerations_of_moment_formulas} demonstrates how our main result degenerates to various similar known results for the models which are hierarchically lower: stochastic six vertex model, ASEP, various $q$-TASEPs, and associated zero range processes.

\subsection*{Acknowledgments} 
We are grateful to Ivan Corwin and Vadim Gorin for valuable discussions. 
We also gratefully acknowledge hospitality and support from the
Galileo Galilei Institute for Theoretical Physics during the program
``Statistical Mechanics, Integrability and Combinatorics'' where parts of this work 
were completed.
A.~B. was partially supported by the NSF grant DMS-1056390.


\section{Vertex weights} 
\label{sec:vertex_weights}

\begin{figure}[htpb]
		\begin{tikzpicture}
			[scale=.6,thick]
			\foreach \xxx in {0,4,5,1,2,3,6,7}
			{
				\draw[dotted] (\xxx,5.5)--++(0,-5);
			}
			\foreach \xxx in {1,2,3,4,5}
			{
				\draw[dotted] (-.5,\xxx)--++(8,0);
			}
			\def\dist{.1}
			\draw[line width=1.7pt] (-1,1)--++(1,0);
			\draw[line width=1.7pt] (-1,2)--++(1,0);
			\draw[->, line width=1.7pt] 
			(0,2)--++(1,0)--++(1-\dist,0)--++(\dist,\dist)--++(0,1-\dist)
			--++(1,0)--++(1-\dist,0)--++(\dist,\dist)--++(0,1-\dist)
			--++(0,1)--++(0,1);
			\draw[->, line width=1.7pt] (0,1)--++(1,0)--++(1-2*\dist,0)--++(\dist,\dist)--++(0,1-3*\dist)--++(3*\dist,3*\dist)
			--++(1,0)--++(1-3*\dist,0)--++(\dist,\dist)--++(0,1-3*\dist)--++(\dist,\dist)--++(1-\dist,0)--++(0,1)--++(1-\dist,0)--++(\dist,\dist)
			--++(0,1-\dist)--++(0,1);
			\draw[->, line width=1.7pt] 
			(2,0)--++(0,1-\dist)--++(\dist,2*\dist)--++(0,1-3*\dist)--++(\dist,\dist)--++(1,0)--++(1-3*\dist,0)
			--++(2*\dist,\dist)
			--++(1,0)--++(1-2*\dist,0)--++(\dist,\dist)--++(0,1-\dist)--++(0,1-\dist)--++(\dist,\dist)--++(1-3*\dist,0)
			--++(\dist,\dist)--++(0,2-\dist);
			\draw[->, line width=1.7pt] (5,0)--++(0,1)--++(1,0)--++(0,1-\dist)--++(\dist,\dist)
			--++(1-\dist,0)--++(0,1)--++(0,1-1*\dist)--++(\dist,2*\dist)--++(0,2-\dist);
			\draw[->, line width=1.7pt] (2,0)--++(0,.5);
			\draw[->, line width=1.7pt] (5,0)--++(0,.5);
			\draw[->, line width=1.7pt] (-1,1)--++(.5,0);
			\draw[->, line width=1.7pt] (-1,2)--++(.5,0);
			\foreach \ppt in {(0,1),(1,1),(5,1),(6,1),
			(0,2),(1,2),(2,3),(3,3),(4,4),(4,5),(5,2),(5,3),(5,4),
			(6,3),(6,5),(7,2),(7,3)}
			{
				\draw \ppt circle(5pt);
			}
			\foreach \ppt in {(2,1),(2,2),(3,2),(4,2),(4,3),(6,2),(6,4),(7,4),(7,5)}
			{
				\draw \ppt circle(10pt);
			}
		\end{tikzpicture}
	\caption{An example of a collection of up-right paths in a region in $\Z^{2}$.
	Note that several paths are allowed to pass along the same edge.
	At each vertex the total number of incoming arrows ($=$~coming from the left or from below)
	must be equal to the total number of outgoing ones ($=$~pointing to the right or upwards), cf. Fig.~\ref{fig:vertex}.
	The circles indicate nonempty vertices
	which contribute
	to the weight of the path collection.}
	\label{fig:paths_example}
\end{figure}

\subsection{Higher spin six vertex model} 
\label{sub:higher_spin_vertex_model}

The higher spin six vertex model can be viewed as a way of assigning weights to collections of up-right paths
in a finite region of $\Z^{2}$, subject to certain boundary conditions. 
An example of such a collection of paths is given in Fig.~\ref{fig:paths_example}.
The \emph{weight} of a path collection is equal to the product of weights of all the vertices that
belong to the paths.
We will always assume that the weight of the empty vertex
\emptyvertex{.6}
is equal to $1$. Thus, the weight of a path collection
can be equivalently defined
as
the product of weights of all vertices in $\Z^{2}$.

Note that the weight of a collection of paths is in general
not equal to the product of weights of individual paths (defined in an obvious way). But if
the paths in a collection have no 
vertices in common, then the weight of this collection
will in fact be equal to the product of weights of individual paths.


\subsection{Vertex weights} 
\label{sub:vertex_weights}

We choose the weights of vertices in a special way. 
First, we postulate that the number of incoming arrows $i_1+j_1$ into any vertex
must be the same as the number of outgoing arrows $i_2+j_2$, see Fig.~\ref{fig:vertex}. 
\begin{figure}[htbp]
	\begin{tikzpicture}
		[scale=1.5, ultra thick]
		\def\d{.1}
		\draw[->] (-1,3*\d)--++(1,0)--++(0,1);
		\draw[->] (-1,2*\d)--++(1,0)--++(\d,\d)--++(0,1);
		\draw[->] (-1,\d)--++(1,0)--++(2*\d,2*\d)--++(0,1);
		\draw[->] (-1,0)--++(1,0)--++(3*\d,3*\d)--++(0,1);
		\draw[->] (-1,-\d)--++(1,0)--++(4*\d,4*\d)--++(0,1);
		\draw[->] (-1,-2*\d)--++(1,0)--++(4*\d,4*\d)--++(1,0);
		\draw[->] (-1,-3*\d)--++(1,0)--++(4*\d,4*\d)--++(1,0);
		\draw[->] (1.5*\d,-1-2.5*\d)--++(0,1)--++(2.5*\d,2.5*\d)--++(1,0);
		\draw[->] (2.5*\d,-1-2.5*\d)--++(0,1)--++(1.5*\d,1.5*\d)--++(1,0);
		\node at (.2,-1.55) {$i_1=2$};
		\node at (-2,0) {$j_1=7$};
		\node at (.2,1.65) {$i_2=5$};
		\node at (2.2,0) {$j_2=4$};
		\node[rotate=-45] at (-.75,-1) {input};
		\node[rotate=-45] at (1.1,1.1) {output};
		\draw[->] (-1,3*\d)--++(.2,0);
		\draw[->] (-1,2*\d)--++(.2,0);
		\draw[->] (-1,1*\d)--++(.2,0);
		\draw[->] (-1,0*\d)--++(.2,0);
		\draw[->] (-1,-1*\d)--++(.2,0);
		\draw[->] (-1,-2*\d)--++(.2,0);
		\draw[->] (-1,-3*\d)--++(.2,0);
		\draw[->] (1.5*\d,-1-2.5*\d)--++(0,.2);
		\draw[->] (2.5*\d,-1-2.5*\d)--++(0,.2);
		\draw[very thick] (1.75*\d,0) circle (15pt);
	\end{tikzpicture}
	\caption{Incoming and outgoing vertical and horizontal arrows
	at a vertex
	which we will denote by
	$(i_1,j_1;i_2,j_2)=(2,7;5,4)$.}
	\label{fig:vertex}
\end{figure}
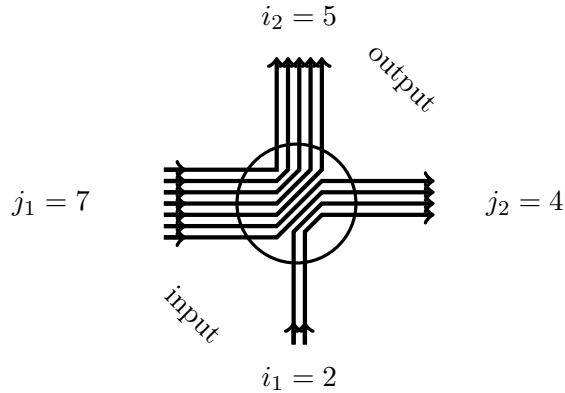
\begin{remark}
	This \emph{arrow preservation} condition obviously fails at the boundaries, so one should either
	fix boundary conditions
	in some way, or specify weights on the boundary independently.
\end{remark}
The vertex weights will depend on two (generally speaking, complex)
parameters that we denote by $q$ and $\SP$,
and 
on an additional
\emph{spectral parameter}
$u\in\C$.
All these parameters 
are assumed to be \emph{generic}\footnote{That is, vanishing of certain algebraic expressions 
in the parameters
may make some of our statements meaningless.
We will not focus on these special cases.}.
The vertex weights are explicitly given by (see also Fig.~\ref{fig:vertex_weights})
\begin{align}\label{weights}
	\begin{array}{rclrcl}
		w_{u,\SP}(g,0;g,0)&:=&\dfrac{1-\SP q^{g}u}{1-\SP u},&\qquad \qquad
		\rule{0pt}{22pt}
		w_{u,\SP}(g+1,0;g,1)&:=&\dfrac{(1-\SP ^{2}q^{g})u}{1-\SP u},\\
		\rule{0pt}{22pt}
		w_{u,\SP}(g,1;g,1)&:=&\dfrac{u-\SP q^{g}}{1-\SP u},&\qquad \qquad
		\rule{0pt}{22pt}
		w_{u,\SP}(g,1;g+1,0)&:=&\dfrac{1-q^{g+1}}{1-\SP u},
	\end{array}
\end{align}
where $g$ is any nonnegative integer.
All other weights are assumed to be zero.
Note that 
the weight of the empty vertex
\emptyvertex{.6} (that is, $(0,0;0,0)$)
is indeed equal to $1$.
Throughout the text, 
the parameter $q$ is assumed fixed, and 
$u$ will be regarded as an indeterminate. The dependence on $\SP$ will also be reflected in the notation.

Observe that the weights $w_{u,\SP}$ \eqref{weights} are nonzero
only for $j_1,j_2\in\{0,1\}$, that is, the multiplicities of 
horizontal edges are bounded by $1$. This restriction will be 
removed later (in \S\ref{sec:stochastic_weights_and_fusion}).
\begin{figure}[htbp]
	\begin{tabular}{c|c|c|c|c}
	&\vertexoo{1}
	&\vertexol{1}
	&\vertexll{1}
	&\vertexlo{1}
	\\\hline\rule{0pt}{20pt}
	$w_{u,\SP}$&
	$\dfrac{1-\SP q^{g}u}{1-\SP u}$&
	$\dfrac{(1-\SP ^{2}q^{g-1})u}{1-\SP u}$&
	$\dfrac{u-\SP q^{g}}{1-\SP u}$&
	$\dfrac{1-q^{g+1}}{1-\SP u}$
	\phantom{\Bigg|}\\
	\end{tabular}
	\caption{Vertex weights
	\eqref{weights}.
	Here $g\in\Z_{\ge0}$ and
	by agreement, $w_{u,\SP}(0,0;-1,1)=0$.}
	\label{fig:vertex_weights}
\end{figure}


\subsection{Motivation} 
\label{sub:motivation}

Weights defined in \eqref{weights} are closely related to matrix
elements of the higher spin R-matrix 
associated with $U_q(\widehat{\mathfrak{sl}_2})$
(e.g., see \cite{Mangazeev2014} and also \cite{baxter2007exactly}, \cite{reshetikhin2010lectures} for a general introduction).
Because of this, they satisfy
a version of the Yang--Baxter equation
which we discuss in \S \ref{sec:yang_baxter_relation} below.
The exact connection of weights \eqref{weights} with R-matrices is written down
in \cite[\S2]{Borodin2014vertex}, and here we follow the notation of that paper.

For the weights \eqref{weights}, the R-matrix in question
corresponds to one of the highest weight representations (the ``vertical'' one) being 
a generic Verma module (associated with the parameter~$\SP$),
while the other representation (``horizontal'')
is two-dimensional. This choice of the ``horizontal''
representation dictates the restriction on the horizontal multiplicities 
$j_1,j_2\in\{0,1\}$. Vertex weights associated with other ``horizontal'' 
representations 
(finite-dimensional of dimension $J+1$, or generic Verma modules)
are discussed in \S \ref{sec:stochastic_weights_and_fusion}
below.

If we set 
$\SP^{2}=q^{-I}$ with $I$ a positive integer,
then matrix elements of the 
generic Verma module turn into those of the
$(I+1)$-dimensional
highest weight representation (of weight $I$), and thus the 
multiplicities of vertical edges will be bounded by $I$.
In particular, setting $I=1$ leads to the well-known six vertex
model (we discuss it in \S \ref{sub:asep_degeneration}). 
Throughout most of the text we will assume,
however, that the parameter $\SP$ is generic, and so 
there is no restriction on the vertical multiplicity.


\subsection{Conjugated weights and stochastic weights} 
\label{sub:S2_stochastic_weights}

Here we write down two related versions 
of the vertex weights which will be later useful for probabilistic
applications.

Throughout the text we will employ the $q$-Pochhammer symbols
\begin{align*}
	(z;q)_n := \begin{cases} \prod_{k=0}^{n-1} (1-zq^k), & n>0,\\ 1,& n=0,\\ \prod_{k=0}^{-n-1}(1-zq^{n+k})^{-1}, &n<0.\end{cases}
\end{align*}
If $|q|<1$ and $n=+\infty$, then the $q$-Pochhammer 
symbol $(z;q)_\infty$ also makes sense. We will also use the $q$-binomial coefficients
\begin{align*}
	\binom{n}{k}_{q}:=\frac{(q;q)_{n}}{(q;q)_{k}(q;q)_{n-k}}.
\end{align*}

Define the following \emph{conjugated vertex weights}:
\begin{align}
	w^{\conj}_{u,\SP}(i_1,j_1;i_2,j_2):=
	\frac{(\SP ^{2};q)_{i_2}}{(q;q)_{i_2}}
	\frac{(q;q)_{i_1}}{(\SP ^{2};q)_{i_1}}\, w_{u,\SP}(i_1,j_1;i_2,j_2).
	\label{weights_conj}
\end{align}
Also define
\begin{align}\label{vertex_weights_stoch}
	\Lmatr_{u,\SP}(i_1,j_1;i_2,j_2):= 
	(-\SP)^{j_2}w^{\conj}_{u,\SP}(i_1,j_1;i_2,j_2).
\end{align}
The above quantities are given in Fig.~\ref{fig:vertex_weights_conj_stoch}.
Note that for any $i_1\in\Z_{\ge0}$, $j_1\in\{0,1\}$ we have
\begin{align}\label{L_sum_to_one}
	\sum_{i_2,j_2\in\Z_{\ge0}\colon i_2+j_2=i_1+j_1}
	\Lmatr_{u,\SP}(i_1,j_1;i_2,j_2)=1.
\end{align}
Therefore, if the $\Lmatr_{u,\SP}$'s are nonnegative, 
they can be interpreted
as defining a \emph{probability distribution} 
on all possible output configurations
$\{(i_2,j_2)\colon i_2+j_2=i_1+j_1\}$
given the input configuration $(i_1,j_1)$, cf.~Fig.~\ref{fig:vertex}.
We will discuss values of parameters leading to nonnegative
$\Lmatr_{u,\SP}$'s in \S \ref{sub:stochastic_weights} below.

A motivation for introducing the conjugated weights $w_{u,\SP}^{\conj}$ can be found in 
\S \ref{sub:semi_infinite_operator_D} below.

\begin{figure}[htbp]
	\begin{tabular}{c|c|c|c|c}
	&\vertexoo{1}
	&\vertexol{1}
	&\vertexll{1}
	&\vertexlo{1}
	\\
	\hline\rule{0pt}{20pt}
	$w_{u,\SP}^{\conj}$&
	$\dfrac{1-\SP q^{g}u}{1-\SP u}$&
	$\dfrac{(1-q^{g})u}{1-\SP u}$&
	$\dfrac{u-\SP q^{g}}{1-\SP u}$&
	$\dfrac{1-\SP ^{2}q^{g}}{1-\SP u}$
	\phantom{\Bigg|}\\\hline\rule{0pt}{20pt}
	$\Lmatr_{u,\SP}$&
	$\dfrac{1-\SP q^{g}u}{1-\SP u}$&
	$\dfrac{-\SP u+\SP q^{g}u}{1-\SP u}$&
	$\dfrac{-\SP u+\SP ^{2}q^{g}}{1-\SP u}$&
	$\dfrac{1-\SP ^{2}q^{g}}{1-\SP u}$
	\phantom{\Bigg|}
	\end{tabular}
	\caption{Vertex weights
	\eqref{weights_conj} and \eqref{vertex_weights_stoch}.
	Note that they automatically vanish at the forbidden
	configuration $(0,0;-1,1)$.}
	\label{fig:vertex_weights_conj_stoch}
\end{figure}



\section{The Yang--Baxter equation} 
\label{sec:yang_baxter_relation}

\subsection{The Yang--Baxter equation in coordinate language} 
\label{sub:yang_baxter_relation_in_coordinate_form}

The Yang--Baxter equation deals with 
vertex weights at two vertices connected by a vertical edge,
with spectral parameters $\ybspec_1,\ybspec_2$.
Define the two-vertex weights by
\begin{align}\label{two-vertex}
	w_{\ybspec_1,\ybspec_2;\,\SP}^{(m,n)}(k_1,k_2;k_1',k_2'):=\sum_{l\ge 0} w_{\ybspec_1,\SP}(m,k_1;l,k_1')w_{\ybspec_2,\SP}(l,k_2;n,k_2'),
	\qquad
	k_1,k_2,k_1',k_2'\in\{0,1\}.
\end{align}
The expression \eqref{two-vertex}
is the weight of the two-vertex configuration
as in Fig.~\ref{fig:YB}, left, 
with numbers of incoming and outgoing arrows
$m,n,k_{1,2},k'_{1,2}$ fixed.
The number of arrows $l\ge0$ along the inside edge is arbitrary,
but due to the arrow preservation,
no more than one value of $l$ contributes to the sum.

Also define
\begin{align}\label{two-vertex_tilde}
	\widetilde w_{\ybspec_1,\ybspec_2;\,\SP}^{(m,n)}(k_1,k_2;k_1',k_2'):=
	w_{\ybspec_1,\ybspec_2;\,\SP}^{(m,n)}(k_2,k_1;k_2',k_1');
\end{align}
this is the weight of the configuration as in 
Fig.~\ref{fig:YB}, right.
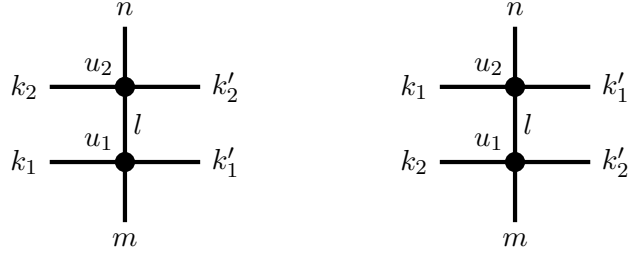
\begin{figure}[htbp]
	\begin{tabular}{cc}
	\begin{tikzpicture}
		[scale=1,ultra thick]
		\draw (0,0)--++(1,0) node [right] {$k_1'$};
		\draw (0,0)--++(-1,0) node [left] {$k_1$};
		\draw (0,1)--++(1,0) node [right] {$k_2'$};
		\draw (0,1)--++(-1,0) node [left] {$k_2$};
		\draw (0,0)--++(0,-.8) node[below] {$m$};
		\draw (0,1)--++(0,.8) node[above] {$n$};
		\draw (0,0)--++(0,1) node[midway, xshift=5pt] {$l$};
		\draw[fill] (0,0) circle (3pt) node [above left] {$\ybspec_1$};
		\draw[fill] (0,1) circle (3pt) node [above left] {$\ybspec_2$};
	\end{tikzpicture}
	&\hspace{40pt}
	\begin{tikzpicture}
		[scale=1,ultra thick]
		\draw (0,0)--++(1,0) node [right] {$k_2'$};
		\draw (0,0)--++(-1,0) node [left] {$k_2$};
		\draw (0,1)--++(1,0) node [right] {$k_1'$};
		\draw (0,1)--++(-1,0) node [left] {$k_1$};
		\draw (0,0)--++(0,-.8) node[below] {$m$};
		\draw (0,1)--++(0,.8) node[above] {$n$};
		\draw (0,0)--++(0,1) node[midway, xshift=5pt] {$l$};
		\draw[fill] (0,0) circle (3pt) node [above left] {$\ybspec_1$};
		\draw[fill] (0,1) circle (3pt) node [above left] {$\ybspec_2$};
	\end{tikzpicture}
	\end{tabular}
	\caption{Two-vertex configurations
	corresponding to \eqref{two-vertex} and \eqref{two-vertex_tilde}, 
	respectively.}
	\label{fig:YB}
\end{figure}

Let us organize the weights \eqref{two-vertex} into $4\times 4$ matrices
\begin{align*}
w_{\ybspec_1,\ybspec_2;\,\SP}^{(m,n)}=\begin{bmatrix} 
w_{\ybspec_1,\ybspec_2;\,\SP}^{(m,n)}(0,0;0,0)&w_{\ybspec_1,\ybspec_2;\,\SP}^{(m,n)}(0,0;0,1)&w_{\ybspec_1,\ybspec_2;\,\SP}^{(m,n)}(0,0;1,0)&w_{\ybspec_1,\ybspec_2;\,\SP}^{(m,n)}(0,0;1,1)\\
w_{\ybspec_1,\ybspec_2;\,\SP}^{(m,n)}(0,1;0,0)&w_{\ybspec_1,\ybspec_2;\,\SP}^{(m,n)}(0,1;0,1)&w_{\ybspec_1,\ybspec_2;\,\SP}^{(m,n)}(0,1;1,0)&w_{\ybspec_1,\ybspec_2;\,\SP}^{(m,n)}(0,1;1,1)\\
w_{\ybspec_1,\ybspec_2;\,\SP}^{(m,n)}(1,0;0,0)&w_{\ybspec_1,\ybspec_2;\,\SP}^{(m,n)}(1,0;0,1)&w_{\ybspec_1,\ybspec_2;\,\SP}^{(m,n)}(1,0;1,0)&w_{\ybspec_1,\ybspec_2;\,\SP}^{(m,n)}(1,0;1,1)\\
w_{\ybspec_1,\ybspec_2;\,\SP}^{(m,n)}(1,1;0,0)&w_{\ybspec_1,\ybspec_2;\,\SP}^{(m,n)}(1,1;0,1)&w_{\ybspec_1,\ybspec_2;\,\SP}^{(m,n)}(1,1;1,0)&w_{\ybspec_1,\ybspec_2;\,\SP}^{(m,n)}(1,1;1,1)
\end{bmatrix},
\end{align*}
and similarly for $\widetilde w_{\ybspec_1,\ybspec_2;\,\SP}^{(m,n)}$.
\begin{proposition}[The Yang--Baxter equation]\label{prop:YB}
	We have
	\begin{align}\label{YB_main_relation}
		\widetilde w_{\ybspec_2,\ybspec_1;\,\SP}^{(m,n)}=
		Xw_{\ybspec_1,\ybspec_2;\,\SP}^{(m,n)}
		X^{-1},
	\end{align}
	where the matrix $X$ depending on 
	$\ybspec_1$ and $\ybspec_2$
	is given by
	\begin{align}
		X=\begin{bmatrix}  
			\ybspec_1-q\ybspec_2&0&0&0\\ \\
			0&
			q(\ybspec_1-\ybspec_2)
			& 
			(1-q)\ybspec_1
			&0\\ \\
			0&
			(1-q)\ybspec_2
			&
			\ybspec_1-\ybspec_2
			&0\\ \\
			0&0&0&\ybspec_1-q\ybspec_2
		\end{bmatrix}.
		\label{matrix_X}
	\end{align}
\end{proposition}
Note that $X$ is independent of $m$ and $n$, and it is this fact that 
makes the weights $w_{u,\SP}$ \eqref{weights} very special.
Note also that $X$ matters only up to an overall factor
(which can depend on $\ybspec_1$ and $\ybspec_2$).
Therefore, $X$ in fact depends only on the ratio
of spectral parameters $\ybspec_1$ and $\ybspec_2$.
Finally, observe that $X$ is independent of the parameter $\SP$.
\begin{proof}
	This equation can be checked directly.
	Alternatively, as shown in 
	\cite[Prop.\;2.5]{Borodin2014vertex},
	it can be derived from the 
	Yang--Baxter equation for the R-matrices.
\end{proof}
The conjugated and the stochastic weights 
(\eqref{weights_conj} and \eqref{vertex_weights_stoch}, respectively),
also satisfy certain versions of the Yang--Baxter equation,
see, e.g. \cite[Appendix C]{CorwinPetrov2015}.

\begin{remark}
	The matrix $X$ \eqref{matrix_X}
	itself can be viewed as a 
	version of the R-matrix corresponding to 
	both representations being two-dimensional 
	(details may be found in
	the proof of 
	Proposition 2.5 in 
	\cite{Borodin2014vertex}). 
\end{remark}


\subsection{The Yang--Baxter equation in operator language} 
\label{sub:yang_baxter_relation_in_operator_language}

Before drawing corollaries from the Yang--Baxter equation, let us restate it in a different language which is sometimes more
convenient. 

Consider a vector space $V=\Span\{\bv_i\colon i=0,1,2,\ldots\}$,
and linear operators 
$\AY(u)$,
$\BY(u)$,
$\CY(u)$,
$\DY(u)$ on this space which depend on a spectral parameter $u\in\C$ 
and act in this basis as follows (cf. \S \ref{sub:vertex_weights}):
\begin{align}\label{ABCD_elementary_operators}
	\begin{array}{ll}
		\AY(u)\,\bv_g
		:=w_{u,\SP}\Big(\raisebox{-15pt}{\mbox{\vertexoo{.5}}}\Big)\,\bv_g
		=\dfrac{1-\SP q^{g}u}{1-\SP u}\,\bv_g,&
		\DY(u)\,\bv_g:=
		w_{u,\SP}\Big(\raisebox{-15pt}{\mbox{\vertexll{.5}}}\Big)\,\bv_g
		=\dfrac{u-\SP q^{g}}{1-\SP u}\,\bv_g,\\
		\BY(u)\,\bv_g:=
		w_{u,\SP}\Big(\raisebox{-15pt}{\mbox{\vertexlo{.5}}}\Big)\,\bv_{g+1}
		=\dfrac{1-q^{g+1}}{1-\SP u}\,\bv_{g+1},&
		\CY(u)\,\bv_g:=
		w_{u,\SP}\Big(\raisebox{-15pt}{\mbox{\vertexol{.5}}}\Big)\,\bv_{g-1}
		=\dfrac{(1-\SP ^{2}q^{g-1})u}{1-\SP u}\,\bv_{g-1},
	\end{array}
\end{align}
where $g\in\Z_{\ge0}$, and, by agreement, $w_{u,\SP}(0,0;-1,1)=0$.
Note that in every vertex $(i_1,j_1;i_2,j_2)$
above, 
$i_1$ corresponds to the index of the vector that 
the operator is applied to, and $i_2$ corresponds to the index of the image vector.
The four possibilities for $j_1,j_2\in\{0,1\}$  
correspond to the four operators.

These four operators are conveniently united into a $2\times 2$
matrix with operator entries
\begin{align*}
	\TY(u):=\begin{bmatrix}
		\AY(u)&\BY(u)\\
		\CY(u)&\DY(u)\\
	\end{bmatrix},
\end{align*}
known as the \emph{monodromy matrix}.
It can be viewed as an operator 
$\TY(u)\colon \C^{2}\otimes V\to\C^{2}\otimes V$.
The space $\C^{2}$ is often called the \emph{auxiliary space},
and $V$ is referred to as the \emph{physical}, or \emph{quantum space}.

In terms of the monodromy matrices, the Yang--Baxter equation (Proposition \ref{prop:YB})
takes the form
\begin{align}\label{YB_operator_form}
	\big(
	\TY(\ybspec_1)\otimes\TY(\ybspec_2)
	\big)=Y
	\big(
	\TY(\ybspec_2)\otimes\TY(\ybspec_1)
	\big)Y^{-1},
\end{align}
where
\begin{align*}
	Y:=(X^{-1})^{\textnormal{transpose}}
	=
	\frac{1}{(\ybspec_1-q\ybspec_2)(\ybspec_2-q \ybspec_1)}
	\begin{bmatrix}
		 \ybspec_2-q \ybspec_1 & 0 & 0 & 0 \\
		 0 & \ybspec_2-\ybspec_1 & (1-q)\ybspec_2 & 0 \\
		 0 & (1-q)\ybspec_1 & q (\ybspec_2-\ybspec_1) & 0 \\
		 0 & 0 & 0 & \ybspec_2-q \ybspec_1
	\end{bmatrix}
	,
\end{align*}
with $X$ given by \eqref{matrix_X}.

The tensor product in both sides of \eqref{YB_operator_form} is taken with respect to the
two different auxiliary spaces corresponding to $(k_1,k_2)$ in Fig.~\ref{fig:YB}. 
Namely, we have
\begin{align}
	\TY(\ybspec_1)\otimes\TY(\ybspec_2)=\begin{bmatrix}
		\begin{matrix}
			\AY(\ybspec_1)\AY(\ybspec_2)&\AY(\ybspec_1)\BY(\ybspec_2)\\
			\AY(\ybspec_1)\CY(\ybspec_2)&\AY(\ybspec_1)\DY(\ybspec_2)
		\end{matrix}
		&
		\begin{matrix}
			\BY(\ybspec_1)\AY(\ybspec_2)&\BY(\ybspec_1)\BY(\ybspec_2)\\
			\BY(\ybspec_1)\CY(\ybspec_2)&\BY(\ybspec_1)\DY(\ybspec_2)
		\end{matrix}
		\\\\
		\begin{matrix}
			\CY(\ybspec_1)\AY(\ybspec_2)&\CY(\ybspec_1)\BY(\ybspec_2)\\
			\CY(\ybspec_1)\CY(\ybspec_2)&\CY(\ybspec_1)\DY(\ybspec_2)
		\end{matrix}
		&
		\begin{matrix}
			\DY(\ybspec_1)\AY(\ybspec_2)&\DY(\ybspec_1)\BY(\ybspec_2)\\
			\DY(\ybspec_1)\CY(\ybspec_2)&\DY(\ybspec_1)\DY(\ybspec_2)
		\end{matrix}
	\end{bmatrix},
	\label{T1T2}
\end{align}
and similarly,
\begin{align}
	\label{T2T1}
	\TY(\ybspec_2)\otimes\TY(\ybspec_1)=\begin{bmatrix}
		\begin{matrix}
			\AY(\ybspec_2)\AY(\ybspec_1)&\BY(\ybspec_2)\AY(\ybspec_1)\\
			\CY(\ybspec_2)\AY(\ybspec_1)&\DY(\ybspec_2)\AY(\ybspec_1)
		\end{matrix}
		&
		\begin{matrix}
			\AY(\ybspec_2)\BY(\ybspec_1)&\BY(\ybspec_2)\BY(\ybspec_1)\\
			\CY(\ybspec_2)\BY(\ybspec_1)&\DY(\ybspec_2)\BY(\ybspec_1)
		\end{matrix}
		\\\\
		\begin{matrix}
			\AY(\ybspec_2)\CY(\ybspec_1)&\BY(\ybspec_2)\CY(\ybspec_1)\\
			\CY(\ybspec_2)\CY(\ybspec_1)&\DY(\ybspec_2)\CY(\ybspec_1)
		\end{matrix}
		&
		\begin{matrix}
			\AY(\ybspec_2)\DY(\ybspec_1)&\BY(\ybspec_2)\DY(\ybspec_1)\\
			\CY(\ybspec_2)\DY(\ybspec_1)&\DY(\ybspec_2)\DY(\ybspec_1)
		\end{matrix}
	\end{bmatrix}.
\end{align}
See also Fig.~\ref{fig:YB_operator} for an example.
\begin{figure}[htbp]
	\begin{tabular}{cc}
	\begin{tikzpicture}
		[scale=1,ultra thick]
		\draw (0,0)--++(1,0) node [right] {$0$};
		\draw (0,0)--++(-1,0) node [left] {$1$};
		\draw (0,1)--++(1,0) node [right] {$0$};
		\draw (0,1)--++(-1,0) node [left] {$0$};
		\draw (0,0)--++(0,-.8) node[below] {$g$};
		\draw (0,1)--++(0,.8) node[above] {$g+1$};
		\draw (0,0)--++(0,1) node[midway, xshift=18pt] {$g+1$};
		\draw[fill] (0,0) circle (3pt) node [above left] {$\ybspec_2$};
		\draw[fill] (0,1) circle (3pt) node [above left] {$\ybspec_1$};
	\end{tikzpicture}
	&\hspace{40pt}
	\begin{tikzpicture}
		[scale=1,ultra thick]
		\draw (0,0)--++(1,0) node [right] {$0$};
		\draw (0,0)--++(-1,0) node [left] {$1$};
		\draw (0,1)--++(1,0) node [right] {$0$};
		\draw (0,1)--++(-1,0) node [left] {$0$};
		\draw (0,0)--++(0,-.8) node[below] {$g$};
		\draw (0,1)--++(0,.8) node[above] {$g+1$};
		\draw (0,0)--++(0,1) node[midway, xshift=18pt] {$g+1$};
		\draw[fill] (0,0) circle (3pt) node [above left] {$\ybspec_1$};
		\draw[fill] (0,1) circle (3pt) node [above left] {$\ybspec_2$};
	\end{tikzpicture}
	\end{tabular}
	\caption{The operator $\AY(\ybspec_1)\BY(\ybspec_2)$ applied to the basis vector $\bv_g$ corresponds to 
	the configuration on the left, and 
	$\AY(\ybspec_1)\BY(\ybspec_2)\,\bv_{g}=\widetilde w^{(g,g+1)}_{\ybspec_2,\ybspec_1;\,\SP}(0,1;0,0)\,\bv_{g+1}$.
	Similarly, we have
	$\AY(\ybspec_2)\BY(\ybspec_1)\,\bv_{g}=w^{(g,g+1)}_{\ybspec_1,\ybspec_2;\,\SP}(1,0;0,0)\,\bv_{g+1}$,
	which corresponds to the configuration on the right.}
	\label{fig:YB_operator}
\end{figure}
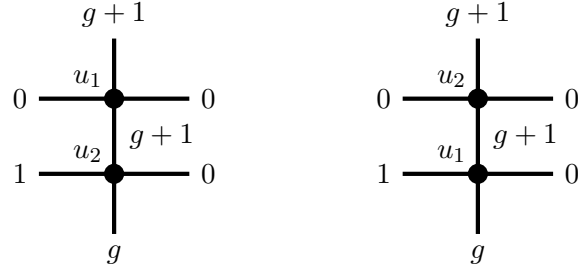

The Yang--Baxter equation \eqref{YB_operator_form} in the matrix form allows to extract individual commutation relations 
between the operators $\AY$, $\BY$, $\CY$, and $\DY$. 
Let us write down relations which will be useful in what follows.
Comparing matrix elements $(1,1)$ on both sides of \eqref{YB_operator_form} implies
\begin{align}\label{A_commute}
	\AY(\ybspec_1)\AY(\ybspec_2)=\AY(\ybspec_2)\AY(\ybspec_1).
\end{align}
Similarly, looking at matrix elements $(1,4)$ and $(4,4)$ gives rise to 
\begin{align}\label{B_commute}
	\BY(\ybspec_1)\BY(\ybspec_2)& =\BY(\ybspec_2)\BY(\ybspec_1),
	\\
	\label{D_commute}
	\DY(\ybspec_1)\DY(\ybspec_2)& =\DY(\ybspec_2)\DY(\ybspec_1),
\end{align}
respectively. Looking at matrix elements $(2,4)$ leads to
\begin{align}
	\label{YB_relationBD}
	\BY(\ybspec_1)\DY(\ybspec_2)=
	\frac{\ybspec_1-\ybspec_2}{q\ybspec_1-\ybspec_2}\DY(\ybspec_2)\BY(\ybspec_1)+\frac{(1-q)\ybspec_2}{\ybspec_2-q\ybspec_1}\BY(\ybspec_2)\DY(\ybspec_1)
	.
\end{align}
Finally, considering matrix elements 
$(2,1)$ and
$(1,3)$ implies, respectively,
\begin{align}
	\label{YB_relationAC}
	\AY(\ybspec_1)\CY(\ybspec_2)&=
	\frac{u_1-u_2}{qu_1-u_2}
	\CY(\ybspec_2)\AY(\ybspec_1)
	+
	\frac{(1-q)u_2}{u_2-qu_1}
	\AY(\ybspec_2)\CY(\ybspec_1)
	,
	\\
	\label{YB_relationAB}
	\BY(\ybspec_1)\AY(\ybspec_2)&=
	\frac{\ybspec_1-\ybspec_2}{\ybspec_1-q\ybspec_2}\AY(\ybspec_2)\BY(\ybspec_1)+\frac{(1-q)\ybspec_2}{\ybspec_1-q\ybspec_2}\BY(\ybspec_2)\AY(\ybspec_1)
	.
\end{align}

\begin{remark}
	All operators entering the matrices \eqref{T1T2} and \eqref{T2T1}
	and the relations
	\eqref{A_commute}--\eqref{YB_relationAB}
	correspond to one and the same physical space, and 
	contain the same parameter $\SP$.
\end{remark}


\subsection{Attaching vertical columns} 
\label{sub:attaching_vertical_columns}

A very important 
property of the Yang--Baxter equation is that it survives 
when one attaches several vertical columns on the side, 
with the requirement that the
conjugating matrix
is the same across the vertical columns.
Let us consider the situation of two vertical columns (cf.~Fig.~\ref{fig:YB_attach}),
with different $\SP$--parameters $\SP_1$ and $\SP_2$, and
spectral parameters $\ip_1\ybspec_1$ and $\ip_1\ybspec_2$ in one column
and
$\ip_2\ybspec_1$ and $\ip_2\ybspec_2$ in the other column, respectively.
Here the parameters $u_{1}$ and $u_2$ are 
as usual constant along horizontal rows,
and $\ip_{1}$ and $\ip_2$ are the inhomogeneity parameters 
which are constant along vertical columns.
The spectral parameter at a vertex is the product of the corresponding
``$u$'' and ``$\ip$'' parameters.
Recall that the conjugating matrix $X$ of Proposition \ref{prop:YB}
depends only on the ratio of spectral parameters in a column and does not depend on $\SP$, 
so it is the same for our two vertical columns.

Attaching two vertical columns on the side involves summing over all possible 
intermediate numbers of arrows
$k_1'$ and $k_2'$,
i.e., this corresponds to taking the product of two $4\times 4$ matrices
$w_{\ip_1\ybspec_1,\ip_1\ybspec_2;\,\SP_1}^{(m_1,n_1)}$ and
$w_{\ip_2\ybspec_1,\ip_2\ybspec_2;\,\SP_2}^{(m_2,n_2)}$.
Clearly, for this product the Yang--Baxter equation 
\eqref{YB_main_relation} is not going to change. 
One can similarly attach an arbitrary finite number of vertical columns
with $\SP$--parameters $\SP_j$ and 
spectral parameters $\ip_j\ybspec_1$ and $\ip_j \ybspec_2$ in the $j$-th column,
and the Yang--Baxter equation will continue to hold.
\begin{figure}[htbp]
	\begin{tikzpicture}
		[scale=1,ultra thick]
		\draw (0,0)--++(2.9,0) node [right] {$k_1''$};
		\draw (0,0)--++(-1.3,0) node [left] {$k_1$};
		\draw (0,1)--++(2.9,0) node [right] {$k_2''$};
		\draw (0,1)--++(-1.3,0) node [left] {$k_2$};
		\draw (0,0)--++(0,-.8) node[below] {$m_1$};
		\draw (0,1)--++(0,.8) node[above] {$n_1$};
		\draw (1.6,0)--++(0,-.8) node[below] {$m_2$};
		\draw (1.6,1)--++(0,.8) node[above] {$n_2$};
		\draw (0,0)--++(0,1) node[midway, xshift=7pt] {$l_1$};
		\draw (1.6,0)--++(0,1) node[midway, xshift=-6pt] {$l_2$};
		\node at (.8,-.3) {$k_1'$};
		\node at (.8,1.3) {$k_2'$};
		\draw[fill] (0,0) circle (3pt) node [above left] {$\ip_1\ybspec_1$};
		\draw[fill] (0,1) circle (3pt) node [above left] {$\ip_1\ybspec_2$};
		\draw[fill] (1.6,1) circle (3pt) node [above right] {$\ip_2\ybspec_2$};
		\draw[fill] (1.6,0) circle (3pt) node [above right] {$\ip_2\ybspec_1$};
		\node[circle, draw, line width=.6pt,
		minimum height=1em] at (0,-1.8) {$\SP_1$};
		\node[circle, draw, line width=.6pt,
		minimum height=1em] at (1.6,-1.8) {$\SP_2$};
	\end{tikzpicture}
	\caption{Attaching two vertical columns with 
	$\SP$--parameters $\SP_j$
	and spectral parameters
	$\ip_j\ybspec_1$ and $\ip_j \ybspec_2$ in the $j$-th column, $j=1,2$.
	The corresponding Yang--Baxter equation continues to hold.}
	\label{fig:YB_attach}
\end{figure}
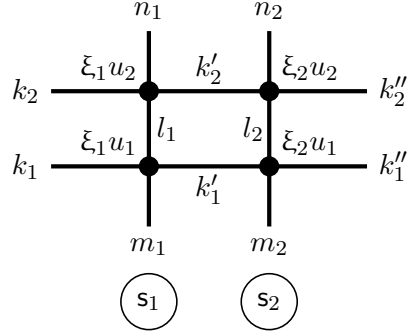

In the operator language 
attaching two vertical columns 
is equivalent to 
taking a tensor product $V=V_1\otimes V_2$ of two different physical spaces $V_1$ and $V_2$ with the same conjugating matrix $X$. 
The monodromy matrix in the space $V$ has the form
\begin{align}
	\TY=\begin{bmatrix}
		\AY&\BY\\
		\CY&\DY
	\end{bmatrix}
	=
	\begin{bmatrix}
		\AY_2&\BY_2\\
		\CY_2&\DY_2
	\end{bmatrix}\begin{bmatrix}
		\AY_1&\BY_1\\
		\CY_1&\DY_1
	\end{bmatrix}
	=\begin{bmatrix}
		\AY_2\AY_1+\BY_2\CY_1&\AY_2\BY_1+\BY_2\DY_1\\
		\CY_2\AY_1+\DY_2\CY_1&\CY_2\BY_1+\DY_2\DY_1
	\end{bmatrix}.
	\label{T_V1V2}
\end{align}
Here the lower index $1$ or $2$ in the operators above corresponds to the 
vertical (=~physical) space in which they act, i.e., $\AY_2=\AY_2(\ip_2 \ybspec\md\SP_2)$ acts in the second
vertical column corresponding to the parameter $\SP_2$, and $\AY_1=\AY_1(\ip_1 \ybspec\md\SP_1)$ acts in the same way in the first vertical column corresponding to $\SP_1$,
and similarly for $\BY_{1,2}$, $\CY_{1,2}$, and $\DY_{1,2}$ (we have omitted
spectral parameters in the notation in \eqref{T_V1V2}).
Note that any two operators with different lower indices commute.

The monodromy matrix $\TY=\TY(\ybspec\md\ipb,\SPB)$ in \eqref{T_V1V2}, 
where $\ipb=(\ip_1,\ip_2)$ and $\SPB=(\SP_1,\SP_2)$, corresponds to one horizontal row of vertices, and $\ybspec$
is the spectral parameter that is constant along this horizontal row.
That is, the four matrix elements of $\TY(\ybspec\md\ipb,\SPB)$ correspond to the following four configurations:
\begin{align*}
	\begin{bmatrix}
		\begin{tikzpicture}
			[scale=1,very thick]
			\draw[densely dotted] (0,0)--++(2.9,0) node [right] {$0$};
			\draw[densely dotted] (0,0)--++(-1.3,0) node [left] {$0$};
			\draw (0,.5)--++(0,-1);
			\draw (1.6,.5)--++(0,-1);
			\node at (.8,-.3) {$k'$};
			\draw[fill] (0,0) circle (3pt) node [above left] {$\ip_1\ybspec$};
			\draw[fill] (1.6,0) circle (3pt) node [above right] {$\ip_2\ybspec$};
		\end{tikzpicture}
		&
		\begin{tikzpicture}
			[scale=1,very thick]
			\draw[densely dotted] (0,0)--++(2.9,0) node [right] {$0$};
			\draw[densely dotted] (0,0)--++(-1.3,0) node [left] {$1$};
			\draw (0,.5)--++(0,-1);
			\draw (1.6,.5)--++(0,-1);
			\node at (.8,-.3) {$k'$};
			\draw[fill] (0,0) circle (3pt) node [above left] {$\ip_1\ybspec$};
			\draw[fill] (1.6,0) circle (3pt) node [above right] {$\ip_2\ybspec$};
			\draw (-1.3,0)--++(1.3,0);
			\draw[->] (-1.3,0)--++(.3,0);
		\end{tikzpicture}
		\\\\
		\begin{tikzpicture}
			[scale=1,very thick]
			\draw[densely dotted] (0,0)--++(2.9,0) node [right] {$1$};
			\draw[densely dotted] (0,0)--++(-1.3,0) node [left] {$0$};
			\draw (0,.5)--++(0,-1);
			\draw (1.6,.5)--++(0,-1);
			\node at (.8,-.3) {$k'$};
			\draw[fill] (0,0) circle (3pt) node [above left] {$\ip_1\ybspec$};
			\draw[fill] (1.6,0) circle (3pt) node [above right] {$\ip_2\ybspec$};
			\draw[->] (1.6,0)--++(1.3,0);
		\end{tikzpicture}
		&
		\begin{tikzpicture}
			[scale=1,very thick]
			\draw[densely dotted] (0,0)--++(2.9,0) node [right] {$1$};
			\draw[densely dotted] (0,0)--++(-1.3,0) node [left] {$1$};
			\draw (0,.5)--++(0,-1);
			\draw (1.6,.5)--++(0,-1);
			\node at (.8,-.3) {$k'$};
			\draw[fill] (0,0) circle (3pt) node [above left] {$\ip_1\ybspec$};
			\draw[fill] (1.6,0) circle (3pt) node [above right] {$\ip_2\ybspec$};
			\draw (-1.3,0)--++(1.3,0);
			\draw[->] (-1.3,0)--++(.3,0);
			\draw[->] (1.6,0)--++(1.3,0);
		\end{tikzpicture}
	\end{bmatrix},
\end{align*}
and the two summands in each matrix element in \eqref{T_V1V2} correspond to 
$k'$ being $0$ or $1$.

Furthermore, tensor products of two monodromy matrices like \eqref{T1T2} and \eqref{T2T1} correspond to 
configurations as in Fig.~\ref{fig:YB_attach}. As follows from the above discussion, these tensor
products satisfy the same Yang--Baxter equation \eqref{YB_operator_form}.



\section{Symmetric rational functions} 
\label{sec:symmetric_rational_functions}

We will now discuss
how the setup of \S \ref{sec:yang_baxter_relation}
can be applied to the physical space 
corresponding to the semi-infinite horizontal strip. 
This will lead to an introduction of 
certain symmetric rational functions 
which are one of our main objects.

\subsection{Signatures} 
\label{sub:signatures}

Let us first introduce some necessary notation.
By a \emph{signature}
of length $N$ we mean a sequence 
$\la=(\la_1\ge \la_2\ge \ldots\ge\la_N)$, $\la_i\in\Z$.
The set of all signatures of length $N$
will be denoted by $\sign{N}$, 
and $\signp{N}$ will stand for the 
set of signatures with $\la_N\ge0$.
By agreement, by $\sign{0}=\signp{0}$ we will denote the 
set consisting of the single empty signature $\varnothing$
of length $0$. Also let 
$\signpe{}:=\bigsqcup_{N\ge0}\signp{N}$
denote the set of all possible nonnegative signatures (including the empty one).
We will also use the multiplicative notation
$\mu=0^{m_0}1^{m_1}2^{m_2}\ldots\in\signpe$
for signatures, which means that $m_j:=|\{i\colon\mu_i=j\}|$
is the number of parts in $\mu$ that are equal to $j$ ($m_j$ is called the \emph{multiplicity} of $j$).


\subsection{Semi-infinite operators $\AY$ and $\BY$ and definition of symmetric rational functions} 
\label{sub:semi_infinite_operators_ay_and_by_definition_of_FG}

Let us consider the physical space $V=V_0\otimes V_1\otimes V_2\otimes \ldots$,
i.e., a tensor product 
of countably many ``elementary'' physical spaces (each of the latter has basis $\{\bv_j\}_{j\ge0}$ marked by $\Z_{\ge0}$). 
We will think that $V$ corresponds to the semi-infinite (to the right) row
of vertices attached to one another on the side. 
We will make sense of the infinite tensor product $V$
by requiring that we only consider \emph{finitary} vectors $V^{\mathrm{fin}}\subset V$, i.e., 
those in which almost all tensor factors are equal to $\bv_0$.
Therefore, a natural basis
in the space $V^{\mathrm{fin}}$ is indexed by nonnegative signatures:
\begin{align*}
	\bv_{\mu}=\bv_{m_0}\otimes \bv_{m_1}\otimes \bv_{m_2}\otimes \ldots,
	\qquad \mu=0^{m_0}1^{m_1}2^{m_2}\ldots\in\signpe
\end{align*}
($m_0+m_1+\ldots$ is the length of the signature $\mu$ which is finite).
We will work in the space
${\bar V}^{\mathrm{fin}}$
of all possible linear combinations of $\bv_{\mu}$ with complex coefficients.

\begin{definition}\label{def:a_param}
	Let us 
	fix (generic complex nonzero) 
	inhomogeneity parameters $\ipb=\{\ip_{j}\}_{j=0,1,2,\ldots}$
	and $\SP$--parameters
	$\SPB=\{\SP_{j}\}_{j=0,1,2,\ldots}$,
	similarly to what was done in \S \ref{sub:attaching_vertical_columns} before.
	Parameters $\ip_j$ and $\SP_j$ correspond to the ``elementary'' physical space $V_j$
	representing the $j$-th column in our semi-infinite horizontal row of vertices, $j=0,1,\ldots$.
	Also by $\ipbb=\{\ip_{j}^{-1}\}_{j=0,1,2,\ldots}$ we will denote 
	the inverses of the parameters $\ipb$ (and similarly for $\SPBB$).
\end{definition}

Defining the operators $\AY$ and $\BY$ acting in ${\bar V}^{\mathrm{fin}}$ causes no problems.
Indeed, we have for any $N\in\Z_{\ge0}$ and $\la\in\signp N$:
\begin{align}\label{AY_one_var_semi_infinite}
	\AY(u\md\ipb,\SPB)\,\bv_{\la}=\sum_{\mu\in\signp N}
	\textnormal{weight}_{\SPB}
	\Bigg(
	\raisebox{-29pt}{\scalebox{.9}{\begin{tikzpicture}
		[scale=1.2,very thick]
		\def\d{.07}
		\draw[densely dotted] (-.5,0)--++(8.2,0);
		\foreach \ii in {0,1,2,3,4,5,6}
		{
			\draw[densely dotted] (1.2*\ii,-.5)--++(0,1);
			\node[above left, xshift=-1pt] at (1.2*\ii,0) {$\ip_{\ii}u$};
		}
		\node[below] at (0,-.5) {0};
		\node[below] at (1.2,-.5) {$\la_N$};
		\node[below] at (3.6,-.5) {$\la_3$};
		\node[below] at (6,-.5) {$\la_2=\la_1$};
		\node[above] at (7.2,.5) {$\mu_1$};
		\node[above] at (6,.5) {$\mu_3=\mu_2$};
		\node[above] at (2.4,.5) {$\mu_N$};
		\draw[line width=2pt,->] (6+\d,-.5) --++ (0,.2);
		\draw[line width=2pt,->] (6+\d,-.5) --++ (0,.5-\d)--++(\d,\d)--++(1.2-2*\d,0)--++(0,.5);
		\draw[line width=2pt,->] (6-\d,-.5) --++ (0,.2);
		\draw[line width=2pt,->] (6-\d,-.5) --++ (0,.5-\d)--++(2*\d,2*\d)--++(0,.5-\d);
		\draw[line width=2pt,->] (3.6,-.5) --++ (0,.2);
		\draw[line width=2pt,->] (3.6,-.5) --++ (0,.5)--++(2.4-2*\d,0)--++(\d,\d)--++(0,.5-\d);
		\draw[line width=2pt,->] (1.2,-.5) --++ (0,.2);
		\draw[line width=2pt,->] (1.2,-.5) --++ (0,.5)--++(1.2,0)--++(0,.5);
	\end{tikzpicture}}}
	\Bigg)\,\bv_{\mu},
\end{align}
and 
\begin{align}\label{BY_one_var_semi_infinite}
	\BY(u\md\ipb,\SPB)\,\bv_{\la}=\sum_{\mu\in\signp {N+1}}
	\textnormal{weight}_{\SPB}
	\Bigg(
	\raisebox{-29pt}{\scalebox{.9}{\begin{tikzpicture}
		[scale=1.2,very thick]
		\def\d{.07}
		\draw[densely dotted] (-.5,0)--++(8.2,0);
		\foreach \ii in {0,1,2,3,4,5,6}
		{
			\draw[densely dotted] (1.2*\ii,-.5)--++(0,1);
			\node[above left, xshift=-1pt] at (1.2*\ii,0) {$\ip_{\ii}u$};
		}
		\node[below] at (0,-.5) {0};
		\node[below] at (1.2,-.5) {$\la_N$};
		\node[below] at (3.6,-.5) {$\la_3$};
		\node[below] at (6,-.5) {$\la_2=\la_1$};
		\node[above] at (7.2,.5) {$\mu_1$};
		\node[above] at (6,.5) {$\mu_3=\mu_2$};
		\node[above] at (2.4,.5) {$\mu_N$};
		\node[above] at (1.2,.5) {$\mu_{N+1}$};
		\draw[line width=2pt,->] (6+\d,-.5) --++ (0,.2);
		\draw[line width=2pt,->] (6+\d,-.5) --++ (0,.5-\d)--++(\d,\d)--++(1.2-2*\d,0)--++(0,.5);
		\draw[line width=2pt,->] (6-\d,-.5) --++ (0,.2);
		\draw[line width=2pt,->] (6-\d,-.5) --++ (0,.5-\d)--++(2*\d,2*\d)--++(0,.5-\d);
		\draw[line width=2pt,->] (3.6,-.5) --++ (0,.2);
		\draw[line width=2pt,->] (3.6,-.5) --++ (0,.5)--++(2.4-2*\d,0)--++(\d,\d)--++(0,.5-\d);
		\draw[line width=2pt,->] (1.2,-.5) --++ (0,.2);
		\draw[line width=2pt,->] (1.2,-.5) --++ (0,.5-\d)--++(\d,\d)--++(1.2-\d,0)--++(0,.5);
		\draw[line width=2pt,->] (-.5,0) --++ (.2,0);
		\draw[line width=2pt,->] (-.5,0) --++ (.5,0)--++(1.2-\d,0)--++(\d,\d)--++(0,.5-\d);
	\end{tikzpicture}}}
	\Bigg)\,\bv_{\mu}.
\end{align}
That is, in \eqref{AY_one_var_semi_infinite} and \eqref{BY_one_var_semi_infinite}
we sum over 
all possible signatures $\mu$, and for each fixed $\mu$
the coefficient is equal to the weight of the unique path collection connecting the arrow configuration
$\la$ to the configuration $\mu$, as shown pictorially
(the coefficient is $0$ if no admissible path collection 
exists).\footnote{Recall that the weight of a path collection is defined as
the product of weights of all (nonempty) vertices in the corresponding
region of $\Z^{2}$, and that the
weight of the empty vertex 
\emptyvertex{.5} is 1.}
The subscript $\SPB$ in the weights 
corresponds to taking parameter $\SP_j$ in the $j$-th
vertex, $j=0,1,\ldots$.
The difference between the action of
the operators
\eqref{AY_one_var_semi_infinite} and \eqref{BY_one_var_semi_infinite}
is that in \eqref{AY_one_var_semi_infinite} the path collection contains $N$
paths connecting $\la_j$ to $\mu_j$, $j=1,\ldots,N$,
and in \eqref{BY_one_var_semi_infinite} there is one additional path 
starting horizontally at the left boundary, 
and ending at $\mu_{N+1}$.


Let us denote the coefficients in the sums in 
\eqref{AY_one_var_semi_infinite} and \eqref{BY_one_var_semi_infinite}
by $\G_{\mu/\la}(u\md\ipb,\SPB)$ and $\F_{\mu/\la}(u\md\ipb,\SPB)$, respectively.
Here $u$ is the spectral parameter, and we also explicitly 
indicate the dependence on the parameters 
$\ipb$ and $\SPB$ (like in the operators $\AY(u\md\ipb,\SPB)$ and $\BY(u\md\ipb,\SPB)$).
\begin{remark}\label{rmk:A_B_semi_infinite_definition}
	Each coefficient $\G_{\mu/\la}(u\md\ipb,\SPB)$ and $\F_{\mu/\la}(u\md\ipb,\SPB)$
	in the semi-infinite setting
	is the same as if we took it in a finite tensor product,
	with the number of factors $\ge \mu_1+1$.
	It follows that the semi-infinite operators \eqref{AY_one_var_semi_infinite} and 
	\eqref{BY_one_var_semi_infinite} satisfy the commutation relations
	\begin{align}\label{AB_semi_infinite_commute}
		\AY(\ybspec_1\md\ipb,\SPB)\AY(\ybspec_2\md\ipb,\SPB)=\AY(\ybspec_2\md\ipb,\SPB)\AY(\ybspec_1\md\ipb,\SPB),
		\quad
		\BY(\ybspec_1\md\ipb,\SPB)\BY(\ybspec_2\md\ipb,\SPB)=\BY(\ybspec_2\md\ipb,\SPB)\BY(\ybspec_1\md\ipb,\SPB)
	\end{align}
	(cf. \eqref{A_commute} and \eqref{B_commute}). 
	Indeed, to check the commutation relations, apply them to 
	$\bv_\la$ and read off the coefficient by each $\bv_\mu$.
	One readily sees that 
	each such coefficient by $\bv_\mu$
	involves only finite summation.
\end{remark}

Similarly, we define the coefficients 
$\G_{\mu/\la}(u_1,\ldots,u_n\md\ipb,\SPB)$ and $\F_{\mu/\la}(u_1,\ldots,u_n\md\ipb,\SPB)$ arising 
from products of our operators in the following way:
\begin{align}
	\label{AY_semi_infinite}
	\AY(u_1\md\ipb,\SPB)\ldots\AY(u_n\md\ipb,\SPB)\,\bv_\la&=\sum_{\mu\in\signp N}
	\G_{\mu/\la}(u_1,\ldots,u_n\md\ipb,\SPB)\,\bv_{\mu},\\
	\BY(u_1\md\ipb,\SPB)\ldots\BY(u_n\md\ipb,\SPB)\,\bv_\la&=\sum_{\mu\in\signp {N+n}}
	\F_{\mu/\la}(u_1,\ldots,u_n\md\ipb,\SPB)\,\bv_{\mu},
	\label{BY_semi_infinite}
\end{align}
where $N\in\Z_{\ge0}$ and $\la\in\signp N$ are arbitrary. 

\medskip

Equivalently, the quantities
$\G_{\mu/\la}(u_1,\ldots,u_n\md\ipb,\SPB)$
and 
$\F_{\mu/\la}(u_1,\ldots,u_n\md\ipb,\SPB)$ 
can be defined as certain partition functions in the higher spin six vertex model:

\begin{definition}\label{def:G}
	Let $N,n\in\Z_{\ge0}$, $\la,\mu\in\signp{N}$.
	Assign to each 
	vertex $(x,y)\in\Z\times\{1,2,\ldots,n\}$ the spectral parameter
	$\ip_x u_y$ and the $\SP$--parameter $\SP_x$.
	Define
	$\G_{\mu/\la}(u_1,\ldots,u_{n}\md\ipb,\SPB)$
	to be the sum of weights of all possible
	collections of $N$ up-right paths such that they
	\begin{itemize}
		\item start with $N$ vertical edges
		$(\la_i,0)\to(\la_i,1)$, $i=1,\ldots,N$,
		\item end with $N$ vertical edges
		$(\mu_i,n)\to(\mu_i,n+1)$, $i=1,\ldots,N$.
	\end{itemize}
	See Fig.~\ref{fig:paths_FG}, left.
	We will also use the abbreviation $\G_{\mu}:=\G_{\mu/(0,0,\ldots,0)}$,
	which corresponds to the decomposition of $\AY(u_1\md\ipb,\SPB)\ldots\AY(u_n\md\ipb,\SPB)(\bv_{N}\otimes \bv_{0}\otimes \bv_{0}\otimes\ldots)$.
\end{definition}
\begin{definition}\label{def:F}
	Let $N,n\in\Z_{\ge0}$, $\la\in\signp{N}$,
	$\mu\in\signp{N+n}$. As before, assign to each 
	vertex $(x,y)\in\Z\times\{1,2,\ldots,n\}$ the spectral parameter
	$\ip_x u_y$ and the $\SP$--parameter $\SP_x$.
	Define
	$\F_{\mu/\la}(u_1,\ldots,u_{n}\md\ipb,\SPB)$
	to be the sum of weights of all possible
	collections of $N+n$ up-right paths such that they
	\begin{itemize}
		\item start with $N$ vertical edges $(\la_i,0)\to(\la_i,1)$, $i=1,\ldots,N$,
		and with $n$ horizontal edges $(-1,y)\to(0,y)$, $y=1,\ldots,n$,
		\item end with $N+n$ vertical edges
		$(\mu_{i},n)\to(\mu_{i},n+1)$, $i=1,\ldots,N+n$.
	\end{itemize}
	See Fig.~\ref{fig:paths_FG}, right.
	We will also use the abbreviation $\F_{\mu}:=\F_{\mu/\varnothing}$,
	which corresponds to the decomposition of 
	$\BY(u_1\md\ipb,\SPB)\ldots\BY(u_n\md\ipb,\SPB)(\bv_{0}\otimes \bv_{0}\otimes\ldots)$.
\end{definition}

In both definitions above,
if a collection of paths has no interior vertices, we define its weight to be~$1$. 
Also, the weight of an empty collection of paths is $0$.
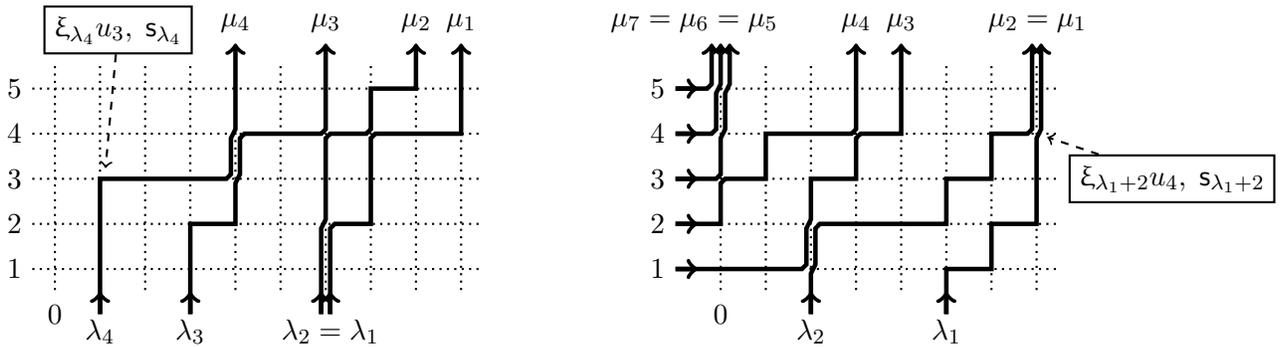
\begin{figure}[htbp]
	\begin{tabular}{cc}
		\begin{tikzpicture}
			[scale=.6,thick]
			\def\d{.1}
			\foreach \xxx in {0,4,5,1,2,3,6,7,8,9}
			{
				\draw[dotted] (\xxx,5.5)--++(0,-5);
			}
			\foreach \xxx in {1,2,3,4,5}
			{
				\draw[dotted] (-.5,\xxx)--++(10,0);
				\node[left] at (-.5,\xxx) {$\xxx$};
			}
			\draw[->, line width=1.7pt] (3,0)--++(0,.5) node [below, yshift=-6pt] {$\la_3$};
			\draw[->, line width=1.7pt] (1,0)--++(0,.5) node [below, yshift=-6pt] {$\la_4$};
			\draw[->, line width=1.7pt] (6+\d,0)--++(0,.5) node [below, yshift=-6pt] {$\la_2=\la_1$};
			\draw[->, line width=1.7pt] (6-\d,0)--++(0,.5);
			\draw[->, line width=1.7pt] (1,0) --++(0,3)--++(3-2*\d,0)--++(\d,\d)--++(0,1-2*\d)--++(\d,2*\d)--++(0,1-\d)--++(0,1) node [above] {$\mu_4$};
			\draw[->, line width=1.7pt] (3,0) --++(0,2)--++(1,0) --++(0,1-\d)--++(\d,2*\d)--++(0,1-2*\d)--++(\d,\d)
			--++(2-3*\d,0)--++(\d,\d)--++(0,2-\d) node [above] {$\mu_3$};
			\draw[->, line width=1.7pt]	(6-\d,0)--++(0,2-\d)--++(\d,2*\d)--++(0,2-2*\d)--++(\d,\d)
			--++(1-2*\d,0)--++(\d,\d)--++(0,1-\d)--++(1,0)--++(0,1)
			node [above] {$\mu_2$};
			\draw[->, line width=1.7pt]	(6+\d,0)--++(0,2-\d)--++(\d,\d)--++(1-2*\d,0)--++(0,2-\d)--++(\d,\d)--++(2-\d,0)--++(0,2)
			node [above] {$\mu_1$};
			\node[rectangle, draw, fill=white] (au_label) at (1.4,6.3) {$\ip_{\la_4}u_3,\; \SP_{\la_4}$};
			\node (au) at (1.1,3) {};
			\draw[->, dashed] (au_label) -- (au);
			\node(0,0) {$0$};
		\end{tikzpicture}
		&\hspace{30pt}
		\begin{tikzpicture}
			[scale=.6,thick]
			\def\d{.1}
			\foreach \xxx in {0,4,5,1,2,3,6,7}
			{
				\draw[dotted] (\xxx,5.5)--++(0,-5);
			}
			\foreach \xxx in {1,2,3,4,5}
			{
				\draw[dotted] (0,\xxx)--++(7.5,0);
				\node[left] at (-1,\xxx) {$\xxx$};
				\draw[->, line width=1.7pt] (-1,\xxx)--++(.5,0);
			}
			\draw[->, line width=1.7pt] (2,0)--++(0,.5) node [below, yshift=-6pt] {$\la_2$};
			\draw[->, line width=1.7pt] (5,0)--++(0,.5) node [below, yshift=-6pt] {$\la_1$};
			\node(0,0) {$0$};
			\draw[->, line width=1.7pt] (-1,5)--++(1-3*\d,0)--++(\d,\d)--++(0,1-\d);
			\draw[->, line width=1.7pt] (-1,4)--++(1-2*\d,0)--++(\d,\d)--++(0,1-2*\d)--++(\d,2*\d)--++(0,1-\d);
			\draw[->, line width=1.7pt] (-1,3)--++(1-\d,0)--++(\d,\d)--++(0,1-2*\d)--++(\d,2*\d)--++(0,1-2*\d)--++(\d,2*\d)--++(0,1-\d)
			node[above, xshift=-10pt] at (0,6) {$\mu_7=\mu_6=\mu_5$};
			\draw[->, line width=1.7pt] (-1,2)--++(1,0)--++(0,1-\d)--++(\d,\d)--++(1-\d,0)
			--++(0,1)--++(2-\d,0)--++(\d,\d)--++(0,2-\d)
			node[above] {$\mu_4$};
			\draw[->, line width=1.7pt] 
			(-1,1)--++(3-2*\d,0)--++(\d,\d)--++(0,1-2*\d)--++(\d,2*\d)--++(0,1-\d)--++(1,0)--++(0,1-\d)--++(\d,\d)--++(1-\d,0)
			--++(0,2) node [above] {$\mu_3$};
			\draw[->, line width=1.7pt] 
			(2,0)--++(0,1-\d)--++(\d,2*\d)--++(0,1-2*\d)--++(\d,\d)--++(3-2*\d,0)--++(0,1)--++(1,0)--++(0,1)
			--++(1-2*\d,0)--++(\d,\d)--++(0,2-\d) node[above, xshift=2pt] {$\mu_2=\mu_1$};
			\draw[->, line width=1.7pt] 
			(5,0)--++(0,1)--++(1,0)--++(0,1)--++(1,0)--++(0,2-\d)--++(\d,2*\d)
			--++(0,2-\d);
			\node[rectangle, draw, fill=white] (au_label) at (10,3) {$\ip_{\la_1+2}u_4,\; \SP_{\la_1+2}$};
			\node (au) at (7,4) {};
			\draw[->, dashed] (au_label) -- (au);
		\end{tikzpicture}
	\end{tabular}
	\caption{Path collections
	used in the definitions of  
	$\G_{\mu/\la}$ (left)
	and $\F_{\mu/\la}$ (right).
	The weight of a path collection is the product of weights of 
	all nonempty vertices (cf. Fig.~\ref{fig:paths_example}).
	Labels in boxes
	show examples
	of 
	spectral parameters $\ip_x u_y$ and parameters $\SP_x$ 
	at the corresponding vertex $(x,y)$.}
	\label{fig:paths_FG}
\end{figure}

Clearly, both quantities $\G_{\mu/\la}(u_1,\ldots,u_n\md\ipb,\SPB)$ and $\F_{\mu/\la}(u_1,\ldots,u_n\md\ipb,\SPB)$
depend on the spectral parameters $u_1,\ldots,u_n$ in a \emph{rational} way.
When the parameters $\ip_j\equiv 1$ and $\SP_j\equiv \SP$ are constant, the functions 
$\G_{\mu/\la}$ and $\F_{\mu/\la}$
reduce to the ones defined in
\cite[\S3]{Borodin2014vertex}.
Therefore, here we consider
\emph{inhomogeneous} versions of the functions
from \cite{Borodin2014vertex}.

\begin{proposition}\label{prop:symmetric}
	The rational functions 
	$\F_{\mu/\la}(u_1,\ldots,u_n\md\ipb,\SPB)$ and
	$\G_{\mu/\la}(u_1,\ldots,u_n\md\ipb,\SPB)$ 
	defined above
	are symmetric with respect to permutations of the 
	$u_j$'s.
\end{proposition}
\begin{proof}
	This immediately follows from the commutation relations
	\eqref{AB_semi_infinite_commute}.
\end{proof}

The functions
$\F_{\mu/\la}$
and $\G_{\mu/\la}$
satisfy the following branching rules:
\begin{proposition}
\label{prop:branching}
	\noindent{\bf1.\/}
	For any 
	$N,n_1,n_2\in\Z_{\ge0}$,
	$\la\in\signp{N}$, and $\mu\in \signp{N+n_1+n_2}$,
	one has
	\begin{equation}\label{branching-F}
		\F_{\mu/\la}(u_1,\dots,u_{n_1+n_2}\md\ipb,\SPB)=
		\sum_{\kappa\in\signp{N+n_1}}  
		\F_{\mu/\kappa}(u_{n_1+1},\dots,u_{n_1+n_2}\md\ipb,\SPB)\F_{\kappa/\la}(u_1,\ldots,u_{n_1}\md\ipb,\SPB). 
	\end{equation}
\smallskip

	\noindent{\bf2.\/}
	For any $N,n_1,n_2\in\Z_{\ge0}$, and $\la,\mu\in\signp{N}$, one has
	\begin{equation}\label{branching-G}
		\G_{\mu/\la}(u_1,\ldots,u_{n_1+n_2}\md\ipb,\SPB)
		=
		\sum_{\kappa\in\signp N} 
		\G_{\mu/\kappa}(u_{n_1+1},\ldots,u_{n_1+n_2}\md\ipb,\SPB)
		\G_{\kappa/\la}(u_1,\ldots,u_{n_1}\md\ipb,\SPB)
		.
	\end{equation}
\end{proposition}
\begin{proof}
	Follows from the definitions \eqref{AY_semi_infinite} and \eqref{BY_semi_infinite} in a straightforward way.
	In other words, 
	identities \eqref{branching-F} and \eqref{branching-G} 
	simply mean 
	the splitting of summation over path collections
	in $\F_{\mu/\la}$
	and $\G_{\mu/\la}$, such that
	the signature
	$\kappa$ keeps track of the cross-section 
	of the path collection at height $n_1$.
\end{proof}

Along with Proposition \ref{prop:symmetric}, one can 
also establish the following partial symmetry 
property of the functions $\F_{\mu/\la}$
and $\G_{\mu/\la}$
with respect to the inhomogeneity parameters:
\begin{proposition}\label{prop:a_symmetry}
	For any interval of consecutive integers
	$\{i,i+1,\ldots,j\}\subset\Z\setminus (\la\cup\mu)$,
	the functions 
	$\G_{\mu/\la}(u_1,\ldots,u_{n}\md\ipb,\SPB)$
	and 
	$\F_{\mu/\la}(u_1,\ldots,u_{n}\md\ipb,\SPB)$
	from Definitions \ref{def:G} and \ref{def:F}
	are symmetric with respect to any permutation of the inhomogeneity parameters
	$\ip_i,\ip_{i+1},\ldots,\ip_{j}$ and the same (simultaneous) permutation
	of the $\SP$--parameters $\SP_i,\SP_{i+1},\ldots,\SP_{j}$.
\end{proposition}
\begin{proof}
	This is a straightforward corollary of a ``horizontal'' version of the Yang--Baxter equation
	which is similar to Proposition \ref{prop:YB}, but with a more complicated 
	conjugating matrix
	related to the general $J$ vertex weights 
	(about those see \S \ref{sec:stochastic_weights_and_fusion} below).
	We will not discuss details of this Yang--Baxter equation here,
	but will note that for $\la=\varnothing$ the 
	claim would alternatively follow from the 
	explicit formulas for our symmetric functions, see Theorem \ref{thm:symmetrization} below.
\end{proof}


\subsection{Semi-infinite operator $\DY$} 
\label{sub:semi_infinite_operator_D}

It is slightly more difficult to define the action of the other two operators,
$\CY$ and $\DY$, in the semi-infinite context. 
We will not need the operator $\CY$,
so let us focus on $\DY=\DY(u\md\ipb,\SPB)$. 
The action of $\DY$ (in a finite tensor product) corresponds to the following 
configuration (cf. \eqref{AY_one_var_semi_infinite} and \eqref{BY_one_var_semi_infinite}): 
\begin{align*}
	\raisebox{-29pt}{\scalebox{.9}{\begin{tikzpicture}
		[scale=1.2,very thick]
		\def\d{.07}
		\draw[densely dotted] (-.5,0)--++(8.2,0);
		\foreach \ii in {0,1,2,3,4,5,6}
		{
			\draw[densely dotted] (1.2*\ii,-.5)--++(0,1);
			\node[above left, xshift=-1pt] at (1.2*\ii,0) {$\ip_{\ii}u$};
		}
		\node[below] at (0,-.5) {0};
		\node[below] at (1.2,-.5) {$\la_N$};
		\node[below] at (3.6,-.5) {$\la_3$};
		\node[below] at (6,-.5) {$\la_2=\la_1$};
		\draw[line width=2pt,->] (6+\d,-.5) --++ (0,.2);
		\draw[line width=2pt,->] (6+\d,-.5) --++ (0,.5-\d)--++(\d,\d)--++(1.8-2*\d,0);
		\draw[line width=2pt,->] (6-\d,-.5) --++ (0,.2);
		\draw[line width=2pt,->] (6-\d,-.5) --++ (0,.5-\d)--++(\d,2*\d)--++(0,.5-\d) node [above] {$\mu_1$};
		\draw[line width=2pt,->] (3.6,-.5) --++ (0,.2);
		\draw[line width=2pt,->] (3.6,-.5) --++ (0,.5)--++(1.2,0)--++(0,.5) node [above] {$\mu_2$};
		\draw[line width=2pt,->] (1.2,-.5) --++ (0,.2);
		\draw[line width=2pt,->] (1.2,-.5) --++ (0,.5-\d)--++(\d,\d)--++(1.2-\d,0)--++(0,.5)
		node [above] {$\mu_{N-1}$};
		\draw[line width=2pt,->] (-.5,0) --++ (.2,0);
		\draw[line width=2pt,->] (-.5,0) --++ (.5,0)--++(1.2-\d,0)--++(\d,\d)--++(0,.5-\d) node [above] {$\mu_N$};
		\node[circle, draw, line width=.6pt,minimum height=1em] at (0,-1.3) {$\SP_0$};
		\node[circle, draw, line width=.6pt,minimum height=1em] at (1.2,-1.3) {$\SP_1$};
		\node[circle, draw, line width=.6pt,minimum height=1em] at (2.4,-1.3) {$\SP_2$};
		\node[circle, draw, line width=.6pt,minimum height=1em] at (3.6,-1.3) {$\SP_3$};
		\node[circle, draw, line width=.6pt,minimum height=1em] at (4.8,-1.3) {$\SP_4$};
		\node[circle, draw, line width=.6pt,minimum height=1em] at (6,-1.3) {$\SP_5$};
		\node[circle, draw, line width=.6pt,minimum height=1em] at (7.2,-1.3) {$\SP_6$};
	\end{tikzpicture}}}
\end{align*}
For the semi-infinite horizontal strip, the weight of this configuration
would involve an infinite product of the form
\begin{align*}
	\prod_{j\ge\la_1+1}w_{\ip_j u,\SP_j}(0,1;0,1)=
	\prod_{j\ge\la_1+1}\frac{\ip_j u-\SP_j}{1-\SP_j\ip_j u},
\end{align*}
which means that one cannot define the operator $\DY(u\md\ipb,\SPB)$ in the semi-infinite setting
directly. 

However, the definition of $\DY$ can be easily corrected, by considering strips of finite length $L+1$
and the operators $\DY(u\md\ipb,\SPB)$ in $V_0\otimes \ldots\otimes V_L$. For a fixed $L$
denote such an operator by $\DY_L=\DY_L(u\md\ipb,\SPB)$.
Dividing $\DY_L$ by
$\prod_{j=0}^{L}w_{\ip_j u,\SP_j}(0,1;0,1)$, and sending $L\to+\infty$, we would arrive at a meaningful object.
Indeed, under this transformations the weights of individual vertices will turn into
\begin{align}
	\begin{array}{rclclcl}
		\dfrac{1}
		{w_{u,\SP}(0,1;0,1)}w_{u,\SP}\Big(\raisebox{-18pt}{\mbox{\vertexoo{.6}}}\Big)
		&=&\dfrac{1-\SP q^{g}u}{u-\SP }
		&=&w_{u^{-1},\SP}\Big(\raisebox{-18pt}{\mbox{\vertexll{.6}}}\Big)&=&
		w_{u^{-1},\SP}^{\conj}\Big(\raisebox{-18pt}{\mbox{\vertexll{.6}}}\Big),
		\\
		\rule{0pt}{22pt}
		\dfrac{1}
		{w_{u,\SP}(0,1;0,1)}w_{u,\SP}\Big(\raisebox{-18pt}{\mbox{\vertexoll{.6}}}\Big)&=&\dfrac{(1-\SP ^{2}q^{g})u}{u-\SP }
		&=&
		w_{u^{-1},\SP}\Big(\raisebox{-18pt}{\mbox{\vertexlo{.6}}}\Big)\dfrac{1-\SP ^{2}q^{g}}{1-q^{g+1}}&=&
		w_{u^{-1},\SP}^{\conj}\Big(\raisebox{-18pt}{\mbox{\vertexlo{.6}}}\Big),\\
		\rule{0pt}{22pt}
		\dfrac{1}
		{w_{u,\SP}(0,1;0,1)}w_{u,\SP}\Big(\raisebox{-18pt}{\mbox{\vertexll{.6}}}\Big)&=&\dfrac{u-\SP q^{g}}{u-\SP }
		&=&w_{u^{-1},\SP}\Big(\raisebox{-18pt}{\mbox{\vertexoo{.6}}}\Big)&=&w_{u^{-1},\SP}^{\conj}\Big(\raisebox{-18pt}{\mbox{\vertexoo{.6}}}\Big),\\
		\rule{0pt}{22pt}
		\dfrac{1}
		{w_{u,\SP}(0,1;0,1)}w_{u,\SP}\Big(\raisebox{-18pt}{\mbox{\vertexlo{.6}}}\Big)&=&\dfrac{1-q^{g+1}}{u-\SP }
		&=&w_{u^{-1},\SP}\Big(\raisebox{-18pt}{\mbox{\vertexoll{.6}}}\Big)\dfrac{1-q^{g+1}}{1-\SP ^{2}q^{g}}
		&=&w_{u^{-1},\SP}^{\conj}\Big(\raisebox{-18pt}{\mbox{\vertexoll{.6}}}\Big),
	\end{array}
	\label{D_Bar_computation}
\end{align}
where we have used the conjugated weights \eqref{weights_conj}.
Note that the numbers of vertical incoming and outgoing arrows at a vertex were \emph{swapped} under the above transformations.
Therefore, for any $L\ge \la_1+1$ we have
\begin{align*}
	\frac{[\textnormal{coefficient of $\bv_\mu$ in
	$\DY_{L}(u\md\ipb,\SPB)\,\bv_\la$}]}{\prod_{j=0}^{L}w_{\ip_j u,\SP_j}(0,1;0,1)}=
	[\textnormal{coefficient of $\bv_\la$ in
	$\AY(u^{-1}\md\ipbb,\SPB)\,\bv_\mu$}]\cdot \frac{\conj_{\SPB}(\la)}{\conj_{\SPB}(\mu)},
\end{align*}
where for any signature $\nu\in\signpe$ we have denoted
\begin{align}\label{conj_definition}
	\conj_{\SPB}(\nu):=\prod_{k=0}^{\infty}\frac{(\SP_{k}^{2};q)_{n_k}}{(q;q)_{n_k}},
	\qquad \nu=0^{n_0}1^{n_1}2^{n_2}\cdots
\end{align}
(this product has finitely many factors not equal to $1$).
Recall that $\ipbb$ means inverting the inhomogeneity parameters, as 
dictated by the transformations \eqref{D_Bar_computation}.
The operator $\AY(u^{-1}\md\ipbb,\SPB)$ above can be regarded as acting either in a finite tensor product,
or in the semi-infinite space ${\bar V}^{\mathrm{fin}}$, since 
matrix elements corresponding to $(\bv_\mu,\bv_\la)$
of 
these two versions of $\AY(u^{-1}\md\ipb,\SPB)$ coincide 
for fixed $\mu,\la$ and large enough $L$ (cf. Remark \ref{rmk:A_B_semi_infinite_definition}).

We see that it is natural to define the normalized operator 
\begin{align}\label{D_Bar}
	\overline\DY(u\md\ipb,\SPB):=
	\lim_{L\to+\infty}
	\frac{
	\DY_{L}(u\md\ipb,\SPB)}{\prod_{j=0}^{L}w_{\ip_j u,\SP_j}(0,1;0,1)},
\end{align}
where the limit is taken in the sense of matrix elements corresponding to the basis vectors
$\{\bv_\la\}_{\la\in\signpe}$. The matrix elements of
$\overline\DY(u\md\ipb,\SPB)$ are (cf. \eqref{AY_one_var_semi_infinite})
\begin{align*}
	\overline\DY(u\md\ipb,\SPB)\,\bv_{\la}=\sum_{\mu\in\signp N}
	\frac{\conj_{\SPB}(\la)}{\conj_{\SPB}(\mu)}\G_{\la/\mu}(u^{-1}\md\ipbb,\SPB)\,\bv_\mu.
\end{align*}
Observe that the above sum over $\mu$ is \emph{finite},
in contrast with the operators \eqref{AY_one_var_semi_infinite}
and \eqref{BY_one_var_semi_infinite}.
From \eqref{D_Bar} and \eqref{D_commute} it follows that 
the operators 
$\overline\DY(u\md\ipb,\SPB)$ commute for different $u$.


In what follows we will use the notation
\begin{align*}
	\F_{\la/\mu}^{\conj}:=
	\frac{\conj_{\SPB}(\la)}{\conj_{\SPB}(\mu)}\F_{\la/\mu}\,,
	\qquad
	\G_{\la/\mu}^{\conj}:=
	\frac{\conj_{\SPB}(\la)}{\conj_{\SPB}(\mu)}\G_{\la/\mu}\,.
\end{align*}


\subsection{Cauchy-type identities from the Yang--Baxter commutation relations} 
\label{sub:limits_of_yang_baxter_commutation_relations_cauchy_type_identities}

Let us consider the semi-infinite limit as $L\to+\infty$ (similar to what was done
in \S \ref{sub:semi_infinite_operator_D} above)
of the Yang--Baxter commutation relation \eqref{YB_relationBD}.
Looking at \eqref{YB_relationBD},
we immediately face the question of what
we need to normalize the two sides by: 
$\prod_{j=0}^{L}w_{\ip_j u_{1},\SP_j}(0,1;0,1)$
or 
$\prod_{j=0}^{L}w_{\ip_j u_{2},\SP_j}(0,1;0,1)$\,? 
Since out of the three terms in 
\eqref{YB_relationBD}
two require the normalization involving $u_2$, let us use that one. 
To be able to take the limit as $L\to+\infty$, 
we will also require that
\begin{align}\label{admissible_bad_condition_general}
	\lim_{L\to+\infty}
	\prod_{j=0}^{L}\left|\frac{w_{\ip_j u_{1},\SP_j}(0,1;0,1)}{w_{\ip_j u_{2},\SP_j}(0,1;0,1)}\right|
	=
	\lim_{L\to+\infty}\prod_{j=0}^{L}
	\left|\frac{\ip_j u_{1}-\SP_j}{1-\SP_j\ip_j u_{1}}
	\cdot
	\frac{1-\SP_j\ip_j u_{2}}{\ip_j u_{2}-\SP_j}\right|=0.
\end{align}
Under \eqref{admissible_bad_condition_general}, we can take the normalized 
(by $\prod_{j=0}^{L}w_{\ip_j u_{2},\SP_j}(0,1;0,1)$)
limit of the 
relation \eqref{YB_relationBD}, and, using \eqref{D_Bar}, conclude that 
\begin{align}\label{B_DBar_YB_relation}
	\BY(u_1\md\ipb,\SPB)\overline\DY(u_{2}\md\ipb,\SPB)=
	\frac{u_1-u_2}{qu_1-u_2}
	\overline\DY(u_{2}\md\ipb,\SPB)\BY(u_1\md\ipb,\SPB).
\end{align}
Indeed, before the limit the normalized 
second term of \eqref{YB_relationBD}
contains
\begin{align*}
	\frac{\DY_{L}(u_{1}\md\ipb,\SPB)}{\prod_{j=0}^{L}w_{\ip_j u_{2},\SP_j}(0,1;0,1)}
	=
	\frac{\DY_{L}(u_{1}\md\ipb,\SPB)}
	{\prod_{j=0}^{L}w_{\ip_j u_{1},\SP_j}(0,1;0,1)}
	\cdot\prod_{j=0}^{L}\frac{w_{\ip_j u_{1},\SP_j}(0,1;0,1)}{w_{\ip_j u_{2},\SP_j}(0,1;0,1)},
\end{align*}
which converges to zero by 
\eqref{admissible_bad_condition_general}.

Using the notation 
$\F_{\mu/\la}$ and $\G_{\mu/\la}$ 
introduced in \S \ref{sub:semi_infinite_operators_ay_and_by_definition_of_FG},
relation \eqref{B_DBar_YB_relation} becomes
\begin{align}\label{skew_Cauchy_bad}
	\sum_{\mu\in\signpe}
	\F_{\la/\mu}(u_1\md\ipb,\SPB)
	\G^{\conj}_{\nu/\mu}(u_2^{-1}\md\ipbb,\SPB)
	= \frac{u_1-u_2}{qu_1-u_2}
	\sum_{\kappa\in\signpe}
	\G^{\conj}_{\kappa/\la}(u_2^{-1}\md\ipbb,\SPB)
	\F_{\kappa/\nu}(u_1\md\ipb,\SPB)
	.
\end{align}
Therefore, we have established the following fact:
\begin{proposition}\label{prop:skew_Cauchy}
	Let $u,v\in\C$ satisfy
	\begin{align}\label{admissible_good_condition_general}
		\lim_{L\to+\infty}\prod_{j=0}^{L}
		\left|\frac{\ip_j u-\SP_j}{1-\SP_j\ip_j u}
		\cdot
		\frac{\ip_j^{-1}v-\SP_j}{1-\SP_j\ip_j^{-1}v}\right|=0.
	\end{align}
	Then for any $\la,\nu\in\signpe$ we have
	\begin{align}\label{skew_Cauchy_good}
		\sum_{\kappa\in\signpe}
		\G^{\conj}_{\kappa/\la}(v\md\ipbb,\SPB)
		\F_{\kappa/\nu}(u\md\ipb,\SPB)
		= \frac{1-quv}{1-uv}
		\sum_{\mu\in\signpe}
		\F_{\la/\mu}(u\md\ipb,\SPB)
		\G^{\conj}_{\nu/\mu}(v\md\ipbb,\SPB)
		.
	\end{align}
\end{proposition}
\begin{proof}
	Indeed, this is just \eqref{skew_Cauchy_bad}
	under the replacement of
	$(u_1,u_2)$
	by $(u,v^{-1})$.
\end{proof}
Identity \eqref{skew_Cauchy_good} is nontrivial 
only if $\nu\in\signp N$ and $\la\in\signp{N+1}$.
In this case the sum in the right-hand side of \eqref{skew_Cauchy_good}
is over $\mu\in\signp N$ and is finite,
while in the left-hand side it
is over $\kappa\in\signp{N+1}$ and is infinite 
(but converges due to \eqref{admissible_good_condition_general}).

We will call \eqref{skew_Cauchy_good}
the \emph{skew Cauchy identity} for the symmetric functions 
$\F_{\mu/\la}$ and $\G_{\mu/\la}$ 
because of its similarity with the skew Cauchy
identities for the Schur, Hall--Littlewood, or 
Macdonald symmetric functions \cite[Ch.\;I.5, Ex.\;26, and Ch.\;VI.7, Ex.\;6]{Macdonald1995}. 
In fact, if $\ip_j\equiv 1$ and $\SP_j\equiv0$, our identity
\eqref{skew_Cauchy_good} becomes the skew Cauchy identity for the Hall--Littlewood
symmetric functions. Further letting $q\to0$,
we recover the Schur case.

When the parameters $\ip_j\equiv 1$ and $\SP_j\equiv \SP$ are constant, 
identity \eqref{skew_Cauchy_good} (and its 
corollaries below in this subsection) appeared in \cite{Borodin2014vertex}.
\begin{definition}\label{def:admissible}
	Let us say that 
	two complex numbers $u,v\in\C$ are \emph{admissible}
	with respect to the parameters $\ipb$ and $\SPB$, denoted $\adm uv$,
	if \eqref{admissible_good_condition_general} holds.
	A sufficient condition which implies \eqref{admissible_good_condition_general} is
	if for some $\epsilon\in(0,1)$, 
	\begin{align}
		\left|\frac{\ip_j u-\SP_j}{1-\SP_j\ip_j u}
		\cdot
		\frac{\ip_j^{-1}v-\SP_j}{1-\SP_j\ip_j^{-1}v}\right|<1-\epsilon
		\qquad\textnormal{for all $j=0,1,2,\ldots$}.
		\label{admissible_sufficient}
	\end{align}
	Note that the relation $\adm uv$ is \emph{not} symmetric in $u$ and $v$:
	$\adm uv\Leftrightarrow \admi vu$.
\end{definition}

The skew Cauchy identity can obviously be iterated with the following result:
\begin{corollary}\label{cor:skew_Cauchy_many}
	Let $u_1,\ldots,u_M$ and $v_1,\ldots,v_N$ be complex
	numbers such that $\adm {u_i}{v_j}$ for all
	$i=1,\ldots,M$ and $j=1,\ldots,N$.
	Then for any $\la,\nu\in\signpe$
	one has
	\begin{multline}\label{skew_Cauchy_many}
		\sum_{\kappa\in\signpe}
		\G^{\conj}_{\kappa/\la}(v_1,\ldots,v_N\md\ipbb,\SPB)
		\F_{\kappa/\nu}(u_1,\ldots,u_M\md\ipb,\SPB)
		\\=\prod_{i=1}^{M}\prod_{j=1}^{N}
		\frac{1-qu_iv_j}{1-u_iv_j}
		\sum_{\mu\in\signpe}
		\F_{\la/\mu}(u_1,\ldots,u_M\md\ipb,\SPB)
		\G^{\conj}_{\nu/\mu}(v_1,\ldots,v_N\md\ipbb,\SPB).
	\end{multline}
\end{corollary}
Furthermore, the skew Cauchy identity \eqref{skew_Cauchy_many} can be simplified
by specializing some of the indices. 
Recall the abbreviations $\G_\mu$ and $\F_\mu$ from Definitions \ref{def:G} and \ref{def:F}.
The identity of Corollary \ref{cor:skew_Cauchy_many}
readily implies the following facts:
\begin{corollary}
\label{cor:Pieri}
	\noindent{\rm{}\bf{}1.\/} For any $N\in\Z_{\ge0}$, $\la\in\signp N$, 
	and any complex $u_1,\ldots,u_N$ and $v$ such that $\adm{u_i}{v}$ for all $i$, 
	we have
	\begin{align}\label{Pieri1}
		\sum_{\kappa\in\signp{N}}
		\G^{\conj}_{\kappa/\la}(v\md\ipbb,\SPB)
		\F_{\kappa}(u_1,\ldots,u_N\md\ipb,\SPB)
		=\prod_{i=1}^{N}
		\frac{1-qu_iv}{1-u_iv}\,
		\F_{\la}(u_1,\ldots,u_N\md\ipb,\SPB).
	\end{align}
\smallskip

	\noindent{\rm{}\bf{}2.\/} For any $N,n\in\Z_{\ge0}$ any $\nu\in\signp N$,
	and any complex $u$ and $v_1,\ldots,v_n$ such that $\adm u{v_j}$ for all $j$, 
	we have
	\begin{align}\label{Pieri2}
		\sum_{\kappa\in\signp{N+1}}
		\G^{\conj}_{\kappa}(v_1,\ldots,v_n\md\ipbb,\SPB)
		\F_{\kappa/\nu}(u\md\ipb,\SPB)
		=
		\frac{1-q^{N+1}}{1-\SP_0\ip_0 u}
		\prod_{j=1}^{n}\frac{1-quv_j}{1-uv_j}\,
		\G^{\conj}_{\nu}(v_1,\ldots,v_n\md\ipbb,\SPB).
	\end{align}
\end{corollary}
\begin{proof}
	Identity \eqref{Pieri1} follows from \eqref{skew_Cauchy_many} by taking $\nu=\varnothing$ 
	and a single $v$-variable.
	Then the sum 
	over $\mu$ in the right-hand side of \eqref{skew_Cauchy_many} reduces to just $\mu=\varnothing$.
		
	Identity \eqref{Pieri2} follows by taking $\la=0^{N+1}$ 
	and a single $u$-variable
	in \eqref{skew_Cauchy_many}, and observing that
	$\F_{0^{N+1}/\mu}(u\md\ipb,\SPB)=\frac{1-q^{N+1}}{1-\SP_0\ip_0 u}
	\mathbf{1}_{\mu=0^N}$ by the very definition of $\F$.
\end{proof}
Identities \eqref{Pieri1} and \eqref{Pieri2}
are analogous to the \emph{Pieri rules}
for Schur, Hall--Littlewood, or Macdonald
symmetric functions \cite[Ch.\;I.5, formula\;(5.16), and Ch.\;VI.6]{Macdonald1995}.

\begin{remark}\label{rmk:Pieri_as_eigenrelations}
	Identity \eqref{Pieri1} shows that the functions 
	$\{\F_{\la}(u_1,\ldots,u_N\md\ipb,\SPB)\}_{\la\in\signp N}$
	for each set of the $u$'s form an eigenvector
	of the transfer matrix $\{\G_{\nu/\la}^{\conj}(v\md\ipbb,\SPB)\}_{\la,\nu\in\signp N}$
	viewed as acting in the spatial variables corresponding to signatures
	(i.e., with rows indexed by $\la$ and columns indexed by $\nu$).
	Equivalently, 
	$\{\F_{\la}^{\conj}(u_1,\ldots,u_N\md\ipb,\SPB)\}_{\la\in\signp N}$
	is an eigenvector of the transfer matrix 
	$\{\G_{\nu/\la}(v\md\ipbb,\SPB)\}_{\la,\nu\in\signp N}$ (i.e., the conjugation ``$\conj$'' can be moved).
	This statement is parallel (and simpler)
	to the fact that on a finite lattice,
	the vector
	$\BY(u_1\md\ipb,\SPB)\ldots\BY(u_n\md\ipb,\SPB)(\bv_{0}\otimes \ldots \otimes\bv_0)$
	is an eigenvector of the operator $\AY(v\md\ipb,\SPB)+\DY(v\md\ipb,\SPB)$
	given certain nonlinear \emph{Bethe equations} on $u_1,\ldots,u_N$.
	In our case the Bethe equations disappeared, and only one of the terms
	in $\AY(v\md\ipb,\SPB)+\DY(v\md\ipb,\SPB)$ has survived.

	One can also obtain analogous statements 
	when the number of $v$-variables in \eqref{Pieri1} 
	is greater than one --- this would correspond to applying a sequence of 
	transfer matrices with varying spectral parameters.
\end{remark}

Taking $\nu=\varnothing$ and $\la=0^{M}$ 
in \eqref{skew_Cauchy_many}
and noting that
\begin{align}\label{F_at_zero_signature}
	\F_{0^{M}}(u_1,\ldots,u_M\md\ipb,\SPB)=
	\frac{(q;q)_{M}}{\prod_{i=1}^{M}(1-\SP_0\ip_0u_i)},
\end{align}
we arrive at the following analogue of the 
\emph{usual} 
(\emph{non-skew}) \emph{Cauchy identity} 
(see \cite[Ch.\;I.4, formula\;(4.3), and Ch.\;VI.4, formula\;(4.13)]{Macdonald1995} for the corresponding Schur and Macdonald 
Cauchy identities):
\begin{corollary}\label{cor:usual_Cauchy}
	For $M,N\ge0$ and complex numbers 
	$u_1,\ldots,u_M$
	and $v_1,\ldots,v_N$
	such that
	$\adm{u_i}{v_j}$ for all $i$ and $j$, one has
	\begin{align}
		\sum_{\mu\in\signp M}
		\F_{\mu}(u_1,\ldots,u_M\md\ipb,\SPB)
		\G^{\conj}_{\mu}(v_1,\ldots,v_N\md\ipbb,\SPB)
		=
		\frac{(q;q)_{M}}{\prod_{i=1}^{M}(1-\SP_0\ip_0u_i)}
		\prod_{i=1}^{M}\prod_{j=1}^{N}
		\frac{1-qu_iv_j}{1-u_iv_j}
		.
		\label{usual_Cauchy}
	\end{align}
\end{corollary}


\subsection{Symmetrization formulas} 
\label{sub:symmetrization_formulas}

So far, our definition of the symmetric functions 
$\F_{\mu/\la}$ and $\G_{\nu/\la}$ was not too explicit --- they were defined
as large sums over all possible path collections with certain boundary conditions (see Definitions \ref{def:G} and \ref{def:F}).
However, it turns out that the non-skew symmetric functions $\F_\mu$ and $\G_\nu$
can be evaluated more explicitly.

We will need some notation. Set
\begin{align}\label{pow_function}
	\pow_k(u)&=\pow_k(u\md\ipb,\SPB):=
	\frac{1-q}{1-\SP_k\ip_k u}\prod_{j=0}^{k-1}\frac{\ip_j u-\SP_j}{1-\SP_j\ip_ju},\qquad k\ge0.
\end{align}
By agreement, for $k=0$ the empty product in \eqref{pow_function}
is equal to $1$.
Note that $\pow_k(u)$ is equal to $\F_{(k)}(u\md\ipb,\SPB)$, where $(k)$ is the 
signature with a single part equal to $k$.
Indeed, for such a signature the path collection of Definition \ref{def:F}
consists of a single path whose weight is \eqref{pow_function},
cf. \eqref{weights}.

\begin{theorem}\label{thm:symmetrization}
	\noindent{\bf1.\/}
	For any $M\ge0$, any $\mu\in\signp M$,
	and any $u_1,\ldots,u_M\in\C$ we 
	have\footnote{In both formulas 
	\eqref{F_symm_formula}
	and \eqref{G_symm_formula} the 
	permutation $\sigma$ (belonging, respectively, to $\Sym_M$ or $\Sym_N$) acts by permuting the 
	indeterminates $u_i$ or $v_j$, respectively.
	The same convention is 
	used throughout the text.}
	\begin{align}\label{F_symm_formula}
		\F_\mu(u_1,\ldots,u_M\md\ipb,\SPB)=
		\sum_{\sigma\in\Sym_M}
		\sigma\Bigg(
		\prod_{1\le \aind<\bind\le M}\frac{u_\aind-qu_\bind}{u_\aind-u_\bind}
		\prod_{i=1}^{M}
		\pow_{\mu_i}(u_i\md\ipb,\SPB)\Bigg).
	\end{align}
	
\smallskip
	\noindent{\bf2.\/}
	For any $n\ge0$, $\nu\in\signp n$,
	let $k$ be the number of 
	zero coordinates in $\nu$, 
	i.e., $\nu_{n-k+1}=\ldots=\nu_n=0$.
	Then for any $N\ge n-k$
	and any
	$v_1,\ldots,v_N\in\C$ we have
	\begin{multline}\label{G_symm_formula}
		\G_\nu(v_1,\ldots,v_N\md\ipb,\SPB)=
		\frac{(\SP_0^{2};q)_{n}}{(q;q)_{N-n+k}(\SP_0^{2};q)_{k}}
		\\\times
		\sum_{\sigma\in \Sym_N}
		\sigma\Bigg(
		\prod_{1\le \aind<\bind\le N}\frac{v_\aind-qv_\bind}{v_\aind-v_\bind}
		\prod_{j=1}^{N}\pow_{\nu_j}(v_j\md\ipb,\SPB)
		\prod_{i=1}^{n-k}
		\frac{\ip_0v_i}{\ip_0v_i-\SP_0}
		\prod_{j=n-k+1}^{N}
		({1-\SP_0 q^{k}\ip_0 v_j})
		\Bigg).
	\end{multline}
	By agreement, if needed to make sense of the expressions
	$\pow_{\nu_j}(v_j)$ for $j>n$, 
	the signature $\nu$ is appended by zeros.
	If $N<n-k$, the function $\G_\nu(v_1,\ldots,v_N\md\ipb,\SPB)$
	vanishes for trivial reasons.
\end{theorem}
When $\ip_j\equiv 1$ and $\SP_j\equiv\SP$, 
this theorem was established in \cite{Borodin2014vertex}.
Here we present a different proof which involves 
the operators $\AY,\BY,\CY,\DY$ from 
\S \ref{sub:yang_baxter_relation_in_operator_language},
and closely follows the algebraic Bethe ansatz 
framework
\cite{FelderVarchenko1996},
\cite{QISM_book}. 
Let us first discuss certain straightforward corollaries of Theorem \ref{thm:symmetrization}. We will denote by $\sh_r$, $r\in\Z_{\ge0}$, 
the \emph{shift operation} applied to
the sequence $\ipb$ or $\SPB$:
\begin{align}\label{sh_operation}
	(\sh_r \ipb)_j:=\ip_{j+r},\qquad
	(\sh_r \SPB)_j:=\SP_{j+r}.
\end{align}
Also, for $\mu\in\signp M$, let $\mu+r^{M}$ denote the shifted
signature $(\mu_1+r,\mu_2+r,\ldots,\mu_M+r)$.

\begin{corollary}\label{cor:shifts}
	\noindent{\bf1.\/} For any $\mu\in\signp M$ and any $r\in\Z_{\ge0}$ one has
	\begin{align}\label{F_shifts}
		\F_{\mu+r^{M}}(u_1,\ldots,u_M\md\ipb,\SPB)=
		\bigg(\prod_{i=1}^{M}\prod_{j=0}^{r-1}\frac{\ip_ju_i-\SP_j}
		{1-\SP_j\ip_j u_i}\bigg)		
		\F_{\mu}(u_1,\ldots,u_M\md\sh_r\ipb,\sh_r\SPB).
	\end{align}
\smallskip

	\noindent{\bf2.\/}
	For any $\nu\in\signp N$ with $\nu_N\ge1$ one has
	\begin{align}\label{G_via_F_shifts}
		\G_\nu(v_1,\ldots,v_N\md\ipb,\SPB)
		=(\SP_0^{2};q)_{N}\bigg(\prod_{i=1}^{N}
		\frac{\ip_0v_i}{\ip_0v_i-\SP_0}\bigg)
		\F_\nu(v_1,\ldots,v_N\md\ipb,\SPB).
	\end{align}
	That is, when $k=0$ and $N=n$ in \eqref{G_symm_formula}, the function
	$\G_\nu(v_1,\ldots,v_N\md\ipb,\SPB)$ almost coincides 
	with $\F_{\nu}$.
\end{corollary}
\begin{proof}
	A straightforward 
	verification using \eqref{F_symm_formula} and \eqref{G_symm_formula}.
	Alternatively, the claims immediately follow
	from the definitions of the functions $\F$ and $\G$
	as partition functions of path collections
	(Definitions \ref{def:F} and \ref{def:G}).
\end{proof}
The next corollary utilizes the explicit formulas \eqref{F_symm_formula}
and \eqref{G_symm_formula} in an essential way:
\begin{corollary}\label{cor:prin_spec}
	\noindent{\bf1.\/}
	For any $M\ge0$, $\mu\in\signp M$, and $u\in\C$ we have
	\begin{align}\label{F_prin_spec_simple}
		\F_\mu(u,qu,\ldots,q^{M-1}u\md\ipb,\SPB)=
		(q;q)_{M}
		\prod_{i=1}^{M}\bigg(
		\frac{1}{1-\SP_{\mu_i}\ip_{\mu_i} q^{i-1}u}
		\prod_{j=0}^{{\mu_i}-1}\frac{\ip_j q^{i-1}u-\SP_j}{1-\SP_j\ip_jq^{i-1}u}
		\bigg).
	\end{align}
\smallskip

	\noindent{\bf2.\/}
	For any $n\ge0$ and $\nu\in\signp n$
	with $k$ zero coordinates,
	any $N\ge n-k$, and any $v\in\C$ we have
	\begin{multline}\label{G_prin_spec_simple}
		\G_\nu(v,qv,\ldots,q^{N-1}v\md\ipb,\SPB)=
		\frac{(q;q)_{N}}{(q;q)_{N-n+k}}
		\frac{(\SP_0\ip_0v;q)_{N+k}}{(\SP_0\ip_0v;q)_{n}}
		\frac{(\SP_0^{2};q)_{n}}{(\SP_0^{2};q)_{k}}
		\frac{1}{(\SP_0/(\ip_0v);q^{-1})_{n-k}}
		\\\times\prod_{j=1}^{N}
		\bigg(
		\frac{1}{1-\SP_{\nu_j}\ip_{\nu_j} q^{j-1}v}
		\prod_{\ell=0}^{{\nu_j}-1}
		\frac{\ip_\ell q^{j-1}v-\SP_\ell}{1-\SP_\ell\ip_\ell q^{j-1}v}
		\bigg).
	\end{multline}
\end{corollary}
Substituting a geometric sequence with ratio $q$
into a function 
$\F$ or $\G$ will be referred to as the 
\emph{principal specialization} of these symmetric functions.
\begin{proof}
	The substitutions of geometric sequences 
	into $\F$ or $\G$ make all terms except the one with 
	$\sigma=\mathrm{id}$ vanish due to the presence
	of the cross term 
	$\sigma\Big(\prod_{1\le\aind<\bind\le M}
	\frac{u_\aind-qu_\bind}{u_\aind-u_\bind}\Big)$.
	For $\sigma=\mathrm{id}$ this cross term is equal to 
	$(q;q)_M/(1-q)^M$.
	The rest is obtained in a straightforward way 
	by evaluating the remaining parts of the formulas.
\end{proof}

The proof of Theorem \ref{thm:symmetrization} occupies the rest of this subsection.
\begin{proof}[Proof of \eqref{F_symm_formula}]
\textbf{Step 1.}
To obtain an explicit formula for 
$\F_\mu(u_1,\ldots,u_M\md\ipb,\SPB)$,
we need to understand how 
the operator
$\BY(u_1\md\ipb,\SPB)\ldots\BY(u_M\md\ipb,\SPB)$
acts on the vector
$(\bv_{0}\otimes \bv_{0}\otimes\ldots)$. 
Let us first consider
what happens in the physical space 
containing just two tensor factors,
which puts us into the setting described in \S \ref{sub:attaching_vertical_columns}.
Let the inhomogeneity parameters in this setting be denoted by $(\ip_1,\ip_2)$, and the $\SP$--parameters be
$(\SP_1,\SP_2)$, as usual.
We have from \eqref{T_V1V2}:
\begin{align}\label{B_two_factor_tensor_F_proof}
	\BY(u\md\ipb,\SPB)=
	\BY_1(\ip_1 u\md\SP_1)\AY_2(\ip_2 u\md\SP_2)+\DY_1(\ip_1 u\md\SP_1)\BY_2(\ip_2 u\md\SP_2),
\end{align}
where the lower indices in the operators 
in the right-hand side stand for
the spaces in which they act (and also determine which of the parameters
$\SP_1$ or $\SP_2$ we take). 
The operators in the right-hand side 
act as in \eqref{ABCD_elementary_operators}.
Recall that any two operators with different lower indices commute.

When we multiply together a number of operators $\BY(u\md\ipb,\SPB)$
(with different spectral $u$-parameters)
and open the parentheses, we collect several 
factors $\BY_1$ and $\DY_1$, 
and several other factors
$\AY_2$ and $\BY_2$. Using the Yang--Baxter commutation
relations \eqref{YB_relationBD} and \eqref{YB_relationAB}, 
we can swap these operators
at the expense of picking certain prefactors, 
and also this swapping of operators could lead to an exchange of their spectral parameters.
Therefore, we can write 
$\BY(u_1\md\ipb,\SPB)\ldots\BY(u_M\md\ipb,\SPB)(\bv_{0}\otimes \bv_{0})$
as a linear combination of vectors of the form
\begin{multline}\label{vectors_of_the_form_F_proof}
	\BY_1(\ip_1 u_{k_1}\md\SP_1)\ldots
	\BY_1(\ip_1 u_{k_{M-s}}\md\SP_1)
	\DY_1(\ip_1 u_{\ell_1}\md\SP_1)\ldots
	\DY_1(\ip_1 u_{\ell_s}\md\SP_1)\,\bv_0\\\otimes
	\BY_2(\ip_2 u_{i_1}\md\SP_2)\ldots
	\BY_2(\ip_2 u_{i_s}\md\SP_2)
	\AY_2(\ip_2 u_{j_1}\md\SP_2)\ldots
	\AY_2(\ip_2 u_{j_{M-s}}\md\SP_2)\,\bv_0,
\end{multline}
with
\begin{align*}
	\begin{array}{lll}
		\IS=\{i_1<\ldots<i_s\},&\quad \JS=\{j_1<\ldots<j_{M-s}\},&\quad
		\IS\sqcup\JS=\{1,\ldots,M\},\\\rule{0pt}{14pt}
		\KS=\{k_1<\ldots<k_{M-s}\},&\quad \LS=\{\ell_1<\ldots<\ell_{s}\},&\quad
		\KS\sqcup\LS=\{1,\ldots,M\}.
	\end{array}
\end{align*}

\smallskip\noindent\textbf{Step 2.}
The coefficients of the vectors \eqref{vectors_of_the_form_F_proof} 
are computed using only the commutation relations 
\eqref{YB_relationBD} and \eqref{YB_relationAB},
and we argue that these coefficients 
\emph{do not depend} on how exactly we apply the commutation relations
to reach the result.
This property is based on the fact that for generic spectral parameters, there exists a representation
of $\begin{bmatrix}
	\AY(u)&\BY(u)\\\CY(u)&\DY(u)
\end{bmatrix}$
subject to the same commutation relations, and a \emph{highest weight vector} $\mathsf{v}_0$
in that representation,\footnote{Meaning that $\mathsf{v}_0$ is annihilated by $\CY(u)$ and is an eigenvector
for $\AY(u)$ and $\DY(u)$.} such that 
vectors $\Big(\prod_{j\in\JS}\BY(u_j)\Big)\mathsf{v}_0$, with $\JS$ ranging over all subsets of $\{1,2,\ldots,M\}$,
are linearly independent.
This is shown in \cite[Lemma\;14]{FelderVarchenko1996}, and we will not repeat the argument here.

Knowing this fact, if we have two ways of applying the commutation relations 
which yield different coefficients of the
vectors \eqref{vectors_of_the_form_F_proof},
then we can apply these commutation relations in the above highest weight representation,
which leads to a contradiction with the linear independence property.\footnote{Note that we perform the commutations in each of the two tensor factors
separately, and thus the statement that the coefficients are uniquely determined 
is not affected by the presence of the parameters
$\ipb$ and $\SPB$.}

\smallskip\noindent\textbf{Step 3.}
Our next goal is to show that the 
coefficient of each vector of the form \eqref{vectors_of_the_form_F_proof}
vanishes unless $\IS\cap\KS=\varnothing$. We argue by induction on $M$.
For $M=1$, the application of the operator \eqref{B_two_factor_tensor_F_proof}
(with $u=u_1$)
to $\bv_0\otimes\bv_0$
obviously has this property.
When we apply the next operator 
$\BY(u_2\md\ipb,\SPB)$, we see that 
the sets $\IS$ and $\KS$ could grow by the element $2$,
and that they can also lose 
the element $1$
in the process of commuting the 
$\DY$'s and the $\AY$'s to the right.
However, the sets $\IS$ and $\KS$
cannot gain the element $1$. This means that $1\notin \IS\cap\KS$.
However, we could have applied  
$\BY(u_1\md\ipb,\SPB)\BY(u_2\md\ipb,\SPB)=\BY(u_2\md\ipb,\SPB)\BY(u_1\md\ipb,\SPB)$
in the opposite order, which implies (by the uniqueness of the coefficients)
that $2\notin\IS\cap\KS$. Therefore, 
$\IS\cap\KS=\varnothing$ for $M=2$. Clearly, we can continue this argument
with more factors in the same way, and conclude that 
$\IS\cap\KS=\varnothing$ for any $M$. 

\smallskip\noindent\textbf{Step 4.}
Since $\IS\sqcup\JS=\KS\sqcup\LS=\{1,\ldots,M\}$, we see that $\IS=\LS$ and $\KS=\JS$.
This implies that the desired action of a product of the $\BY$ operators takes the form
\begin{multline}
	\BY(u_1\md\ipb,\SPB)\ldots
	\BY(u_M\md\ipb,\SPB)(\bv_0\otimes\bv_0)\\=
	\sum_{\KS\subseteq\{1,2,\ldots,M\}}
	C_{\KS}
	\bigg(
	\prod_{k\in\KS}\BY_1(\ip_1 u_k\md\SP_1)
	\prod_{\ell\notin\KS}\DY_1(\ip_1 u_\ell\md\SP_1)
	\bigg)\,\bv_0\otimes
	\bigg(
	\prod_{\ell\notin\KS}\BY_2(\ip_2 u_\ell\md\SP_2)
	\prod_{k\in\KS}\AY_2(\ip_2 u_k\md\SP_2)
	\bigg)\,\bv_0,
	\label{desired_action_F_proof}
\end{multline}
with some uniquely defined coefficients 
$C_\KS(u_1,\ldots,u_M)$, where
$\KS\subseteq\{1,\ldots,M\}$.

Now, since we obviously can permute the spectral parameters
$u_j$ without changing the desired action \eqref{desired_action_F_proof},
by uniqueness of the coefficients
we must have
\begin{align*}
	C_{\KS}(u_{\sigma(1)},\ldots,u_{\sigma(M)})
	=C_{\sigma(\KS)}(u_1,\ldots,u_M)
	\qquad \textnormal{for all $\sigma\in\Sym_M$}.
\end{align*}
Thus, it suffices to compute these coefficients
for $\KS=\{1,2,\ldots,r\}$ for each $r=1,2,\ldots,M$. This can be done by simply opening the 
parentheses in
\begin{multline}
	\big(
		\BY_1(\ip_1 u_1\md\SP_1)\AY_2(\ip_2 u_1\md\SP_2)+\DY_1(\ip_1 u_1\md\SP_1)\BY_2(\ip_2 u_1\md\SP_2)
	\big)\ldots
	\\\cdots
	\big(
		\BY_1(\ip_1 u_M\md\SP_1)\AY_2(\ip_2 u_M\md\SP_2)+\DY_1(\ip_1 u_M\md\SP_1)\BY_2(\ip_2 u_M\md\SP_2)
	\big),
	\label{opening_parentheses_F_proof}
\end{multline}
because the only way to end up with 
the vector
\begin{multline*}
	\BY_1(\ip_1u_1\md\SP_1)\cdots\BY_1(\ip_1u_r\md\SP_1)
	\DY_1(\ip_1u_{r+1}\md\SP_1)\cdots\DY_1(\ip_1u_{M}\md\SP_1)\,\bv_0
	\\\otimes
	\BY_2(\ip_2 u_{r+1}\md\SP_2)\cdots\BY_2(\ip_2 u_{M}\md\SP_2)
	\AY_2(\ip_2 u_1\md\SP_2)\cdots\AY_2(\ip_2 u_r\md\SP_2)\,\bv_0
\end{multline*}
is to use the first summand in \eqref{opening_parentheses_F_proof}
for $j=1,\ldots,r$, the second summand 
for $j=r+1,\ldots,M$, and commute all the $\AY_2$'s through the $\BY_2$'s
without swapping the spectral parameters.
From \eqref{YB_relationAB} we readily have
\begin{align}
	\AY(w_1)\BY(w_2)=
	\frac{w_2-qw_1}{w_2-w_1}\,
	\BY(w_2)\AY(w_1)
	-
	\frac{(1-q)w_1}{w_2-w_1}\,
	\BY(w_1)\AY(w_2)
	,\label{AB=BA_BA_YB_relation}
\end{align}
and we are only interested in the first summand above.
Our commutations thus give the coefficient
\begin{align*}
	C_{\{1,2,\ldots,r\}}(u_1,\ldots,u_M)=\prod_{\aind=1}^{r}\prod_{\bind=r+1}^{M}\frac{\ip_2u_{\bind}-q\ip_2u_{\aind}}{\ip_2u_{\bind}-\ip_2u_{\aind}}
	=\prod_{\aind=1}^{r}\prod_{\bind=r+1}^{M}\frac{u_{\bind}-qu_{\aind}}{u_{\bind}-u_{\aind}},
\end{align*}
and so we have
\begin{multline}
	\BY(u_1\md\ipb,\SPB)\ldots
	\BY(u_M\md\ipb,\SPB)(\bv_0\otimes\bv_0)
	\\=\sum_{\KS\subseteq\{1,\ldots,M\}}
	\prod_{\substack{\aind\in\KS\\\bind\notin \KS}}
	\frac{u_{\bind}-qu_{\aind}}{u_{\bind}-u_{\aind}}
	\bigg(
	\prod_{k\in\KS}\BY_1(\ip_1 u_k\md\SP_1)
	\prod_{\ell\notin\KS}\DY_1(\ip_1 u_\ell\md\SP_1)
	\bigg)\,\bv_0\\\otimes
	\bigg(
	\prod_{\ell\notin\KS}\BY_2(\ip_2 u_\ell\md\SP_2)
	\prod_{k\in\KS}\AY_2(\ip_2 u_k\md\SP_2)
	\bigg)\,\bv_0.\label{pre_final_two_factors_F_proof}
\end{multline}
Recall that $\bv_0$ is an eigenvector for $\DY_1$ and $\AY_2$, and
introduce the notation $\ay_{1,2}$ and $\dy_{1,2}$ by
\begin{align}
	\DY_{j}(\ip_j u\md\SP_j)\,\bv_0=\dy_{j}(\ip_ju\md\SP_j)\,\bv_0,
	\qquad
	\AY_{j}(\ip_j u\md\SP_j)\,\bv_0=\ay_{j}(\ip_ju\md\SP_j)\,\bv_0,
	\qquad j=1,2.
	\label{ay_dy_notation}
\end{align}
Thus, $\ay_{1,2}$ and $\dy_{1,2}$
are eigenvalues (scalars).\footnote{
When $\bv_0$ is the highest weight vector in the 
representation $V$ of \S \ref{sub:yang_baxter_relation_in_operator_language}, these eigenvalues
can be read off \eqref{ABCD_elementary_operators}.
However, in Step 5 below we will use
notation \eqref{ay_dy_notation}
for highest weight vectors of representations obtained
by tensoring several such $V$'s.}
Hence our final result \eqref{pre_final_two_factors_F_proof} 
for two tensor factors can be rewritten in the following form:
\begin{align}
	\BY(u_1\md\ipb,\SPB)\ldots
	\BY(u_M\md\ipb,\SPB)(\bv_0\otimes\bv_0)
	=\sum_{\KS\subseteq\{1,\ldots,M\}}
	\dy_1(\KSB)\ay_2(\KS)
	\prod_{\substack{\aind\in\KS\\\bind\notin \KS}}
	\frac{u_{\bind}-qu_{\aind}}{u_{\bind}-u_{\aind}}\,
	(\BY_1(\KS)\,\bv_0)\otimes
	(\BY_2(\KSB)\,\bv_0),
	\label{final_two_factors_F_proof}
\end{align}
where we have abbreviated
\begin{align}\label{abbreviation_F_proof}
	\KSB:=\{1,\ldots,M\}\setminus\KS,\qquad
	f_j(\KS):=\prod_{k\in\KS}f_j(\ip_ju_k\md\SP_j),\qquad j=1,2,
\end{align}
so $\dy_1$ and $\BY_1$ include the parameters $\ip_1$ and $\SP_1$,
and $\ay_2$ and $\BY_2$ contain $\ip_2$ and $\SP_2$.

\smallskip\noindent\textbf{Step 5.}
In this form the 
formula \eqref{final_two_factors_F_proof} 
for two tensor factors
can be immediately extended to arbitrarily many tensor factors.
Indeed, let us think of the second vector $\bv_0$
as $\tilde\bv_0\otimes\tilde\bv_0$. Then we can use 
\eqref{final_two_factors_F_proof}
to evaluate $\BY_2(\KSB)\,\bv_0=\BY_2(\KSB)(\tilde\bv_0\otimes\tilde\bv_0)$,
split the second $\tilde \bv_0$ again, and so on. 

Therefore, we obtain the 
final formula for the action of
$\BY(u_1\md\ipb,\SPB)\ldots\BY(u_M\md\ipb,\SPB)$
on the vector
$(\bv_{0}\otimes \bv_{0}\otimes\ldots)$:
\begin{multline}
	\BY(u_1\md\ipb,\SPB)\ldots\BY(u_M\md\ipb,\SPB)
	(\bv_{0}\otimes \bv_{0}\otimes\ldots)
	\\=\sum_{\substack{\KS_0,\KS_1,\ldots\subseteq\{1,\ldots,M\}\\
	\KS_0\sqcup\KS_1\sqcup \ldots
	=\{1,2,\ldots,M\}}}
	\prod_{0\le i<j}\dy_i(\KS_j)\ay_j(\KS_i)
	\prod_{\substack{\aind\in\KS_i\\\bind\in \KS_j}}
	\frac{u_{\bind}-qu_{\aind}}{u_{\bind}-u_{\aind}}
	\Big(
	\BY_0(\KS_0)\,\bv_0\otimes
	\BY_1(\KS_1)\,\bv_0\otimes \ldots
	\Big),
	\label{final_many_factors_F_proof}
\end{multline}
where
from now on
we denote
the inhomogeneity parameters
and the $\SP$--parameters
by
$\ip_0,\ip_1,\ip_2,\ldots$
and 
$\SP_0,\SP_1,\SP_2,\ldots$,
respectively
(as in \eqref{F_symm_formula}).

To finish the derivation of \eqref{F_symm_formula}, we need to recall the
action \eqref{ABCD_elementary_operators}
of the operators $\AY,\BY$, and $\DY$
in the ``elementary'' physical space $\Span\{\bv_i\colon i=0,1,2,\ldots\}$.
We have 
\begin{align}\nonumber
	\ay_j(\KS)&=1
	,\qquad \qquad
	\dy_j(\KS)=\prod_{k\in\KS}
	\frac{\ip_ju_k-\SP_j}{1-\SP_j\ip_ju_k},
	\\
	\BY_j(\KS)\,\bv_0&=
	\frac{(q;q)_{|\KS|}}{\prod_{k\in\KS}(1-\SP_j\ip_j u_k)}\,\bv_{|\KS|}
	=
	\frac{(1-q)^{|\KS|}}{\prod_{k\in\KS}(1-\SP_j\ip_j u_k)}
	\bigg(\sum_{\sigma\in\Sym(\KS)}
	\sigma\Big(\prod_{\substack{\aind<\bind\\\aind,\bind\in\KS}}
	\frac{u_\aind-qu_\bind}{u_\aind-u_\bind}\Big)
	\bigg)\,
	\bv_{|\KS|}
	,\label{BjK_cross_F_proof}
\end{align}
where for $\BY_j(\KS)$ we have used the
symmetrization formula
\cite[Ch.\;III.1, formula\;(1.4)]{Macdonald1995}\footnote{That is, for any $r\in\Z_{\ge1}$,
we have
$\displaystyle \sum_{\omega\in\Sym_r}\prod_{1\le \aind<\bind\le r}
\frac{u_{\omega(\aind)}-q u_{\omega(\bind)}}
{u_{\omega(\aind)}-u_{\omega(\bind)}}=\frac{(q;q)_{r}}{(1-q)^{r}}$.\label{symm_footnote}}
to insert an additional sum over permutations
of $\KS$ (here $\Sym(\KS)$ denotes the group of permutations of $\KS$,
and $\sigma$ acts by permuting the corresponding variables).

To read off the coefficient of
$\bv_\mu=\bv_{m_0}\otimes \bv_{m_1}\otimes \bv_{m_2}\otimes \ldots$,
$\mu\in\signp M$, in \eqref{final_many_factors_F_proof}, we must have
$|\KS_i|=m_i$ for all $i\ge0$.
Let us fix one such partition 
$\KS_0\sqcup\KS_1\sqcup \ldots=\{1,\ldots,M\}$.
For each $\aind\in\{1,\ldots,M\}$, let $\ks(\aind)\in\Z_{\ge0}$ 
denote the number $j$ such that $\aind\in\KS_j$.
Then we can write
\begin{align*}
	\prod_{0\le i<j}\dy_i(\KS_j)
	=
	\prod_{r=1}^{M}\prod_{\ell=0}^{\ks(r)-1}
	\frac{\ip_\ell u_r-\SP_\ell}{1-\SP_\ell\ip_\ell u_r},
\end{align*}
which, combined with the factors
$\dfrac{(1-q)^{|\KS_{j}|}}{\prod_{k\in\KS_{j}}(1-\SP_j\ip_j u_k)}$
coming from $\BY_j(\KS_{j})\,\bv_0$,
produces
$\displaystyle\prod_{i=1}^{M}\pow_{\ks(i)}(u_i)$.
Note that this product does
not change if we permute 
the $u_i$'s within the sets $\KS_j$.
Furthermore,
we can also write
\begin{align}
	\prod_{0\le i<j}
	\prod_{\substack{\aind\in\KS_i\\\bind\in \KS_j}}
	\frac{u_{\bind}-qu_{\aind}}{u_{\bind}-u_{\aind}}
	=
	\prod_{\substack{1\le \aind,\bind\le M\\\ks(\aind)<\ks(\bind)}}
	\frac{u_{\bind}-qu_{\aind}}{u_{\bind}-u_{\aind}}.
	\label{dj_cross_F_proof}
\end{align}
We then combine this with the remaining coefficients 
coming from $\BY_j(\KS_{j})\,\bv_0$
which involve summations over
permutations within the sets $\KS_j$, and 
compare the result with the desired formula \eqref{F_symm_formula}.

Clearly, fixing a partition into the $\KS_j$'s
corresponds to considering
only permutations $\sigma\in\Sym_M$ in \eqref{F_symm_formula}
which place each $i\in\{1,\ldots,M\}$ into $\KS_{\ks(i)}$.
This is the mechanism which gives rise 
to the summations
over
permutations within the sets $\KS_j$ as in \eqref{BjK_cross_F_proof}.
One can readily check that the summands agree, and thus \eqref{F_symm_formula} 
is established.
\end{proof}

\begin{remark}
	Formula \eqref{F_symm_formula} that we just established
	links the algebraic and the coordinate Bethe ansatz.
	Its proof given above closely follows 
	the proof of Theorem 5 in Section 8 of
	\cite{FelderVarchenko1996}.
	The key relation \eqref{final_many_factors_F_proof} without proof can be found in 
	\cite[Appendix VII.2]{QISM_book}.
\end{remark}

\begin{proof}[Proof of \eqref{G_symm_formula}]
\textbf{Step 1.}
We will use the same approach as in the proof of \eqref{F_symm_formula}
to get an explicit formula for 
$\G_\nu(v_1,\ldots,v_N\md\ipb,\SPB)$.
That is, we need to compute
$\AY(v_1\md\ipb,\SPB)\ldots\AY(v_N\md\ipb,\SPB)(\bv_{n}\otimes \bv_{0}\otimes \bv_{0}\otimes\ldots)$.
We start with just two tensor factors, and consider the application of 
this operator to $\bv_n\otimes\bv_0$. 
After that we will use \eqref{F_symm_formula} to turn the second $\bv_0$
into $\bv_0\otimes\bv_0\otimes \ldots$.

For two tensor factors we have from \eqref{T_V1V2}:
\begin{align}\label{AY_two_factor_G_proof}
	\AY(v\md\ipb,\SPB)=\AY_1(\ip_1 v\md\SP_1)\AY_2(\ip_2 v\md\SP_2)+\CY_1(\ip_1 v\md\SP_1)\BY_2(\ip_2 v\md\SP_2).
\end{align}
Taking the product $\AY(v_1\md\ipb,\SPB)\ldots\AY(v_N\md\ipb,\SPB)$ and opening the parentheses,
we can use the commutation relations \eqref{YB_relationAC} and \eqref{YB_relationAB} to 
express the result as a linear combination of vectors of the form
\begin{multline}\label{vectors_of_the_form_G_proof}
	\AY_1(\ip_1 v_{k_1}\md\SP_1)\ldots
	\AY_1(\ip_1 v_{k_{N-s}}\md\SP_1)
	\CY_1(\ip_1 v_{\ell_1}\md\SP_1)\ldots
	\CY_1(\ip_1 v_{\ell_s}\md\SP_1)\,\bv_n\\\otimes
	\BY_2(\ip_2 v_{i_1}\md\SP_2)\ldots
	\BY_2(\ip_2 v_{i_s}\md\SP_2)
	\AY_2(\ip_2 v_{j_1}\md\SP_2)\ldots
	\AY_2(\ip_2 v_{j_{N-s}}\md\SP_2)\,\bv_0,
\end{multline}
with 
\begin{align*}
	\begin{array}{lll}
		\IS=\{i_1<\ldots<i_s\},&\quad \JS=\{j_1<\ldots<j_{N-s}\},&\quad
		\IS\sqcup\JS=\{1,\ldots,N\},\\\rule{0pt}{14pt}
		\KS=\{k_1<\ldots<k_{N-s}\},&\quad \LS=\{\ell_1<\ldots<\ell_{s}\},&\quad
		\KS\sqcup\LS=\{1,\ldots,N\}.
	\end{array}
\end{align*}

\smallskip\noindent\textbf{Step 2.}
Again, the key point is that the 
coefficients by vectors of the form \eqref{vectors_of_the_form_G_proof}
are uniquely determined by the commutation relations, and do not depend on the order of commuting.
The uniqueness argument here is very similar
to the one in Step 2 of the proof of \eqref{F_symm_formula}, and we will 
not repeat it.

\smallskip\noindent\textbf{Step 3.}
We now observe that we must have $\IS=\LS$ and $\JS=\KS$.
Indeed, let us show that $\IS\cap\KS=\varnothing$, which would imply the claim. 
We argue by induction. The case of $N=1$ is obvious. 
When we then apply the next operator
$\AY(v_2\md\ipb,\SPB)$ to \eqref{AY_two_factor_G_proof}
(with $v=v_1$)
and use the commutation relations to write all vectors in the required form
\eqref{vectors_of_the_form_G_proof},
neither $\IS$ nor $\KS$ can gain index $1$, 
exactly in the same way as 
in Step 3 of the proof of \eqref{F_symm_formula}. The fact that
$1\notin\IS\cap\KS$ does not change after we apply all other operators
$\AY(v_j\md\ipb,\SPB)$, $j=3,\ldots,N$. Since the order of factors
in 
$\AY(v_1\md\ipb,\SPB)\ldots\AY(v_N\md\ipb,\SPB)$ does not matter, we conclude that $\IS\cap\KS=\varnothing$
for any $N$.

\smallskip\noindent\textbf{Step 4.}
We thus conclude that 
\begin{align}\label{step4_action_of_A_G_proof}
	\AY(v_1\md\ipb,\SPB)\ldots\AY(v_N\md\ipb,\SPB)
	(\bv_n\otimes\bv_0)
	=
	\sum_{\KS\subseteq\{1,2,\ldots,N\}}
	C_\KS
	\big(\AY_1(\KS)\CY_1(\KSB)\big)\,\bv_n\otimes
	\big(\BY_2(\KSB)\AY_2(\KS)\big)\,\bv_0,
\end{align}
where we are using the abbreviation \eqref{abbreviation_F_proof}.
Here the coefficients $C_\KS$
are uniquely determined,
and satisfy
\begin{align*}
	C_{\KS}(v_{\sigma(1)},\ldots,v_{\sigma(N)})
	=C_{\sigma(\KS)}(v_1,\ldots,v_N)
	\qquad \textnormal{for all $\sigma\in\Sym_N$}.
\end{align*}
Thus, we need to compute only the coefficients $C_\KS$
for $\KS=\{r+1,\ldots,N\}$, where $r=1,2,\ldots,N$. 
Since $\nu\in\signp n$ has exactly $k$ zero coordinates
(see \eqref{G_symm_formula}), we must have $|\KSB|=r=n-k$. 
Indeed, this is because $\CY_1(\KSB)$ is responsible for 
moving some of $n$ arrows to the right from the location $0$.

The coefficients $C_{\{r+1,\ldots,N\}}$ can thus 
be computed by simply opening the 
parentheses in
\begin{multline*}
	\big(\AY_1(\ip_1 v_N\md\SP_1)\AY_2(\ip_2 v_N\md\SP_2)+
	\CY_1(\ip_1 v_N\md\SP_1)\BY_2(\ip_2 v_N\md\SP_2)\big)\ldots
	\\\ldots
	\big(\AY_1(\ip_1 v_1\md\SP_1)\AY_2(\ip_2 v_1\md\SP_2)+
	\CY_1(\ip_1 v_1\md\SP_1)\BY_2(\ip_2 v_1\md\SP_2)\big)
	,
\end{multline*}
and noting that there is a unique way of reaching $\KS=\{r+1,\ldots,N\}$: 
pick the first summands in the first $N-r=N-n+k$ factors, 
the second summands in the last $r= n-k$ factors,
and after that move $\AY_2(\KS)$ to the right of $\BY_2(\KSB)$
without swapping the spectral parameters in the process of commuting.
Using \eqref{AB=BA_BA_YB_relation} (where we are 
interested only in the first term in the right-hand side),
we thus get the product
\begin{align*}
	C_{\{n-k+1,\ldots,N\}}(v_1,\ldots,v_N)=
	\prod_{\aind=1}^{n-k}\prod_{\bind=n-k+1}^{N}
	\frac{\ip_2v_\aind-q \ip_2v_\bind}{\ip_2v_\aind-\ip_2v_\bind}
	=
	\prod_{\aind=1}^{n-k}\prod_{\bind=n-k+1}^{N}
	\frac{v_\aind-q v_\bind}{v_\aind-v_\bind}.
\end{align*}
Next, we note that $\AY_2(\KS)\,\bv_0=\bv_0$ and that (from 
\eqref{ABCD_elementary_operators})
\begin{align*}
	\AY_1(\KS)\CY_1(\KSB)\,\bv_n=
	\frac{(1-\SP_1^{2}q^{n-1})\ldots(1-\SP_1^{2}q^{k})\ip_1^{n-k}v_1 \ldots v_{n-k}}{(1-\SP_1\ip_1v_1)\ldots(1-\SP_1\ip_1v_{n-k})}
	\prod_{j=n-k+1}^{N}\frac{1-\SP_1 q^{k}\ip_1 v_j}{1-\SP_1\ip_1 v_j}\cdot\bv_k.
\end{align*}

\smallskip\noindent\textbf{Step 5.}
What remains unaccounted for in \eqref{step4_action_of_A_G_proof} 
is $\BY_2(\KS)\,\bv_0=\BY_2(v_1\md\ipb,\SPB)\ldots\BY_2(v_{n-k}\md\ipb,\SPB)\,\bv_0$. 
But this was computed earlier in the proof of 
\eqref{F_symm_formula}, and we can also immediately 
take the second vector to be $\bv_0\otimes\bv_0\otimes \ldots$ instead of just $\bv_0$. 
Let now
the 
parameters be denoted by $\ip_0,\ip_1,\ip_2,\ldots$ 
and $\SP_0,\SP_1,\SP_2,\ldots$,
as in \eqref{G_symm_formula},
and note that 
in the part corresponding to $\BY_2(\KS)$
we need to take the shifted parameters $\sh_1\ipb$ and $\sh_1\SPB$ 
(cf. \eqref{sh_operation}).
Thus, by \eqref{BY_semi_infinite}, we have
\begin{align}\label{B_action_G_proof}
	\BY(v_1\md\sh_1\ipb,\sh_1\SPB)\ldots\BY(v_{n-k}\md\sh_1\ipb,\sh_1\SPB)\big(\bv_0\otimes\bv_0\otimes \ldots\big)
	=\sum_{\kappa\in\signp {n-k}}
	\F_\kappa(v_1,\ldots,v_{n-k}\md\sh_1\ipb,\sh_1\SPB)\,\bv_\kappa,
\end{align}
and so the coefficient of $\bv_\nu$ (with 
$\nu\in\signp n$ having exactly $k$ zero coordinates)
in 
\begin{align*}
	\AY(v_1\md\ipb,\SPB)\ldots\AY(v_N\md\ipb,\SPB)(\bv_{n}\otimes \bv_{0}\otimes \bv_{0}\otimes\ldots)
\end{align*}
is equal to
\begin{align}\label{G_coeff_by_enu_G_proof}
	\frac{(\SP_0^{2};q)_{n}}{(\SP_0^{2};q)_{k}}
	\sum_{\substack{\KSB\subseteq\{1,\ldots,N\}\\|\KSB|=n-k}}
	\prod_{i\in\KSB}\frac{\ip_0 v_i}{1-\SP_0\ip_0 v_i}
	\prod_{j\in\KS}\frac{1-\SP_0 q^{k}\ip_0v_j}{1-\SP_0\ip_0v_j}
	\prod_{\substack{\aind\in\KSB\\\bind\in\KS}}
	\frac{v_\aind-q v_\bind}{v_\aind-v_\bind}
	\F_{(\nu_1-1,\ldots,\nu_{n-k}-1)}(\{v_i\}_{i\in\KSB}\md\sh_1\ipb,\sh_1\SPB).
\end{align}
Indeed, the signature $\kappa$ in \eqref{B_action_G_proof} 
corresponds to nonzero parts in $\nu$, and coordinates in $\kappa$ are counted starting from location 1
(hence the shifts $\nu_i-1$).

To match \eqref{G_coeff_by_enu_G_proof} to \eqref{G_symm_formula},
we use formula \eqref{F_symm_formula} to write 
$\F_{(\nu_1-1,\ldots,\nu_{n-k}-1)}$ as a sum over permutations of $\KSB$, and 
insert an additional symmetrization over $\KS$ 
(see footnote$^{\ref{symm_footnote}}$):
\begin{align*}
	1=\frac{(1-q)^{N-n+k}}{(q;q)_{N-n+k}}
	\sum_{\substack{\sigma\colon\KS\to\KS\\\textnormal{$\sigma$ is a bijection}}}
	\sigma\Bigg(\prod_{\substack{\aind,\bind\in\KS\\\aind<\bind}}
	\frac{v_{\aind}-qv_{\bind}}{v_{\aind}-v_{\bind}}
	\Bigg).
\end{align*}
After that one readily checks that \eqref{G_coeff_by_enu_G_proof}
coincides with the desired expression.
This concludes the proof of 
Theorem \ref{thm:symmetrization}.
\end{proof}



\section{Stochastic weights and fusion} 
\label{sec:stochastic_weights_and_fusion}

One key object we will consider
is the set of probability measures afforded by the Cauchy identities of 
\S \ref{sub:limits_of_yang_baxter_commutation_relations_cauchy_type_identities}.
We will describe and study 
them in \S \ref{sec:markov_kernels_and_stochastic_dynamics} below.
The present section is devoted to a preliminary 
discussion of the \emph{fusion procedure} 
on which some 
of the constructions of \S \ref{sec:markov_kernels_and_stochastic_dynamics}
are based.

\subsection{Stochastic weights $\Lmatr_{u,\SP}$} 
\label{sub:stochastic_weights}

If we assume that
\begin{align}\label{stochastic_weights_condition_qsxi}
\parbox{.85\textwidth}{
	\begin{itemize}
	\item
	$0<q<1$;
	\smallskip\item
	all inhomogeneity parameters $\ip_j$ are positive
	and are uniformly (in $j$) bounded away from 
	$0$ and $+\infty$;
	\smallskip\item
	all $\SP$--parameters $\SP_j$
	belong to $(-1,0)$
	and are 
	uniformly (in $j$) bounded away from 
	$-1$ and $0$,
	\end{itemize}
}
\end{align}
and, moreover, that
\begin{align}\label{stochastic_weights_condition_u}
\parbox{.85\textwidth}{
	\begin{itemize}
	\item
	all spectral parameters $u_i$
	are nonnegative,
	\end{itemize}
}
\end{align}
then all the vertex weights
$w_{u,\SP}$, 
$w_{u,\SP}^{\conj}$, 
and 
$\Lmatr_{u,\SP}$
(see Fig.~\ref{fig:vertex_weights} and \ref{fig:vertex_weights_conj_stoch})
are nonnegative.
Under these assumptions, \eqref{L_sum_to_one} implies that 
the stochastic weights
$\Lmatr_{u,\SP}(i_1,j_1;i_2,j_2)$, where $i_1,i_2\in\Z_{\ge0}$
and $j_1,j_2\in\{0,1\}$,
define a \emph{probability distribution}
on all possible output arrow configurations
$\big\{(i_2,j_2)\in\Z_{\ge0}\times\{0,1\}\colon i_2+j_2=i_1+j_1\big\}$
given the input arrow configuration $(i_1,j_1)$.
We will use conditions 
\eqref{stochastic_weights_condition_qsxi}--\eqref{stochastic_weights_condition_u}
to define Markov dynamics in \S \ref{sec:markov_kernels_and_stochastic_dynamics} below.

The conditions \eqref{stochastic_weights_condition_qsxi}--\eqref{stochastic_weights_condition_u} are sufficient but not necessary 
for the nonnegativity of the $\Lmatr_{u}$'s; 
for other conditions see \cite[Prop.\;2.3]{CorwinPetrov2015}
and also \S \ref{sub:asep_degeneration} and
\S \ref{sub:general_j_dynamics_and_q_hahn_degeneration} below.

\begin{remark}\label{rmk:q_or_s_zero}
	We will
	always 
	assume that the parameters $q$ and $\SP_j$
	are nonzero. In fact, without this assumption
	the weights
	$\Lmatr_{u,\SP}$ may still
	define probability distributions. If $q$ or the $\SP_j$'s
	vanish, then some of our statements remain valid and simplify, but
	we will not focus on the necessary modifications.
\end{remark}

\begin{remark}\label{rmk:sign_of_SP}
	Since the stochastic weights $\Lmatr_{u,\SP}$
	depend on $\SP$ and $u$ only through 
	$\SP u$ and $\SP^{2}$, they are invariant under 
	the simultaneous change of sign of both $\SP$ and $u$.
	We have chosen $\SP$ to be negative, and $u$ will be 
	nonnegative.
\end{remark}


\subsection{Fusion of stochastic weights} 
\label{sub:fusion_of_stochastic_weights}

For each $J\in\Z_{\ge1}$,
we will now define certain
more general stochastic vertex weights 
$\LJ{J}_{u,\SP}(i_1,j_1;i_2,j_2)$, 
where $(i_1,j_1), (i_2,j_2)\in\Z_{\ge0}\times\{0,1,\ldots,J\}$.
That is, we want to relax the restriction that the horizontal arrow
multiplicities are bounded by 1, and consider multiplicities bounded by any fixed $J\ge1$.
When $J=1$, the vertex weights $\LJ{J}_{u,\SP}$ will coincide with $\Lmatr_{u,\SP}$.
Of course, we want the new weights
$\LJ{J}_{u,\SP}$ to share some of the nice properties of the $\Lmatr_{u,\SP}$'s;
most importantly, the $\LJ{J}_{u,\SP}$'s should satisfy a version of the Yang--Baxter equation.
The construction of the weights 
$\LJ{J}_{u,\SP}$ follows the so-called \emph{fusion procedure}, which was invented 
in a representation-theoretic context \cite{KulishReshSkl1981yang} (see also \cite{KR1987Fusion})
to produce higher-dimensional solutions of the Yang--Baxter equation
from lower-dimensional ones. Following \cite{CorwinPetrov2015}, here
we describe the fusion procedure in purely combinatorial/probabilistic terms.

We will need the following definition.
\begin{definition}\label{def:q_exch}
	A probability distribution $P$ on $\{0,1\}^{J}$
	is called \emph{$q$-exchangeable} if 
	the probability weights $P(\vec h)$, 
	$\vec h=(h^{\scriptscriptstyle(1)},\ldots,h^{\scriptscriptstyle(J)})\in\{0,1\}^{J}$,
	depend on $\vec h$ in the following way:
	\begin{align}\label{P_bar_exchangeable}
		P(\vec h)=\tilde P(j)\cdot \frac{q^{\sum_{r=1}^{J}(r-1)h^{\scriptscriptstyle(r)}}}{Z_{j}(J)},
		\qquad \qquad
		j:=\sum_{r=1}^{J}h^{\scriptscriptstyle(r)},
	\end{align}
	where $\tilde P$ is a probability distribution on $\{0,1,\ldots,J\}$.
	In words, 
	for 
	a fixed sum of coordinates $j$,
	the weights of the conditional distribution  
	of $\vec h$
	are proportional to the product of the factors
	$q^{r-1}$ for each 
	coordinate ``1''
	at location $r\in\{1,\ldots,J\}$. 
	The normalization constant $Z_j(J)$
	is given by the following expression involving the $q$-binomial 
	coefficient:\footnote{Indeed, 
	$Z_j(J)$ is the sum of
	$q^{\sum_{r=1}^{J}(r-1)h^{(r)}}$ over all $\vec h\in\{0,1\}^{J}$ with 
	$\sum_{r=1}^{J}h^{(r)}=j$.
	Considering two cases $h^{(J)}=1$ or $h^{(J)}=0$, we see that it satisfies the recursion
	$Z_j(J)=q^{J-1}Z_{j-1}(J-1)+Z_{j}(J-1)$,
	with 
	$Z_0(J)=1$. This recursion is solved by \eqref{Zjj_formula}.}
	\begin{align}\label{Zjj_formula}
		Z_j(J)=q^{\frac{j(j-1)}2}\binom{J}{j}_{q}=
		q^{\frac{j(j-1)}2}\frac{(q;q)_{J}}{(q;q)_{j}(q;q)_{J-j}}.
	\end{align}

	The name ``$q$-exchangeable'' refers to the fact that 
	any exchange in $\vec h$
	of the form $10\to 01$
	multiplies the weight of $\vec h$
	by $q$.
	See \cite{Gnedin2009}, \cite{GnedinOlsh2009q} for a detailed treatment of 
	$q$-exchangeable distributions.
\end{definition}

Returning to vertex weights,
a key probabilistic feature observed in \cite{CorwinPetrov2015}
which triggers the fusion procedure is the following. Attach vertically
$J$
vertices with spectral parameters $u,qu,\ldots,q^{J-1}u$
(see Fig.~\ref{fig:J_vertex}), and assign 
to them the corresponding weights $\Lmatr_{q^{i}u,\SP}$ given by \eqref{vertex_weights_stoch}.
Fixing the numbers $i_1$ and $i_2$ of arrows at the bottom and at the top, 
we see that this vertex configuration
maps probability distributions $P_1$
on incoming arrows
$h_1^{\scriptscriptstyle(1)},\ldots,h_1^{\scriptscriptstyle(J)}$
to probability distributions
$P_2$
on outgoing arrows
$h_2^{\scriptscriptstyle(1)},\ldots,h_2^{\scriptscriptstyle(J)}$.
\begin{figure}[htbp]
	\scalebox{.9}{\begin{tikzpicture}
		[scale=1,very thick]
		\draw (0,0)--++(1.4,0) node [right] {$h_2^{\scriptscriptstyle(1)}$};
		\draw (0,0)--++(-1.4,0) node [left] {$h_1^{\scriptscriptstyle(1)}$};
		\draw (0,1)--++(1.4,0) node [right] {$h_2^{\scriptscriptstyle(2)}$};
		\draw (0,1)--++(-1.4,0) node [left] {$h_1^{\scriptscriptstyle(2)}$};
		\draw (0,3.2)--++(1.4,0) node [right] {$h_2^{\scriptscriptstyle(J)}$};
		\draw (0,3.2)--++(-1.4,0) node [left] {$h_1^{\scriptscriptstyle(J)}$};
		\draw (0,0)--++(0,-.8) node[below] {$i_1$};
		\draw (0,1)--++(0,.8) node [right, yshift=-9pt] {$l^{\scriptscriptstyle(2)}$};
		\draw (0,2.4)--++(0,.8) node [right, yshift=-13pt] {$l^{\scriptscriptstyle(J-1)}$};;
		\draw (0,3.2)--++(0,.8) node[above] {$i_2$};
		\draw (0,0)--++(0,1) node [right, yshift=-15pt] {$l^{\scriptscriptstyle(1)}$};
		\draw[fill] (0,0) circle (3pt) node [above left] {$\ybspec$};
		\draw[fill] (0,1) circle (3pt) node [above left] {$q\ybspec$};
		\draw[fill] (0,3.2) circle (3pt) node [above left] {$q^{J-1}\ybspec$};
		\node at (0.05,2.1) {$\cdots$};
	\end{tikzpicture}}
	\caption{Attaching $J$ vertices 
	with spectral parameters $u,qu,\ldots,q^{J-1}u$
	vertically.}
	\label{fig:J_vertex}
\end{figure}
\begin{proposition}\label{prop:L_q_exch}
	The 
	mapping $P_1\mapsto P_2$ 
	described above 
	preserves the class of
	$q$-exchangeable distributions.
\end{proposition}
\begin{proof}
	Let us fix 
	the numbers $i_1,i_2\in\Z_{\ge0}$
	of bottom and top arrows, 
	as well as 
	the total number 
	$j_1=\sum_{\ell=1}^{J}h_1^{\scriptscriptstyle(\ell)}\in\{0,1,\ldots,J\}$
	of incoming arrows from the left.
	Under these conditions, the 
	incoming $q$-exchangeable distribution $P_1$
	is unique (its partition function is $Z_{j_1}(J)$).
	It suffices to show that 
	for any $\vec h_2\in\{0,1\}^{J}$ 
	with $h_2^{\scriptscriptstyle(r)}=0$, $h_2^{\scriptscriptstyle(r+1)}=1$ for some $r$, we have
	\begin{align*}
		P_2(h_2^{\scriptscriptstyle(1)},\ldots,h_2^{\scriptscriptstyle(r)},h_2^{\scriptscriptstyle(r+1)},\ldots
		h_2^{\scriptscriptstyle(J)})=
		q\cdot P_2
		(h_2^{\scriptscriptstyle(1)},\ldots,h_2^{\scriptscriptstyle(r+1)},h_2^{\scriptscriptstyle(r)},\ldots
		h_2^{\scriptscriptstyle(J)}).
	\end{align*}

	Since this property involves only two neighboring vertices, it suffices to consider the case $J=2$.
	The desired statement now follows from the relations
	(here $g$ is arbitrary):
	\def\dd{.12}\begin{align*}
		\textnormal{weight}
		\Bigg(\raisebox{-27pt}{\scalebox{.6}{\begin{tikzpicture}
			[scale=1, very thick]
			\node (i1) at (0,-1) {$g$};
			\node (i1p) at (\dd,-1) {\phantom{$g$}};
			\node (i1m) at (-\dd,-1) {\phantom{$g$}};
			\node (i2) at (0,2) {$g-1$};
			\node (i2p) at (\dd,2) {\phantom{$g-1$}};
			\node (i2m) at (-\dd,2) {\phantom{$g-1$}};
			\node (h11) at (-1,0) {$0$};
			\node (h12) at (-1,1) {$0$};
			\node (h21) at (1,0) {$0$};
			\node (h22) at (1,1) {$1$};
			\draw[densely dotted] (h11) -- (h21);
			\draw[densely dotted] (h12) -- (h22);
			\draw[->] (i1m) -- (i2m);
			\draw[->] (i1) -- (i2);
			\draw[->] (i1p) -- (\dd,1) -- (h22);
		\end{tikzpicture}}}\Bigg)&=q\cdot
		\textnormal{weight}
		\Bigg(\raisebox{-27pt}{\scalebox{.6}{\begin{tikzpicture}
			[scale=1, very thick]
			\node (i1) at (0,-1) {$g$};
			\node (i1p) at (\dd,-1) {\phantom{$g$}};
			\node (i1m) at (-\dd,-1) {\phantom{$g$}};
			\node (i2) at (0,2) {$g-1$};
			\node (i2p) at (\dd,2) {\phantom{$g-1$}};
			\node (i2m) at (-\dd,2) {\phantom{$g-1$}};
			\node (h11) at (-1,0) {$0$};
			\node (h12) at (-1,1) {$0$};
			\node (h21) at (1,0) {$1$};
			\node (h22) at (1,1) {$0$};
			\draw[densely dotted] (h11) -- (h21);
			\draw[densely dotted] (h12) -- (h22);
			\draw[->] (i1m) -- (i2m);
			\draw[->] (i1) -- (i2);
			\draw[->] (i1p) -- (\dd,0) -- (h21);
		\end{tikzpicture}}}\Bigg),\\
		\textnormal{weight}
		\Bigg(\raisebox{-27pt}{\scalebox{.6}{\begin{tikzpicture}
			[scale=1, very thick]
			\node (i1) at (0,-1) {$g$};
			\node (i1p) at (\dd,-1) {\phantom{$g$}};
			\node (i1m) at (-\dd,-1) {\phantom{$g$}};
			\node (i2) at (0,2) {$g+1$};
			\node (i2m1) at (-\dd/2,2) {\phantom{$g+1$}};
			\node (i2m2) at (-3*\dd/2,2) {\phantom{$g+1$}};
			\node (i2p2) at (3*\dd/2,2) {\phantom{$g+1$}};
			\node (i2p1) at (\dd/2,2) {\phantom{$g+1$}};
			\node (h11) at (-1,0) {$1$};
			\node (h12) at (-1,1) {$1$};
			\node (h21) at (1,0) {$0$};
			\node (h22) at (1,1) {$1$};
			\draw[densely dotted] (h11) -- (h21);
			\draw[densely dotted] (h12) -- (h22);
			\draw[->] (i1) -- (0,1-\dd) --++ (3/2*\dd,2*\dd) -- (i2p2);
			\draw[->] (i1m) -- (-\dd,1-\dd) --++ (3/2*\dd,2*\dd)  -- (i2p1);
			\draw[->] (h11) -- (-2*\dd,0) --++ (0,1-\dd)--++ (3/2*\dd,2*\dd)  -- (i2m1);
			\draw[->] (h12) -- (-2.5*\dd,1)--++(\dd,\dd)--(i2m2);
			\draw[->] (i1p) -- (\dd,1-\dd) --++(\dd,\dd)--(h22);
		\end{tikzpicture}}}\Bigg)&=q\cdot
		\textnormal{weight}
		\Bigg(\raisebox{-27pt}{\scalebox{.6}{\begin{tikzpicture}
			[scale=1, very thick]
			\node (i1) at (0,-1) {$g$};
			\node (i1p) at (\dd,-1) {\phantom{$g$}};
			\node (i1m) at (-\dd,-1) {\phantom{$g$}};
			\node (i2) at (0,2) {$g+1$};
			\node (i2m1) at (-\dd/2,2) {\phantom{$g+1$}};
			\node (i2m2) at (-3*\dd/2,2) {\phantom{$g+1$}};
			\node (i2p2) at (3*\dd/2,2) {\phantom{$g+1$}};
			\node (i2p1) at (\dd/2,2) {\phantom{$g+1$}};
			\node (h11) at (-1,0) {$1$};
			\node (h12) at (-1,1) {$1$};
			\node (h21) at (1,0) {$1$};
			\node (h22) at (1,1) {$0$};
			\draw[densely dotted] (h11) -- (h21);
			\draw[densely dotted] (h12) -- (h22);
			\draw[->] (i1) -- (0,0-\dd) --++ (3/2*\dd,2*\dd) -- (i2p2);
			\draw[->] (i1m) -- (-\dd,0-\dd) --++ (3/2*\dd,2*\dd)  -- (i2p1);
			\draw[->] (h11) -- (-1.5*\dd,0) --++ (\dd,\dd)  -- (i2m1);
			\draw[->] (h12) -- (-1.5*\dd,1)--(i2m2);
			\draw[->] (i1p) -- (\dd,0-\dd) --++(\dd,\dd)--(h21);
		\end{tikzpicture}}}\Bigg),\\
		q\cdot \textnormal{weight}
		\Bigg(\raisebox{-27pt}{\scalebox{.6}{\begin{tikzpicture}
			[scale=1, very thick]
			\node (i1) at (0,-1) {$g$};
			\node (i1p) at (\dd,-1) {\phantom{$g$}};
			\node (i1m) at (-\dd,-1) {\phantom{$g$}};
			\node (i2) at (0,2) {$g$};
			\node (i2m) at (-\dd,2) {\phantom{$g$}};
			\node (i2p) at (\dd,2) {\phantom{$g$}};
			\node (h11) at (-1,0) {$0$};
			\node (h12) at (-1,1) {$1$};
			\node (h21) at (1,0) {$0$};
			\node (h22) at (1,1) {$1$};
			\draw[densely dotted] (h11) -- (h21);
			\draw[densely dotted] (h12) -- (h22);
			\draw[->] (i1p) -- (\dd,1-\dd) --++(\dd,\dd) --++ (h22);
			\draw[->] (i1) -- (0,1-\dd)--++(\dd,2*\dd) -- (i2p);
			\draw[->] (i1m) -- (-\dd,1-\dd)--++(\dd,2*\dd)-- (i2);
			\draw[->] (h12) -- (-2*\dd,1) --++(\dd,\dd) -- (i2m);
		\end{tikzpicture}}}\Bigg){}+{}\textnormal{weight}
		\Bigg(\raisebox{-27pt}{\scalebox{.6}{\begin{tikzpicture}
			[scale=1, very thick]
			\node (i1) at (0,-1) {$g$};
			\node (i1p) at (\dd,-1) {\phantom{$g$}};
			\node (i1m) at (-\dd,-1) {\phantom{$g$}};
			\node (i2) at (0,2) {$g$};
			\node (i2m) at (-\dd,2) {\phantom{$g$}};
			\node (i2p) at (\dd,2) {\phantom{$g$}};
			\node (h11) at (-1,0) {$1$};
			\node (h12) at (-1,1) {$0$};
			\node (h21) at (1,0) {$0$};
			\node (h22) at (1,1) {$1$};
			\draw[densely dotted] (h11) -- (h21);
			\draw[densely dotted] (h12) -- (h22);
			\draw[->] (i1p) -- (\dd,-\dd) --++(\dd/2,2*\dd)--(3/2*\dd,1-\dd)--++(\dd,\dd)--(h22);
			\draw[->] (i1) -- (0,-\dd) --++(\dd/2,2*\dd)--(1/2*\dd,1-\dd)--++(\dd/2,2*\dd)--(i2p);
			\draw[->] (i1m) -- (-\dd,-\dd) --++(\dd/2,2*\dd)--(-1/2*\dd,1-\dd)--++(\dd/2,2*\dd)--(i2);
			\draw[->] (h11) -- (-2.5*\dd,0) --++(\dd,\dd) -- (-3/2*\dd,1-\dd)--++(\dd/2,2*\dd)--(i2m);
		\end{tikzpicture}}}\Bigg)&=
		q\cdot \left(q\cdot 
		\textnormal{weight}
		\Bigg(\raisebox{-27pt}{\scalebox{.6}{\begin{tikzpicture}
			[scale=1, very thick]
			\node (i1) at (0,-1) {$g$};
			\node (i1p) at (\dd,-1) {\phantom{$g$}};
			\node (i1m) at (-\dd,-1) {\phantom{$g$}};
			\node (i2) at (0,2) {$g$};
			\node (i2m) at (-\dd,2) {\phantom{$g$}};
			\node (i2p) at (\dd,2) {\phantom{$g$}};
			\node (h11) at (-1,0) {$0$};
			\node (h12) at (-1,1) {$1$};
			\node (h21) at (1,0) {$1$};
			\node (h22) at (1,1) {$0$};
			\draw[densely dotted] (h11) -- (h21);
			\draw[densely dotted] (h12) -- (h22);
			\draw[->] (i1p) -- (\dd,0-\dd) --++(\dd,\dd) --++ (h21);
			\draw[->] (i1) -- (0,-\dd)--++(\dd/2,2*\dd)--(\dd/2,1-\dd)--++(\dd/2,2*\dd)--(i2p);
			\draw[->] (i1m) -- (-\dd,-\dd)--++(\dd/2,2*\dd)--(-\dd/2,1-\dd)--++(\dd/2,2*\dd)--(i2);
			\draw[->] (h12) -- (-2*\dd,1) --++(\dd,\dd) -- (i2m);
		\end{tikzpicture}}}\Bigg){}+{}\textnormal{weight}
		\Bigg(\raisebox{-27pt}{\scalebox{.6}{\begin{tikzpicture}
			[scale=1, very thick]
			\node (i1) at (0,-1) {$g$};
			\node (i1p) at (\dd,-1) {\phantom{$g$}};
			\node (i1m) at (-\dd,-1) {\phantom{$g$}};
			\node (i2) at (0,2) {$g$};
			\node (i2m) at (-\dd,2) {\phantom{$g$}};
			\node (i2p) at (\dd,2) {\phantom{$g$}};
			\node (h11) at (-1,0) {$1$};
			\node (h12) at (-1,1) {$0$};
			\node (h21) at (1,0) {$1$};
			\node (h22) at (1,1) {$0$};
			\draw[densely dotted] (h11) -- (h21);
			\draw[densely dotted] (h12) -- (h22);
			\draw[->] (i1p) -- (\dd,-\dd) --++(\dd,\dd) --++ (h21);
			\draw[->] (i1) -- (0,-\dd)--++(\dd,2*\dd) -- (i2p);
			\draw[->] (i1m) -- (-\dd,-\dd)--++(\dd,2*\dd)-- (i2);
			\draw[->] (h11) -- (-2*\dd,0) --++(\dd,\dd) -- (i2m);
		\end{tikzpicture}}}\Bigg)
		\right).
	\end{align*}
	In each of the relations 
	the right-hand side differs by 
	moving the outgoing arrow
	down,
	and ``weight'' means the product of the weights
	$\Lmatr_{u,\SP}$ at the bottom and $\Lmatr_{qu,\SP}$ at the top vertex.
	The above 
	relations are readily verified 
	from the definition of 
	$\Lmatr_{u,\SP}$
	\eqref{vertex_weights_stoch}
	(see also Fig.~\ref{fig:vertex_weights_conj_stoch}).
\end{proof}

This proposition implies that 
for any fixed $i_1,i_2\in\Z_{\ge0}$,
the Markov operator mapping 
$P_1$ to $P_2$
(where $P_1$, $P_2$ are probability distributions on $\{0,1\}^{J}$), 
can be projected to another Markov operator
which maps $\tilde P_1$ to $\tilde P_2$ (cf. \eqref{P_bar_exchangeable}),
i.e., acts on probability distributions on the smaller space $\{0,1,\ldots,J\}$.
We will denote the matrix elements of this ``collapsed'' 
Markov operator by 
$\LJ{J}_{u,\SP}(i_1,j_1;i_2,j_2)$, 
where $(i_1,j_1), (i_2,j_2)\in\Z_{\ge0}\times\{0,1,\ldots,J\}$.

The definition of $\LJ{J}_{u,\SP}$ implies that
these matrix elements
satisfy a certain
recursion relation in $J$. 
This relation is obtained by 
considering two cases,
whether there is a left-to-right arrow at the very bottom, or not
(i.e., $h_1^{\scriptscriptstyle(1)}=0$ or $h_1^{\scriptscriptstyle(1)}=1$).
Therefore, we obtain the following recursion:
\begin{align}
	\LJ{J}_{u,\SP}(i_1,j_1;i_2,j_2)
	=\sum_{a,b\in\{0,1\}}\sum_{l\ge0}
	P(h_{1}^{\scriptscriptstyle(1)}=a)\,
	\Lmatr_{u,\SP}(i_1,a;l,b)
	\LJ{J-1}_{qu,\SP}(l,j_1-a;i_2,j_2-b).
	\label{recursion_relation}
\end{align}
Here the probability $P(h_{1}^{\scriptscriptstyle(1)}=a)$ 
corresponds to our division into two cases. 
It can be readily computed using Definition \ref{def:q_exch}:
\begin{align*}
	P(h_{1}^{\scriptscriptstyle(1)}=0)=\frac{q^{j_1}Z_{j_1}(J-1)}{Z_{j_1}(J)}=
	\frac{q^{j_1}-q^{J}}{1-q^{J}},
	\qquad
	P(h_{1}^{\scriptscriptstyle(1)}=1)=
	\frac{q^{j_1-1}Z_{j_1-1}(J-1)}{Z_{j_1}(J)}=
	\frac{1-q^{j_1}}{1-q^{J}}.
\end{align*}

The recursion relation \eqref{recursion_relation} has a solution
expressible in terms of 
terminating $q$-hypergeometric functions 
(here we follow
the notation of \cite{Mangazeev2014}, \cite{Borodin2014vertex}):
\begin{multline*}
	{}_{r+1}\bar{\varphi}_r\left(\begin{matrix} q^{-n};a_1,\dots,a_r\\b_1,\dots,b_r\end{matrix}
	\Bigl| q,z\right):=\sum_{k=0}^n z^k\,\frac{(q^{-n};q)_k}{(q;q)_k}	\prod_{i=1}^r (a_i;q)_k(b_iq^k;q)_{n-k}\\
	=\prod_{i=1}^r (b_i;q)_n\cdot {}_{r+1}{\varphi}_r\left(\begin{matrix} q^{-n},a_1,\dots,a_r\\b_1,\dots,b_r\end{matrix}
	\Bigl| q,z\right),
\end{multline*}
where here $n\in\Z_{\ge0}$.
The solution $\LJ{J}_{u,\SP}$
looks as follows:
\begin{multline}
	\LJ{J}_{u,\SP}(i_1,j_1;i_2,j_2)
	=
	\mathbf{1}_{i_1+j_1=i_2+j_2}
	\frac{(-1)^{i_1}q^{
	\frac12 i_1(i_1+2j_1-1)}
	u^{i_1}\SP^{j_1+j_2-i_2}(u \SP^{-1};q)_{j_2-i_1}}
	{(q;q)_{i_2} (\SP u;q)_{i_2+j_2}
	(q^{J+1-j_1};q)_{j_1-j_2}}
	\\\times{}_{4}\bar{\varphi}_3\left(\begin{matrix} q^{-i_2};q^{-i_1},\SP u q^{J},q \SP/u\\
	\SP^{2},q^{1+j_2-i_1},q^{J+1-i_2-j_2}\end{matrix}
	\Bigl|\, q,q\right).\label{LJ_formula}
\end{multline}
Formula \eqref{LJ_formula} for fused vertex weights 
is essentially due to \cite{Mangazeev2014}.
In the present form \eqref{LJ_formula}
it was obtained 
in \cite[Thm.\;3.15]{CorwinPetrov2015}
by matching the recursion
\eqref{recursion_relation}
to the recursion relation for the classical
$q$-Racah orthogonal polynomials.\footnote{The 
parameters in \eqref{LJ_formula} match those in
\cite[Thm.\;3.15]{CorwinPetrov2015}
as $\beta=\al q^{J}$, $\al=-\SP u$, and $\nu=\SP^{2}$.}
About the latter see \cite[Ch.\;3.2]{Koekoek1996}.


\subsection{Principal specializations of skew functions} 
\label{sub:principal_specializations_of_skew_functions}

The fused stochastic vertex weights discussed
in \S \ref{sub:fusion_of_stochastic_weights}
can be used to describe principal specializations of the 
skew functions $\F_{\mu/\la}$ and $\G_{\mu/\la}$,
in analogy to the non-skew principal specializations
of Corollary \ref{cor:prin_spec}.

Mimicking \eqref{weights_conj}--\eqref{vertex_weights_stoch}, 
we will use the general $J$
stochastic weights $\LJ{J}_{u,\SP}$ \eqref{LJ_formula} to define the 
weights which are general $J$ versions of the $w_u$'s \eqref{weights}:
\begin{align}
	\WJ{J}_{u,\SP}(i_1,j_1;i_2,j_2):=\frac1{(-\SP)^{j_2}}
	\frac{(q;q)_{i_2}}{(\SP ^{2};q)_{i_2}}
	\frac{(\SP ^{2};q)_{i_1}}{(q;q)_{i_1}}
	\LJ{J}_{u,\SP}(i_1,j_1;i_2,j_2),
	\label{WJ_defn}
\end{align}
where $i_1,i_2\in\Z_{\ge0}$ and $j_1,j_2\in\{0,1,\ldots,J\}$
are such that $i_1+j_1=i_2+j_2$ (otherwise the above weight is set to zero).
These weights are expressed
via the $q$-hypergeometric function
as follows:
\begin{multline}
	\WJ{J}_{u,\SP}(i_1,j_1;i_2,j_2)=
	\frac{(-1)^{i_1+j_2}q^{\frac12 i_1(i_1+2j_1-1)}\SP^{j_2-i_1}u^{i_1}
	(q;q)_{j_1}(u\SP^{-1};q)_{j_1-i_2}}{(q;q)_{i_1}(q;q)_{j_2} (u\SP;q)_{i_1+j_1}}
	\\\label{WJ_formula}
	\times
	{}_{4}\bar{\varphi}_3\left(\begin{matrix} q^{-i_1};q^{-i_2},q^J \SP u,q\SP u^{-1}\\\SP^2,q^{1+j_1-i_2},q^{1+J-i_1-j_1}\end{matrix}
	\Bigl|\, q,q\right).
\end{multline}
We see that the weights $\WJ{J}_{u,\SP}$
depend on the spectral parameter
$u$ in a rational manner.
One can also check that for $J=1$, the weights $\WJ{J}_{u,\SP}$ turn into \eqref{weights}.

\begin{proposition}\label{prop:skew_prin_spec}
	\noindent{\bf1.\/}
	For any $J\in\Z_{\ge1}$, $N\in\Z_{\ge0}$, 
	$\la\in\signp N$, $\mu\in\signp{N+J}$, and $u\in\C$, the principal specialization
	of the skew function
	\begin{align*}
		\F_{\mu/\la}(u,qu,\ldots,q^{J-1}u\md\ipb,\SPB)
	\end{align*}
	is equal to the weight of the unique collection of $N+J$ up-right paths in 
	the semi-infinite horizontal strip of height $1$, 
	where the weight at each vertex $x\in\Z_{\ge0}$ is
	equal to $\WJ{J}_{\ip_x u,\SP_x}$ (so that at most $J$ horizontal arrows
	per edge are allowed). The paths in the collection start with 
	$N$ vertical edges $(\la_i,0)\to(\la_i,1)$ and with $J$ horizontal edges $(-1,1)\to(0,1)$,
	and end with $N+J$ vertical edges 
	$(\mu_j,1)\to(\mu_j,2)$, see Fig.~\ref{fig:J_paths}, top.

\smallskip

	\noindent{\bf2.\/}
	For any $J\in\Z_{\ge1}$, any $\la,\mu\in\sign N$, and any $v\in\C$,
	the principal specialization 
	of the skew function
	\begin{align*}
		\G_{\mu/\la}(v,qv,\ldots,q^{J-1}v\md\ipb,\SPB)
	\end{align*}
	is equal to the weight of the unique collection of $N$ up-right paths in 
	the semi-infinite horizontal strip of height $1$, 
	where the weight at each vertex $x\in\Z_{\ge0}$ is
	equal to $\WJ{J}_{\ip_x v,\SP_x}$ (so that at most $J$ horizontal arrows
	per edge are allowed). 
	The paths in the collection start with 
	$N$ vertical edges $(\la_i,0)\to(\la_i,1)$ and end with $N$ vertical edges 
	$(\mu_i,1)\to(\mu_i,2)$, see Fig.~\ref{fig:J_paths}, bottom.
\end{proposition}
\begin{figure}[htbp]
	\scalebox{1}{\begin{tikzpicture}
	[scale=1.16,ultra thick]
	\def\d{.105}
	\def\h{1.2} 
	\draw[densely dotted] (-.8,0)--++(8.8+2*\h,0);
	\foreach \ii in {0,1,2,3,4,5,6,7,8}
	{
		\draw[densely dotted] (\h*\ii,-.8)--++(0,1.6);
		\node[above left, xshift=-.5pt,yshift=1pt] at (1.2*\ii,0) {$\ip_{\ii}u$};
	}
	\node[below] at (-.1,-.8) {0};
	\node[below] at (\h,-.8) {$\la_{N-1}=\la_N$};
	\node[below] at (6*\h,-.8) {$\la_{3}=\la_{2}$};
	\node[below] at (7*\h,-.8) {$\la_{1}$};
	\node[below] at (3*\h,-.8) {$\la_{4}$};
	\node[above] at (8*\h,.8) {$\mu_{2}=\mu_1$};
	\node[above] at (6*\h,.8) {$\mu_3$};
	\node[above] at (4*\h,.8) {$\mu_4$};
	\node[above] at (0*\h-.1,.8) {$\mu_{N+J}$};
	\node[above] at (1*\h,.8) {$\mu_{N+J-1}$};
	\draw[->] (7*\h,-.8)--++(0,.8-\d)--++(\d,\d/2)--++(\h-\d*3/2,0)--++(\d,\d)--++(0,.8-\d/2);
	\draw[->] (6*\h+\d/2,-.8)--++(0,.8-\d)--++(\d,\d)--++(\h-2*\d,0)--++(2*\d,\d/2)--++(\h-2*\d,0)--++(0,.8-\d/2);
	\draw[->] (6*\h-\d/2,-.8)--++(0,.8-\d)--++(\d/2,2*\d)--++(0,.8-\d);
	\draw[->] (3*\h,-.8)--++(0,.8)--++(\h,0)--++(0,.8);
	\draw[->] (-.8,\d)--++(.8-\d,0)--++(\d,\d)--++(0,.8-2*\d);
	\draw[->] (-.8,0)--++(.8-\d,0)--++(2*\d,\d/2)--++(\h-\d*5/2,0)--++(3/2*\d,3/2*\d)--++(0,.8-2*\d);
	\draw[->] (-.8,-\d)--++(.8-\d,0)--++(2*\d,\d/2)--++(\h-2*\d,0)--++(3/2*\d,3/2*\d)--++(\h-\d/2-\d-\d,0)--++(\d,\d)
	--++(0,.8-2*\d);
	\draw[->] (\h-\d/2,-.8)--++(0,.8-3/2*\d)--++(3/2*\d,3/2*\d)--++(\h-2*\d,0)--++(\d,\d)--++(0,.8-\d);
	\draw[->] (\h+\d/2,-.8)--++(0,.8-\d*2)--++(\d,\d)--++(\h-3/2*\d,0)--++(\d,\d)--++(0,.8);
\end{tikzpicture}}
\\\bigskip
\scalebox{1}{\begin{tikzpicture}
	[scale=1.16,ultra thick]
	\def\d{.105}
	\def\h{1.2} 
	\draw[densely dotted] (-.8,0)--++(8.8+2*\h,0);
	\foreach \ii in {0,1,2,3,4,5,6,7,8}
	{
		\draw[densely dotted] (\h*\ii,-.8)--++(0,1.6);
		\node[above left, xshift=-.5pt,yshift=1pt] at (1.2*\ii,0) {$\ip_{\ii}v$};
	}
	\node[below] at (-.1,-.8) {0};
	\node[below] at (\h,-.8) {$\la_{N-1}=\la_N$};
	\node[below] at (6*\h,-.8) {$\la_{3}=\la_{2}$};
	\node[below] at (7*\h,-.8) {$\la_{1}$};
	\node[above] at (8*\h,.8) {$\mu_{2}=\mu_1$};
	\node[above] at (6*\h,.8) {$\mu_3$};
	\node[above] at (5*\h,.8) {$\mu_4$};
	\node[below] at (3*\h,-.8) {$\la_4$};
	\node[above] at (3*\h,.8) {$\mu_N$};
	\draw[->] (7*\h,-.8)--++(0,.8-\d)--++(\d,\d/2)--++(\h-\d*3/2,0)--++(\d,\d)--++(0,.8-\d/2);
	\draw[->] (6*\h+\d/2,-.8)--++(0,.8-\d)--++(\d,\d)--++(\h-2*\d,0)--++(2*\d,\d/2)--++(\h-2*\d,0)--++(0,.8-\d/2);
	\draw[->] (6*\h-\d/2,-.8)--++(0,.8-\d)--++(\d/2,2*\d)--++(0,.8-\d);
	\draw[->] (\h-\d/2,-.8)--++(0,.8-\d/2)--++(\d,\d)--++(\h-\d,0)--++(2*\d,\d/2)--++(\h-\d*3/2-\d,0)--++(\d,\d)
	--++(0,.8-3/2*\d);
	\draw[->] (\h+\d/2,-.8)--++(0,.8-\d)--++(\d/2,\d/2)--++(\h-\d*3/2,0)--++(2*\d,\d/2)--++(\h-\d*3/2-\d,0)
	--++(2*\d,\d)--++(\h-\d-\d*3/2,0)--++(\d,\d)--++(0,.8-\d*3/2);
	\draw[->] (2*\h,-.8)--++(0,.8-\d*2)--++(\d,\d)--++(\h-\d-\d,0)--++(2*\d,\d)--++(\h-\d-\d/2,0)--++(\d,\d)
	--++(0,.8-\d/2);
	\draw[->] (3*\h,-.8)--++(0,.8-2*\d)--++(\d,\d)--++(\h-\d*2,0)--++(2*\d,\d)--++(\h-\d,0)--++(0,.8);
\end{tikzpicture}}
	\caption{
	A unique path collection with horizontal multiplicities bounded by $J=3$
	corresponding to the function 
	$\F_{\mu/\la}(u,qu,\ldots,q^{J-1}u\md\ipb,\SPB)$ (top)
	or
	$\G_{\mu/\la}(v,qv,\ldots,q^{J-1}v\md\ipb,\SPB)$ (bottom).
	Each $j$-th vertex having the spectral parameter $\ip_j u$ or $\ip_j v$, respectively,
	contains the $\SP$--parameter $\SP_j$.}
	\label{fig:J_paths}
\end{figure}
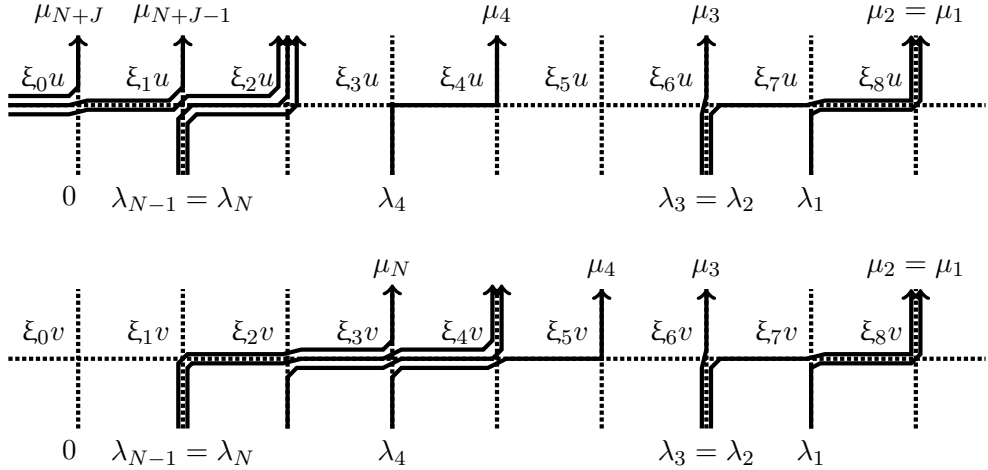
\begin{proof}
	For $\la\in\signp N$ and our $\SP$--parameters $\SPB=\{\SP_{j}\}$, denote
	\begin{align}\label{SPB_la_notation}
		(-\SPB)^{\la}:=\prod_{i=1}^{N}\prod_{j=0}^{\la_i-1}(-\SP_j).
	\end{align}
	
	Let us focus on the second claim.
	Relation \eqref{weights_conj}--\eqref{vertex_weights_stoch} 
	between the $J=1$ weights 
	$w_{u,\SP}$ and
	$\Lmatr_{u,\SP}$
	readily implies that for any
	$z\in\C$,
	any $L\in\Z_{\ge0}$ and
	$\nu,\kappa\in\signp L$,
	the quantity
	\begin{align*}
		\frac{(-\SPB)^{\kappa}}{(-\SPB)^{\nu}}
		\frac{\conj_{\SPB}(\kappa)}{\conj_{\SPB}(\nu)}\G_{\kappa/\nu}(z\md\ipb,\SPB)
		=\frac{(-\SPB)^{\kappa}}{(-\SPB)^{\nu}}
		\G^{\conj}_{\kappa/\nu}(z\md\ipb,\SPB)
	\end{align*}
	is equal to the weight of the unique collection
	of $L$ paths 
	in the semi-infinite horizontal strip of height~$1$
	connecting $\nu$ to $\kappa$,
	but with horizontal arrow multiplicities bounded by~$1$.
	The weight at a vertex $x\in\Z_{\ge0}$ in 
	this path collection is the stochastic weight $\Lmatr_{\ip_x z,\SP_x}$.

	Therefore, by Proposition \ref{prop:branching}, the quantity
	\begin{align}\label{ssconjconj_G_fusion_proof}
		\frac{(-\SPB)^{\mu}}{(-\SPB)^{\la}}
		\G_{\mu/\la}^{\conj}(v,qv,\ldots,q^{J-1}v\md\ipb,\SPB)
	\end{align}
	is equal to the sum of weights of 
	collections of $N$ paths in $\{1,2,\ldots,J\}\times\Z_{\ge0}$
	connecting $\la$ to $\mu$ (as in Definition \ref{def:G}), 
	in which the weight at each vertex $(j,x)\in\{1,2,\ldots,J\}\times\Z_{\ge0}$
	is $\Lmatr_{\ip_x q^{j-1}v,\SP_x}$
	(hence the horizontal arrow multiplicities are bounded by~$1$).
	In \eqref{ssconjconj_G_fusion_proof},
	the configuration of input horizontal arrows
	at location $0$ is empty, and hence its distribution is $q$-exchangeable.
	Thus, we may use
	the fusion of stochastic weights from \S \ref{sub:fusion_of_stochastic_weights}
	to collapse the $J$ horizontals 
	into one, with horizontal arrow multiplicities bounded by $J$, and 
	with fused vertex weights 
	$\LJ{J}_{\ip_x v,\SP_x}$.
	In this way all path collections in
	$\{1,2,\ldots,J\}\times\Z_{\ge0}$
	connecting $\la$ to $\mu$
	map to the unique collection 
	in $\{1\}\times\Z_{\ge0}$
	with horizontal edge multiplicities
	bounded by $J$.
	Using \eqref{WJ_defn}, we conclude that the second claim holds.

	The first claim about
	$\F_{\mu/\la}(u,qu,\ldots,q^{J-1}u\md\ipb,\SPB)$
	is analogous, because 
	the corresponding configuration of input arrows
	in $\{1,2,\ldots,J\}\times\Z_{\ge0}$
	is the fully packed one, whose distribution is also $q$-exchangeable.
	This completes the proof.
\end{proof}

\begin{remark}\label{rmk:anal_cont}
	Since the general $J$ vertex weights 
	$\WJ{J}_{u,\SP}(i_1,j_1;i_2,j_2)$
	\eqref{WJ_formula}
	depend on $q^{J}$ in a rational manner, they make sense for
	$q^{J}$ being an arbitrary (generic) complex parameter.
	Thus, we can consider the 
	principal specializations
	$\G_{\mu/\la}(v,qv,\ldots,q^{J-1}v\md\ipb,\SPB)$
	for a generic $q^{J}\in\C$. In other words, 
	this quantity admits an analytic continuation 
	in $q^{J}$.
	The second part of Proposition \ref{prop:skew_prin_spec}
	thus states that when $J\in\Z_{\ge1}$,
	the function $\G_{\mu/\la}(v,qv,\ldots,q^{J-1}v\md\ipb,\SPB)$
	can be expressed as a substitution of
	the values $(v,qv,\ldots,q^{J-1}v)$ into the symmetric function
	$\G_{\mu/\la}(v_{1},v_{2},\ldots,v_{J}\md\ipb,\SPB)$.

	In contrast with the functions $\G_{\mu/\la}$
	in which the number of indeterminates
	does not depend on $\la$ and $\mu$,
	the
	number of arguments in the
	functions $\F_{\mu/\la}$ 
	is completely determined by the signatures $\la$ and $\mu$. Therefore, 
	the parameter $J$ in 
	$\F_{\mu/\la}(u,qu,\ldots,q^{J-1}u\md\ipb,\SPB)$
	has to 
	remain a positive integer.
\end{remark}

\begin{remark}
	The weights \eqref{WJ_defn} are related to 
	the weights 
	$\WTJ{J}_{v}(i_1,j_1;i_2,j_2)$ of \cite[(6.8)]{Borodin2014vertex}
	via
	\begin{align*}
		\WJ{J}_{v,\SP}(i_1,j_1;i_2,j_2)=
		q^{\frac14 (j_2^{2}-j_1^{2})}
		(-\SP)^{j_2-j_1}\frac{(q;q)_{j_1}}{(q;q)_{j_2}}
		\WTJ{J}_{v}(i_1,j_1;i_2,j_2).
	\end{align*}
	The structure of the path collection
	for $\G_{\mu/\la}$
	implies that 
	for the purposes of computing
	$\G_{\mu/\la}$,
	any prefactors
	in the vertex weights
	of the form $f(j_1)/f(j_2)$
	are irrelevant. 
	Therefore, the principal specializations
	$\G_{\mu/\la}(v,qv,\ldots,q^{J-1}v\md\ipb,\SPB)$
	for $\ip_j\equiv1$ and $\SP_j\equiv \SP$
	coincide with those in \cite[\S6]{Borodin2014vertex}.
	Note that, however, factors
	of the form $f(j_1)/f(j_2)$ in the vertex weights
	do make a difference for the functions
	$\F_{\mu/\la}(u,qu,\ldots,q^{J-1}u\md\ipb,\SPB)$.
\end{remark}



\section{Markov kernels and stochastic dynamics} 
\label{sec:markov_kernels_and_stochastic_dynamics}

In this section we describe 
probability distributions on signatures arising from the 
Cauchy identity (Corollary \ref{cor:usual_Cauchy}),
as well as
discrete time stochastic systems
(i.e., discrete time Markov chains)
which act nicely on these measures.
Some of the stochastic systems we consider 
are inhomogeneous generalizations of 
the ones from \cite{CorwinPetrov2015}.

\subsection{Probability measures associated with the Cauchy identity} 
\label{sub:probability_measures_associated_with_the_cauchy_identity}

The idea that summation identities for symmetric functions
lead to interesting
probability measures dates back at least to 
\cite{fulman1997probabilistic}, \cite{okounkov2001infinite},
and it was further developed in
\cite{okounkov2003correlation},
\cite{vuletic2007shifted},
\cite{Borodin2010Schur},
\cite{BorodinCorwin2011Macdonald},
\cite{BCGS2013}.
Similar ideas in our context lead to
the definition of the following probability measures
which are analogous to the Schur or Macdonald measures:

\begin{definition}\label{def:MM}
	Let $M,N\in\Z_{\ge0}$,\footnote{Here and below if $M=0$,
	then $\signp M$ consists of the single empty signature $\varnothing$, 
	and thus all probability measures and Markov operators on this space are trivial.}
	and the parameters 
	$q$, $\ipb$, $\SPB$, and 
	$\UU=(u_1,\ldots,u_M)$,
	$\VV=(v_1,\ldots,v_N)$
	satisfy 
	\eqref{stochastic_weights_condition_qsxi}--\eqref{stochastic_weights_condition_u}.\footnote{These 
	conditions are assumed throughout \S \ref{sec:markov_kernels_and_stochastic_dynamics}
	except \S \ref{sub:asep_degeneration}
	and \S \ref{sub:general_j_dynamics_and_q_hahn_degeneration}.}
	Moreover, assume that 
	$\adm{u_i}{v_j}$ 
	for all $i,j$ (for the admissibility 
	it is enough to require that all $u_i$ and 
	$v_j$ are sufficiently small, cf. Definition \ref{def:admissible}).
	Define the 
	probability measure on 
	$\signp M$ via
	\begin{align}\label{MM_measure}
		\MM_{\UU;\VV}(\nu\md\ipb,\SPB)
		=\MM_{\UU;\VV}(\nu):=\frac{1}{Z(\UU;\VV\md\ipb,\SPB)}\,
		\F_\nu(u_1,\ldots,u_M\md\ipb,\SPB)
		\G_\nu^{\conj}(v_1,\ldots,v_N\md\ipbb,\SPB),
		\qquad \nu\in\signp M,
	\end{align}
	where the normalization constant is given by
	\begin{align}\label{Z_formula}
		Z(\UU;\VV\md\ipb,\SPB)
		:=(q;q)_{M}\prod_{i=1}^{M}
		\bigg(\frac1{1-\SP_0\ip_0u_i}
		\prod_{j=1}^{N}
		\frac{1-qu_iv_j}{1-u_iv_j}\bigg).
	\end{align}
\end{definition}
The fact that the unnormalized weights 
$Z(\UU;\VV\md\ipb,\SPB)\,\MM_{\UU;\VV}(\nu\md\ipb,\SPB)$
are nonnegative follows from \eqref{stochastic_weights_condition_qsxi}--\eqref{stochastic_weights_condition_u}.
Indeed, these conditions 
imply that the vertex weights 
$w_{u,\SP}$
and 
$w^{\conj}_{u,\SP}$
are nonnegative, and hence so are the functions 
$\F_\nu$ and $\G_\nu^{\conj}$.
The form \eqref{Z_formula} of the normalization constant
follows from the Cauchy identity 
\eqref{usual_Cauchy}.\footnote{The sum of the unnormalized weights 
converges due to the admissibility conditions, and hence the normalization
constant
$Z(\UU;\VV\md\ipb,\SPB)$ is nonnegative.
This constant is positive whenever the measure 
$\MM_{\UU;\VV}$ is nontrivial.}
Note that the length of the tuple $\UU$ determines the length of the 
signatures on which the measure $\MM_{\UU;\VV}$ lives. In contrast, the 
length of the tuple $\VV$ may be arbitrary.

In two degenerate cases, 
$\MM_{\varnothing;\VV}$ is the delta measure
at the empty configuration (for any $\VV$),
and $\MM_{(u_1,\ldots,u_M);\varnothing}$ is the delta measure
at the configuration $0^{M}$ (that is, all $M$ particles are at zero).

\begin{figure}[htbp]
	\scalebox{1}{\begin{tikzpicture}
		[scale=.7,thick]
		\def\d{.11}
		\foreach \xxx in {0,...,16}
		{
			\draw[dotted] (\xxx,.5)--++(0,9);
		}
		\foreach \xxx in {1,...,9}
		{
			\draw[dotted] (-.85,\xxx)--++(17.35,0);
			\draw[->, line width=1.7pt] (-.85,\xxx)--++(.5,0);
		}
		\node[left] at (-.85,1) {$u_1$};
		\node[left] at (-.85,2) {$u_{2}$};
		\node[left] at (-.85,3.15) {$\vdots$};
		\node[left] at (-.85,4) {$u_M$};
		\node[left] at (-.85,5) {$v_N$};
		\node[left] at (-.85,6.15) {$\vdots$};
		\node[left] at (-.85,7) {$v_3$};
		\node[left] at (-.85,8) {$v_2$};
		\node[left] at (-.85,9) {$v_1$};
		\node at (0,0) {$0$};
		\node at (1,0) {$1$};
		\node at (2,0) {$2$};
		\node at (3,0) {$3$};
		\node at (4,0) {$\ldots$};
		\draw[->, line width=1.7pt] (-.85,5)--++(.85-\d,0)--++(\d,\d)
		--++(0,1-2*\d)--++(\d,\d)--++(4-3*\d,0)--++(0,1-\d)
		--++(\d,2*\d)--++(0,1-\d)--++(\d,\d)--++(0,1-\d)--++(12.85,0);
		\draw[->, line width=1.7pt] (-.85,6)--++(.85-\d,0)
		--++(\d,\d)--++(0,1-2*\d)--++(\d,2*\d)--++(0,1-2*\d)
		--++(\d,2*\d)--++(0,1-2*\d)--++(\d,2*\d)--++(0,1-\d);
		\draw[->, line width=1.7pt] (-.85,7)--++(.85-\d,0)--++(0,1-\d)--++(\d,2*\d)--++(0,1-2*\d)--++(\d,2*\d)--++(0,1-\d);
		\draw[->, line width=1.7pt] (-.85,8)--++(.85-2*\d,0)--++(0,1-\d)
		--++(\d,2*\d)--++(0,1-\d);
		\draw[->, line width=1.7pt] (-.85,9)--++(.85-3*\d,0)--++(0,1);
		\draw[->, line width=1.7pt] (-.85,4)--++(.85,0)--++(0,1-\d)--++(\d,\d)
		--++(1-\d,0)--++(3-\d,0)--++(0,1-\d)--++(\d,2*\d)--++(0,1-2*\d)
		--++(\d,2*\d)--++(0,1-\d)--++(3-\d,0)--++(9.85,0);
		\draw[->, line width=1.7pt] (-.85,3)--++(.85,0)--++(2-\d,0)
		--++(\d,\d)--++(0,1-\d)--++(2,0)
		--++(0,1-\d)--++(\d,2*\d)--++(0,1-2*\d)--++(\d,2*\d)--++(0,1-\d)
		--++(4-2*\d,0)--++(8.85,0);
		\draw[->, line width=1.7pt] (-.85,2)--++(.85,0)--++(2,0)
		--++(0,1-\d)--++(\d,\d)--++(5-2*\d,0)--++(\d,\d)--++(0,1-\d)
		--++(1-\d,0)--++(\d,\d)--++(0,1-\d)--++(2,0)--++(0,1)
		--++(6.85,0);
		\draw[->, line width=1.7pt] (-.85,1)--++(.85,0)
		--++(4,0)--++(0,1)--++(3,0)--++(0,1-\d)--++(\d,\d)
		--++(1-\d,0)--++(0,1-\d)--++(\d,\d)--++(5-\d,0)
		--++(0,1)--++(3.85,0);
		\node[right, xshift=2pt] at (0,4.5) {$\nu_M$};
		\node[right, xshift=2pt] at (8,4.5) {$\nu_{2}$};
		\node[right, xshift=2pt] at (13,4.5) {$\nu_{1}$};
		\draw[->, line width=3.1pt, opacity=.55, red]
		(3*\d,9.85)--++(0,-1+.15)--++(4-3*\d,0)--++(0,-1+\d)
		--++(\d,-2*\d)--++(0,-1+2*\d)--++(\d,-2*\d)--++(0,-1+\d)
		--++(6-\d*2,0)--++(0,-1)--++(3,0)--++(0,-.75);
		\draw[->, line width=3.1pt, opacity=.55, draw=red]
		(1*\d,9.85)--++(0,-1+\d+.15)--++(\d,-2*\d)--++(0,-1+\d)--++(4-3*\d,0)
		--++(0,-1+\d)--++(\d,-2*\d)--++(0,-1+2*\d)--++(\d,-2*\d)--++(0,-1+\d)
		--++(4-\d,0)--++(0,-.75);
		\draw[->, line width=3.1pt, opacity=.55, red]
		(-1*\d,9.85)--++(0,-1+\d+.15)--++(\d,-2*\d)--++(0,-1+2*\d)
		--++(\d,-2*\d)--++(0,-1+\d)--++(4-3*\d,0)--++(0,-1+\d)
		--++(\d,-2*\d)--++(0,-1+2*\d)--++(\d,-2*\d)--++(0,-.75+\d);
		\draw[->, line width=3.1pt, opacity=.55, red]
		(-3*\d,9.85)--++(0,-1+\d+.15)--++(\d,-2*\d)--++(0,-1+2*\d)
		--++(\d,-2*\d)--++(0,-1+2*\d)
		--++(\d,-2*\d)--++(0,-1+\d)--++(0,-1.75);
	\end{tikzpicture}}
	\caption{
	Probability weights
	$\MM_{\UU;\VV}(\nu\md\ipb,\SPB)$
	as partition functions.
	The bottom part of height $M$ corresponds to 
	$\BY(u_{1}\md\ipb,\SPB)\ldots\BY(u_{M}\md\ipb,\SPB)$,
	and the top half of height $N$ --- to
	$\DY(v_1^{-1}\md\ipb,\SPB)
	\ldots
	\DY(v_N^{-1}\md\ipb,\SPB)$.
	The initial configuration at the bottom
	is empty, and the 
	final configuration of the solid (black) paths at the top
	is $\bv_{(0^{M})}=\bv_M\otimes \bv_0\otimes \bv_0 \ldots$.
	The locations where the
	solid paths cross the horizontal division line
	correspond to the signature $\nu=(\nu_1,\ldots,\nu_M)$.
	The opaque (red) paths complement the 
	solid (black) paths in the top part, this corresponds 
	to the renormalization \eqref{D_Bar_computation}
	employed in the passage from the operators 
	$\DY(v_j^{-1}\md\ipb,\SPB)$ to $\overline\DY(v_j^{-1}\md\ipbb,\SPB)$
	(the latter
	involves coefficients
	$\G^{\conj}_{\la/\mu}(v_j\md\ipbb,\SPB)$).
	After the renormalization, we let the width 
	of the grid go to infinity.}
	\label{fig:MM_pictorially}
\end{figure}
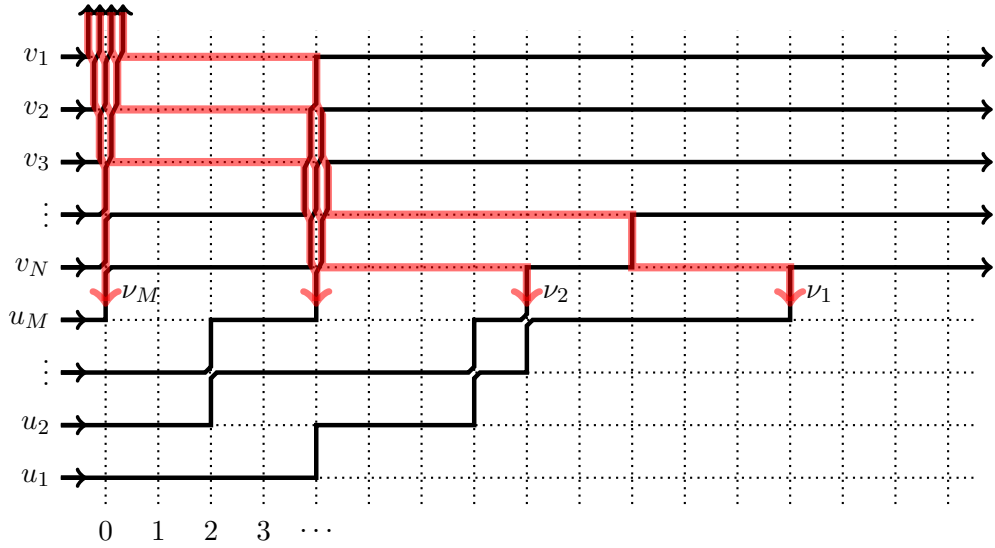

The measures $\MM_{\UU;\VV}$ can be 
represented pictorially, see Fig.~\ref{fig:MM_pictorially}.
Let us look at the bottom part of the path collection as in 
Fig.~\ref{fig:MM_pictorially}, and 
let us denote the positions of the vertical edges
at the $k$-th horizontal by
$\nu^{(k)}_{i}$, $1\le i\le k\le M$
see Fig.~\ref{fig:GT_scheme_paths}, bottom.
In the top part of the path collection, let us denote the 
coordinates of the vertical edges by 
$\tnu^{(\ell)}_{j}$,
$1\le \ell\le N$, $1\le j\le M$
(see Fig.~\ref{fig:GT_scheme_paths}, top).
We have
$\nu^{(M)}_{i}=\tnu^{(N)}_{i}=\nu_i$, $i=1,\ldots,M$.
By construction, these coordinates 
satisfy \emph{interlacing constraints}:
\begin{align}\label{interlacing_constraints}
	\nu^{(k)}_{i}\le \nu^{(k-1)}_{i-1}\le \nu^{(k)}_{i-1},
	\qquad
	\tnu^{(\ell)}_{j}\le \tnu^{(\ell-1)}_{j-1}\le \tnu^{(\ell)}_{j-1}
\end{align}
for all meaningful values of $k,i$ and $\ell,j$.
Arrays of the form 
$\{\nu^{(k)}_{i}\}_{1\le i\le k\le M}$
satisfying the above interlacing properties
are also sometimes called
\emph{Gelfand--Tsetlin schemes/patterns}.
By the very definition of the skew 
$\F$ and $\G$ functions, the distribution
of the sequence of signatures
$(\nu^{(1)},\ldots,\nu^{(M)}=\tnu^{(N)},\ldots,\tnu^{(1)})$
has the form
\begin{multline}\label{MM_Process}
	\MM_{\UU;\VV}(\nu^{(1)},\ldots,\nu^{(M)}=\tnu^{(N)},\ldots,\tnu^{(1)}\md\ipb,\SPB)
	\\=
	\frac{1}{Z(\UU;\VV\md\ipb,\SPB)}\,
	\F_{\nu^{(1)}}(u_1\md\ipb,\SPB)
	\F_{\nu^{(2)}/\nu^{(1)}}(u_2\md\ipb,\SPB)
	\ldots
	\F_{\nu^{(M)}/\nu^{(M-1)}}(u_M\md\ipb,\SPB)
	\\\times
	\G_{\tnu^{(1)}}^{\conj}(v_1\md\ipbb,\SPB)
	\G_{\tnu^{(2)}/\tnu^{(1)}}^{\conj}(v_2\md\ipbb,\SPB)\ldots
	\G_{\tnu^{(N)}/\tnu^{(N-1)}}^{\conj}(v_N\md\ipbb,\SPB).
\end{multline}
The probability distribution \eqref{MM_Process} on 
interlacing arrays is an analogue of 
Schur or Macdonald processes of 
\cite{okounkov2003correlation}, 
\cite{BorodinCorwin2011Macdonald}.
It readily follows from the Pieri rules
(Corollary \ref{cor:Pieri})
that under \eqref{MM_Process}, the marginal distribution of 
$\nu^{(k)}$ for any $k=1,\ldots,M$ is 
$\MM_{(u_1,\ldots,u_k); (v_1,\ldots,v_N)}$, and
similarly the marginal distribution of 
$\tnu^{(\ell)}$, $\ell=1,\ldots,N$,
is
$\MM_{(u_1,\ldots,u_M); (v_1,\ldots,v_\ell)}$.

\begin{figure}[htbp]
	\scalebox{.9}{\begin{tikzpicture}
		[scale=.8,thick]
		\def\d{.11}
		\foreach \xxx in {0,...,14}
		{
			\draw[dotted, opacity=.7] (\xxx,4.5)--++(0,5);
		}
		\foreach \xxx in {5,...,9}
		{
			\draw[dotted, opacity=.7] (-.85,\xxx)--++(15.35,0);
		}
		\draw[->, line width=1.7pt] (3*\d,9.85)--++(0,-1+.5);
		\draw[->, line width=1.7pt] (1*\d,9.85)--++(0,-1+.5);
		\draw[->, line width=1.7pt] (-1*\d,9.85)--++(0,-1+.5);
		\draw[->, line width=1.7pt] (-3*\d,9.85)--++(0,-1+.5);
		\draw[->, line width=1.7pt]
		(3*\d,9.85)--++(0,-1+.15)--++(4-3*\d,0)--++(0,-1+\d)
		--++(\d,-2*\d)--++(0,-1+2*\d)--++(\d,-2*\d)--++(0,-1+\d)
		--++(6-\d*2,0)--++(0,-1)--++(3,0)--++(0,-.75) 
		node[left] {$\tnu^{(6)}_{1}$};
		\draw[->, line width=1.7pt]
		(1*\d,9.85)--++(0,-1+\d+.15)--++(\d,-2*\d)--++(0,-1+\d)--++(4-3*\d,0)
		--++(0,-1+\d)--++(\d,-2*\d)--++(0,-1+2*\d)--++(\d,-2*\d)--++(0,-1+\d)
		--++(4-\d,0)--++(0,-.75) node[left] {$\tnu^{(6)}_{2}$};
		\draw[->, line width=1.7pt]
		(-1*\d,9.85)--++(0,-1+\d+.15)--++(\d,-2*\d)--++(0,-1+2*\d)
		--++(\d,-2*\d)--++(0,-1+\d)--++(4-3*\d,0)--++(0,-1+\d)
		--++(\d,-2*\d)--++(0,-1+2*\d)--++(\d,-2*\d)--++(0,-.75+\d)
		node[left] {$\tnu^{(6)}_{3}$};
		\draw[->, line width=1.7pt]
		(-3*\d,9.85)--++(0,-1+\d+.15)--++(\d,-2*\d)--++(0,-1+2*\d)
		--++(\d,-2*\d)--++(0,-1+2*\d)
		--++(\d,-2*\d)--++(0,-1+\d)--++(0,-1.75)
		node[left] {$\tnu^{(6)}_{4}$};
		\node[left, anchor=east] at (0, 5.55) {$\tnu^{(5)}_{4}$};
		\node[left, anchor=east] at (0, 6.55) {$\tnu^{(4)}_{4}$};
		\node[left, anchor=east] at (-\d, 7.55) {$\tnu^{(3)}_{4}=\tnu^{(3)}_{3}$};
		\node[left, anchor=east] at (-2*\d, 8.55) {$\tnu^{(2)}_{4}=\tnu^{(2)}_{3}=\tnu^{(2)}_{2}$};
		\node[left, anchor=east] at (-3*\d, 9.55) {$\tnu^{(1)}_{4}=\tnu^{(1)}_{3}=\tnu^{(1)}_{2}=\tnu^{(1)}_{1}$};
		\node[left, anchor=east] at (4-\d, 5.55) {$\tnu^{(5)}_{3}=\tnu^{(5)}_{2}$};
		\node[left, anchor=east] at (10, 5.55) {$\tnu^{(5)}_{1}$};
		\node[right, anchor=west] at (4+2*\d, 6.55) 
		{$\tnu^{(4)}_{3}=\tnu^{(4)}_{2}=\tnu^{(4)}_{1}$};
		\node[left, anchor=east] at (4-\d, 7.55) {$\tnu^{(3)}_{2}=\tnu^{(3)}_{1}$};
		\node[left, anchor=east] at (4, 8.55) {$\tnu^{(2)}_{1}$};
	\end{tikzpicture}}

	\scalebox{.9}{\begin{tikzpicture}
		[scale=.8,thick]
		\def\d{.1}
		\foreach \xxx in {0,...,14}
		{
			\draw[dotted, opacity=.7] (\xxx,.5)--++(0,4);
		}
		\foreach \xxx in {1,...,4}
		{
			\draw[dotted, opacity=.7] (-.85,\xxx)--++(15.35,0);
			\draw[->, line width=1.7pt] (-.85,\xxx)--++(.5,0);
		}
		\draw[->, line width=1.7pt] (-.85,4)--++(.85,0)--++(0,1)
		node[left] {${\nu^{(4)}_{4}}$};
		\draw[->, line width=1.7pt] (-.85,3)--++(.85,0)--++(2-\d,0)
		--++(\d,\d)--++(0,1-\d)--++(2,0)
		--++(0,1) node[left] {${\nu^{(4)}_{3}}$};
		\draw[->, line width=1.7pt] (-.85,2)--++(.85,0)--++(2,0)
		--++(0,1-\d)--++(\d,\d)--++(5-2*\d,0)--++(\d,\d)--++(0,1-\d)
		--++(1-\d,0)--++(\d,\d)--++(0,1-\d)
		node[left] {${\nu^{(4)}_{2}}$};
		\draw[->, line width=1.7pt] (-.85,1)--++(.85,0)
		--++(4,0)--++(0,1)--++(3,0)--++(0,1-\d)--++(\d,\d)
		--++(1-\d,0)--++(0,1-\d)--++(\d,\d)--++(5-\d,0)
		--++(0,1) node[left] {${\nu^{(4)}_{1}}$};
		\node[left] at (2,3.55) {$\nu^{(3)}_{3}$};
		\node[left] at (7,3.55) {$\nu^{(3)}_{2}$};
		\node[right] at (8,3.55) {$\nu^{(3)}_{1}$};
		\node[left] at (7,2.55) {$\nu^{(2)}_{1}$};
		\node[left] at (2,2.55) {$\nu^{(2)}_{2}$};
		\node[left] at (4,1.55) {$\nu^{(1)}_{1}$};
		\node[left, anchor=east] at (-3*\d, 3.55) {\phantom{$\tnu^{(1)}_{4}=\tnu^{(1)}_{3}=\tnu^{(1)}_{2}=\tnu^{(1)}_{1}$}};
		\node at (0,0) {$0$};
		\node at (1,0) {$1$};
		\node at (2,0) {$2$};
		\node at (3,0) {$3$};
		\node at (4,0) {$\ldots$};
	\end{tikzpicture}}
	\caption{A pair of interlacing arrays from path collections.
	Horizontal parts
	of the paths are for illustration.}
	\label{fig:GT_scheme_paths}
\end{figure}
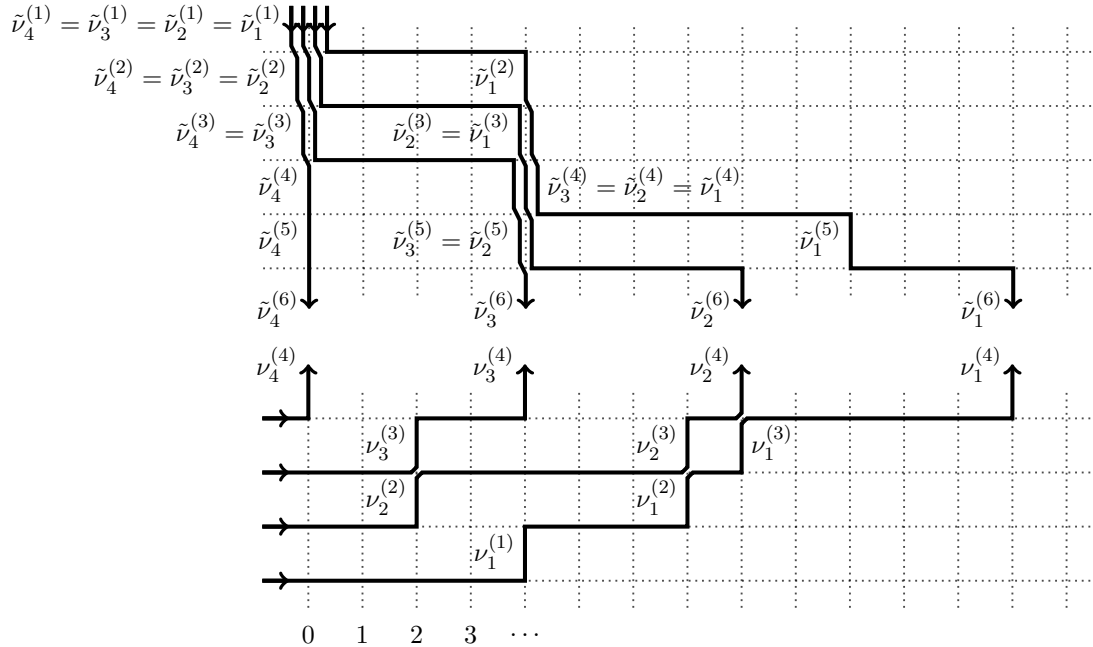


\subsection{Four Markov kernels} 
\label{sub:four_markov_kernels}

Let us now define four 
Markov kernels
which map the measure
$\MM_{\UU;\VV}$
to a measure of the same form, but with 
modified parameters $\UU$ or $\VV$.

The first two Markov kernels, $\Lam$ and $\Lae$, correspond 
to taking conditional distributions given $\nu^{(M)}=\tnu^{(N)}$
of 
$\nu^{(k)}$ or
$\tnu^{(\ell)}$, respectively,
in the ensemble \eqref{MM_Process}.
Namely, let us define for any $m$:
\begin{align}\label{Lam}
	\Lam_{u\md\UU}(\nu\to\mu):=
	\frac{\F_\mu(u_1,\ldots,u_m\md\ipb,\SPB)}{\F_\nu(u_1,\ldots,u_m,u\md\ipb,\SPB)}\,
	\F_{\nu/\mu}(u\md\ipb,\SPB),
\end{align}
where $\UU=(u_1,\ldots,u_m)$, and 
$\nu\in\signp{m+1}$, $\mu\in\signp{m}$.
Also, let us define for any $n$:
\begin{align}\label{Lae}
	\Lae_{v\md\VV}(\la\to\mu):=
	\frac{\G^{\conj}_{\mu}(v_1,\ldots,v_n\md\ipbb,\SPB)}
	{\G^{\conj}_{\la}(v_1,\ldots,v_n,v\md\ipbb,\SPB)}\,
	\G^{\conj}_{\la/\mu}(v\md\ipbb,\SPB),
\end{align}
where $\VV=(v_1,\ldots,v_n)$, and $\la,\mu\in\signp m$ for some $m$.
The facts that the quantities \eqref{Lam} and \eqref{Lae}
sum to 1 (over all $\mu\in\signp m$; note that these sums are finite) 
follow from 
the branching rules (Proposition \ref{prop:branching}). Hence, 
$\Lam_{u\md\UU}\colon\signp{m+1}\dashrightarrow\signp m$
and 
$\Lae_{v\md\VV}\colon\signp{m}\dashrightarrow\signp m$
define Markov kernels.\footnote{We use the notation ``$\dashrightarrow$''
to indicate that $\Lam_{u\md\UU}$ and 
$\Lae_{v\md\VV}$ are Markov kernels, i.e., they are
functions in the first variable 
(belonging to the space on the left of ``$\dashrightarrow$'')
and probability distributions in the second variable
(belonging to the space on the right of ``$\dashrightarrow$'').}
Note that in \eqref{Lam} and \eqref{Lae}
one can replace all functions by $\F^{\conj}$
or $\G$, respectively, and get the same kernels.

The kernels $\Lam$ and $\Lae$
act on the measures \eqref{MM_measure}
as 
\begin{align}\label{M_and_Lambdas}
	\MM_{\UU\cup u;\VV}\Lam_{u\md\UU}=
	\MM_{\UU;\VV}
	,\qquad \qquad
	\MM_{\UU;\VV\cup v}\Lae_{v\md\VV}=
	\MM_{\UU;\VV},
\end{align}
this follows from the Pieri rules (Corollary \ref{cor:Pieri}).
The matrix products above are understood in a natural way, for example,
$\big(\MM_{\UU\cup u;\VV}\Lam_{u\md\UU}\big)(\mu)=
\sum_{\nu}\MM_{\UU\cup u;\VV}(\nu)\Lam_{u\md\UU}(\nu\to\mu)$.

\begin{remark}[Gibbs measures]
	Conditioned on 
	any $\nu^{(k)}$ (where $k=1,\ldots,M$),
	the distribution 
	of the lower levels
	$\nu^{(1)},\ldots,\nu^{(k-1)}$
	under \eqref{MM_Process}
	is independent of $\VV$ and is given by 
	\begin{align}\label{Gibbs_property}
		\Lam_{u_k\md(u_1,\ldots,u_{k-1})}(\nu^{(k)}\to\nu^{(k-1)})\ldots
		\Lam_{u_3\md(u_1,u_2)}(\nu^{(3)}\to\nu^{(2)})
		\Lam_{u_2\md(u_1)}(\nu^{(2)}\to\nu^{(1)}),
	\end{align}
	and a similar expression can be 
	written for conditioning on $\tnu^{(\ell)}$, 
	yielding a distribution
	which is independent
	of $\UU$ 
	and involves the kernels $\Lae$.

	It is natural to call a measure on 
	a sequence of interlacing signatures
	$(\nu^{(1)},\ldots,\nu^{(M)})$
	whose conditional distributions are given by \eqref{Gibbs_property}
	a \emph{Gibbs measure} (with respect to the $\UU$ parameters).
	In fact, when $q=0$, $\SP_j\equiv 0$, $\ip_j\equiv 1$, and $u_i\equiv 1$, 
	this Gibbs property turns into the following:
	conditioned on any 
	$\nu^{(k)}$,
	the distribution 
	of the lower levels
	$\nu^{(1)},\ldots,\nu^{(k-1)}$
	is \emph{uniform}
	among all sequences of signatures satisfying the interlacing constraints
	\eqref{interlacing_constraints}.

	This Gibbs property (as well as commutation relations discussed 
	below in this subsection) can be used to construct ``multivariate''
	Markov kernels on
	arrays of interlacing signatures which act nicely on 
	distributions of the form \eqref{MM_Process}, 
	but we will not address this construction here
	(about similar constructions
	see references given in 
	Remark \ref{rmk:Doob_h_transform} below). For details of 
	such constructions in the case of Macdonald
	processes see \cite{BorodinPetrov2013NN},
	\cite{MatveevPetrov2014}.
\end{remark}

The other two Markov kernels, $\Qp$ and $\Qe$, increase the number of 
parameters in the measures $\MM_{\UU;\VV}$, as opposed to \eqref{M_and_Lambdas}, where the number of parameters is decreased. 
These kernels are defined as follows.
For any $n,m\in\Z_{\ge0}$,
define
\begin{align}\label{Qp}
	\Qp_{u;\VV}(\la\to\nu):=\frac{1-\SP_0\ip_0u}{1-q^{m+1}}
	\bigg(\prod_{j=1}^{n}\frac{1-uv_j}{1-quv_j}\bigg)
	\frac{\G_{\nu}^{\conj}(v_1,\ldots,v_n\md\ipbb,\SPB)}
	{\G_{\la}^{\conj}(v_1,\ldots,v_n\md\ipbb,\SPB)}
	\,\F_{\nu/\la}(u\md\ipb,\SPB),
\end{align}
where 
$\VV=(v_1,\ldots,v_n)$ such that
$\adm u{v_j}$ for all $j$, with 
$\la\in\signp m$ and $\nu\in\signp {m+1}$.
Also, for any $m\in\Z_{\ge0}$, define
\begin{align}\label{Qe}
	\Qe_{\UU;v}(\mu\to\nu):=
	\bigg(\prod_{i=1}^{m}\frac{1-u_iv}{1-qu_iv}\bigg)
	\frac{\F_{\nu}(u_1,\ldots,u_m\md\ipb,\SPB)}{\F_{\mu}(u_1,\ldots,u_m\md\ipb,\SPB)}
	\,\G^{\conj}_{\nu/\mu}(v\md\ipbb,\SPB),
\end{align}
where $\UU=(u_1,\ldots,u_M)$
such that $\adm{u_i}v$ for all $i$,
with $\mu,\nu\in\signp m$.
By the Pieri rules of Corollary \ref{cor:Pieri},
$\Qp_{u;\VV}\colon\signp m\dashrightarrow \signp {m+1}$
and 
$\Qe_{\UU;v}\colon\signp m\dashrightarrow \signp {m}$
define Markov kernels (i.e., they sum to one in the
second argument).

\begin{remark}
	In $\Qe_{\UU;v}$ \eqref{Qe}, moving the conjugation from the function $\G$ to 
	both functions
	$\F$ does not change the kernel. 
	However, doing so in $\Qp_{u;\VV}$
	\eqref{Qp} requires modifying the prefactor:
	\begin{align*}
		\Qp_{u;\VV}(\la\to\nu)=\frac{1-\SP_0\ip_0u}{1-\SP_0^{2}q^{m}}
		\bigg(\prod_{j=1}^{n}\frac{1-uv_j}{1-quv_j}\bigg)
		\frac{\G_{\nu}(v_1,\ldots,v_n\md\ipbb,\SPB)}
		{\G_{\la}(v_1,\ldots,v_n\md\ipbb,\SPB)}
		\,\F^{\conj}_{\nu/\la}(u\md\ipb,\SPB).
	\end{align*}
\end{remark}
From the branching rules (Proposition \ref{prop:branching})
it readily follows that the kernels
$\Qp$ and $\Qe$ act on the measures \eqref{MM_measure} as
\begin{align}\label{MM_and_Q}
	\MM_{\UU;\VV}\Qp_{u;\VV}=\MM_{\UU\cup u;\VV},\qquad
	\qquad
	\MM_{\UU;\VV}\Qe_{\UU;v}=\MM_{\UU;\VV\cup v}.
\end{align}

The Markov kernels defined above enter the following 
commutation relations:
\begin{proposition}\label{prop:Q_La_commutation}
\noindent{\bf1.\/}
For any $\UU=(u_1,\ldots,u_m)$ and $u,v\in\C$ such that $\adm uv$,
we have 
$\Qe_{\UU\cup u;v}\Lam_{u\md\UU}=
\Lam_{u\md\UU}\Qe_{\UU;v}$ 
(as Markov kernels $\signp {m+1}\dashrightarrow\signp m$), or, in more detail,
\begin{align*}
	\sum_{\la\in\signp{m+1}}
	\Qe_{\UU\cup u;v}(\nu\to\la)\Lam_{u\md\UU}(\la\to\mu)
	=
	\sum_{\kappa\in\signp m}
	\Lam_{u\md\UU}(\nu\to\kappa)\Qe_{\UU;v}(\kappa\to\mu),
\end{align*}
where $\nu\in\signp{m+1}$ 
and 
$\mu\in\signp m$.

\noindent{\bf2.\/}
For any $\VV=(v_1,\ldots,v_n)$
and $u,v\in\C$ such that $\adm uv$, we have
$\Qp_{u;\VV\cup v}\Lae_{v\md\VV}
=
\Lae_{v\md\VV}\Qp_{u;\VV}$
(as Markov kernels $\signp {m}\dashrightarrow\signp {m+1}$),
which is unabbreviated in the same way as the first relation.
\end{proposition}
\begin{proof}
	A straightforward corollary of the skew Cauchy identity
	(Proposition \ref{prop:skew_Cauchy}).
\end{proof}

\begin{remark}
	In the context of Schur functions, the 
	Markov kernels 
	$\Qp$ and $\Lam$
	are often referred to as \emph{transition} and \emph{cotransition probabilities}.
	In \cite[\S9]{Borodin2010Schur} and \cite[\S2.3.3]{BorodinCorwin2011Macdonald}
	similar kernels are denoted by $p^{\uparrow}$ and $p^{\downarrow}$, respectively. 
	The kernels 
	$\Qp$ and $\Lam$ involve the skew functions $\F$ in the $\UU$ parameters, 
	and similarly
	$\Qe$ and $\Lae$ correspond to the $\G$'s in the $\VV$ parameters.
	The latter operators differ form the former ones because 
	(unlike in the Schur or Macdonald setting)
	the functions $\F$ and $\G$ are not proportional to each other.
\end{remark}

We will treat the Markov kernels $\Qe_{\UU;v}$ and $\Qp_{u;\VV}$
as one-step transition operators of certain discrete time
Markov chains. 

\begin{remark}\label{rmk:EF_relation_Qe}
	One can readily write down eigenfunctions
	of $\Qe_{\UU;v}$ viewed as an operator
	on functions on $\signp m$.
	Here we mean 
	algebraic (or formal) eigenfunctions, 
	i.e., we do not address the question 
	of how they decay at infinity.
	We have for any $\mu\in\signp m$:
	\begin{multline}
		\big(\Qe_{\UU;v}\EF{\UU}\bullet(z_1,\ldots,z_m)\big)(\mu)
		=\sum_{\nu\in\signp m}\Qe_{\UU;v}(\mu\to\nu)
		\EF{\UU}\nu(z_1,\ldots,z_m)\\=
		\bigg(\prod_{i=1}^{m}\frac{1-qz_i v}{1-z_i v}
		\frac{1-u_iv}{1-qu_iv}
		\bigg)\EF{\UU}\mu(z_1,\ldots,z_m),
		\label{EF_relation_Qe}
	\end{multline}
	where 
	the eigenfunction 
	$\EF\UU\la(\ZZZ)$
	depends on the spectral variables $\ZZZ=(z_1,\ldots,z_m)$
	satisfying 
	the admissibility conditions $\adm{z_i}v$ for all $i$,
	and is defined as follows:
	\begin{align}
		\label{EF_Qe}
		\EF\UU\la(z_1,\ldots,z_m):=
		\frac{1}{{\F_\la(u_1,\ldots,u_m\md\ipb,\SPB)}}\,{\F_\la(z_1,\ldots,z_m\md\ipb,\SPB)}.
	\end{align}
	Relation \eqref{EF_relation_Qe} readily follows from the Pieri rules (Corollary \ref{cor:Pieri}).
	This eigenrelation can be employed to 
	write down a 
	spectral decomposition of the operator $\Qe$,
	see Remark \ref{rmk:spec_decomp} below.
\end{remark}


\subsection{Specializations} 
\label{sub:specializations_pi_and_rho_}

Let us now discuss special choices of parameters $\UU$ and $\VV$ which 
greatly simplify the Markov kernels 
$\Qe_{\UU;v}$ and $\Qp_{u;\VV}$, respectively. First, observe that
for any $m\in\Z_{\ge0}$ and any $\mu\in\signp m$ we have
\begin{align*}
	\F_\mu(\PI m\md\ipb,\SPB)=
	\F_\mu(\underbrace{0,0,\ldots,0}_{\textnormal{$m$ times}}\md\ipb,\SPB)
	=(-\SPB)^{\mu}(q;q)_m,
\end{align*}
where the last equality is due to 
\eqref{F_prin_spec_simple}
because we can take $u=0$ 
in that formula 
(note that by \eqref{F_symm_formula}, the function $\F_\mu(u_1,\ldots,u_m)$ is continuous
at $\UU=\PI m$).
We have also used the notation \eqref{SPB_la_notation}.

A similar limit for the functions $\G_\mu$ is given in the next proposition:
\begin{proposition}\label{prop:RHO_spec}
	For any $n\in\Z_{\ge0}$ and $\nu\in\signp n$, we have
	\begin{align}
		\label{RHO_spec}
		\G_\nu(\RHO\md\ipbb,\SPB):=
		\lim_{\epsilon\to0}\Big(\G_\nu(\epsilon,q\epsilon,\ldots,q^{J-1}\epsilon\md\ipbb,\SPB)\Big\vert_{q^{J}=
		\ip_0/(\SP_0\epsilon)}\Big)=
		\begin{cases}
			(-\SPB)^{\nu}(\SP_0^{2};q)_{n}\SP_0^{-2n},&\textnormal{if $\nu_n>0$};
			\\
			0,&\textnormal{if $\nu_n=0$}.
		\end{cases}
	\end{align}
\end{proposition}
Note that 
\eqref{RHO_spec}
does not depend on the inhomogeneity parameters $\ipb$.
\begin{proof}
	Let $k$ be the number of zero coordinates in $\nu$.
	From \eqref{G_prin_spec_simple} we have for $J\ge n-k$:
	\begin{multline*}
		\G_\nu(\epsilon,q\epsilon,\ldots,q^{J-1}\epsilon\md\ipbb,\SPB)=
		\frac{(q;q)_{J}}{(q;q)_{J-n+k}}
		\frac{(\SP_0\ip_0^{-1}\epsilon ;q)_{J+k}}{(\SP_0\ip_0^{-1}\epsilon ;q)_{n}}
		\frac{(\SP_0^{2};q)_{n}}{(\SP_0^{2};q)_{k}}
		\frac{1}{(\ip_0\SP_0/\epsilon;q^{-1})_{n-k}}
		\\\times
		\frac{1}{(\SP_0\ip_0^{-1}q^{n-k}\epsilon;q)_{J-n+k}}
		\prod_{j=1}^{n-k}
		\bigg(
		\frac{1}{1-\SP_{\nu_j}\ip_{\nu_j}^{-1} q^{j-1}\epsilon }
		\prod_{\ell=0}^{{\nu_j}-1}
		\frac{\ip_\ell^{-1} q^{j-1}\epsilon -\SP_\ell}{1-\SP_\ell\ip_\ell^{-1} q^{j-1}\epsilon }
		\bigg).
	\end{multline*}
	The $\epsilon\to0$ limit of the product over $j$ above (which is independent of $q^J$)
	gives $(-\SPB)^{\nu}$. We can rewrite the prefactor as follows:
	\begin{align*}
		&\frac{(q;q)_{J}(\SP_0 ^{2};q)_{n}(\SP_0 \epsilon/\ip_0;q)_{J+k}}{(q;q)_{J-n+k}
		(\SP_0 ^{2};q)_{k}(\SP_0 \epsilon/\ip_0;q)_{n}(\SP_0 \ip_0/\epsilon;q^{-1})_{n-k}(\SP_0 \epsilon q^{n-k}/\ip_0;q)_{J-n+k}}
		\\&\hspace{80pt}=
		\frac{(\SP_0 ^{2}q^{k};q)_{n-k}}{(\SP_0 \ip_0/\epsilon;q^{-1})_{n-k}(\SP_0 \epsilon/\ip_0;q)_{n}}
		(q^{J+1+k-n};q)_{n-k}(\SP_0 \epsilon/\ip_0;q)_{n-k}(\SP_0 \epsilon q^{J}/\ip_0;q)_{k}.
	\end{align*}
	One readily sees that 
	the above quantity 
	depends on $q^{J}$ in a rational manner.\footnote{This can also
	be thought of as a consequence of the fusion procedure
	(\S \ref{sub:principal_specializations_of_skew_functions}),
	but the statement of the proposition does not require fusion.}
	This allows to analytically continue in $q^{J}$, and set
	$q^{J}=\ip_0/(\SP_0\epsilon)$. Observe that 
	the result involves 
	$(1;q)_k$, which vanishes unless $k=0$. 
	For $k=0$ we obtain:
	\begin{align*}
		\frac{(\SP_0^{2};q)_{n}}{(\SP_0\ip_0/\epsilon;q^{-1})_{n}(\SP_0\epsilon/\ip_0;q)_{n}}
		(\ip_0(\SP_0\epsilon)^{-1}q^{1-n};q)_{n}(\SP_0\epsilon/\ip_0;q)_{n}=
		\frac{(\SP_0^{2};q)_{n}}{(\SP_0\ip_0/\epsilon;q^{-1})_{n}}
		(\ip_0(\SP_0\epsilon)^{-1}q^{1-n};q)_{n},
	\end{align*}
	and in the $\epsilon\to0$ limit this turns into 
	$(\SP_0^{2};q)_{n}\SP_0^{-2n}$, which completes the proof.
\end{proof}
\begin{remark}
	An alternative proof of Proposition \ref{prop:RHO_spec} 
	(and in fact a computation of a more general specialization) using the 
	integral formula of Corollary \ref{cor:G_integral_formula} below
	is discussed in \S\ref{sub:computation_of_gnurhow}.
\end{remark}

Let us substitute the above specializations $\PI m$ and
$\RHO$ into the Markov kernels.
The kernel $\Qe_{\PI m;v}$ looks as follows:
\begin{align}\label{Qe_PI}
	\Qe_{\PI m;v}(\mu\to\nu)=
	\frac{(-\SPB)^{\nu}}{(-\SPB)^{\mu}}
	\,\G^{\conj}_{\nu/\mu}(v\md\ipbb,\SPB),
	\qquad\mu,\nu\in\signp m.
\end{align}
Similarly, the kernel $\Qp_{u;\RHO}$ has the form
\begin{align*}
	\Qp_{u;\RHO}(\la\to\nu)=
	\frac{1-\SP_0\ip_0u}{\SP_0(\SP_0-u\ip_0)}
	\frac{(-\SPB)^{\nu}}
	{(-\SPB)^{\la}}
	\,\F^{\conj}_{\nu/\la}(u\md\ipb,\SPB),
\end{align*}
where
$\la\in\signp m$, $\nu\in\signp{m+1}$ are such that $\la_m,\nu_{m+1}>0$.
Because of this latter condition, we can subtract $1$ from all parts of $\la$ and $\nu$, and rewrite 
$\Qp_{u;\RHO}$ as follows:
\begin{align}
	\label{Qp_RHO}
	\Qp_{u;\RHO}(\la\to\nu)=
	\frac{(-\SPB)^{\nu}}
	{(-\SPB)^{\la}}
	\frac{1}{\Lmatr_{\ip_0u,\SP_0}(0,1;0,1)}
	\,\F^{\conj}_{\nu/\la}(u\md\ipb,\SPB)=
	\frac{(-\sh_1\SPB)^{\nu-1^{m+1}}}
	{(-\sh_1\SPB)^{\la-1^{m}}}
	\,\F^{\conj}_{(\nu-1^{m+1})/(\la-1^{m})}(u\md\sh_1\ipb,\sh_1\SPB),
\end{align}
where $\sh_1$ is the shift \eqref{sh_operation}.


\subsection{Interacting particle systems} 
\label{sub:interacting_particle_systems}

Fix $M\in\Z_{\ge0}$. Let us interpret $\signp M$ 
as the space of $M$-particle configurations on $\Z_{\ge0}$,
in which putting an arbitrary number of particles per site is allowed (particles are assumed to be identical).
That is, each
$\la=0^{\ell_0}1^{\ell_1}2^{\ell_2}\ldots\in\signp M$ corresponds to 
having $\ell_0$ particles at site $0$, $\ell_1$ particles at site $1$, and so on.

We can interpret the Markov kernels
$\Qe_{\UU;v}$ and $\Qp_{u;\VV}$ (for any $\UU$ or $\VV$)
as one-step transition operators of two discrete time Markov chains.
Denote these Markov chains by 
$\Xe_{\UU;\{v_t\}}$ 
and $\Xp_{\{u_t\};\VV}$, respectively.
Here $\{v_t\}_{t\in\Z_{\ge0}}$
and $\{u_t\}_{t\in\Z_{\ge0}}$
are time-dependent parameters which are 
added during one step of 
$\Xe$ or $\Xp$, respectively (we tacitly assume that all parameters $u_i$ and $v_j$
satisfy the necessary admissibility conditions as in \S \ref{sub:four_markov_kernels}).

For generic $\UU$ and $\VV$ parameters, the Markov chains
$\Xe_{\UU;\{v_t\}}$ 
and $\Xp_{\{u_t\};\VV}$, respectively, are \emph{nonlocal}, 
i.e., transitions at a given location depend on the whole particle configuration.
However, taking $\UU=\PI m$ or $\VV=\RHO$ in the corresponding chain
makes them \emph{local}
(in fact, we will get certain \emph{sequential update} rules). 
\begin{remark}\label{rmk:Doob_h_transform}
	The origin of nonlocality in the above Markov chains 
	is the conjugation of the skew functions that is necessary
	for the transition probabilities to add up to 1
	(cf. \eqref{Qp} and \eqref{Qe}). This conjugation may be viewed as 
	an instance of the 
	classical Doob's $h$-transform
	(we refer to, e.g., \cite{konig2002non}, \cite{Konig2005} for details).
	
	Another way of introducing locality to 
	Markov chains $\Xe_{\UU;\{v_t\}}$ 
	and $\Xp_{\{u_t\};\VV}$ that works for generic $\UU$ and~$\VV$, respectively, 
	could be to consider 
	``multivariate'' chains on whole interlacing arrays (similarly to, e.g., 
	\cite{OConnell2003Trans}, \cite{OConnell2003}, \cite{BorFerr2008DF}, \cite{BorodinPetrov2013NN},
	\cite{MatveevPetrov2014},
	with \cite{BorodinBufetov2015} 
	providing an application
	to the six vertex model on the torus),
	but we will not discuss this here.
\end{remark}

Let us discuss update rules 
of the dynamics 
$\Xe_{\PI M;\{v_t\}}$ 
and $\Xp_{\{u_t\};\RHO}$ in detail. They follow from 
\eqref{Qe_PI} and \eqref{Qp_RHO}
combined with 
the interpretation
of functions $\F$ and $\G$
as partition functions of path collections with stochastic vertex weights \eqref{vertex_weights_stoch}.

\subsubsection{Dynamics $\Xe_{\PI M;\{v_t\}}$}
\label{ssub:dyn_xe}

Fix $M\in\Z_{\ge0}$.
During each time step $t\to t+1$ of
the chain $\Xe_{\PI M;\{v_t\}}$,
the current configuration
$\mu=0^{m_0}1^{m_1}2^{m_2}\ldots\in\signp M$ 
is randomly changed to
$\nu=0^{n_0}1^{n_1}2^{n_2}\ldots\in\signp M$ 
according to the following sequential (left to right) update.
First, 
choose $n_0\in\{0,1,\ldots,m_0\}$ from the probability distribution
$$\Lmatr_{\ip_0^{-1}v_{t+1},\SP_0}(m_0,0;n_0,m_0-n_0),$$ and set $h_1:=m_0-n_0\in\{0,1\}$,
Then, having $h_1$
and $m_1$, choose $n_1\in\{0,1,\ldots,m_1+h_1\}$ 
from the probability distribution
$$\Lmatr_{\ip_1^{-1}v_{t+1},\SP_1}(m_1,h_1;n_1,m_1+h_1-n_1),$$
and set $h_2:=m_1+h_1-n_1\in\{0,1\}$.
Continue in the same manner for $x=2,3,\ldots$ 
by 
choosing
$n_{x}\in\{0,1,\ldots,m_x+h_x\}$ 
from the distribution
$$\Lmatr_{\ip_x^{-1}v_{t+1},\SP_x}(m_x,h_x;n_{x},m_x+h_x-n_{x}),$$
and setting $h_{x+1}:=m_x+h_x-n_{x}\in\{0,1\}$.
Since at each step the probability that 
$h_{x+1}=1$ is strictly less than $1$,
eventually for some $x>\mu_1$ we will have $h_{x+1}=0$, which means that the update will  
terminate (all the above choices are independent).
See Fig.~\ref{fig:particle_systems}, left, for an example.

\subsubsection{Dynamics $\Xp_{\{u_t\};\RHO}$}
\label{ssub:dyn_xp}

During each time step $t\to t+1$ of
the chain $\Xp_{\{u_t\};\RHO}$,
the current configuration
$\mu=1^{m_1}2^{m_2}\ldots\in\signp M$ 
is randomly changed to
$\nu=1^{n_1}2^{n_2}\ldots\in\signp {M+1}$
according to the following sequential (left to right) update
(note that here $M$ is increased with time,
and also that there cannot be any particles at location $0$).

First, choose $n_1\in\{0,1,\ldots,m_1+1\}$ from the probability distribution
$$\Lmatr_{\ip_1u_{t+1},\SP_1}(m_1,1;n_1,m_1+1-n_1),$$ and set $h_2:=m_1+1-n_1\in\{0,1\}$. 
The fact that $j_1=1$ in this stochastic vertex weight 
accounts for the incoming arrow from $0$.
For $x=2,3,\ldots$ continue in the same way, for each $x$ choosing 
$n_{x}\in\{0,1,\ldots,m_x+h_x\}$ from the probability distribution
$$\Lmatr_{\ip_x u_{t+1},\SP_x}(m_x,h_x;n_x,m_x+h_x-n_x),$$
and setting $h_{x+1}=m_x+h_x-n_x\in\{0,1\}$.
The update will eventually terminate when $h_{x+1}=0$ for some $x>\mu_1$.
See Fig.~\ref{fig:particle_systems}, right, for an example.

\begin{figure}[htbp]
	\begin{tabular}{cc}
		\scalebox{.9}{\begin{tikzpicture}
		[scale=1, very thick]
		\def\hh{-1.8}
		\def\d{.06}
		\draw[->] (-.5,0)--++(8.2,0);
		\draw[thick, densely dotted, opacity=.7] (-.5,\hh)--++(8,0);
		\foreach \iii in {0,1,2,3,4,5,6,7}
		{
			\draw (\iii,.1)--++(0,-.2) node [below] {$\iii$};
			\draw[thick, densely dotted, opacity=.7] (\iii,\hh-.5)--++(0,1);
		}
		\foreach \ppt in {(0,.3),(0,.6),(1,.3),(3,.3),(3,.6),(3,.9),(6,.3),(6,.6)}
		{
			\draw[thick] \ppt circle(3pt);
		}
		\draw[->] (0,.8) to [in=100, out=45] (1,.6);
		\draw[->] (1,.5) to [in=100, out=45] (2,.3);
		\draw[->] (3,1.1) to [out=35, in=180] (4.5,1.5) to [in=120, out=0] (6,.9);
		\draw[->] (-\d,\hh-.5)--++(0,.5-\d)--++(\d,2*\d)--++(0,.5-\d);
		\draw[->] (\d,\hh-.5)--++(0,.5)--++(1-\d-\d,0)--++(\d,\d)--++(0,.5-\d);
		\draw[->] (1,\hh-.5)--++(0,.5-\d)--++(\d,\d)--++(1-\d,0)--++(0,.5);
		\draw[->] (3-2*\d,\hh-.5)--++(0,.5-\d)--++(\d,2*\d)--++(0,.5-\d);
		\draw[->] (3,\hh-.5)--++(0,.5-\d)--++(\d,2*\d)--++(0,.5-\d);
		\draw[->] (3+2*\d,\hh-.5)--++(0,.5)--++(1-2*\d,0)--++(2-2*\d,0)--++(0,.5);
		\draw[->] (6-\d,\hh-.5)--++(0,.5-\d)--++(\d,2*\d)--++(0,.5-\d);
		\draw[->] (6+\d,\hh-.5)--++(0,.5-\d)--++(\d,2*\d)--++(0,.5-\d);
		\node at (1.9,-3.5)
		{\parbox{100pt}
		{
			$\Lmatr_0(2,0;1,1)\Lmatr_1(1,1;1,1)
						$

			$\qquad\times\Lmatr_2(0,1;1,0)\Lmatr_3(3,0;2,1)\Lmatr_4(0,1;0,1)
						$
			
			$\qquad \qquad\times\Lmatr_5(0,1;0,1)\Lmatr_6(2,1;3,0)	$
		}
		};
	\end{tikzpicture}}
	&\hspace{20pt}
	\scalebox{.9}{\begin{tikzpicture}
		[scale=1, very thick]
		\def\hh{-1.8}
		\def\d{.06}
		\draw[->] (-.5,0)--++(8.2,0);
		\draw[thick, densely dotted, opacity=.7] (-.5,\hh)--++(8,0);
		\foreach \iii in {0,1,2,3,4,5,6,7}
		{
			\draw (\iii,.1)--++(0,-.2) node [below] {$\iii$};
			\draw[thick, densely dotted, opacity=.7] (\iii,\hh-.5)--++(0,1);
		}
		\foreach \ppt in {(1,.3),(1,.6),(3,.3),(3,.6),(3,.9),(6,.3),(6,.6)}
		{
			\draw[thick] \ppt circle(3pt);
		}
		\draw[->] (1,.8) to [in=100, out=45] (2,.3);
		\draw[thick, densely dotted] (-1,.3) circle(3pt);
		\draw[->] (-1,.5) to [in=120, out=45] (1,.9);
		\draw[->] (3,1.1) to [out=35, in=180] (4.5,1.5) to [in=120, out=0] (6,.9);
		\draw[->] (-1,\hh)--++(2-3*\d,0)--++(2*\d,\d)--++(0,.5-\d);
		\draw[->] (1-\d,\hh-.5)--++(0,.5-\d)--++(2*\d,2*\d)--++(0,.5-\d);
		\draw[->] (1+\d,\hh-.5)--++(0,.5-\d)--++(2*\d,\d)--++(1-\d*3,0)--++(0,.5);
		\draw[->] (3-2*\d,\hh-.5)--++(0,.5-\d)--++(\d,2*\d)--++(0,.5-\d);
		\draw[->] (3,\hh-.5)--++(0,.5-\d)--++(\d,2*\d)--++(0,.5-\d);
		\draw[->] (3+2*\d,\hh-.5)--++(0,.5)--++(1-2*\d,0)--++(2-2*\d,0)--++(0,.5);
		\draw[->] (6-\d,\hh-.5)--++(0,.5-\d)--++(\d,2*\d)--++(0,.5-\d);
		\draw[->] (6+\d,\hh-.5)--++(0,.5-\d)--++(\d,2*\d)--++(0,.5-\d);
		\node at (1.9,-3.5)
		{\parbox{100pt}
		{
			$\Lmatr_1(2,1;2,1)\Lmatr_2(0,1;1,0)
						$

			$\qquad\times\Lmatr_3(3,0;2,1)\Lmatr_4(0,1;0,1)
						$
			
			$\qquad \qquad\times\Lmatr_5(0,1;0,1)\Lmatr_6(2,1;3,0)	$
		}
		};
	\end{tikzpicture}}
	\end{tabular}
	\caption{
	Left: a possible move under the chain $\Xe_{\PI M;\{v_t\}}$ with $M=8$
	(depicted in terms of particle and path configurations). 
	The probability of this move is also given, 
	where $\Lmatr_j=\Lmatr_{\ip_j^{-1}v_{t+1},\SP_{j}}$.
	Right: a possible move under the chain 
	$\Xp_{\{u_t\};\RHO}$
	with $M=7$ (so that the resulting configuration has $8$ particles).
	The probability of this move is 
	also given, with 
	$\Lmatr_j=\Lmatr_{\ip_ju_{t+1},\SP_{j}}$.
	}
	\label{fig:particle_systems}
\end{figure}
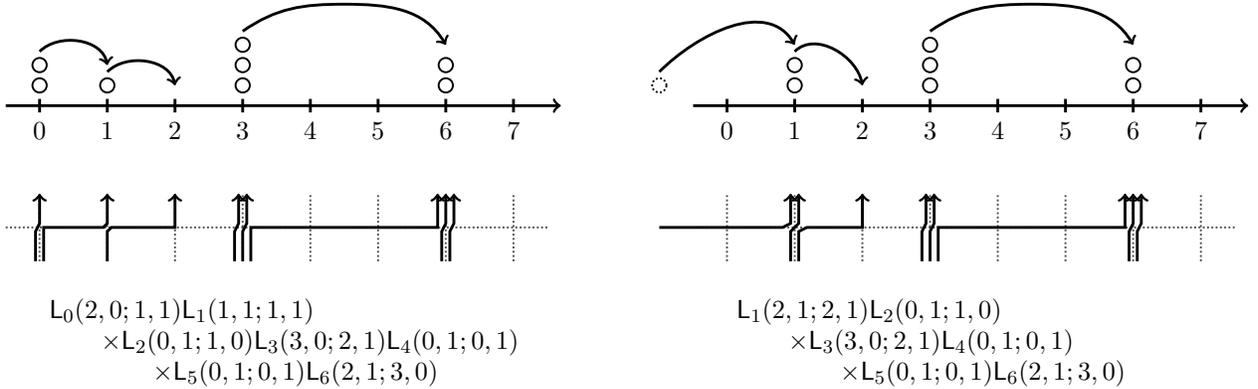

\subsubsection{Properties of dynamics} 
\label{ssub:properties_of_dynamics}

We will now list a number of immediate properties of the Markov chains 
$\Xe_{\PI M;\{v_t\}}$ and $\Xp_{\{u_t\};\RHO}$ described above.

\smallskip\noindent$\bullet$
Under both dynamics, particles move only to the right.
Moreover, at most one particle can leave 
any given stack of particles
and it can move only as far as the next nonempty stack of particles.

\smallskip\noindent$\bullet$
The property that 
at most one particle can leave 
any given stack of particles
is a $J=1$ feature. 
One can readily define \emph{fused dynamics}
involving stochastic vertex weights 
$\LJ{J}_{z,\SP}$
for any $J\in\Z_{\ge1}$ (see \S \ref{sub:fusion_of_stochastic_weights}). In these general $J$ dynamics,
at most $J$ particles 
can leave 
any given stack.
One step of a general $J$ dynamics 
(say, an analogue of $\Xp$)
can be thought of as simply 
combining $J$ steps of the $J=1$ dynamics with parameters
$u_t,qu_t,\ldots,q^{J-1}u_t$.
Results of \S \ref{sub:fusion_of_stochastic_weights}
show that one can then forget about the intermediate configurations 
during these $J$ steps, and still obtain a Markov chain. 
We will utilize these general $J$ Markov chains 
in \S \ref{sub:general_j_dynamics_and_q_hahn_degeneration} below.


\smallskip\noindent$\bullet$
If the dynamics 
$\Xe_{\UU;\{v_t\}}$
is started from the initial configuration 
$0^{M}$ 
(that is, all $M$ particles are at zero),
then at any time $t$ the distribution of the 
particle configuration is 
given by $\MM_{\UU;(v_1,\ldots,v_t)}$.
Similarly, if 
$\Xp_{\{u_t\};\VV}$ starts from the empty initial configuration, 
then at any time $t$ the distribution of the 
particle configuration is 
given by $\MM_{(u_1,\ldots,u_t);\VV}$. This follows
from \eqref{MM_and_Q}.

\smallskip\noindent$\bullet$
Let us return to local dynamics.
As follows from \eqref{EF_relation_Qe}--\eqref{EF_Qe},
the eigenfunctions of the 
transition operator 
$\Qe_{\PI M;v}$ corresponding to the dynamics $\Xe$ (on $M$-particle configurations) are 
\begin{align*}
	\EF{}\la(z_1,\ldots,z_M)=
	\frac{1}{(q;q)_{M}(-\SPB)^{\la}}\,{\F_\la(z_1,\ldots,z_M\md\ipb,\SPB)},
	\qquad
	\Qe_{\PI M;v}\EF{}\la(\ZZZ) = 
	\bigg(
	\prod_{j=1}^{M}\frac{1-qz_jv}{1-z_jv}
	\bigg)\EF{}\la(\ZZZ)
\end{align*}
(here and below for $\UU=\PI M$ we write 
$\EF{}\la$ instead of
$\EF{\UU}\la$).

\smallskip\noindent$\bullet$
In the homogeneous case $\ip_j\equiv 1$ and $\SP_j\equiv \SP$, the dynamics 
$\Xe_{\PI M;\{v_t\}}$
on $M$-particle configurations appeared in 
\cite{CorwinPetrov2015}
(under the name \emph{$J=1$ higher spin zero range 
process}). In this homogeneous setting, 
\cite{CorwinPetrov2015} established certain duality relations
for this dynamics.\footnote{Similar duality results 
also appeared earlier in
\cite{BorodinCorwinSasamoto2012}, \cite{BorodinCorwin2013discrete},
\cite{Corwin2014qmunu}
for $q$-TASEP and $q$-Hahn degenerations of the general higher spin six vertex model.
They also hold in an inhomogeneous setting, cf. 
\S\ref{sub:duality_from_observables} below.}
Some of the results in \cite{CorwinPetrov2015}
also deal with infinite-particle 
process like $\Xe_{\PI M;\{v_t\}}$,
which starts from the
initial configuration $0^{\infty}1^{0}2^{0}\ldots$
(interpreting the zero range process as an exclusion process, this 
would correspond to the most well-studied \emph{step initial data}).
In this case, during each time step $t\to t+1$,
one particle can escape the location $0$ 
with probability $\Lmatr_{\ip_0^{-1}v_{t+1},\SP_0}
(\infty,0;\infty,1)=(-\SP_0\ip_0^{-1}v_{t+1})/(1-\SP_0\ip_0^{-1}v_{t+1})$ (note that under
\eqref{stochastic_weights_condition_qsxi}--\eqref{stochastic_weights_condition_u}
this number is between $0$ and $1$). 
In \S \ref{sub:general_j_dynamics_and_q_hahn_degeneration} below we will discuss how 
this initial condition can be obtained by a straightforward
limit transition from the dynamics 
$\Xp_{\{u_t\};\RHO}$.
Thus, considering the latter dynamics 
without this limit transition 
adds a new boundary condition, 
under which during each time step,
a new particle is \emph{always} added at the leftmost location.



\subsection{Degeneration to the six vertex model and the ASEP} 
\label{sub:asep_degeneration}

In this subsection we do not assume that 
our parameters satisfy \eqref{stochastic_weights_condition_qsxi}--\eqref{stochastic_weights_condition_u}.
However, all algebraic statements discussed above in this section
(e.g., Proposition \ref{prop:Q_La_commutation}) 
continue to hold without this assumption --- they just become statements
about linear operators. Moreover, 
one can say that these are statements about \emph{formal} Markov
operators, i.e., in which the matrix elements sum up to one
along each row,
but are not necessarily nonnegative.

Observe that taking $\SP^{2}=q^{-I}$ for $I\in\Z_{\ge1}$
makes the weight
\begin{align*}
	\Lmatr_{u,\SP}(I,1;I+1,0)=\frac{1-\SP^{2}q^{I}}{1-\SP u}
\end{align*}
vanish, regardless of $u$.
If, moreover, all other weights 
$\Lmatr_{u,\SP}(i_1,j_1;i_2,j_2)$ with
$i_{1,2}\in\{0,1,\ldots,I\}$ and 
$j_{1,2}\in\{0,1\}$
are nonnegative, 
then we can restrict our 
attention to path ensembles in which 
the multiplicities of all
vertical edges are bounded by $I$,
and still talk about interacting particle systems
as in \S \ref{sub:interacting_particle_systems} above.

Let us consider the simplest case and take $I=1$, so 
$\SP=q^{-\frac12}$.
For this choice of $\SP$,
there are six possible arrow configurations
at a vertex, and their weights are given in 
Fig.~\ref{fig:six_vertex}. 
These weights are nonnegative if
either $0<q<1$ and $u\ge q^{-\frac12}$,
or
$q>1$ and $0\le u\le q^{-\frac12}$ (these are the 
new nonnegativity conditions
replacing \eqref{stochastic_weights_condition_qsxi}--\eqref{stochastic_weights_condition_u} for $\SP=q^{-\frac12}$).
Observe the following symmetry of the vertex weights:
\begin{align}\label{stoch_6V_symmetry}
	\Lmatr_{u,q^{-\frac12}}(i_1,j_1;i_2,j_2)
	=
	\Lmatr_{u^{-1},q^{-\frac12}}(1-i_1,1-j_1;1-i_2,1-j_2)\big\vert_{q\to q^{-1}},
	\qquad i_1,j_1,i_2,j_2\in\{0,1\}.
\end{align}

\begin{figure}[htbp]
	\def\d{0.1}\def\scl{.8}
	\begin{tabular}{c|c|c|c|c|c|c}
	&\scalebox{\scl}{\begin{tikzpicture}
		[scale=1, very thick]
		\node (i1) at (0,-1) {$0$};
		\node (j1) at (-1,0) {$0$};
		\node (i2) at (0,1) {$0$};
		\node (j2) at (1,0) {$0$};
		\draw[densely dotted] (j1) -- (j2);
		\draw[densely dotted] (i1) -- (i2);
	\end{tikzpicture}}
	&\scalebox{\scl}{\begin{tikzpicture}
		[scale=1, very thick]
		\node (i1) at (0,-1) {$1$};
		\node (j1) at (-1,0) {$0$};
		\node (i2) at (0,1) {$1$};
		\node (j2) at (1,0) {$0$};
		\draw[densely dotted] (j1) -- (j2);
		\draw[densely dotted] (i1) -- (i2);
		\draw[->, line width=1.7pt] (i1)--(i2);
	\end{tikzpicture}}
	&\scalebox{\scl}{\begin{tikzpicture}
		[scale=1, very thick]
		\node (i1) at (0,-1) {$1$};
		\node (j1) at (-1,0) {$0$};
		\node (i2) at (0,1) {$0$};
		\node (j2) at (1,0) {$1$};
		\draw[densely dotted] (j1) -- (j2);
		\draw[densely dotted] (i1) -- (i2);
		\draw[->, line width=1.7pt] (i1)--(0,0)--(j2);
	\end{tikzpicture}}
	&\scalebox{\scl}{\begin{tikzpicture}
		[scale=1, very thick]
		\node (i1) at (0,-1) {$0$};
		\node (j1) at (-1,0) {$1$};
		\node (i2) at (0,1) {$0$};
		\node (j2) at (1,0) {$1$};
		\draw[densely dotted] (j1) -- (j2);
		\draw[densely dotted] (i1) -- (i2);
		\draw[->, line width=1.7pt] (j1)--(j2);
	\end{tikzpicture}}
	&\scalebox{\scl}{\begin{tikzpicture}
		[scale=1, very thick]
		\node (i1) at (0,-1) {$0$};
		\node (j1) at (-1,0) {$1$};
		\node (i2) at (0,1) {$1$};
		\node (j2) at (1,0) {$0$};
		\draw[densely dotted] (j1) -- (j2);
		\draw[densely dotted] (i1) -- (i2);
		\draw[->, line width=1.7pt] (j1)--(0,0)--(i2);
	\end{tikzpicture}}
	&\scalebox{\scl}{\begin{tikzpicture}
		[scale=1, very thick]
		\node (i1) at (0,-1) {$1$};
		\node (j1) at (-1,0) {$1$};
		\node (i2) at (0,1) {$1$};
		\node (j2) at (1,0) {$1$};
		\draw[densely dotted] (j1) -- (j2);
		\draw[densely dotted] (i1) -- (i2);
		\draw[->, line width=1.7pt] (i1)--(0,-\d)--++(\d,\d)--(j2);
		\draw[->, line width=1.7pt] (j1)--(-\d,0)--++(\d,\d)--(i2);
	\end{tikzpicture}}
	\\
	\hline\rule{0pt}{20pt}
	$\Lmatr_{u,q^{-\frac12}}$
	&1
	&$b_1$
	&$1-b_1$
	&$b_2$
	&$1-b_2$
	&1
	\phantom{\Bigg|}
	\end{tabular}
	\caption{All six stochastic vertex weights
	corresponding to $\SP=q^{-\frac12}$.
	The weights $b_{1,2}$ are
	expressed through $u$ and $q$ as
	$b_1=\dfrac{1-uq^{\frac12}}{1-uq^{-\frac12}}$
	and 
	$b_2=\dfrac{-uq^{-\frac12}+q^{-1}}{1-uq^{-\frac12}}$.}
	\label{fig:six_vertex}
\end{figure}

In the semi-infinite horizontal strip
one must set $\SP_x\equiv q^{-\frac12}$ for all $x\in\Z_{\ge0}$.
While this eliminates the inhomogeneity in the $\SP$-parameters,
one can still take inhomogeneous spectral parameters, so that
at the intersection of the $i$-th horizontal and 
the $j$-th vertical lines the parameter
is equal to $\ip_j u_i$ (cf. Fig.~\ref{fig:paths_FG}). 
This leads to the \emph{inhomogeneous stochastic
six vertex model}.
A homogeneous version of the model 
(corresponding to $u_i\equiv u$ and $\ip_j\equiv 1$) was
introduced in \cite{GwaSpohn1992} and studied recently in
\cite{BCG6V}.
Simulations of the stochastic six vertex model (both homogeneous and inhomogeneous)
are given in~Fig.~\ref{fig:stoch6v}.

The paper \cite{BCG6V} deals with the homogeneous stochastic six vertex model
in which the vertical arrows are entering from below, and no arrows enter from the left (cf. Fig.~\ref{fig:stoch6v}, left).
Moreover, to get a nontrivial limit shape,
one should take $q>1$. However, with the help of the symmetry
\eqref{stoch_6V_symmetry} (leading to the swapping
of arrows with empty edges), these boundary conditions
are equivalent to considering the process
$\Xp_{\{u_t\};\RHO}$ with $0<q<1$, which is our usual assumption throughout the 
text. Simulations of the latter dynamics
can be obtained from the pictures in Fig.~\ref{fig:stoch6v}
by reflecting them with respect to the diagonal of the first quadrant.
\begin{figure}[htbp]
	\includegraphics[width=.47\textwidth]{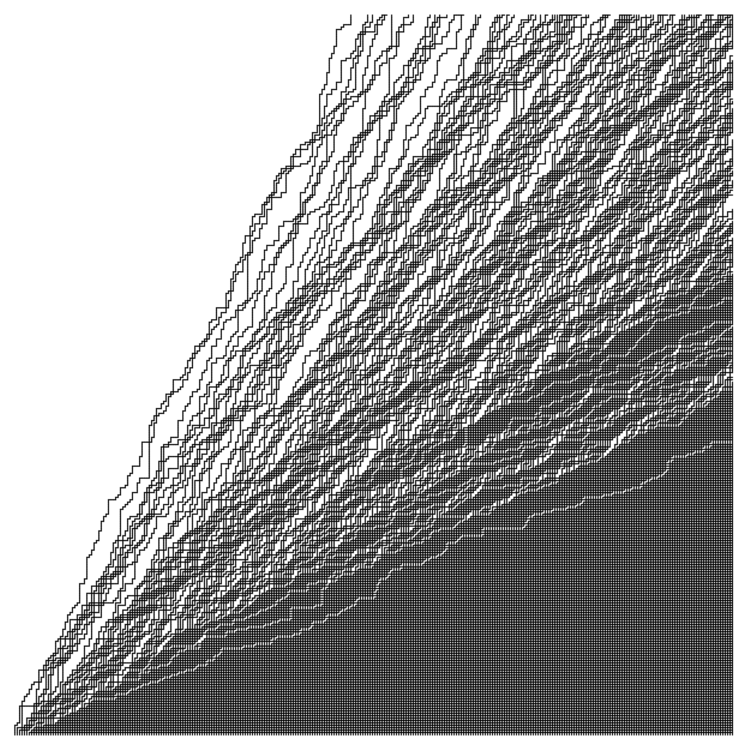}
	\qquad
	\includegraphics[width=.47\textwidth]{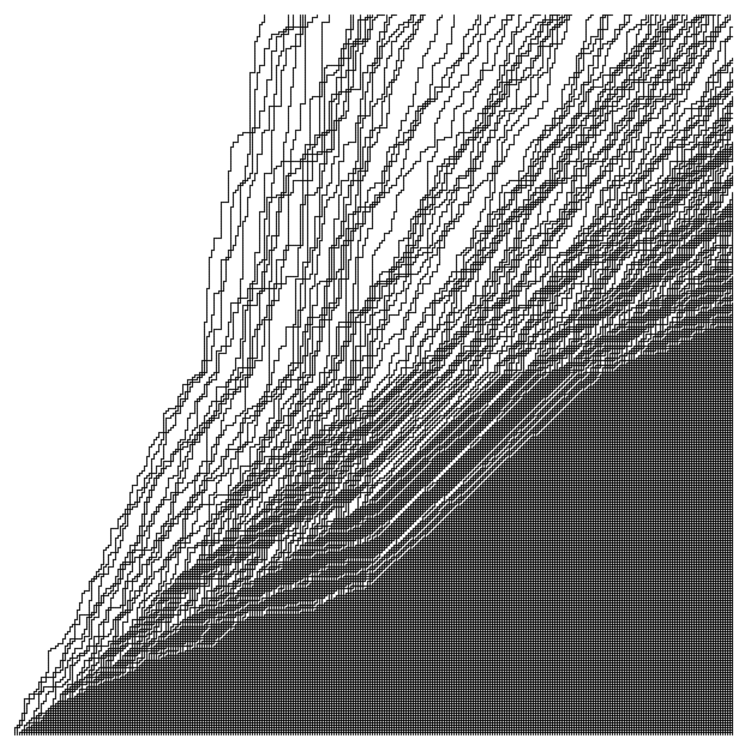}
	\caption{Left: A simulation of 
	the homogeneous
	stochastic six vertex model 
	of size $300$ 
	with 
	boundary conditions as in \cite{BCG6V}
	and parameters
	$\Lmatr(0,1;0,1)=0.3$,
	$\Lmatr(1,0;1,0)=0.7$. 
	Right: A simulation of the inhomogeneous stochastic
	six vertex model 
	of size $300$ with the 
	same boundary conditions.
	The parameters 
	$(\Lmatr(0,1;0,1),\Lmatr(1,0;1,0))$
	are
	$(0.3,0.7)$
	in the lower left and the upper right quarters,
	${\approx(0.38, 0.88)}$
	in the upper left quarter,
	and 
	${\approx(0.041, 0.096)}$
	in the lower right quarter (note that
	the ratio $q$ 
	of the parameters must be the same).}
	\label{fig:stoch6v}
\end{figure}

The stochastic 
six vertex model which is inhomogeneous in 
both the vertical and the horizontal directions can be studied 
(in the sense of computing certain observables)
using the technique developed here, see \S \ref{sub:moments_of_six_vertex_model_and_asep} below. 
In fact, the tools of \cite{BCG6V} also allow to 
study the stochastic six vertex model which is inhomogeneous in one direction (varying spectral parameters).

Let us briefly discuss two continuous time 
limits of the stochastic six vertex model. 
Here we restrict our attention to 
systems of the type $\Xe_{\PI M;\{v_t\}}$, i.e., with a fixed finite number of particles 
(about other boundary and initial conditions see also \S \ref{sub:moments_of_six_vertex_model_and_asep} below).
The first of the limits is the well-known 
ASEP (Asymmetric Simple Exclusion Process) introduced in 
\cite{Spitzer1970}
(see Fig.~\ref{fig:ASEP}), which is obtained as follows.
Observe that for $u=q^{-\frac12}+(1-q)q^{-\frac12}\epsilon$, 
we have as $\epsilon\searrow0$:
\begin{align*}
	\Lmatr_{u,q^{-\frac12}}(0,1;0,1)=
	\epsilon
	+O(\epsilon^{2}),\qquad \qquad
	\Lmatr_{u,q^{-\frac12}}(1,0;1,0)=
	q\epsilon
	+O(\epsilon^{2}).
\end{align*}
Therefore, taking 
$\ip_j\equiv 1$
and
$\epsilon$ small,
the particles in the stochastic six vertex model
will mostly travel
to the right by $1$ at every step.
If we subtract this deterministic shift
and look at times of order $\epsilon^{-1}$, 
then the rescaled discrete time process will converge to the 
continuous time
ASEP with $r=1$ and $\ell=q$, see Fig.~\ref{fig:six_to_ASEP}
(note that multiplying both $r$ and $\ell$
by a constant is the same as a deterministic
rescaling of the continuous time in the ASEP, and thus is a harmless
operation).
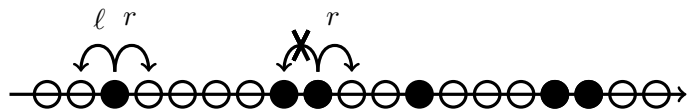
\begin{figure}[htbp]
	\begin{tikzpicture}
		[scale=1,very thick]
			\def\pt{.17}
			\def\ee{.1}
			\def\h{.45}
			\draw[->] (-.5,0) -- (8.5,0);
			\foreach \ii in {(0,0),(\h,0),(3*\h,0),(4*\h,0),(5*\h,0),(6*\h,0) ,(8*\h,0),(10*\h,0),(9*\h,0),(12*\h,0),(13*\h,0),(14*\h,0),(15*\h,0),(16*\h,0),(17*\h,0),(18*\h,0)}
			{
				\draw \ii circle(\pt);
			}
			\foreach \ii in {(2*\h,0),(7*\h,0),(11*\h,0),(15*\h,0),(8*\h,0),(16*\h,0)}
			{
				\draw[fill] \ii circle(\pt);
			}
		    \draw[->, very thick] (2*\h,.3) to [in=180,out=90] (2.5*\h,.65) to [in=90, out=0] (3*\h,.3) node [xshift=-7,yshift=20] {$r$};
		    \draw[->, very thick] (2*\h,.3) to [in=0,out=90] (1.5*\h,.65) to [in=90, out=180] (1*\h,.3) node [xshift=7,yshift=20] {$\ell$};
		    \draw[->, very thick] (8*\h,.3) to [in=0,out=90] (7.5*\h,.65) to [in=90, out=180] (7*\h,.3);
		    \draw[ultra thick] (7.5*\h,.65)--++(.1,.2)--++(-.2,-.4)--++(.1,.2)--++(-.1,.2)
		    --++(.2,-.4);
		    \draw[->, very thick] (8*\h,.3) to [in=180,out=90] (8.5*\h,.65) to [in=90, out=0] (9*\h,.3) node [xshift=-7,yshift=20] {$r$};
	\end{tikzpicture}
  	\caption{The ASEP is a continuous time Markov
  	chain on particle configurations on $\Z$
  	(in which there is at most one particle per site).
  	Each particle 
  	has two 
  	exponential clocks of rates
  	$r$ and $\ell$, respectively (all exponential clocks
  	in the process are assumed independent).
	When the ``$r$''
	clock of a particle rings, it immediately tries to
	jump to the right by one,
	and similarly for the ``$\ell$'' clock and left jumps.
	If the destination of a jump is already occupied, then the jump is blocked. (This describes the
	ASEP with finitely many particles, but one can also 
	construct the infinite-particle ASEP
	following, e.g., the graphical method of \cite{Harris1978}.)}
  	\label{fig:ASEP}
\end{figure}
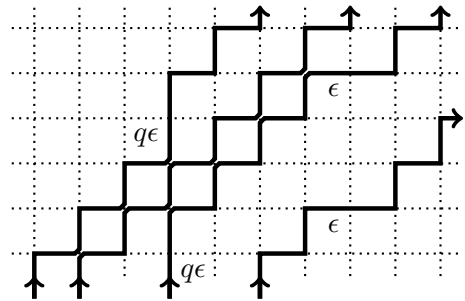
\begin{figure}[htbp]
	\begin{tikzpicture}
			[scale=.6,thick]
			\def\d{.1}
			\foreach \xxx in {0,...,9}
			{
				\draw[dotted] (\xxx,0.5)--++(0,6);
			}
			\foreach \xxx in {1,...,6}
			{
				\draw[dotted] (-.5,\xxx)--++(10,0);
			}
			\draw[->, line width=1.7pt] (0,0)--++(0,.5);
			\draw[->, line width=1.7pt] (1,0)--++(0,.5);
			\draw[->, line width=1.7pt] (3,0)--++(0,.5);
			\draw[->, line width=1.7pt] (5,0)--++(0,.5);
			\draw[->, line width=1.7pt] (0,0)--++(0,1)--++(1-\d,0)
			--++(\d,\d)--++(0,1-\d)--++(1-\d,0)--++(\d,\d)
			--++(0,1-\d)--++(1-\d,0)--++(\d,\d)
			--++(0,1-\d)--++(0,1)--++(1,0)--++(0,1)--++(1,0)
			--++(0,.5);
			\draw[->, line width=1.7pt] (1,0)--++(0,1-\d)
			--++(\d,\d)--++(1-\d,0)--++(0,1-\d)--++(\d,\d)--++(1-2*\d,0)
			--++(\d,\d)--++(0,1-2*\d)--++(\d,\d)
			--++(1-2*\d,0)--++(\d,\d)--++(0,1-\d)
			--++(1-\d,0)--++(\d,\d)--++(0,1-\d)
			--++(1-\d,0)--++(\d,\d)
			--++(0,1-\d)--++(1,0)--++(0,.5);
			\draw[->, line width=1.7pt] (3,0)--++(0,2-\d)
			--++(\d,\d)--++(1-\d,0)--++(0,1-\d)--++(\d,\d)
			--++(1-\d,0)--++(0,1-\d)--++(\d,\d)
			--++(1-\d,0)--++(0,1-\d)--++(\d,\d)
			--++(1-\d,0)--++(1,0)--++(0,1)--++(1,0)--++(0,.5);
			\draw[->, line width=1.7pt] (5,0)--++(0,1)
			--++(1,0)--++(0,1)
			--++(1,0)--++(1,0)--++(0,1)--++(1,0)--++(0,1)--++(.5,0);

			\node[below right] at (3,1) {$q \epsilon$};
			\node[below left] at (3,4) {$q \epsilon$};
			\node[below left] at (7,2) {$\epsilon$};
			\node[below left] at (7,5) {$\epsilon$};
		\end{tikzpicture}
  	\caption{Limit of the six vertex model to the ASEP.}
  	\label{fig:six_to_ASEP}
\end{figure}

\begin{remark}
	Because we are subtracting the deterministic shift, 
	it seems unlikely that one can 
	utilize the inhomogeneous stochastic six vertex model 
	to produce an inhomogeneous 
	extension of the ASEP as a continuous time limit. 
	
	It is worth noting that
	duality for the ASEP with bond-dependent 
	jump rates exists (cf. \cite[Rmk. 4.4]{BorodinCorwinSasamoto2012};
	such a duality was essentially established in \cite{schutz1997dualityASEP}),
	but moment formulas (similar to the ones in 
	\S \ref{sec:_q_moments_of_the_height_function_of_interacting_particle_systems} below)
	for that inhomogeneous ASEP do not seem to be known.
\end{remark}

Another continuous time limit 
is obtained by setting:
\begin{align*}
	q=\frac{1- \epsilon}{\al},\qquad
	u=\frac{\epsilon\al^{\frac12}}{1-\al},
\end{align*}
where $0<\al<1$, so that 
as $\epsilon\to 0$ we have
\begin{align*}
	\Lmatr_{\ip_j u,q^{-\frac12}}(0,1;0,1)=
	\al+\al(1-\ip_j)\epsilon
	+O(\epsilon^{2}),\qquad \qquad
	\Lmatr_{\ip_j u,q^{-\frac12}}(1,0;1,0)=
	1-\ip_j \epsilon
	+O(\epsilon^{2}).
\end{align*}
At times of order $\epsilon^{-1}$, 
the system behaves as follows. 
Each particle at a location $j$
has an exponential clock with rate $\ip_j$.
When the clock rings, the particle wakes up
and performs a jump to the right having the geometric
distribution with parameter $\al$. 
However, if in the process of the jump this particle
runs into another particle (i.e., its first neighbor on the right),
then the moving particle stops at this neighbor's location, 
and the neighbor 
wakes up (and subsequently performs a geometrically distributed
jump). See Fig.~\ref{fig:second_limit}.

\begin{figure}[htbp]
	\begin{tikzpicture}
		[scale=1,very thick]
			\def\pt{.17}
			\def\ee{.1}
			\def\h{.45}
			\draw[->] (-.5,0) -- (8.5,0);
			\foreach \ii in {(0,0),(\h,0),(3*\h,0),(4*\h,0),
			(5*\h,0),(6*\h,0) ,(8*\h,0),(10*\h,0),(9*\h,0),(7*\h,0),
			(12*\h,0),(14*\h,0),(15*\h,0),(16*\h,0),(17*\h,0),
			(18*\h,0)}
			{
				\draw \ii circle(\pt);
			}
			\foreach \ii in {(2*\h,0),(13*\h,0),(11*\h,0),(15*\h,0),(5*\h,0),(16*\h,0)}
			{
				\draw[fill] \ii circle(\pt);
			}
			\node[below, yshift=-5pt] at (2*\h,0) {$x_{k}$};
			\node[below, yshift=-5pt] at (5*\h,0) {$x_{k+1}$};
			\node[above, yshift=20pt, anchor=east, 
			xshift=2pt] at (2*\h,0) 
			{$\textnormal{waking up rate}=\ip_{x_k}$};
			\node[above, yshift=35pt, xshift=-20pt] at (4.5*\h,0) 
			{$\textnormal{geometric}(\al)$ jump};
			\node[above, yshift=22pt] at (12*\h,0) 
			{$\textnormal{geometric}(\al)$ jump};
			\draw[->, very thick] (2*\h,.3) 
			to [in=180,out=90] (4.5*\h,1.25) 
			to [in=90, out=0] (7*\h,.3);
			\draw[->, very thick, densely dotted] (2*\h,.3) 
			to [in=180,out=90] (3.5*\h,.8) 
			to [in=90, out=0] (5*\h,.3);
			\draw[->, very thick] (5*\h,.3) 
			to [in=180,out=90] (7*\h,1) 
			to [in=90, out=0] (9*\h,.3);
	\end{tikzpicture}
  	\caption{A possible jump in the second
  	limit of the (inhomogeneous)
  	stochastic six vertex model.
  	The particle at $x_k$ wakes up at rate
  	$\ip_{x_k}$ and decides to jump by 5 with probability
  	$(1-\al)\al^{4}$ (waking up means that the particle 
  	will jump by at least one). 
  	However, $x_{k+1}$ is closer than the intended jump of $x_{k}$,
  	and so $x_k$ stops at the location of
  	$x_{k+1}$, and the latter particle wakes up.
  	Then $x_{k+1}$ decides to jump by 4 with probability
  	$(1-\al)\al^{3}$.}
  	\label{fig:second_limit}
\end{figure}
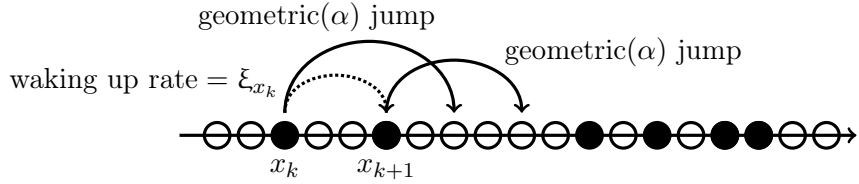


\subsection{Degeneration to $q$-Hahn and $q$-Boson systems} 
\label{sub:general_j_dynamics_and_q_hahn_degeneration}

In this subsection we will consider another 
family of 
degenerations of the higher spin six vertex model
which puts no restrictions on the vertical multiplicities.
For these degenerations we will need to employ the general $J$
stochastic vertex weights 
$\LJ{J}_{u,\SP}(i_1,j_1;i_2,j_2)$
described in \S \ref{sub:fusion_of_stochastic_weights}.

\begin{proposition}
	\label{prop:LJ_simplifying}
	When $u=\SP$, formula \eqref{LJ_formula}
	for the weights 
	$\LJ{J}_{u,\SP}$
	simplifies to the following product form:
	\begin{align}\label{LJ_simplifying}
		\LJ{J}_{\SP,\SP}(i_1,j_1;i_2,j_2)=
		\mathbf{1}_{i_1+j_1=i_2+j_2}\cdot
		\mathbf{1}_{j_2\le i_1}\cdot
		(\SP^{2}q^{J})^{j_2}\frac{(q^{-J};q)_{j_2}(\SP^{2}q^{J};q)_{i_1-j_2}}{(\SP^2;q)_{i_1}} 
		\frac{(q;q)_{i_1}}{(q;q)_{j_2}(q;q)_{i_1-j_2}}.
	\end{align}
\end{proposition}
\begin{proof}
	To show this, one can directly check that 
	\eqref{LJ_simplifying} satisfies the 
	corresponding recursion relation for $u=\SP$ \eqref{recursion_relation}.
	Alternatively, one can transform the 
	${}_{4}\bar{\varphi}_3$ $q$-hypergeometric
	function to the desired form.
	We refer to \cite[Prop.\;6.7]{Borodin2014vertex}
	for the complete proof following the second approach.
\end{proof}

We see that this degeneration turns the 
higher spin interacting particle systems
described in \S \ref{sub:interacting_particle_systems}
with sequential update into simpler systems with parallel update.

\subsubsection{Distribution $\phi_{q,\smu,\snu}$} 
\label{ssub:distribution_phi}

Before discussing interacting particle systems
arising from the vertex weights \eqref{LJ_simplifying},
let us focus on the \emph{$q$-deformed Beta-binomial distribution} 
appearing in the right-hand side of that formula:
\begin{align}\label{qHahn_distr}
	\phi_{q,\smu,\snu}(j\md m):=
	\smu^j\frac{(\snu/\smu;q)_{j}(\smu;q)_{m-j}}{(\snu;q)_{m}} \frac{(q;q)_{m}}{(q;q)_{j}(q;q)_{m-j}},
	\qquad
	j\in\{0,1,\ldots,m\}.
\end{align}
Here $m\in\Z_{\ge0}\cup\{+\infty\}$,
and the 
case $m=+\infty$ corresponds to a straightforward limit of \eqref{qHahn_distr}, 
see \eqref{qHahn_distr_infinity} below.
If the parameters 
belong to one of the following families:
\begin{align}\label{qHahn_weights_condition}
\parbox{.85\textwidth}{
	\begin{enumerate}
	\item $0<q<1$, $0\le \smu\le 1$, and $\snu\le \smu$;
	\smallskip\item $0<q<1$, $\smu=q^{J}\snu$ for some $J\in\Z_{\ge0}$, and $\snu\le 0$;
	\smallskip\item $m$ is finite, $q>1$, $\smu=q^{-J}\snu$ 
	for some $J\in\Z_{\ge0}$, and $\snu\le 0$;
	\smallskip\item $m$ is finite, $q>0$, $\smu=q^{\tilde\smu}$, 
	and $\snu=q^{\tilde\snu}$ 
	with $\tilde\smu,\tilde\snu\in\Z$, such that 
	\begin{itemize}
		\smallskip\item either $\tilde\smu,\tilde\snu\ge0$, and $\tilde\snu\ge\tilde\smu$,
		\smallskip\item or $\tilde\smu,\tilde\snu\le0$, $\tilde\snu\le -m$, and $\tilde\snu\le\tilde\smu$,
	\end{itemize}
\end{enumerate}}
\end{align}
then the 
weights \eqref{qHahn_distr} are 
nonnegative.\footnote{These are sufficient conditions for nonnegativity, 
and in fact some of these families intersect nontrivially.
We do not 
attempt to list all the necessary conditions (as, for example, for $q<0$ there also exist 
values of $\smu$ and $\snu$ leading to nonnegative weights).}
The above conditions \eqref{qHahn_weights_condition} replace the 
nonnegativity conditions \eqref{stochastic_weights_condition_qsxi}--\eqref{stochastic_weights_condition_u}
for this subsection.

We will now discuss several interpretations of the distribution \eqref{qHahn_distr}
which, in particular, will justify its name. 
The significance of the probability distribution $\phi_{q,\smu,\snu}$
for interacting particle systems
was first realized by Povolotsky \cite{Povolotsky2013},
who showed that it corresponds to the 
most general ``chipping model'' (i.e., a particle system as in
Fig.~\ref{fig:particle_systems} with possibly multiple particles 
leaving a given stack at a time)
having
parallel update, product-form steady state,
and such that the system is solvable by the coordinate Bethe ansatz. 
He also provided an algebraic interpretation of this distribution:
\begin{proposition}[{\cite[Thm.\;1]{Povolotsky2013}}]
	Let $A$ and $B$ be two letters satisfying the following quadratic 
	commutation relation:
	\begin{align*}
		BA=\al A^2+\be AB+\gamma B^2,\qquad
		\al+\be+\gamma=1.
	\end{align*}
	Then
	\begin{align}\label{q_binomial_relation}
		\big(pA+(1-p)B\big)^{m}=\sum_{j=0}^{m}
		\phi_{q,\smu,\snu}(j\md m) A^{j}B^{m-j},
	\end{align}
	where
	\begin{align*}
		\al=\frac{\snu(1-q)}{1-q\snu}
		,\qquad
		\be=\frac{q-\snu}{1-q\snu}
		,\qquad
		\gamma=
		\frac{1-q}{1-q\snu}
		,
		\qquad
		\smu=p+\snu(1-p).
	\end{align*}
\end{proposition}
In particular, taking $A=B=1$ 
in \eqref{q_binomial_relation} 
implies that the weights \eqref{qHahn_distr} sum to 
$1$ over $j=0,1,\ldots,m$.
The proof of the above statement is nontrivial, and we will not reproduce it here.

Another interpretation of the $q$-deformed Beta-binomial distribution 
can be given via a $q$-version of the 
P\'olya's urn process due to Gnedin and Olshanski \cite{Gnedin2009}. Consider the Markov chain
on the Pascal triangle 
\begin{align*}
	\bigsqcup_{m=0}^{\infty}\{(k,\ell)\in\Z^{2}_{\ge0}\colon k+\ell=m\}
\end{align*}
with the following transition probabilities (here $m=k+\ell$ is the time in this chain)
\begin{align*}
	\begin{tikzpicture}
		[scale=1, thick]
		\node (a1) at (0,0) {$(k,\ell)$};
		\node (a2) at (6,.8) {$(k,\ell+1)$};
		\node (a3) at (6,-.8) {$(k+1,\ell)$};
		\draw[->] (a1)--(a2.west) node[midway, yshift=12pt] 
		{$\frac{1-q^{b+\ell}}{1-q^{a+b+m}}$};
		\draw[->] (a1)--(a3.west) node[midway, yshift=-15pt] 
		{$q^{\ell+b}\frac{1-q^{a+k}}{1-q^{a+b+m}}$};
	\end{tikzpicture}
\end{align*}
Then the distribution of this Markov chain (started from the initial vertex
$(0,0)$) at time $m$ is 
\begin{align*}
	\Prob\big((k,\ell)\big)=\phi_{q,q^{b},q^{a+b}}(k\md m),\qquad
	k=0,1,\ldots,m.
\end{align*}
More general Markov chains (on the space of interlacing arrays)
based on the distributions
$\phi_{q,q^{b},q^{a+b}}$ with negative $a$ and $b$
which have a combinatorial significance (they are $q$-deformations 
of the classical Robinson--Schensted--Knuth insertion algorithm) were constructed
recently in \cite{MatveevPetrov2014}.

Another feature of the distribution 
$\phi_{q,\smu,\snu}$ is that it is the weight function for the 
so-called \emph{$q$-Hahn orthogonal polynomials}. See \cite[\S3.6]{Koekoek1996}
about the polynomials, and \cite[\S5.2]{BCPS2014}
for the exact matching between $\phi_{q,\smu,\snu}$ and the notation
related to the $q$-Hahn polynomials.


\subsubsection{$q$-Hahn particle system} 
\label{ssub:_q_hahn_particle_system}

We will now discuss what the dynamics
$\Xe_{\PI M;\{v_t\}}$ and $\Xp_{\{u_t\};\RHO}$
look like under the degeneration
described in Proposition \ref{prop:LJ_simplifying}.
We will first consider
the dynamics 
$\Xe_{\PI M;\{v_t\}}$ which lives on 
particle configurations with a fixed number of particles
(say, $M\in\Z_{\ge0}$), and then will 
deal with $\Xp_{\{u_t\};\RHO}$. 
The resulting dynamics will be commonly referred to as the 
\emph{$q$-Hahn particle system} 
with different initial or boundary 
conditions.

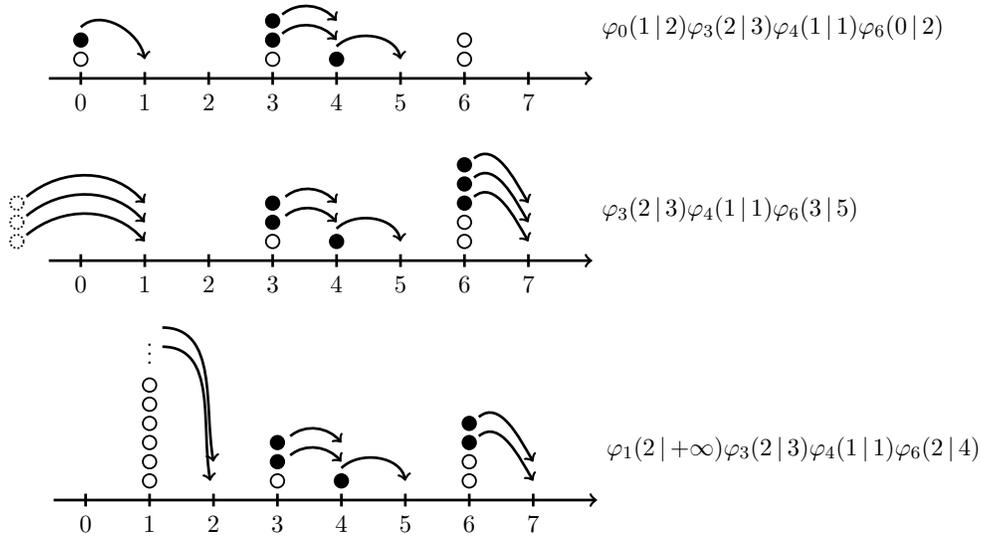
\begin{figure}[htbp]
	\scalebox{.85}{\begin{tikzpicture}
		[scale=1, very thick]
		\def\hh{-1.8}
		\draw[->] (-.5,0)--++(8.5,0);
		\node at (-1,0) {};
		\foreach \iii in {0,1,2,3,4,5,6,7}
		{
			\draw (\iii,.1)--++(0,-.2) node [below] {$\iii$};
		}
		\foreach \ppt in {(0,.3),(3,.3),(6,.3),(6,.6)}
		{
			\draw[thick] \ppt circle(3pt);
		}
		\foreach \ppt in {(0,.6),(3,.6),(3,.9),(4,.3)}
		{
			\draw[thick, fill] \ppt circle(3pt);
		}
		\draw[->] (0,.8) to [in=100, out=45] (1,.3);
		\draw[->] (4,.5) to [in=100, out=45] (5,.3);
		\draw[->] (3.15,.6+.1) to [in=130, out=40] (4,.6);
		\draw[->] (3.15,.9+.1) to [in=130, out=40] (4,.9);
		\node[anchor=west] at (8,.8) {$\phi_0(1\md 2)\phi_3(2\md 3)\phi_4(1\md 1)\phi_6(0\md 2)$};
		\node[anchor=west] at (8,.8) {\phantom{$\phi_1(2\md{+\infty})\phi_3(2\md 3)\phi_4(1\md 1)\phi_6(2\md 4)$}};
	\end{tikzpicture}}
	\vspace{7pt}

	\scalebox{.85}{\begin{tikzpicture}
		[scale=1, very thick]
		\def\hh{-1.8}
		\draw[->] (-.5,0)--++(8.5,0);
		\draw[thick, densely dotted] (-1,.3) circle(3pt);
		\draw[thick, densely dotted] (-1,.6) circle(3pt);
		\draw[thick, densely dotted] (-1,.9) circle(3pt);
		\node at (-1,0) {};
		\foreach \iii in {0,1,2,3,4,5,6,7}
		{
			\draw (\iii,.1)--++(0,-.2) node [below] {$\iii$};
		}
		\foreach \ppt in {(3,.3),(6,.3),(6,.6)}
		{
			\draw[thick] \ppt circle(3pt);
		}
		\foreach \ppt in {(3,.6),(3,.9),(4,.3),(6,.9),(6,1.2),(6,1.5)}
		{
			\draw[thick, fill] \ppt circle(3pt);
		}
		\draw[->] (-.85,.9+.1) to [in=130, out=40] (1,.9);
		\draw[->] (-.85,.6+.1) to [in=130, out=40] (1,.6);
		\draw[->] (-.85,.3+.1) to [in=130, out=40] (1,.3);
		\draw[->] (6.15,.9+.1) to [in=120, out=40] (7,.3);
		\draw[->] (6.15,1.2+.1) to [in=120, out=40] (7,.6);
		\draw[->] (6.15,1.5+.1) to [in=120, out=40] (7,.9);
		\draw[->] (4,.5) to [in=100, out=45] (5,.3);
		\draw[->] (3.2,.6+.1) to [in=130, out=40] (4,.6);
		\draw[->] (3.2,.9+.1) to [in=130, out=40] (4,.9);
		\node[anchor=west] at (8,.8) {$\phi_3(2\md 3)\phi_4(1\md 1)\phi_6(3\md 5)$};
		\node[anchor=west] at (8,.8) {\phantom{$\phi_1(2\md{+\infty})\phi_3(2\md 3)\phi_4(1\md 1)\phi_6(2\md 4)$}};
	\end{tikzpicture}}
	\vspace{7pt}

	\scalebox{.85}{\begin{tikzpicture}
		[scale=1, very thick]
		\def\hh{-1.8}
		\node at (-1,0) {};
		\draw[->] (-.5,0)--++(8.5,0);
		\foreach \iii in {0,1,2,3,4,5,6,7}
		{
			\draw (\iii,.1)--++(0,-.2) node [below] {$\iii$};
		}
		\foreach \ppt in {(3,.3),(6,.3),(6,.6),
		(1,.3),(1,.6),(1,.9),(1,1.2),(1,1.5),(1,1.8)}
		{
			\draw[thick] \ppt circle(3pt);
		}
		\foreach \ppt in {(3,.6),(3,.9),(4,.3),(6,.9),(6,1.2)}
		{
			\draw[thick, fill] \ppt circle(3pt);
		}
		\node at (1,2.4) {$\vdots$};
		\draw[->] (6.15,.9+.1) to [in=120, out=40] (7,.3);
		\draw[->] (6.15,1.2+.1) to [in=120, out=40] (7,.6);
		\draw[->] (4,.5) to [in=100, out=45] (5,.3);
		\draw[->] (3.2,.6+.1) to [in=130, out=40] (4,.6);
		\draw[->] (3.2,.9+.1) to [in=130, out=40] (4,.9);
		\draw[->] (1.2,2.4) to [in=100, out=-00] (2-0.05,.3);
		\draw[->] (1.2,2.7) to [in=100, out=-00] (2,.6);
		\node[anchor=west] at (8,.8) {$\phi_1(2\md{+\infty})\phi_3(2\md 3)\phi_4(1\md 1)\phi_6(2\md 4)$};
	\end{tikzpicture}}
  	\caption{Possible transitions of the $q$-Hahn particle system,
  	with probabilities given on the right (here $\phi_x\equiv \phi_{q,q^{J}\SP_x^{2},\SP_x^{2}}$). 
  	Top: dynamics $\Xeh$
	living on configurations with a fixed number of particles.
  	Middle: dynamics $\Xph$,
  	in which at each time step $J$ new particles 
  	are added at location $1$ ($J=3$ on the picture).
  	Bottom: dynamics 
  	$\Xih$ with the initial condition
  	$1^{\infty}2^{0}3^{0}\ldots$.}
  	\label{fig:qhahn}
\end{figure}

In order to perform the desired degeneration of $\Xe$, let 
us fix the $\SP$-parameters $\{\SP_j\}_{j\in\Z_{\ge0}}$
indexed by the semi-infinite 
lattice, and 
set $\ip_j=\SP_j^{-1}$ for all $j$.
Also fix $J\in\Z_{\ge1}$, and 
take the time-dependent parameters $\{v_t\}$ to be
\begin{align*}
	(v_1,v_2,\ldots)=(1,q,\ldots,q^{J-1},1,q,\ldots,q^{J-1},\ldots).
\end{align*}
We will consider the \emph{fused dynamics}
in which one time step corresponds
to $J$ steps of the original dynamics.
The fused dynamics is Markovian
due to the results of \S \ref{sub:fusion_of_stochastic_weights}.
As follows from \S \ref{ssub:dyn_xe},
each time step 
$\mu=0^{m_0}1^{m_1}2^{m_2}\ldots
\to
\nu=0^{n_0}1^{n_1}2^{n_2}\ldots$ 
($\mu,\nu\in\signp M$) 
of the fused dynamics
looks as follows.
For each location $x\in\Z_{\ge0}$, sample $j_x\in\{0,1,\ldots,m_x\}$
independently of other locations according to the probability 
distribution
$\phi_{q,q^{J}\SP_x^{2},\SP_x^{2}}(j_x\md m_x)$
(clearly, $j_x=m_x=0$ for all large enough $x$). 
Then, in parallel, move $j_x$ particles from location $x$ to location $x+1$
for each $x\in\Z_{\ge0}$, that is, set $n_x=m_x-j_x+j_{x+1}$.
Denote this dynamics by $\Xeh$ (see Fig.~\ref{fig:qhahn}, top).

Note that the weights $\phi_{q,q^{J}\SP_x^{2},\SP_x^{2}}$ are nonnegative
for $J\in\Z_{\ge1}$ if $\SP_x^{2}\le 0$ 
(case 2 in 
\eqref{qHahn_weights_condition}), 
and we are assuming this in our construction.

\begin{remark}\label{rmk:analityc_continuation_in_Xe}
	For $J\in\Z_{\ge1}$, at most $J$ particles can leave 
	any given location during one time step. However, since the weights of
	the distribution $\phi_{q,q^{J}\SP_x^{2},\SP_x^{2}}$ depend
	on $q^{J}$ in a rational way, we may analytically continue 
	$\Xeh$
	from the case $q^J\in q^{\Z_{\ge1}}$ and $\SP_x^{2}\le 0$,
	and let 
	the parameters
	$\smu_x=q^{J}\snu_x$ and $\snu_x=\SP_x^{2}$ 
	belong to one of the other families in \eqref{qHahn_weights_condition}.
	If $\smu_x/\snu_x\notin q^{\Z_{\ge1}}$, 
	then an arbitrary number of particles
	can leave 
	any given location during one time step.
\end{remark}

Let us now discuss the degeneration of the dynamics 
$\Xp_{\{u_t\};\RHO}$. 
Fix $J\in\Z_{\ge1}$,
take 
$\ip_j=\SP_j$ for all $j\in\Z_{\ge0}$, 
and let the time-dependent parameters $\{u_t\}$ be
\begin{align}
	(u_1,u_2,\ldots)=(1,q,\ldots,q^{J-1},1,q,\ldots,q^{J-1},\ldots).
	\label{u_specialization_to_generalize}
\end{align} 
From \S \ref{ssub:dyn_xp} we see that 
the corresponding fused dynamics
is very similar to $\Xeh$,
and the only difference is
in the behavior at locations $0$ and $1$.
Namely, 
location $0$ cannot be occupied,
and at each time step,
exactly $J$ new particles are added at location $1$.
Denote this degeneration of 
$\Xp_{\{u_t\};\RHO}$ by $\Xph$
(see Fig.~\ref{fig:qhahn}, middle).

Because $J\in\Z_{\ge1}$ particles are added 
to the configuration at each time step,
dynamics $\Xph$ cannot be analytically continued in $J$
similarly to Remark \ref{rmk:analityc_continuation_in_Xe}.
However, we can simplify this dynamics, by generalizing
\eqref{u_specialization_to_generalize} to
\begin{align*}
	(u_1,u_2,\ldots)=
	(1,q,\ldots,q^{K-1},1,q,\ldots,q^{J-1},1,q,\ldots,q^{J-1},\ldots),
\end{align*}
where $K\in\Z_{\ge1}$ is a new parameter. 
If we start the corresponding fused dynamics
from the empty initial configuration 
$1^{0}2^{0}\ldots$, then after the first step of this dynamics
the configuration will be simply $1^{K}2^{0}3^{0}\ldots$. 
Moreover, during the evolution, the number of particles at location 
$1$ will always be $\ge K$.
Assume that $|q|<1$, and take $K\to+\infty$.
Under the limiting dynamics, at all subsequent times the 
number of particles leaving 
location $1$ has the distribution
\begin{align}\label{qHahn_distr_infinity}
	\lim_{i_1\to+\infty}\LJ{J}_{\SP_{1},\SP_{1}}(i_1,j_1;i_2,j_2)=
	\mathbf{1}_{i_2=+\infty}\cdot 
	\phi_{q,q^{J}\SP_1^{2},\SP_1^{2}}(j_2\md {+\infty}),
	\quad
	\phi_{q,\smu,\snu}(j\md {+\infty})=
	\smu^j\frac{(\snu/\smu;q)_{j}}{(q;q)_{j}} 
	\frac{(\smu;q)_{\infty}}{(\snu;q)_{\infty}}.
\end{align}
The limit as $K\to+\infty$ clearly does not affect 
probabilities of particle jumps at all other locations. 
We will denote the limiting dynamics by 
$\Xih$ (see Fig.~\ref{qHahn_distr}, bottom).
When the parameters $\SP_j\equiv \SP$ are homogeneous, 
this particle system was introduced in \cite{Povolotsky2013}.
The system $\Xih$ 
readily admits an analytic continuation from
$q^J\in q^{\Z_{\ge1}}$ and $\SP_x^{2}\le 0$
as in Remark \ref{rmk:analityc_continuation_in_Xe}.

\begin{remark}\label{rmk:infinitely_at_one}
	It is possible to start any dynamics
	$\Xp_{\{u_t\};\RHO}$ 
	(i.e., with arbitrary admissible parameters
	$\ipb$, $\SPB$, and $\{u_t\}$)
	from the initial configuration 
	$1^{\infty}2^{0}3^{0}\ldots$. Indeed, for that
	one simply must take $\ip_1=\SP_1$, and 
	take $K\to+\infty$
	as above. 
	Under the resulting 
	(non-fused, $J=1$)
	dynamics, the
	number of particles leaving 
	location~$1$ during time step $t\to t+1$ has the distribution
	\begin{align*}
		\lim_{i_1\to+\infty}\Lmatr_{\SP_1u_{t+1},\SP_1}(i_1,j_1;i_2,j_2)=
		\mathbf{1}_{i_2=+\infty}\cdot
		\begin{cases}
			\dfrac{1}{1-\SP_1^{2}u_{t+1}},&j_2=0;\\\rule{0pt}{20pt}
			\dfrac{-\SP_1^{2}u_{t+1}}{1-\SP_1^{2}u_{t+1}},&j_2=1.
		\end{cases}
	\end{align*}
	Note that these probabilities are between $0$ and $1$ if
	$\SP_1^{2}\le0$ and $u_{t+1}\ge0$, as in the above discussion.
	We will denote this dynamics started from the infinite number of particles at
	location $1$ by 
	$\Xii_{\{u_t\}}$. We analyze 
	its observables in \S \ref{sub:starting_from_infinitely_many_particles_at_zero}
	below.
\end{remark}

\begin{figure}[htbp]
	\begin{tikzpicture}
		[scale=1,very thick]
			\def\pt{.17}
			\def\ee{.1}
			\def\h{.45}
			\draw[->] (-1.7,0) -- (8.5,0);
			\foreach \ii in {18,17,16,13,12,11,9,6,5,4,3}
			{
				\draw (\ii*\h,0) circle(\pt);
			}
			\foreach \ii in {15,14,10,8,7,2,1,0,-1,-2,-3}
			{
				\draw[fill] (\ii*\h,0) circle(\pt);
			}
			\node[below] at (15*\h,-.2) {$x_1$};
			\node[below] at (14*\h,-.2) {$x_2$};
			\node[below] at (10*\h,-.2) {$x_3$};
			\node[below] at (8*\h,-.2) {$x_4$};
			\node[below] at (7*\h,-.2) {$x_5$};
			\node[below] at (2*\h,-.2) {$x_6$};
			\node[below] at (1*\h,-.2) {$x_7$};
			\node[below] at (-1*\h,-.27) {$\ldots$};
		    \draw[->, very thick] (2*\h,.3) to [in=180,out=90] (3*\h,.85) to [in=90, out=0] (4*\h,.3) node [xshift=-7,yshift=25] {$\phi_6(2\md4)$};
		    \draw[->, very thick] (8*\h,.3) to [in=180,out=90] (8.5*\h,.65) to [in=90, out=0] (9*\h,.3) node [xshift=-10,yshift=20] {$\phi_4(1\md1)$};
		    \draw[->, very thick] (10*\h,.3) to [in=180,out=90] (11*\h,.85) to [in=90, out=0] (12*\h,.3) node [xshift=0,yshift=25] {$\phi_3(2\md3)$};
		    \draw[->, very thick] (15*\h,.3) to [in=180,out=90] (16*\h,.85) to [in=90, out=0] (17*\h,.3) node [xshift=-7,yshift=25] {$\phi_1(2\md{+\infty})$};
	\end{tikzpicture}
  	\caption{The transition of the $q$-Hahn system $\Xih$ from Fig.~\ref{fig:qhahn}, bottom, 
  	interpreted in terms of the $q$-Hahn TASEP
  	(here $\phi_i\equiv \phi_{q,q^{J}\SP_i^{2},\SP_i^{2}}$).}
  	\label{fig:qhahn_tasep}
\end{figure}
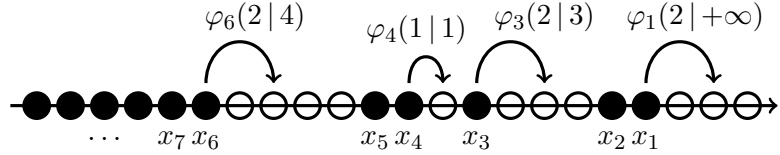

Thanks to infinitely many particles at location $1$,
the system $\Xih$ admits another nice particle interpretation. 
Namely, consider right-finite particle configurations 
$\{x_1>x_2>x_3>\ldots\}$
in $\Z$,
in which there can be at most one particle at a given location.
For a configuration $\la=1^{\infty}2^{\ell_2}3^{\ell_3}\ldots$
of $\Xih$, let 
$\ell_j=x_{j-1}-x_j-1$ (with $x_0=+\infty$)
be the number of empty spaces between consecutive particles.
Let the process start from the \emph{step initial configuration}
$x_i(0)=-i$ for all $i$ 
(corresponding to $\la=1^{\infty}2^{0}3^{0}\ldots$). 
Then during each time step, 
each particle $x_i$ jumps to the right according to the distribution
$\phi_{q,q^{J}\SP_i^{2},\SP^{2}_i}(\cdot\md\gap_i)$,
where $\gap_i=\ell_i$ is the distance to the nearest right neighbor of $x_i$
(the first particle uses the distribution with $\gap_1=+\infty$).
This system is called the \emph{$q$-Hahn TASEP}, it was also introduced
in \cite{Povolotsky2013} (in the homogeneous case $\SP_j\equiv \SP$). 
See Fig.~\ref{fig:qhahn_tasep}.

\begin{remark}
	In all the above $q$-Hahn systems, one can clearly let the 
	parameter $J$ to depend on time. 
\end{remark}


\subsubsection{$q$-TASEP and $q$-Boson} 
\label{ssub:qBoson}

Let us now perform a further degeneration of the $q$-Hahn
TASEP corresponding to the parameters
$q$ and $\{\smu_i\}$, $\{\snu_i\}$,
by setting $\smu_i=q\snu_i$ for all $i$
(that is, we take $J=1$, and thus must consider $\snu\le 0$).
Then \eqref{qHahn_distr} implies that 
$\phi_{q,\smu,\snu}(j\md m)$ 
vanishes unless $j\le 1$, and 
\begin{align*}
	&\phi_{q,\smu,\snu}(0\md m)=\frac{1-q^{m}\snu}{1-\snu},
	\qquad \qquad&
	\phi_{q,\smu,\snu}(1\md m)=\frac{-\snu(1-q^{m})}{1-\snu},
	\\
	&\phi_{q,\smu,\snu}(0\md {+\infty})=\frac{1}{1-\snu},
	\qquad \qquad&
	\phi_{q,\smu,\snu}(1\md {+\infty})=\frac{-\snu}{1-\snu}.
\end{align*}
Taking $\snu_i=- \epsilon a_i$ (with $a_i>0$) 
and speeding up
the time by $\epsilon^{-1}$, we
arrive at the \emph{$q$-TASEP} --- a continuous time particle
system on 
configurations 
$\{x_1>x_2>x_3>\ldots\}$
on
$\Z$ (with no more than one particle per location)
in which each particle $x_i$ jumps to the right by one at rate 
$a_i(1-q^{\gap_{i}})$, where, as before, 
$\gap_i=x_{i-1}-x_i-1$
is the distance to the right neighbor of $x_i$.

The $q$-TASEP was introduced in \cite{BorodinCorwin2011Macdonald}
(see also \cite{BorodinCorwinSasamoto2012}),
and an ``arrow'' interpretation of the
$q$-TASEP (on configurations
$\la=1^{\infty}2^{\gap_2}3^{\gap_3}\ldots$)
had been considered much earlier
\cite{BogoliubovBulloughTimonen94},
\cite{BogoliubovIzerginKitanine98},
\cite{SasamotoWadati1998} under the name of the 
(\emph{stochastic}) \emph{$q$-Boson system}.

It is worth noting that the $q$-Hahn system has a variety of other degenerations,
see \cite{Povolotsky2013} and \cite{CorwinBarraquand2015Beta} for examples.




\section{Orthogonality relations} 
\label{sec:orthogonality_relations}

In this section we describe 
two types of (bi)orthogonality 
relations for the symmetric rational functions
$\F_\la$
from \S \ref{sec:symmetric_rational_functions}.
These relations imply certain Plancherel
isomorphism theorems.
We also apply biorthogonality to get an integral representation 
for the functions $\G_\mu$.
The results of this section provide us with tools which will
eventually allow to explicitly evaluate averages of certain observables 
of the interacting particle systems described in 
\S \ref{sec:markov_kernels_and_stochastic_dynamics} above.

\subsection{Spatial biorthogonality} 
\label{sub:spatial_biorthogonality}

First, we will need the following general statement:

\begin{lemma}\label{lemma:general_lemma}
	Let $\{f_m(u)\}$,
	$\{g_\ell(u)\}$ be two families of rational functions in $u\in\C$
	such that there exist two disjoint sets 
	$P_1,P_2\subset\C\cup\{\infty\}$\footnote{For the applications below, these sets should be countable. 
	However, since our integrands are 
	rational functions, only finitely many of the points of $P_1$ or $P_2$
	can serve as singularities of the integrand, and so there are no issues 
	with accumulation points of $P_{1,2}$.}
	and positively oriented pairwise nonintersecting
	closed contours $c_1,\ldots,c_k$
	with the following properties:
	\begin{itemize}
		\item All singularities of all the functions $f_m(u),g_\ell(u)$ lie inside $P_1\cup P_2$.
		\item The product $f_m(u)g_\ell(u)$ does not have singularities in $P_1$ if $m<\ell$, and 
		the same product does not have singularities in $P_2$ if $m>\ell$.
		\item For any $i>j$ the contour $c_i$ can be shrunk to $P_1$ without intersecting the 
		contour\footnote{Here and below $R \boldsymbol\gamma$ denotes the image of the 
		contour $\boldsymbol\gamma$ under the multiplication by the constant $R$.} $q^{-1} c_j$
		(equivalently, for any $j<i$ the contour $c_j$ can be shrunk to $P_2$ without
		intersecting $q\cdot c_i$). 
		Shrinking takes place on the Riemann
		sphere $\C\cup\{\infty\}$.
	\end{itemize}
	Fix $k\in\Z_{\ge1}$ and two signatures $\mu,\la\in\signp k$.
	If $\mu\ne\la$, then for any permutation $\sigma\in \Sym_k$ we have
	\begin{align}\label{general_lemma_integral}
		\oint\limits_{c_1}d u_1
		\ldots
		\oint\limits_{c_k}d u_k
		\prod_{1\le \aind<\bind\le k}
		\bigg(
		\frac{u_\aind-u_\bind}{u_\aind-q u_\bind}\cdot
		\frac{u_{\sigma(\aind)}-q u_{\sigma(\bind)}}
		{u_{\sigma(\aind)}-u_{\sigma(\bind)}}
		\bigg)
		\prod_{j =1}^{k}f_{\mu_j}(u_j)g_{\la_{\sigma^{-1}(j)}}(u_j)=0.
	\end{align}
\end{lemma}
\begin{proof}
	This is a straightforward generalization of Lemma 3.5 in \cite{BCPS2014}.
	Let us outline the steps of the proof.

	We will assume that the integral \eqref{general_lemma_integral}
	is nonzero, and will show that then it must be that $\la=\mu$.
	First, we observe that, by our assumed structure of the poles, 
	\begin{itemize}
		\item If it is possible to shrink the contour $c_i$ to $P_1$, then for the 
		integral to be nonzero we must have $\mu_i\ge\la_{\sigma^{-1}(i)}$.
		\item If it is possible to shrink the contour $c_i$ to $P_2$, then for the 
		integral to be nonzero we must have $\mu_i\le\la_{\sigma^{-1}(i)}$.
	\end{itemize}
	Next, using
	\begin{align}
		\prod_{1\le \aind<\bind\le k}
		\frac{u_\aind- u_\bind}{u_\aind-q u_\bind}\cdot
		\frac{u_{\sigma(\aind)}-q u_{\sigma(\bind)}}
		{u_{\sigma(\aind)}-u_{\sigma(\bind)}}
		=\sgn(\sigma)
		\prod_{\aind<\bind\colon \sigma(\aind)>\sigma(\bind)}\frac{u_{\sigma(\aind)}-q u_{\sigma(\bind)}}{u_\aind-q u_\bind}
		\label{general_lemma_proof1}
	\end{align}
	and the properties of the integration contours, we see that
	\begin{itemize}
		\item If for some $i\in\{1,\ldots,k\}$ one has $\sigma(i)>\max(\sigma(1),\ldots,\sigma(i-1))$, then 
		the numerator in the left-hand side of \eqref{general_lemma_proof1}
		contains all terms of the form 
		$(u_{\sigma(i)}-q u_{\sigma(i)+1}),(u_{\sigma(i)}-q u_{\sigma(i)+2}),\ldots,
		(u_{\sigma(i)}-q u_{k})$, and thus the expression in the right-hand side of
		\eqref{general_lemma_proof1} does not have poles at $u_{\sigma(i)}=q u_j$ for all $j>\sigma(i)$.
		This means that we can shrink the contour $c_{\sigma(i)}$ to $P_1$
		without picking any residues.
		This implies that for the integral \eqref{general_lemma_integral} to 
		be nonzero, we must have $\mu_{\sigma(i)}\ge \la_{i}$.
		\item Similarly, if for some 
		$i\in\{1,\ldots,k\}$ one has $\sigma(i)<\min(\sigma(i+1),\ldots,\sigma(k))$, then 
		the contour $c_{\sigma(i)}$ can be shrunk to $P_2$, and so the integral
		\eqref{general_lemma_integral} can be nonzero only if 
		$\mu_{\sigma(i)}\le \la_{i}$.
	\end{itemize}

	This completes 
	the argument analogous to
	Step I of the proof of 
	\cite[Lemma\;3.5]{BCPS2014}. Further steps of the proof have a purely 
	combinatorial nature and can be repeated without change. 
	We will not reproduce the full combinatorial argument here, 
	but will 
	illustrate it on a concrete example. 

	Take $\sigma=(3,2,4,5,1,8,6,7)$, and consider an arrow diagram
	as in Fig.~\ref{fig:arrow_diagram}. That is,
	think of the bottom row as $\la_1\ge\la_2\ge \ldots\ge\la_8$
	and of the top row as $\mu_1\ge\mu_2\ge \ldots\ge\mu_8$,
	and draw the corresponding horizontal arrows from larger to smaller 
	integers. Labels in the nodes correspond to 
	the permutation $\sigma$ itself.\begin{figure}[htbp]
		\begin{center}
		\parbox{.8\textwidth}{\begin{tikzpicture}
				[scale=1,nodes={draw},very thick]
				\def\h{2.7}
				\def\x{1.5}
				\node[draw=none] at (-1*\x,\h) {$\mu$};
				\node[circle] (x1) at (0,\h) {$1$};
				\node[circle] (x2) at (\x,\h) {$2$};
				\node[circle] (x3) at (2*\x,\h) {$3$};
				\node[circle] (x4) at (3*\x,\h) {$4$};
				\node[circle] (x5) at (4*\x,\h) {$5$};
				\node[circle] (x6) at (5*\x,\h) {$6$};
				\node[circle] (x7) at (6*\x,\h) {$7$};
				\node[circle] (x8) at (7*\x,\h) {$8$};
				\node[draw=none] at (-1*\x,0) {$\la$};
				\node[circle] (y1) at (0,0) {$3$};
				\node[circle] (y2) at (\x,0) {$2$};
				\node[circle] (y3) at (2*\x,0) {$4$};
				\node[circle] (y4) at (3*\x,0) {$5$};
				\node[circle] (y5) at (4*\x,0) {$1$};
				\node[circle] (y6) at (5*\x,0) {$8$};
				\node[circle] (y7) at (6*\x,0) {$6$};
				\node[circle] (y8) at (7*\x,0) {$7$};
				\draw[ultra thick, ->] (x1)--(x2);
				\draw[ultra thick, ->] (x2)--(x3);
				\draw[ultra thick, ->] (x3)--(x4);
				\draw[ultra thick, ->] (x4)--(x5);
				\draw[ultra thick, ->] (x5)--(x6);
				\draw[ultra thick, ->] (x6)--(x7);
				\draw[ultra thick, ->] (x7)--(x8);
				\draw[ultra thick, ->] (y1)--(y2);
				\draw[ultra thick, ->] (y2)--(y3);
				\draw[ultra thick, ->] (y3)--(y4);
				\draw[ultra thick, ->] (y4)--(y5);
				\draw[ultra thick, ->] (y5)--(y6);
				\draw[ultra thick, ->] (y6)--(y7);
				\draw[ultra thick, ->] (y7)--(y8);
				\draw[line width = 2, ->, dashed, red] (x3) -- (y1);
				\draw[line width = 2, ->, dashed, red] (x4) -- (y3);
				\draw[line width = 2, ->, dashed, red] (x5) -- (y4);
				\draw[line width = 2, ->, dashed, red] (x8) -- (y6);
				\draw[line width = 2, ->, blue] (y5) -- (x1);
				\draw[line width = 2, ->, blue] (y7) -- (x6);
				\draw[line width = 2, ->, blue] (y8) -- (x7);
			\end{tikzpicture}}
			\end{center}
		\caption{Arrow diagram for $\sigma=(3,2,4,5,1,8,6,7)$.}
  		\label{fig:arrow_diagram}
	\end{figure}
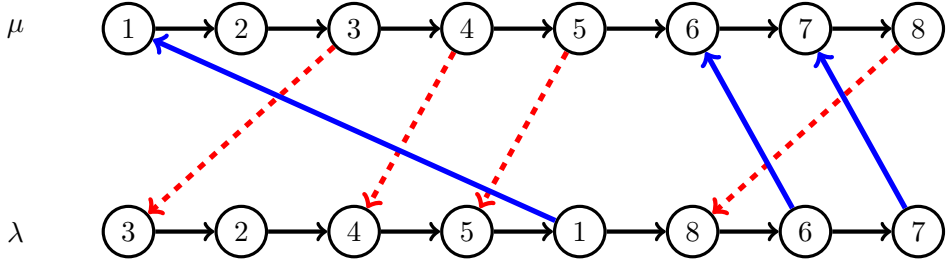
	As we read this permutation from left to right, we see
	running maxima
	$\sigma(1)$, $\sigma(3)>\max(\sigma(1),\sigma(2))$, 
	$\sigma(4)>\max(\sigma(1),\sigma(2),\sigma(3))$,
	and 
	$\sigma(6)>\max(\sigma(1),\sigma(2),\sigma(3),\sigma(4),\sigma(5))$,
	and correspondingly we add 
	(red dashed) arrows
	$\mu_{\sigma(i)}\to\la_i$.
	Similarly, reading the 
	permutation from right to left, we see running minima
	$\sigma(8)$, $\sigma(7)<\sigma(8)$, and 
	$\sigma(5)<\min(\sigma(6),\sigma(7),\sigma(8))$, 
	and we add (blue solid) arrows
	$\la_i\to\mu_{\sigma(i)}$.
	As one examines the arrow diagram, it becomes obvious
	that for the integral to be nonzero, we must have
	\begin{align*}
		\la_1=\ldots=\la_5=\mu_1=\ldots=\mu_5,
		\qquad
		\la_6=\la_7=\la_8=\mu_6=\mu_7=\mu_8.
	\end{align*}
	We also see that any permutation $\sigma$
	yielding a nonzero integral \eqref{general_lemma_integral}
	splits into two blocks 
	permuting $\{1,2,3,4,5\}$ and $\{5,6,7\}$, 
	which 
	are the clusters of equal 
	parts in $\la$ and $\mu$.
	A similar clustering occurs in the general situation, too.

	This implies that for the 
	integral \eqref{general_lemma_integral} to be nonzero,
	we must have
	$\la=\mu$, as desired.
\end{proof}

Before discussing 
the first of the 
orthogonality statements,
we will introduce some notation and assumptions.
For any set 
$\mathbf{a}=\{a_i\}_{i\in\Z_{\ge0}}$, 
denote 
\begin{align}\label{min_max_notation}
	M_{\mathbf{a}}:=\sup_{i\in\Z_{\ge0}} a_i,
	\qquad \qquad
	m_{\mathbf{a}}:=\inf_{i\in\Z_{\ge0}} a_i. 
\end{align}
Throughout the text, we will use 
this notation in expressions like
$m_{\ipbb|\SPB|}=\inf_{i\in\Z_{\ge0}}(\ip_{i}^{-1}|\SP_i|)$
(here and below $\ipbb|\SPB|$ stands for the 
set of products of elements of $\ipbb$ and $\SPB$ with the 
\emph{same indices}, i.e., we \emph{do not include all possible pairwise products}; and similarly
for other sets of products).

For the orthogonality statements below in this subsection we need to make 
certain assumptions about our parameters.
Namely, we assume that 
$q$, $\SPB$, and $\ipb$ satisfy \eqref{stochastic_weights_condition_qsxi}, 
and, moreover,
\begin{align}
	\label{assumptions_one_better}
	m_{\ipbb|\SPB|}>q M_{\ipbb|\SPB|},
	\qquad
	m_{\ipbb|\SPBB|}>M_{\ipbb|\SPB|},
\end{align}
which is needed for the existence of the contours 
in the following definition.

\begin{definition}\label{def:orthogonality_contours}
	Let 
	$q$, $\SPB$, and $\ipb$
	satisfy \eqref{stochastic_weights_condition_qsxi} and \eqref{assumptions_one_better}.
	For any $k\ge1$,
	let $\contq {\ipbb\SPB}1,\ldots,\contq {\ipbb\SPB}k$
	be 
	positively oriented closed contours
	such that
	(see Fig.~\ref{fig:contours})
	\begin{itemize}
		\item 
		Each contour
		$\contq {\ipbb\SPB}\aind$ encircles all the points of the set
		$P_1:=\ipbb\SPB=\{\ip_i^{-1}\SP_i\}_{i\in\Z_{\ge0}}$, 
		while leaving outside all the points of
		$P_2:=\ipbb\SPBB=\{\ip_i^{-1}\SP_i^{-1}\}_{i\in\Z_{\ge0}}$.
		This is possible because $m_{\ipbb|\SPBB|}>M_{\ipbb|\SPB|}$.
		\item For any $\bind>\aind$, the interior of $\contq {\ipbb\SPB}\aind$ contains 
		the contour $q\contq {\ipbb\SPB}\bind$. 
		Note that this is possible because 
		$m_{\ipbb|\SPB|}>q M_{\ipbb|\SPB|}$.
		\item 
		The contour $\contq {\ipbb\SPB}k$ is sufficiently small so that it does not intersect with $q\contq {\ipbb\SPB}k$.
	\end{itemize}
	Also, let $\contn{\ipbb\SPB;0}$ be a positively oriented closed 
	contour encircling all points
	of
	$\ipbb\SPB\cup q\ipbb\SPB\cup \ldots\cup q^{k-1}\ipbb\SPB$,
	which also contains $q\contn{\ipbb\SPB;0}$,
	and leaves outside the points of
	$\ipbb\SPBB$. Note that $0$ must be inside this contour.
\end{definition}
\begin{remark}
	The superscript ``$+$'' in the contours 
	of Definition \ref{def:orthogonality_contours}
	refers to the property that they 
	are $q$-nested, as opposed to 
	$q^{-1}$-nested contours $\contqi{\bullet}j$ 
	which we will consider later in \S \ref{sec:observables_of_interacting_particle_systems}.
	Throughout the text the set of points encircled by a contour is
	explicitly indicated in the square brackets.
\end{remark}

\begin{figure}[htbp]
	\begin{center}
	\begin{tikzpicture}
		[scale=3]
		\def\pt{0.02}
		\def\q{.6}
		\def\ss{.56}
		\draw[->, thick] (-2.8,0) -- (.8,0);
	  	\draw[->, thick] (0,-.9) -- (0,.9);
	  	\draw[fill] (-1.3,0) circle (\pt) node [below, xshift=3pt] {$\ip_i^{-1}\SP_i$};
	  	\draw[fill] (-1.1,0) circle (\pt);
	  	\draw[fill] (-1.4,0) circle (\pt);
	  	\draw[fill] (-1.3/\ss,0) circle (\pt) node [below, xshift=7pt] {$\ip_i^{-1}\SP_i^{-1}$};
	  	\draw[fill] (-1.1/\ss,0) circle (\pt);
	  	\draw[fill] (-1.4/\ss,0) circle (\pt);
	  	\draw[fill] (0,0) circle (\pt) node [below left] {$0$};
	  	\draw (-1.25,0) circle (.25) node [below,xshift=-7,yshift=24] {$\contqe3$};
	  	\draw[dotted] (-1.25*\q,0) circle (.25*\q);
	  	\draw[dotted] (-1.25*\q*\q,0) circle (.25*\q*\q);
	  	\draw (-.74,0) circle (.9) node [below,xshift=-64,yshift=80] {$\contn {\ipbb\SPB;0}$};
	  	\draw (-1.05,0) ellipse (.48 and .4) node [below,xshift=40,yshift=34] {$\contqe{2}$};
	  	\draw (-.92,0) ellipse (.65 and .5) node [below,xshift=25,yshift=57] {$\contqe{1}$};
	\end{tikzpicture}
	\end{center}
  	\caption{A possible choice of nested integration contours 
  	$\contqe{i}=\contq {\ipbb\SPB}i$, $i=1,2,3$
  	(Definition~\ref{def:orthogonality_contours}).
  	Contours $q\contq{\ipbb\SPB}3$ and $q^{2}\contq{\ipbb\SPB}3$ are shown dotted.
  	The large contour $\contn {\ipbb\SPB;0}$ 
  	is also shown.}
  	\label{fig:contours}
\end{figure}
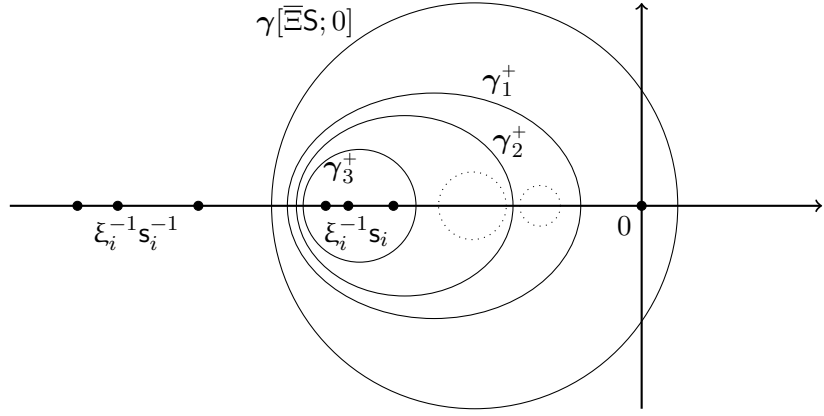

\begin{theorem}\label{thm:orthogonality_F}
	Under assumptions 
	\eqref{stochastic_weights_condition_qsxi} and \eqref{assumptions_one_better}
	on
	$q$, $\SPB$, and $\ipb$,
	for any $k\in\Z_{\ge1}$
	and $\la,\mu\in\signp k$, we have
	\begin{align}
		(1-q)^{-k}
		\oint\limits_{\contq {\ipbb\SPB}1}\frac{d u_1}{2\pi\i}
		\ldots
		\oint\limits_{\contq {\ipbb\SPB}k}\frac{d u_k}{2\pi\i}
		\prod_{1\le \aind<\bind\le k}\frac{u_\aind-u_\bind}{u_\aind-qu_\bind}\,
		\F_{\la}^{\conj}(u_1,\ldots,u_k\md\ipb,\SPB)
		\prod_{i=1}^{k}u_i^{-1}\pow_{\mu_{i}}(u_{i}^{-1}\md\ipbb,\SPB)
		=\mathbf{1}_{\la=\mu},
		\label{orthogonality_formulation}
	\end{align}
	where the $\pow_{\mu_i}$'s and $\F_\la$ are as in \S \ref{sub:symmetrization_formulas}.
\end{theorem}
An immediate corollary of Theorem \ref{thm:orthogonality_F}
is the following ``spatial'' biorthogonality property of the 
functions $\F_\mu$:
\begin{corollary}\label{cor:spatial_biorth}
	Under assumptions 
	\eqref{stochastic_weights_condition_qsxi} and \eqref{assumptions_one_better}
	on
	$q$, $\SPB$, and $\ipb$,
	for any $k\in\Z_{\ge1}$
	and $\la,\mu\in\signp k$, we have
	\begin{align}
		\frac{\conj_\SPB(\la)}{(1-q)^{k}k!}
		\oint\limits_{\contn {\ipbb\SPB;0}}\frac{d u_1}{2\pi\i u_1}
		\ldots
		\oint\limits_{\contn {\ipbb\SPB;0}}\frac{d u_k}{2\pi\i u_k}
		\prod_{\aind\ne \bind}\frac{u_\aind-u_\bind}{u_\aind-qu_\bind}\,
		\F_{\la}(u_1,\ldots,u_k\md\ipb,\SPB)
		\F_{\mu}(u_1^{-1},\ldots,u_k^{-1}\md\ipbb,\SPB)=
		\mathbf{1}_{\la=\mu},
		\label{orthogonality_formulation_cor}
	\end{align}
	with $\conj_\SPB(\cdot)$ as in \eqref{conj_definition}.
\end{corollary}
We call this property \emph{spatial biorthogonality}
because for, say, fixed $\la$ and varying $\mu$, 
the right-hand side 
is the delta function in the \emph{spatial} variables
$\mu$. This should be compared to the \emph{spectral biorthogonality}
of Theorem \ref{thm:spec_orthogonality}
with delta functions
in spectral variables in the right-hand side.

When the parameters $\ip_j\equiv 1$ and $\SP_j\equiv \SP$ are homogeneous, 
Theorem \ref{thm:orthogonality_F} and Corollary \ref{cor:spatial_biorth}
were conjectured in \cite{Povolotsky2013} 
and proved in
\cite[\S3]{BCPS2014}.
See also \cite[Thm. 7.2]{Borodin2014vertex}.
\begin{proof}[Proof of Corollary \ref{cor:spatial_biorth}]
	Assuming Theorem \ref{thm:orthogonality_F}, 
	deform the integration contours 
	$\contq {\ipbb\SPB}1,\ldots,\contq {\ipbb\SPB}k$ (in this order)
	to $\contn {\ipbb\SPB;0}$. One readily sees that this does not lead to  
	any additional residues.
	Next, observe that
	the integral over all $u_j\in\contn {\ipbb\SPB;0}$ 
	is invariant under permutations of the $u_j$'s, and thus
	one can perform the symmetrization and divide by $k!$. This leads to
	\begin{multline*}
		\mathbf{1}_{\la=\mu}=
		\frac{(1-q)^{-k}}{k!}
		\oint\limits_{\contn {\ipbb\SPB;0}}\frac{d u_1}{2\pi\i u_1}
		\ldots
		\oint\limits_{\contn {\ipbb\SPB;0}}\frac{d u_k}{2\pi\i u_k}
		\prod_{1\le \aind\ne \bind\le k}\frac{u_\aind-u_\bind}{u_\aind-qu_\bind}
		\F_{\la}^{\conj}(u_1,\ldots,u_k\md\ipb,\SPB)
		\\\times
		\underbrace{\sum_{\sigma\in \Sym_k}
		\prod_{1\le \bind<\aind\le k}\frac{u_{\sigma(\aind)}-qu_{\sigma(\bind)}}{u_{\sigma(\aind)}-u_{\sigma(\bind)}}
		\prod_{i=1}^{k}\pow_{\mu_{i}}(u_{\sigma(i)}^{-1}\md\ipbb,\SPB)}_
		{\textnormal{$\F_{\mu}(u_1^{-1},\ldots,u_k^{-1}\md\ipbb,\SPB)$ 
		by \eqref{F_symm_formula}}},
	\end{multline*}
	as desired.
\end{proof}

\begin{proof}[Proof of Theorem \ref{thm:orthogonality_F}]
	Fix $k$ and signatures $\la$ and $\mu$.
	In order to apply Lemma \ref{general_lemma_integral}, set 
	\begin{align*}
		f_m(u):=
		\frac{\pow_m(u^{-1}\md\ipbb,\SPB)}{u}
		=\frac{\ip_m(1-q)}{\ip_mu-\SP_m}\prod_{j=0}^{m-1}\frac{1-\SP_j\ip_ju}{\ip_ju-\SP_j},
		\quad
		g_\ell(u):=\pow_\ell(u\md\ipb,\SPB)
		=\frac{1-q}{1-\SP_\ell\ip_\ell u}\prod_{j=0}^{\ell-1}\frac{\ip_ju-\SP_j}{1-\SP_j\ip_ju},
	\end{align*}
	and use the sets $P_1$, $P_2$ described
	in Definition \ref{def:orthogonality_contours}.
	If $m>\ell$, then all singularities of
	\begin{align*}
		f_m(u)g_\ell(u)=
		\frac{\ip_m(1-q)}{\ip_mu-\SP_m}
		\frac{1-q}{1-\SP_\ell\ip_\ell u}\prod_{j=0}^{m-1}\frac{1-\SP_j\ip_ju}{\ip_ju-\SP_j}
		\prod_{j=0}^{\ell-1}\frac{\ip_ju-\SP_j}{1-\SP_j\ip_ju}
	\end{align*} 
	are in $P_1$, and if $m<\ell$, then all singularities of this product are in $P_2$.
	By virtue of the symmetrization formula for $\F_\la$ \eqref{F_symm_formula}, 
	we see that the integrand in \eqref{orthogonality_formulation}
	is 
	(up to a multiplicative constant)
	the same as the one in Lemma \ref{lemma:general_lemma}
	(with the above specialization of $f_m$ and $g_\ell$). 
	The structure of our integration contours $\contq {\ipbb\SPB}j$
	and the fact that $\infty\notin P_1\cup P_2$ (so
	the contours can be dragged through infinity without picking the residues)
	implies that the third condition of Lemma \ref{lemma:general_lemma} is also 
	satisfied.
	Thus, we 
	conclude that the integral in \eqref{orthogonality_formulation} vanishes unless $\mu=\la$.

	\begin{example}
		To better illustrate the application of Lemma 
		\ref{lemma:general_lemma} here, consider contours 
		$\contq {\ipbb\SPB}1,\contq {\ipbb\SPB}2$, and $\contq {\ipbb\SPB}3$
		as in Fig.~\ref{fig:contours}. Depending on $\sigma$, 
		the denominator 
		in the integrand in \eqref{general_lemma_integral} contains 
		some of the factors $u_1-q u_2$, $u_1-q u_3$, and $u_2-q u_3$.
		The contour $\contq {\ipbb\SPB}3$ can always be shrunk to $P_1$ without picking residues at
		$u_3=q^{-1}u_1$ and $u_3=q^{-1}u_2$ (this is the assumption 
		on the contours in Lemma~\ref{lemma:general_lemma}). 
		Moreover, if, for example, the permutation
		$\sigma$ provides a cancellation of the factor 
		$u_2-qu_3$ in the denominator, then the contour $\contq {\ipbb\SPB}2$
		can also be shrunk to $P_1$ without picking the 
		residue at $u_2=qu_3$. Similarly, the contour
		$\contq {\ipbb\SPB}1$ can always be expanded (``shrunk'' on the Riemann sphere) to $P_2$ 
		(recall that infinity does not supply any residues),
		and if the factor $u_1-qu_2$ in the denominator is canceled for a certain $\sigma$,
		then $\contq {\ipbb\SPB}2$ can also be expanded to $P_2$ without picking the residue at $u_2=q^{-1}u_1$.
	\end{example}
	
	Now we must consider the 
	case when
	$\mu=\la$, that is, evaluate the ``squared norm''
	of $\F_\mu$. 
	Arguing similarly to the example in the proof of 
	Lemma \ref{general_lemma_integral}
	(see \cite[Lemma\;3.5]{BCPS2014} for more detail), we see that the
	integral 
	\begin{multline}
		\sum_{\sigma\in\Sym_k}
		\oint\limits_{\contq {\ipbb\SPB}1}\frac{d u_1}{2\pi\i}
		\ldots
		\oint\limits_{\contq {\ipbb\SPB}k}\frac{d u_k}{2\pi\i}
		\prod_{1\le \aind<\bind\le k}
		\bigg(\frac{u_\aind-u_\bind}{u_\aind-qu_\bind}
		\frac{u_{\sigma(\aind)}-qu_{\sigma(\bind)}}{u_{\sigma(\aind)}-u_{\sigma(\bind)}}\bigg)
		\\\times\prod_{i=1}^{k}u_i^{-1}\pow_{\la_{i}}(u_{i}^{-1}\md\ipbb,\SPB)
		\pow_{\la_{\sigma^{-1}(i)}}(u_i\md\ipb,\SPB)
		\label{orthogonality_F_proof1}
	\end{multline}
	(this is the same as the left-hand side of \eqref{orthogonality_formulation} with $\mu=\la$,
	up to the constant $(1-q)^{-k}\conj_{\SPB}(\la)$)
	vanishes unless $\sigma\in\Sym_k$ permutes within
	clusters of the signature $\la$, that is,
	$\sigma$ must preserve each maximal 
	set of indices $\{a,a+1,\ldots,b\}\subseteq\{1,\ldots,k\}$
	for which $\la_{a}=\la_{a+1}=\ldots=\la_{b}$. Let $c_\la$ be the 
	number of such clusters in $\la$. 
	Denote the set of all permutations 
	permuting within
	clusters of $\la$
	by $\Sym_k^{(\la)}$. 
	Any permutation $\sigma\in\Sym_k^{(\la)}$
	can be represented as a product of $c_\la$ permutations
	$\sigma_1,\ldots,\sigma_{c_\la}$, with each $\sigma_{i}$ fixing all elements
	of $\{1,\ldots,k\}$ except those belonging to the $i$-th cluster of $\la$.
	We will denote
	the set of indices within the $i$-th cluster by $C_i(\la)$,
	and write $\sigma_i\in \Sym_k^{(\la,i)}$.

	Therefore, the sum in \eqref{orthogonality_F_proof1}
	is only over $\sigma\in\Sym_k^{(\la)}$. Let us now compute it. We have
	\begin{multline*}
		\sum_{\sigma\in \Sym_k^{(\la)}}
		\prod_{1\le \aind<\bind\le k}
		\frac{u_{\sigma(\aind)}-qu_{\sigma(\bind)}}{u_{\sigma(\aind)}-u_{\sigma(\bind)}}
		\prod_{i=1}^{k}u_i^{-1}\pow_{\la_{i}}(u_{i}^{-1}\md\ipbb,\SPB)
		\pow_{\la_{\sigma^{-1}(i)}}(u_i\md\ipb,\SPB)
		\\=
		\prod_{r=1}^{k}\frac{\ip_{\la_r}(1-q)^{2}}{(\ip_{\la_r}u_r-\SP_{\la_r})(1-\SP_{\la_r}\ip_{\la_r}u_r)}\cdot
		\prod_{1\le i<j\le c_\la}
		\prod_{\substack{\aind\in C_i(\la)\\ \bind\in C_j(\la)}}
		\frac{u_\aind-qu_\bind}{u_\aind-u_\bind}
		\cdot
		\prod_{i=1}^{c_\la}
		\sum_{\sigma_i\in \Sym_k^{(\la,i)}}
		\prod_{\substack{1\le \aind<\bind\le k\\
		\aind,\bind\in C_i(\la)}}
		\frac{u_{\sigma_i(\aind)}-qu_{\sigma_i(\bind)}}
		{u_{\sigma_i(\aind)}-u_{\sigma_i(\bind)}}.
	\end{multline*}
	For each sum 
	over $\sigma_i\in \Sym_k^{(\la,i)}$
	we can use the symmetrization identity
	(footnote$^{\ref{symm_footnote}}$).
	Thus, \eqref{orthogonality_F_proof1}
	becomes
	\begin{multline*}
		\sum_{\sigma\in\Sym_k}
		\oint\limits_{\contq {\ipbb\SPB}1}\frac{d u_1}{2\pi\i}
		\ldots
		\oint\limits_{\contq {\ipbb\SPB}k}\frac{d u_k}{2\pi\i}
		\prod_{\aind<\bind}
		\bigg(\frac{u_\aind-u_\bind}{u_\aind-qu_\bind}
		\frac{u_{\sigma(\aind)}-qu_{\sigma(\bind)}}{u_{\sigma(\aind)}-u_{\sigma(\bind)}}\bigg)
		\prod_{i=1}^{k}u_i^{-1}\pow_{\la_{i}}(u_{i}^{-1}\md\ipbb,\SPB)
		\pow_{\la_{\sigma^{-1}(i)}}(u_i\md\ipb,\SPB)
		\\=
		(1-q)^{k}\prod_{i\ge0}(q;q)_{\ell_i}
		\oint\limits_{\contq {\ipbb\SPB}1}\frac{d u_1}{2\pi\i}
		\ldots
		\oint\limits_{\contq {\ipbb\SPB}k}\frac{d u_k}{2\pi\i}
		\prod_{i=1}^{c_\la}
		\prod_{\substack{1\le \aind<\bind\le k\\
		\aind,\bind\in C_i(\la)}}\frac{u_\aind-u_\bind}{u_\aind-qu_\bind}
		\prod_{r=1}^{k}\frac{\ip_{\la_r}}
		{(\ip_{\la_r}u_r-\SP_{\la_r})(1-\SP_{\la_r}\ip_{\la_r}u_r)},
	\end{multline*}
	where we have used the usual multiplicative notation 
	$\la=0^{\ell_0}1^{\ell_1}2^{\ell_2}\ldots$.
	The integration variables above corresponding to each cluster are now independent, and thus
	the integral reduces to a product of
	$c_\la$ smaller nested contour integrals of similar form.
	In each such separate contour integral
	the inhomogeneity parameters $\ip_{\la_r}$ and $\SP_{\la_r}$
	are the same because the $\la_r$'s belong to the same cluster.
	Thus,
	each of these integrals can be computed
	as follows (here we denote $\ip=\ip_{\la_r}$ and $\SP=\SP_{\la_r}$
	to shorten the notation):\footnote{In fact, a more general 
	integral of this sort can also be computed, see \cite[Prop.\;3.7]{BCPS2014}.}
	\begin{align*}
		&\oint\limits_{\contq {\ipbb\SPB}1}\frac{du_1}{2\pi\i}
		\oint\limits_{\contq {\ipbb\SPB}2}\frac{du_2}{2\pi\i}
		\ldots
		\oint\limits_{\contq {\ipbb\SPB}\ell}\frac{du_\ell}{2\pi\i}
		\prod_{1\le \aind<\bind\le \ell}\frac{u_\aind-u_\bind}{u_\aind-qu_\bind}
		\prod_{i=1}^{\ell}\frac{\ip}{(\ip u_i-\SP)(1-\SP\ip u_i)}
		\\&\hspace{40pt}=\frac{1}{1-\SP^2}
		\oint\limits_{\contq {\ipbb\SPB}2}\frac{du_2}{2\pi\i}
		\ldots
		\oint\limits_{\contq {\ipbb\SPB}\ell}\frac{du_\ell}{2\pi\i}
		\prod_{j=2}^{\ell}\frac{1-\SP\ip u_j}{1-q\SP \ip u_j}
		\prod_{2\le \aind<\bind\le \ell}\frac{u_\aind-u_\bind}{u_\aind-qu_\bind}
		\prod_{i=2}^{\ell}\frac{\ip}{(\ip u_i-\SP)(1-\SP \ip u_i)}
		\\&\hspace{40pt}=
		\frac{1}{1-\SP^2}
		\oint\limits_{\contq {\ipbb\SPB}2}\frac{du_2}{2\pi\i}
		\ldots
		\oint\limits_{\contq {\ipbb\SPB}\ell}\frac{du_\ell}{2\pi\i}
		\prod_{2\le \aind<\bind\le \ell}\frac{u_\aind-u_\bind}{u_\aind-qu_\bind}
		\prod_{i=2}^{\ell}\frac{\ip}{(\ip u_i-\SP)(1-q\SP \ip u_i)}
		\\&\hspace{40pt}=
		\textnormal{etc.}
		\\&\hspace{40pt}=
		\frac{1}{(\SP^{2};q)_{\ell}}.
	\end{align*}
	Indeed, there is only one $u_1$-pole 
	outside the contour $\contq {\ipbb\SPB}1$, namely, $u_1=(\SP\ip)^{-1}$.
	Evaluating the integral over $u_1$ by taking the 
	minus residue at $u_1=(\SP\ip)^{-1}$ leads to a smaller similar integral 
	with the outside pole $(\SP\ip)^{-1}$ replaced by $(q\SP\ip)^{-1}$.
	Continuing in the same way with integration over $u_2,\ldots,u_\ell$, we obtain
	${1}/{(\SP^{2};q)_{\ell}}$.
	Putting together all of the above components, we see that 
	we have established the desired claim.
\end{proof}

\subsection{Plancherel isomorphisms and completeness} 
\label{sub:plancherel_isomorphisms_and_completeness}

Here we discuss Plancherel isomorphism results
related to the functions $\F_\la$. Similar results
were obtained in \cite{BorodinCorwinPetrovSasamoto2013}, \cite{BCPS2014},
and \cite{Borodin2014vertex} in the homogeneous case $\ip_j\equiv1$,
$\SP_j\equiv \SP$. 
Detailed discussion of Plancherel-type results for other integrable
interacting particle systems can also be found in the first two of these
references.

Let us fix the number of variables $n\in\Z_{\ge1}$. 
Let the parameters $\ipb=\{\ip_x\}_{x\in\Z}$ and 
$\SPB=\{\SP_x\}_{x\in\Z}$ be indexed by 
all integers, and assume that $\ip_x^{-1}\SP_x^{\pm1}$
for all $x\in\Z$ are pairwise distinct.\footnote{Most definitions and 
statements below in this 
subsection are still valid (with suitable modifications) when some of these points
coincide, and follow by a simple limit transition. However,
we will not focus on these details.}
We will assume that conditions
\eqref{stochastic_weights_condition_qsxi} and \eqref{assumptions_one_better} hold for these
$\Z$-indexed families of parameters, which implies that
the nested integration contours 
$\contq {\ipbb\SPB}j$ of Definition~\ref{def:orthogonality_contours}~exist.

Extend the definition of the functions
$\F_\la(u_1,\ldots,u_n\md\ipb,\SPB)$ to all $\la\in\sign n$ (i.e., with 
possibly negative parts)
using the shifting property \eqref{F_shifts}.
In other words, define
$\F_\la$ for all $\la\in\sign n$ by the same symmetrization
formula \eqref{F_symm_formula}, but extend
$\pow_k(u\md\ipb,\SPB)$ \eqref{pow_function} to negative integers as
\begin{align*}
	\pow_{-k}(u\md\ipb,\SPB)=
	\frac{1-q}{1-\ip_{-k}\SP_{-k}u}\prod_{j=-k}^{-1}
	\frac{1-\ip_j\SP_ju}{\ip_ju-\SP_j}, \qquad k\in\Z_{\ge1},
\end{align*}
so that $\{\pow_k\}_{k\in\Z}$ also satisfy \eqref{F_shifts} with $M=1$.

\begin{definition}[Function spaces]\label{def:function_spaces}
	Denote the space of functions $f(\la)$ on $\sign n$ with finite support
	by $\funspat{n}$. 
	Also, denote by $\funspec{n}$ the space of symmetric rational functions
	$R(u_1,\ldots,u_n)$ which can have poles only at $u_i=\ip_x^{-1}\SP_x^{-1}$, $x\in\Z_{\ge0}$ 
	and $u_i=\ip_x^{-1}\SP_x$, 
	$x\in\Z_{<0}$, $i=1,\ldots,n$,
	and all the poles are simple. 
	In other words,
	\begin{align*}
		R(u_1,\ldots,u_n)\cdot\prod_{i=1}^{n}\bigg(
		\prod_{x=0}^{M}(u_i-\ip_x^{-1}\SP_x^{-1})
		\prod_{x=-M}^{-1}(u_i-\ip_x^{-1}\SP_x)
		\bigg)
	\end{align*}
	is a polynomial in the $u_i$'s for large enough $M$.
	Moreover, we require that the functions from 
	$\funspec{n}$ converge to zero as $|u_i|\to\infty$ for any $i$.
\end{definition}

Note that as functions in $\la$, the $\F_\la(u_1,\ldots,u_n\md\ipb,\SPB)$'s clearly do not belong to 
$\funspat{n}$. However, as functions in the $u_i$'s, they belong to $\funspec{n}$.
Indeed, to verify the latter statement, use \eqref{F_symm_formula}
and bring all summands to the common denominator. 
This denominator is a product of
factors of the form $u_i-\ip_x\SP_x^{\pm1}$
times the Vandermonde $\prod_{1\le\aind<\bind\le n}(u_\aind-u_\bind)$.
Observe that the numerator is an antisymmetric polynomial in $u_1,\ldots,u_n$, 
so it can be divided by the Vandermonde, thus 
removing it from the denominator.
Finally, $\F_\la(u_1,\ldots,u_n\md\ipb,\SPB)$ clearly goes to zero as $|u_i|\to\infty$.

Let us first formulate two main statements 
(Theorems \ref{thm:spec_orthogonality} and \ref{thm:Plancherel})
which we prove in this subsection. 
\begin{theorem}\label{thm:spec_orthogonality}
	The functions $\F_\la$ satisfy the spectral biorthogonality 
	relation\footnote{
	Throughout the text
	we will use the abbreviated notation
	$\bar\VV=(v_1^{-1},\ldots,v_n^{-1})$,
	and $d\VV$ stands for
	$dv_1 \ldots dv_n$.
	Similarly for $\bar\UU$ and $d\UU$. Also,
	$\sigma\UU$ stands for the
	permutation $\sigma$ of the variables
	$\UU$.}
	\begin{multline}\label{spec_biorth}
		\sum_{\la\in\sign n}\mathop{\oint \ldots \oint}
		\frac{d\UU}{(2\pi\i)^{n}}
		\mathop{\oint \ldots \oint}
		\frac{d\VV}{(2\pi\i)^{n}}
		\varphi(\UU)\psi(\VV)
		\F_\la(\UU\md\ipb,\SPB)
		\F^{\conj}_{\la}(\bar\VV\md\ipbb,\SPB)
		\prod_{1\le\aind<\bind\le n}
		(u_{\aind}-u_{\bind})
		(v_{\aind}-v_{\bind})
		\\=
		(-1)^{\frac{n(n-1)}2}\mathop{\oint \ldots \oint}
		\frac{d\UU}{(2\pi\i)^{n}}
		\prod_{1\le \aind,\bind\le n}(u_{\aind}-qu_{\bind})\cdot
		\varphi(\UU)\sum_{\sigma\in\Sym_n}
		\sgn(\sigma)\psi(\sigma \UU),
	\end{multline}
	where $\varphi$ and $\psi$ are arbitrary test functions satisfying
	(see also Remark \ref{rmk:optimal_phi_psi} below for a
	discussion of these conditions)
	\begin{align}\label{R_Q_conditions}
		\parbox{.85\textwidth}{
		\begin{itemize}
			\item The function $\varphi(\UU)\prod_{\aind<\bind}
			(u_{\aind}-u_{\bind})$ is a rational function in $\{u_i\}$ 
			and all parameters $q$, $\ipb$, and $\SPB$,
			which 
			can have at most simple poles at $u_i=\ip_x^{-1}\SP_x^{\pm1}$, $x\in\Z$, $i=1,\ldots,n$,
			and is regular at infinity,
			$u_i=\ip_x^{-1}\SP_x^{-1}$, $x\ge0$,
			and $u_i=\ip_x^{-1}\SP_x$, $x<0$;
			\item The function
			$\psi(\VV)\prod_{\aind,\bind}(v_{\aind}-qv_{\bind})$ 
			is rational 
			in $\{v_j\}$ 
			and all parameters $q$, $\ipb$, and $\SPB$,
			and 
			can have at most simple poles at $v_j=\ip_x^{-1}\SP_x^{\pm1}$, 
			$x\in\Z$, $j=1,\ldots,n$,
		\end{itemize}}
	\end{align}
	and each integration 
	in \eqref{spec_biorth} is performed over 
	one and the same positively oriented closed contour 
	which encircles
	all of the points $\ip_x^{-1}\SP_x$, $x\in\Z$,
	and leaves all $\ip_x^{-1}\SP_x^{-1}$ outside.
\end{theorem}
\begin{remark}\label{rmk:spec_biorth_inform}
	Informally, the spectral biorthogonality can be written as
	\begin{multline*}
		\prod_{1\le\aind<\bind\le n}
		(u_{\aind}-u_{\bind})
		(v_{\aind}-v_{\bind})
		\sum_{\la\in\sign n}
		\F_\la(\UU\md\ipb,\SPB)
		\F^{\conj}_{\la}(\bar\VV\md\ipbb,\SPB)\\=
		(-1)^{\frac{n(n-1)}2}\prod_{1\le \aind,\bind\le n}(u_{\aind}-qu_{\bind})\cdot
		\det[\delta(v_i-u_j)]_{i,j=1}^{n}.
	\end{multline*}
	This identity should be understood 
	in the integrated sense with suitable test 
	functions as above.
\end{remark}
\begin{definition}[Plancherel transforms]\label{def:Plancherel_transforms}
	The direct transform $\Pltrans n$ maps a function $f$ from $\funspat n$ 
	to $\Pltrans n f\in\funspec n$ and acts as
	\begin{align*}
		(\Pltrans n f)(u_1,\ldots,u_n):=\sum_{\la\in\sign n}f(\la)\F_\la(u_1,\ldots,u_n\md\ipb,\SPB).
	\end{align*}
	The inverse transform $\Pltransi n$ takes $R\in\funspec n$ to 
	$\Pltransi n R\in\funspat n$ and acts as
	\begin{align*}
		(\Pltransi n R)(\la):=
		\frac{\conj_{\SPB}(\la)}{(1-q)^{n}}
		\oint\limits_{\contq {\ipbb\SPB}1}\frac{d u_1}{2\pi\i}
		\ldots
		\oint\limits_{\contq {\ipbb\SPB}n}\frac{d u_n}{2\pi\i}
		\prod_{1\le \aind<\bind\le n}\frac{u_\aind-u_\bind}{u_\aind-qu_\bind}
		\frac{R(u_1,\ldots,u_n)}{u_1 \ldots u_n}
		\prod_{i=1}^{n}\pow_{\la_{i}}(u_{i}^{-1}\md\ipbb,\SPB),
	\end{align*}
	where $\conj_{\SPB}(\la)$ is defined by \eqref{conj_definition}.
	Let us explain why $\Pltransi n R$ has finite support in $\la$. 
	If $\la_1\ge M$
	for sufficiently large $M>0$, then the integrand has no poles $\ip_x^{-1}\SP_x^{-1}$
	outside $\contq {\ipbb\SPB}1$, and thus vanishes. (It is crucial 
	that $R$ vanishes at $u_1=\infty$, so that
	the integrand has no residue at $u_1=\infty$.)
	Similarly,
	if $\la_n\le -M$, then 
	there are no $u_n$-poles 
	$\ip_x^{-1}\SP_x$
	inside 
	$\contq {\ipbb\SPB}n$, and so the integral also vanishes. 
	Clearly, the bound $M$ depends on the function $R$.
\end{definition}

\begin{remark}\label{rmk:bilinear_pairing_Tinv}
	Similarly to \cite[Proposition\;3.2]{BCPS2014}, the
	nested contours in the inverse Plancherel transform
	transform $\Pltransi n$ can be replaced by 
	two different families of 
	identical contours, which allows to 
	symmetrize under the integral and 
	interpret $\Pltransi n R$ as a bilinear 
	pairing between $R$ and $\F^{\conj}_\la(u_1^{-1},\ldots,u_n^{-1}\md\ipbb,\SPB)$.
	One of the choices of these identical contours
	is $\contn{\ipbb\SPB;0}$, cf. Corollary \ref{cor:spatial_biorth}.
	Another one is the small contour $\contq {\ipbb\SPB}1$ 
	around $\ipbb\SPB$, but the formula 
	for $\Pltransi n$ would then involve string specializations of the $u_i$'s.
	We refer to \cite{BorodinCorwinPetrovSasamoto2013} 
	and \cite{BCPS2014} for details.
\end{remark}

\begin{theorem}[Plancherel isomorphisms]\label{thm:Plancherel}
	The operator $f\mapsto \Pltransi n (\Pltrans n f)$ acts as the identity 
	on $\funspat n$. The operator
	$R\mapsto \Pltrans n (\Pltransi n R)$
	acts as the identity on $\funspec n$.
\end{theorem}

The first statement of this theorem 
is clearly equivalent to Theorem \ref{thm:orthogonality_F} 
established above 
(note that by \eqref{F_shifts}, identities
\eqref{orthogonality_formulation} and \eqref{orthogonality_formulation_cor}
are invariant under simultaneous shifts in $\la$ and $\mu$, and 
thus also hold for all $\la,\mu\in\sign n$). The proof of 
the second statement relies on Theorem \ref{thm:spec_orthogonality} and 
is given below in this subsection.

\begin{example}[$n=1$]
	To illustrate
	our strategy of the proofs
	(and relate Theorems \ref{thm:spec_orthogonality}
	and \ref{thm:Plancherel} to the
	results of \S \ref{sub:spatial_biorthogonality} and
	the Cauchy identities
	of \S \ref{sub:limits_of_yang_baxter_commutation_relations_cauchy_type_identities}), 
	let us consider the simplest one-variable homogeneous case.
	For that, let us consider the following variant of the Cauchy identity:
	\begin{align*}
		\sum_{n=0}^{\infty}\frac{z^{n}}{w^{n+1}}=\frac{1}{w-z},\qquad \left|\frac{z}{w}\right|<1.
	\end{align*}
	By shifting the summation index towards $-\infty$, we can write
	\begin{align*}
		\sum_{n=-M}^{\infty}\frac{z^{n}}{w^{n+1}}=\frac{w^{M}}{z^{M}}\frac{1}{w-z},\qquad \left|\frac{z}{w}\right|<1.
	\end{align*}
	Now take contour integrals in $z$ and $w$ (over positively oriented
	circles with $|z|<|w|$) of both sides of this relation multiplied by 
	$P(z)Q(w)$, where $P$ and $Q$ are Laurent polynomials. Then in the left-hand side 
	we obtain the same sum for any $M\gg 1$, and in the right-hand side the 
	$w$ contour can be shrunk to zero, thus picking the residue at $w=z$. Therefore, we have an analogue of Theorem \ref{thm:spec_orthogonality}:
	\begin{align*}
		\sum_{n=-\infty}^{\infty}\oint\oint \frac{z^{n}}{w^{n+1}}\,P(z)Q(w)\frac{dz}{2\pi\i}\frac{dw}{2\pi\i}
		=
		\oint P(z)Q(z)\frac{dz}{2\pi\i}.
	\end{align*}
	The convergence condition $|z|<|w|$ is irrelevant for the left-hand side because the 
	sum over $n$ now contains only finitely many terms.
	The resulting spectral biorthogonality can be informally written as
	\begin{align*}
		\sum_{n=-\infty}^{\infty}\frac{z^{n}}{w^{n+1}}=\delta(w-z),
	\end{align*}
	cf. Remark \ref{rmk:spec_biorth_inform}.

	To get the other biorthogonality relation, integrate both sides of the above identity against 
	$w^{m}\frac{dw}{2\pi\i}$,
	$m\in\Z$. Since $\{z^{n}\}_{n\in\Z}$ are linearly independent, 
	we obtain
	\begin{align*}
		\oint_{|w|=\textnormal{const}}w^{m}\frac{1}{w^{n+1}}\frac{dw}{2\pi\i}=\mathbf{1}_{m=n}.
	\end{align*}
	This is the spatial biorthogonality relation
	(an analogue of Theorem \ref{thm:orthogonality_F}
	and Corollary \ref{cor:spatial_biorth}). This identity also
	readily follows from the Cauchy's integral formula.

	Similar considerations work for Cauchy identities in several variables. The second type of biorthogonality 
	relations can often be verified independently in a simpler fashion (as in the proof of 
	Theorem~\ref{thm:orthogonality_F}).
\end{example}

\begin{proof}[Proof of Theorem \ref{thm:spec_orthogonality}]
	The proof is similar to the one given in \cite{Borodin2014vertex},
	with suitable modifications required in the inhomogeneous case.
	The starting point is the Cauchy identity of Corollary \ref{cor:usual_Cauchy}
	written for parameters $\sh_{-M}\ipb$, $\sh_{-M}\SPB$,
	where $M>0$ is a large integer:
	\begin{align}
		\frac{\prod_{i=1}^{n}(1-\SP_{-M}\ip_{-M}u_i)}{(q;q)_{n}}
		\sum_{\mu\in\signp n}
		\F_{\mu}(\UU\md\sh_{-M}\ipb,\sh_{-M}\SPB)
		\G^{\conj}_{\mu}(\bar\VV\md\sh_{-M}\ipbb,\sh_{-M}\SPB)
		=
		\prod_{i=1}^{n}\prod_{j=1}^{n}
		\frac{v_j-qu_i}{v_j-u_i}.\label{Cauchy_identity_spectral_biorth_proof0}
	\end{align}
	Recall that the
	convergence of this sum
	requires
	$\adm{u_i}{v_j^{-1}}$ for all $i,j$ (Definition \ref{def:admissible}),
	and so we must explain how to achieve these conditions 
	on deformations $C_{\UU}\ni u_i$ 
	and $C_{\VV}\ni v_j$
	of our original integration contour 
	in \eqref{spec_biorth}.
	First, observe that due to the restrictions \eqref{R_Q_conditions} 
	on 
	$\varphi(\UU)$, the sum over 
	$\la$ in \eqref{spec_biorth}
	is finite for fixed test functions (the argument is similar 
	to the part of Definition \ref{def:Plancherel_transforms}
	explaining why the support in $\la$ is finite). Thus,
	\eqref{spec_biorth}
	is an identity of rational functions
	in the parameters $q$, $\ipb$, $\SPB$,
	and it is enough
	to verify it on an open subset of the space of 
	parameters.\footnote{Note that a deformation of contours
	passing through a singularity of the integrand may change the rational function
	represented by the contour integral.
	Thus, verifying an identity of rational functions involving contour integrals
	on an open subset in the space of parameters
	allows to then analytically continue this identity 
	as long as the contour integrals represent the same rational functions.
	We will employ this understanding throughout the text.\label{footnote:rational_continuation}}

	We will deform the contours to achieve 
	the following inequalities
	which imply admissibility:
	\begin{align}\label{spec_bio_admiss}
		\left|
		\frac{u_i-\ip_x^{-1}\SP_x}{u_i-\ip_x^{-1}\SP_x^{-1}}\right|
		<r,\qquad
		\left|
		\frac{v_j-\ip_x^{-1}\SP_x^{-1}}{v_j-\ip_x^{-1}\SP_x}
		\right|
		<R,
		\qquad
		\textnormal{for some $R>1$, $0<r<1$, with $rR<1$, and all $x$}.
	\end{align}
	That is, the points $u_i$ have to be closer to 
	$\ip_x^{-1}\SP_x$ than to $\ip_x^{-1}\SP_x^{-1}$, and the opposite for 
	the $v_j$'s.
	Consider the discs
	\begin{align*}
		B_x^{(r)}:=\{z\in\C\colon|z-\ip_x^{-1}\SP_x|<
		r|z-\ip_x^{-1}\SP_x^{-1}|\},
	\end{align*}
	and note that each $u_i$ must be inside
	$\bigcap_{x\in\Z}B_x^{(r)}$, while each $v_j$
	must be outside
	$\bigcup_{x\in\Z}B_x^{(1/R)}$.
	Thus, for the deformed contours 
	$C_{\UU}$ and $C_{\VV}$ to exist, it must be that
	$\bigcap_{x\in\Z}B_x^{(r)}$ is nonempty and 
	contains all $\ip_x^{-1}\SP_x$,
	while
	$\bigcup_{x\in\Z}B_x^{(1/R)}$
	must not contain any of the points 
	$\ip_x^{-1}\SP_x^{-1}$. Moreover, the deformation of the 
	contour
	in \eqref{spec_biorth} to $C_\VV$ must not cross the possible singularities
	at $v_\aind=q^{\pm1}v_{\bind}$, cf. \eqref{R_Q_conditions}. 
	The latter conditions can be ensured by requiring that $q\in(0,\delta)$
	for a sufficiently small $\delta>0$.

	Assume that
	the other parameters are restricted for all $x$ as follows (with $a,b,c,d$ to be determined):
	\begin{align}\label{spec_bio_admiss2}
		\ip_x\in(a,b),\quad 
		0<a<b<\infty;
		\qquad \qquad
		\SP_x\in(-d,-c),\quad 0<c<d<1.
	\end{align}
	The diameter of each $B_x^{(r)}$ 
	(i.e., $B_x^{(r)}\cap \R$)
	is 
	\begin{align*}
		\bigg(\frac{r+\SP_x^2}{(r+1) \SP_x \ip_x},
		\frac{r-\SP_x^2}{(r-1) \SP_x \ip_x}\bigg),
	\end{align*}
	and similarly for $B_x^{(1/R)}$.
	Under \eqref{spec_bio_admiss2} we can estimate for $r>d^{2}$ and 
	$1<R<d^{-2}$:
	\begin{align*}
		\frac{1/R+d^{2}}{-ac(1+1/R)}<\frac{1/R+\SP_x^2}{(1/R+1) \SP_x \ip_x},\qquad
		\frac{r+\SP_x^2}{(r+1) \SP_x \ip_x}<
		\frac{r+c^{2}}{-bd(1+r)},\qquad
		\frac{r-d^2}{-b d (r-1)}<\frac{r-\SP_x^2}{(r-1) \SP_x \ip_x},
	\end{align*}
	and $-d/a<\ip_x^{-1}\SP_x<-c/b$, $\ip_x^{-1}\SP_x<-1/(bd)$.
	
	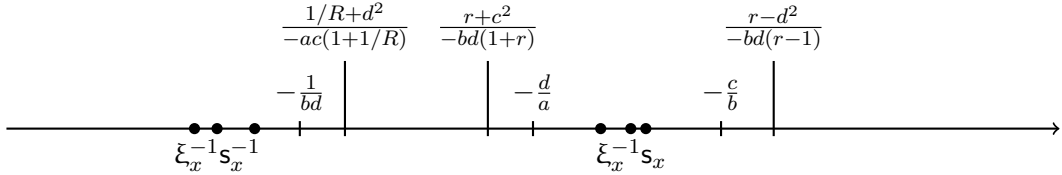
\begin{figure}[htbp]
		\begin{tikzpicture}
			[scale=1, thick]
			\def\pt{0.06}
			\draw[->] (-3,0)--++(14,0);
			\draw[fill] (4.9,0) circle (\pt);
			\draw[fill] (5.3,0) circle (\pt) node[below] {$\ip_x^{-1}\SP_x$};
			\draw[fill] (5.5,0) circle (\pt);
			\draw[fill] (.3,0) circle (\pt);
			\draw[fill] (-.2,0) circle (\pt) node[below] {$\ip_x^{-1}\SP_x^{-1}$};
			\draw[fill] (-0.5,0) circle (\pt);
			\draw (6.5,-.1)--++(0,.2) node[above] {$-\frac cb$};
			\draw (4,-.1)--++(0,.2) node[above] {$-\frac da$};
			\draw (.9,-.1)--++(0,.2) node[above] {$-\frac 1{bd}$};
			\draw (1.5,-.1)--++(0,1) node[above] {$\frac{1/R+d^{2}}{-ac(1+1/R)}$};
			\draw (3.4,-.1)--++(0,1) node[above] {$\frac{r+c^{2}}{-bd(1+r)}$};
			\draw (7.2,-.1)--++(0,1) node[above] {$\frac{r-d^2}{-b d (r-1)}$};
		\end{tikzpicture}
		\caption{Inequalities \eqref{intersections_of_discs_ineq} guarantee 
		this configuration of points, and thus
		the existence of the
		integration contours $C_{\UU}$ and $C_{\VV}$.} 
		\label{fig:intersections_of_discs}
	\end{figure}

	The existence of the contours $C_{\UU}$ and $C_{\VV}$ is thus implied by the following inequalities
	(see Fig.~\ref{fig:intersections_of_discs}):
	\begin{align}\label{intersections_of_discs_ineq}
		-\frac{1}{bd}<\frac{1/R+d^{2}}{-ac(1+1/R)}<
		\frac{r+c^{2}}{-bd(1+r)}<-\frac da<
		-\frac cb<\frac{r-d^2}{-b d (r-1)}.
	\end{align}
	One can check that these inequalities hold for, e.g., 
	\begin{align}\label{spec_bio_admiss3}
		a=\tfrac56,
		\quad b=1,
		\quad c=\tfrac16,
		\quad d=\tfrac14,
		\quad 1<R<16,
		\quad \tfrac1{16}<r<R^{-1}.
	\end{align}
	Thus, for sufficiently small $q$ and for other parameters satisfying
	\eqref{spec_bio_admiss2} and \eqref{spec_bio_admiss3},
	there exist deformations of contours in \eqref{spec_biorth} 
	to $C_{\UU}$ and $C_{\VV}$ not changing the integral, such 
	that on the deformed contours 
	one has $\adm{u_i}{v_j^{-1}}$, and so 
	\eqref{Cauchy_identity_spectral_biorth_proof0} holds.
		
	Let us now rewrite the left-hand side of \eqref{Cauchy_identity_spectral_biorth_proof0} 
	using \eqref{G_via_F_shifts} as follows:
	\begin{multline}\label{Cauchy_identity_spectral_biorth_proof}
		\frac{\prod_{i=1}^{n}(1-\SP_{-M}\ip_{-M}u_i)}{(q;q)_{n}}
		\sum_{{\la\in\signp n\colon\la_n=0}}
		\F_{\la}(\UU\md\sh_{-M}\ipb,\sh_{-M}\SPB)
		\G^{\conj}_{\la}(\bar\VV\md\sh_{-M}\ipbb,\sh_{-M}\SPB)
		\\+
		\prod_{i=1}^{n}\frac{1-\SP_{-M}\ip_{-M}u_i}{1-\SP_{-M}\ip_{-M}v_i}
		\sum_{{\la\in\signp n\colon\la_n\ge1}}
		\F_{\la}(\UU\md\sh_{-M}\ipb,\sh_{-M}\SPB)
		\F_\la^{\conj}(\bar\VV\md\sh_{-M}\ipbb,\sh_{-M}\SPB).
	\end{multline}
	Multiply \eqref{Cauchy_identity_spectral_biorth_proof} 
	by 
	\begin{align}
		\prod_{i=1}^{n}
		\prod_{j=-M}^{-1}
		\frac{1-\ip_j\SP_ju_i}{\ip_ju_i-\SP_j}
		\frac{\ip_jv_i-\SP_j}{1-\ip_j\SP_jv_i}
		\prod_{1\le\aind<\bind\le n}
		(u_{\aind}-u_{\bind})
		(v_{\aind}-v_{\bind})
		\cdot \varphi(u_1,\ldots,u_n)\psi(v_1,\ldots,v_n),
		\label{Cauchy_identity_spectral_biorth_proof1} 
	\end{align}
	where $\varphi$ and $\psi$ are test functions as in
	\eqref{R_Q_conditions}.
	We will integrate \eqref{Cauchy_identity_spectral_biorth_proof}
	multiplied by \eqref{Cauchy_identity_spectral_biorth_proof1}
	over the contours $C_{\UU}$ and $C_{\VV}$ described above, and observe the following:
\smallskip

	{\bf1.} The first summand in \eqref{Cauchy_identity_spectral_biorth_proof} multiplied
		by \eqref{Cauchy_identity_spectral_biorth_proof1} vanishes after the integration for large enough $M$. 
		Indeed,  
		each summand coming from 
		$\prod_{i=1}^{n}(1-\SP_{-M}\ip_{-M}u_i)
		\F_{\la}(\UU\md\sh_{-M}\ipb,\sh_{-M}\SPB)$ with $\la_n=0$
		is regular outside $C_{\UU}$, because there are no poles at $\ip_x^{-1}\SP_x^{-1}$ for $x\in \Z$.
		\smallskip

	{\bf2.} In the second summand we write, 
		using \eqref{F_shifts} and shifting the summation index:
		\begin{multline*}
			\prod_{i=1}^{n}
			\prod_{j=-M}^{-1}
			\frac{1-\ip_j\SP_ju_i}{\ip_ju_i-\SP_j}
			\frac{\ip_jv_i-\SP_j}{1-\ip_j\SP_jv_i}
			\sum_{{\la\in\signp n\colon \la_n\ge1}}
			\F_{\la}(\UU\md\sh_{-M}\ipb,\sh_{-M}\SPB)
			\F_\la^{\conj}(\bar\VV\md\sh_{-M}\ipbb,\sh_{-M}\SPB)\\=
			\sum_{{\la\in\sign n\colon \la_n\ge-M+1}}
			\F_{\la}(\UU\md\ipb,\SPB)
			\F_\la^{\conj}(\bar\VV\md\ipbb,\SPB).
		\end{multline*}
		Therefore, as $M\to+\infty$, we obtain the sum over all $\la\in\sign n$. 
		Since with our test functions the above sum over $\la$
		is actually finite, this limit procedure is a stabilization.
	\smallskip

	{\bf3.} Consider the integral 
		of the right-hand side of \eqref{Cauchy_identity_spectral_biorth_proof0}
		multiplied by \eqref{Cauchy_identity_spectral_biorth_proof1},
		and compute it by evaluating the residues 
		in the  $\VV$-integration variables. 
		Because $C_{\UU}$ is inside $C_{\VV}$,
		for large enough $M$ the integrand has no $v_i$-poles inside $C_{\VV}$
		except $v_i=u_{\sigma(i)}$ for some permutation $\sigma\in\Sym_{n}$.
		Indeed, the same $u_j$ cannot be utilized twice because of the prefactor 
		$\prod_{\aind<\bind}
		(v_{\aind}-v_{\bind})$. The sum over all $\sigma\in\Sym_n$
		yields the desired right-hand side of \eqref{spec_biorth}.
		\smallskip

		We have thus established \eqref{spec_biorth} for small $q$ and 
		other parameters $\ipb$ and $\SPB$ satisfying \eqref{spec_bio_admiss2} and \eqref{spec_bio_admiss3},
		and with integration contours $C_{\UU}$ and $C_{\VV}$. However, since the 
		sum in the left-hand side of \eqref{spec_biorth}
		is finite, we can deform the contours back to one and the same contour
		as described in the claim. Next, since both sides of 
		\eqref{spec_biorth} are rational functions 
		in $q$, $\ipb$, and $\SPB$, we can analytically continue this identity
		to the full range of parameters.
		This completes the proof.
\end{proof}

\begin{proof}[Proof of Theorem \ref{thm:Plancherel}]
	Let us show how the second statement
	follows from the spectral biorthogonality of Theorem \ref{thm:spec_orthogonality}.
	To prove $\Pltrans n (\Pltransi n R)=R$, rewrite the integration in 
	$\Pltransi n R$ using the contours $\contn{\ipbb\SPB;0}$
	(cf.~Remark \ref{rmk:bilinear_pairing_Tinv}). Thus, we must show 
	that
	\begin{align*}
		R(\UU)=
		\sum_{\la\in\sign n}
		\F_\la(\UU\md\ipb,\SPB)
		\frac{1}{(1-q)^{n}n!}
		\mathop{\oint \ldots \oint}\limits_{(\contn {\ipbb\SPB;0})^{n}}
		\frac{d\VV}{2\pi\i}
		\prod_{i=1}^{n}v_i^{-1}
		\prod_{1\le \aind\ne\bind\le n}\frac{v_\aind-v_\bind}{v_\aind-qv_\bind}
		R(\VV)
		\F^{\conj}_{\la}(\bar\VV\md\ipbb,\SPB).
	\end{align*}
	
	It suffices to establish the following integrated version of the above
	identity (we have interchanged summation and integration in $\UU$
	because of the finitely many nonzero terms in the sum):
	\begin{multline*}
		\mathop{\oint \ldots \oint}\limits_{(\contn {\ipbb\SPB;0})^{n}}
		R(\UU)Q(\UU)
		\frac{d\UU}{(2\pi\i)^{n}}
		=
		\sum_{\la\in\sign n}
		\frac{1}{n!}
		\mathop{\oint \ldots \oint}\limits_{(\contn {\ipbb\SPB;0})^{n}}
		\frac{d\UU}{(2\pi\i)^{n}}
		\mathop{\oint \ldots \oint}\limits_{(\contn {\ipbb\SPB;0})^{n}}
		\frac{d\VV}{(2\pi\i)^{n}}
		\\\times
		\prod_{1\le \aind\ne\bind\le n}(v_\aind-v_\bind)
		\prod_{1\le \aind,\bind\le n}\frac1{v_\aind-qv_\bind}
		\F_\la(\UU\md\ipb,\SPB)
		\F^{\conj}_{\la}(\bar\VV\md\ipbb,\SPB)
		R(\VV)Q(\UU).
	\end{multline*}
	Indeed, 
	consider the partial fraction expansion $R(\UU)=\sum_{j,p} r_{j,p}
	(u_j-p)^{-1}$, where $j=1,\ldots,n$, and $p$ runs over a subset
	of possible poles described in Definition \ref{def:function_spaces}.
	To extract a single $r_{j,p}$, choose $Q(\UU)=
	\prod_{\tilde p}(u_j-\tilde p)
	\prod_{i\ne j}(u_i-\ip_x^{-1}\SP_x)^{-1}$, where $\tilde p$
	runs over all other $u_j$-poles of $R(\UU)$,
	and $x\in\Z_{\ge0}$ is arbitrary. The factors 
	$(u_i-\ip_x^{-1}\SP_x)^{-1}$ are needed to ensure that the integrals over
	all other $u_i$'s are nontrivial.
	Thus, it suffices to let $Q(\UU)$ be an arbitrary (not necessarily symmetric) rational function
	with possibly simple poles at $\ip_x^{-1}\SP_x$, $x\ge0$.

	Apply the spectral biorthogonality
	\eqref{spec_biorth} with functions 
	(which clearly satisfy \eqref{R_Q_conditions}):
	\begin{align*}
		\varphi(\UU)=Q(\UU)
		\prod_{1\le \aind <\bind\le n}\frac1{u_\aind-u_\bind}
		\qquad\text{and}\qquad
		\psi(\VV)=R(\VV)\prod_{1\le \bind <\aind\le n}(v_\aind-v_\bind)
		\prod_{1\le \aind,\bind\le n}\frac1{v_\aind-qv_\bind}
	\end{align*}
	to rewrite the above sum as
	\begin{align*}
		\frac{1}{n!}
		\mathop{\oint \ldots \oint}\limits_{(\contn {\ipbb\SPB;0})^{n}}
		\frac{d\UU}{(2\pi\i)^{n}}
		Q(\UU)\sum_{\sigma\in \Sym_k}\sgn(\sigma)R(\sigma\UU)
		\prod_{1\le \aind <\bind\le n}\frac{
		u_{\sigma(\aind)}-u_{\sigma(\bind)}}
		{u_\aind-u_\bind}.
	\end{align*}
	The cancellation in the last product yields another $\sgn(\sigma)$,
	and so we see that the desired identity is established.
	This concludes the proof of the theorem.
\end{proof}

We conclude this subsection with a number of remarks.

\begin{remark}[Completeness of the Bethe ansatz]
	Plancherel isomorphism results (Theorem \ref{thm:Plancherel})
	imply that the (coordinate) Bethe ansatz yielding the 
	eigenfunctions $\F_\la$ of the transfer matrices is \emph{complete}.
	That is, any function $f\in\funspat n$ 
	can be mapped into the spectral space, and then reconstructed back from 
	its image. One of the ways 
	to write down this completeness statement 
	(using the orthogonality relation \eqref{orthogonality_formulation_cor})
	is the following:
	\begin{multline*}
		f(\la)=
		\frac{1}{(1-q)^{n}n!}
		\oint\limits_{\contn {\ipbb\SPB;0}}\frac{d u_1}{2\pi\i u_1}
		\ldots
		\oint\limits_{\contn {\ipbb\SPB;0}}\frac{d u_n}{2\pi\i u_n}
		\prod_{1\le \aind\ne\bind\le n}\frac{u_\aind-u_\bind}{u_\aind-qu_\bind}
		\\\times(\Pltrans n f)(u_1,\ldots,u_n)
		\F^{\conj}_{\la}(u_1^{-1},\ldots,u_n^{-1}\md\ipbb,\SPB).
	\end{multline*}
\end{remark}

\begin{remark}[Spectral decomposition of $\Qe_{\PI m;v}$]
	\label{rmk:spec_decomp}
	Similarly, \eqref{orthogonality_formulation_cor}
	implies a spectral decomposition
	of the operator $\Qe_{\PI m;v}$ acting on functions 
	on $\signp m$, cf. Remark \ref{rmk:EF_relation_Qe}:
	\begin{multline}\label{spec_decomp}
		\Qe_{\PI m;v}(\mu\to\nu)=
		\frac{1}{(1-q)^{m}m!}\frac{(-\SPB)^{\nu}}{(-\SPB)^{\mu}}
		\oint\limits_{\contn {\ipbb\SPB;0}}\frac{d z_1}{2\pi\i z_1}
		\ldots
		\oint\limits_{\contn {\ipbb\SPB;0}}\frac{d z_m}{2\pi\i z_m}
		\prod_{1\le \aind\ne \bind\le m}
		\frac{z_\aind-z_\bind}{z_\aind-qz_\bind}
		\\\times
		\bigg(\prod_{i=1}^{m}\frac{1-qz_iv}{1-z_iv}\bigg)
		\F_{\mu}(z_1,\ldots,z_m\md\ipb,\SPB)
		\F_{\nu}^{\conj}(z_1^{-1},\ldots,z_m^{-1}\md\ipbb,\SPB).
	\end{multline}
	Indeed, by \eqref{EF_relation_Qe} this operator has eigenfunctions
	$\EF{}\la(z_1,\ldots,z_m)=\F_\la(z_1,\ldots,z_m\md\ipb,\SPB)/(-\SPB)^{\la}$
	with eigenvalues 
	$\displaystyle\prod\nolimits_{i=1}^{m}\frac{1-qz_iv}{1-z_iv}$
	(the constant $(q;q)_{m}$ can be ignored).
	Thus, 
	\eqref{spec_decomp} follows
	by multiplying the eigenrelation \eqref{EF_relation_Qe} by 
	$(-\SPB)^{\nu}\F^{\conj}_\nu(z_1^{-1},\ldots,z_m^{-1}\md\ipbb,\SPB)$
	and integrating as in \eqref{orthogonality_formulation_cor}. 
	Since the identity \eqref{EF_relation_Qe} 
	requires the admissibility 
	$\adm{z_i}v$
	(Definition \ref{def:admissible}) before the integration, 
	in \eqref{spec_decomp}
	the point $v^{-1}$ should be outside the integration
	contour ${\contn {\ipbb\SPB;0}}$ (the argument
	for this is similar to the proof of
	Proposition \ref{prop:skew_G_integral_formula} below).
\end{remark}

\begin{remark}[Extensions]
	\label{rmk:optimal_phi_psi}
	Function spaces 
	$\funspat n$ and $\funspec n$,
	as well as test functions
	in Theorem \ref{thm:spec_orthogonality},
	are far from being optimal.
	This is because we only address
	algebraic aspects of Plancherel isomorphisms.
		
	The concrete form of restrictions on the functions $\phi(\UU)$ and $\psi(\VV)$
	in Theorem \ref{thm:spec_orthogonality}
	is motivated by the application of this theorem
	in the proof of Theorem \ref{thm:Plancherel}. 
	However, as can be seen from the proof, these restrictions
	can be relaxed. 
	For example, for
	$\psi(\VV)$ it is enough that
	there
	exists an open subset $\Omega$ in the space of parameters $q$, $\ipb$, and $\SPB$, 
	such that for the parameters in $\Omega$ the
	function $\psi(\VV)$ is holomorphic in the interior of the deformed contour $C_{\VV}$
	(constructed in the proof of Theorem \ref{thm:spec_orthogonality})
	minus the points $\ip_x^{-1}\SP_x$ where $\psi(\VV)$ can have at most finitely many simple poles.

	An extension of the first Plancherel isomorphism 
	to larger spaces is described in \cite[Appendix~A]{CorwinPetrov2015}
	in the homogeneous case; the inhomogeneous situation is completely analogous.
\end{remark}


\subsection{An integral representation for $\G_\mu$} 
\label{sub:integral_formula_for_g_mu_}

Using the orthogonality result of Theorem \ref{thm:orthogonality_F} and the Cauchy identity, 
we can obtain relatively simple nested contour integral formulas for the 
skew functions $\G_{\mu/\kappa}$
(and, in particular, for $\G_{\mu}$), which will
be useful later in \S \ref{sec:observables_of_interacting_particle_systems}
and \S \ref{sec:_q_moments_of_the_height_function_of_interacting_particle_systems}.

\begin{proposition}\label{prop:skew_G_integral_formula}
	Assume that 
	the parameters
	$q$, $\SPB$, and $\ipb$
	satisfy
	\eqref{stochastic_weights_condition_qsxi} and \eqref{assumptions_one_better}.
	For any $k,N\in\Z_{\ge1}$,
	$\mu,\kappa\in\signp k$, and $v_1,\ldots,v_N$ such that the $v_i^{-1}$'s are outside 
	of all the integration
	contours $\contq {\ipbb\SPB}j$ of Definition \ref{def:orthogonality_contours}, we have 
	\begin{multline}\label{skew_G_integral_formula}
		\G_{\mu/\kappa}(v_1,\ldots,v_N\md\ipbb,\SPB)=
		\frac{1}{(1-q)^{k}}
		\oint\limits_{\contq {\ipbb\SPB}1}\frac{d u_1}{2\pi\i}
		\ldots
		\oint\limits_{\contq {\ipbb\SPB}k}\frac{d u_k}{2\pi\i}
		\prod_{1\le \aind<\bind\le k}\frac{u_\aind-u_\bind}{u_\aind-qu_\bind}
		\\\times\F_{\kappa}^{\conj}(u_1,\ldots,u_k\md\ipb,\SPB)
		\prod_{i=1}^{k}u_i^{-1}\pow_{\mu_{i}}(u_{i}^{-1}\md\ipbb,\SPB)
		\prod_{\substack{1\le i\le k\\1\le j\le N}}
		\frac{1-qu_iv_j}{1-u_iv_j}.
	\end{multline}
\end{proposition}
When $\kappa=(0^{k})$, 
with the 
help of \eqref{F_at_zero_signature},
formula \eqref{skew_G_integral_formula} reduces to
\begin{corollary}\label{cor:G_integral_formula}
	Under the same assumptions as in 
	Proposition \ref{prop:skew_G_integral_formula} above, we have
	\begin{multline}\label{G_integral_formula}
		\G_\mu(v_1,\ldots,v_N\md\ipbb,\SPB)=
		\frac{(\SP_0^{2};q)_{k}}{(1-q)^{k}}
		\oint\limits_{\contq {\ipbb\SPB}1}\frac{d u_1}{2\pi\i}
		\ldots
		\oint\limits_{\contq {\ipbb\SPB}k}\frac{d u_k}{2\pi\i}
		\prod_{1\le \aind<\bind\le k}\frac{u_\aind-u_\bind}{u_\aind-qu_\bind}\\\times
		\prod_{i=1}^{k}\frac{u_i^{-1}\pow_{\mu_{i}}(u_{i}^{-1}\md\ipbb,\SPB)}{1-\SP_0\ip_0u_i}
		\prod_{\substack{1\le i\le k\\1\le j\le N}}
		\frac{1-qu_iv_j}{1-u_iv_j}.
	\end{multline}
\end{corollary}
When  the parameters $\ip_j$ and $\SP_j$ are $j$-independent, 
formula \eqref{G_integral_formula} 
appeared in \cite[Prop. 7.3]{Borodin2014vertex}.
\begin{proof}[Proof of Proposition \ref{prop:skew_G_integral_formula}]
	Fix $\mu,\kappa\in\signp k$, 
	multiply both sides of \eqref{orthogonality_formulation}
	by $\G_{\la/\kappa}(v_1,\ldots,v_N\md\ipbb,\SPB)$, and sum over $\la\in\signp k$.
	The right-hand side obviously equals $\G_{\mu/\kappa}(v_1,\ldots,v_N\md\ipbb,\SPB)$, 
	while in the left-hand side we 
	have
	\begin{multline*}
		(1-q)^{-k}\sum_{\la\in\signp k}\,
		\oint\limits_{\contq {\ipbb\SPB}1}\frac{d u_1}{2\pi\i}
		\ldots
		\oint\limits_{\contq {\ipbb\SPB}k}\frac{d u_k}{2\pi\i}
		\prod_{1\le \aind<\bind\le k}\frac{u_\aind-u_\bind}{u_\aind-qu_\bind}
		\\\times\prod_{i=1}^{k}\frac{\pow_{\mu_{i}}(u_{i}^{-1}\md\ipbb,\SPB)}{u_i}
		\F_{\la}^{\conj}(u_1,\ldots,u_k\md\ipb,\SPB)
		\G_{\la/\kappa}(v_1,\ldots,v_N\md\ipbb,\SPB).
	\end{multline*}
	If one can perform the (infinite) 
	summation over $\la$
	inside the integral,
	then by the (iterated)
	Corollary \ref{cor:Pieri}.1 (which follows from the Cauchy identity),
	one readily gets the desired formula for the symmetric function
	$\G_{\mu/\kappa}(v_1,\ldots,v_N\md\ipbb,\SPB)$. It remains to justify 
	that we indeed can interchange summation and integration.

	The (absolutely convergent) summation can be performed inside the integral
	if $\adm{u_i}{v_j}$ for $u_i$ on the contours. The
	admissibility
	follows if
	\begin{align*}
		\left|\frac{\ip_x u_i-\SP_x}{1-\SP_x\ip_x u_i}
		\cdot
		\frac{\ip_x^{-1}v_j-\SP_x}{1-\SP_x\ip_x^{-1}v_j}\right|=
		\left|
		\frac{u_i-\ip_x^{-1}\SP_x}{u_i-\ip_x^{-1}\SP_x^{-1}}
		\cdot
		\frac{v_j^{-1}-\ip_x^{-1}\SP_x^{-1}}{v_j^{-1}-\ip_x^{-1}\SP_x}
		\right|
		<1-\epsilon
		\qquad
		\textnormal{for some $\epsilon>0$ and all $x$}.
	\end{align*}
	Therefore, it would be sufficient if
	$u_i$ is closer to $\ip_x^{-1}\SP_x$
	than to $\ip_x^{-1}\SP_x^{-1}$ for all $x$, 
	and, on the other hand, $v_j^{-1}$
	is closer to 
	$\ip_x^{-1}\SP_x^{-1}$ than to $\ip_x^{-1}\SP_x$.
	Since the $u$-contours encircle the points
	$\ip_x^{-1}\SP_x$, we can readily 
	achieve the above inequality in the case when 
	for each $x$, the midpoint 
	$\frac{1}{2}\ip_x^{-1}(\SP_x+\SP_x^{-1})$
	of
	$\ip_x^{-1}\SP_x^{\pm1}$ lies to the left 
	of the leftmost point $m_{\ipbb\SPB}$
	of $\ip_x^{-1}\SP_x$.
	A sufficient condition for that is
	\begin{align}\label{skew_G_integral_formula_proof1}
		\tfrac12\big(m_{\ipbb|\SPB|}+m_{\ipbb|\SPBB|}\big)>M_{\ipbb|\SPB|},
	\end{align}
	which is 
	more restrictive than the second condition
	in \eqref{assumptions_one_better} because clearly
	$m_{\ipbb|\SPB|}<m_{\ipbb|\SPBB|}$.

	If this more restrictive condition holds, we can slightly deform the 
	contours $\contq {\ipbb\SPB}i$ if needed, and choose the $v_j^{-1}$'s
	outside these contours with 
	real part being
	negative and
	sufficiently large in absolute value. This ensures the admissibility, 
	and so we can perform the summation inside the integral
	and establish the desired identity \eqref{skew_G_integral_formula}.
	
	Let $\Omega$ 
	be the set of parameters $(\{v_{j}\},\ipb,\SPB)$
	such that 
	\eqref{skew_G_integral_formula_proof1}
	holds and 
	the real parts of the $v_j^{-1}$'s are sufficiently negative. 
	This set is clearly nonempty and open, 
	and, moreover, 
	for fixed $\mu$ and $\kappa$
	both sides of \eqref{skew_G_integral_formula}
	are represented by rational functions which
	depend only on finitely many of the $\ip_x$'s and $\SP_x$'s.
	Thus, we can employ
	analytic continuation of rational functions 
	(cf. footnote$^{\ref{footnote:rational_continuation}}$) 
	and continue identity \eqref{skew_G_integral_formula}
	from $\Omega$ to a larger set of variables and parameters as in the claim of the 
	present proposition.
	Indeed, the restrictions on parameters in the claim of the proposition
	ensure that the integration contours
	$\contq {\ipbb\SPB}i$ exist, and also that the points $v_j^{-1}$ are outside these contours.
	Thus, for these parameters the right-hand side of \eqref{skew_G_integral_formula} represents
	the same rational function as on $\Omega$.
\end{proof}


\subsection{Another proof of symmetrization formula for $\G_\mu$} 
\label{sub:another_way_to_get_symmetrization_formula_for_g_mu_}

The nested contour integral formula for $\G_\mu$
of Corollary \ref{cor:G_integral_formula}
may be used as an alternative way to derive the 
symmetrization formula 
for $\G_\mu$
of Theorem \ref{thm:symmetrization}.
Note that 
Corollary \ref{cor:G_integral_formula}
in turn follows from the Cauchy identity
plus the
spatial orthogonality of the $\F_\la$'s,
and the latter is implied by the 
symmetrization formula for $\F_\la$.

\begin{remark}
	Another use of formula \eqref{G_integral_formula}
	is a straightforward alternative proof of
	Proposition \ref{prop:RHO_spec}
	(computation of the specialization
	$\G_\mu(\RHO\md\ipbb,\SPB)$),
	which can also be generalized to 
	other specializations of $\G_\mu$.
	This will be a starting point for averaging of
	observables in \S \ref{sub:computation_of_gnurhow}
	below.
\end{remark}

To get the symmetrization formula for $\G_\mu$,
we follow the approach of \cite[Prop. 7.3]{Borodin2014vertex}
and explicitly compute the integral in 
the right-hand side of \eqref{G_integral_formula}.
To ensure that this formula holds, we
assume \eqref{stochastic_weights_condition_qsxi} 
and \eqref{assumptions_one_better}, 
and that $v_1,\ldots,v_N\in\C$ are pairwise distinct
and are such that the points $v_i^{-1}$ lie 
outside the integration contours $\contq {\ipbb\SPB}j$
of Definition \ref{def:orthogonality_contours}.
 Observe the following properties of the integrand in 
\eqref{G_integral_formula}:
\begin{itemize}
	\item The integrand is regular at $u_j=\infty$ and $u_j=0$.
	\item If $\mu_i>0$, then the integrand is regular
	at $u_i=\ip_j^{-1}\SP_j^{-1}$ for all $j\ge0$ because then
	\begin{align*}
		\frac{\pow_{\mu_{i}}(u_{i}^{-1}\md\ipbb,\SPB)}{1-\SP_0\ip_0u_i}
		=
		\frac{1-q}{1-\SP_{\mu_i}u_i^{-1}/\ip_{\mu_i}}
		\frac{1}{\ip_0u_i-\SP_0}
		\prod_{j=1}^{\mu_i-1}\frac{u_i^{-1}/\ip_j-\SP_j}{1-\SP_ju_i^{-1}/\ip_j}.
	\end{align*}
	\item If $\mu_i=0$, then the integrand is regular at $u_i=\ip_j^{-1}\SP_j^{-1}$
	for all $j\ge1$, because then
	\begin{align*}
		\frac{\pow_{\mu_{i}}(u_{i}^{-1}\md\ipbb,\SPB)}{1-\SP_0\ip_0u_i}
		=\frac{1-q}{(1-\SP_0u_i^{-1}/\ip_0)(1-\SP_0\ip_0u_i)}.
	\end{align*}
\end{itemize}
We will now evaluate the integral in \eqref{G_integral_formula} as follows.
Assume that $\mu\in\signp k$ has exactly $\ell$ zeros:
$\mu_k=\mu_{k-1}=\ldots=\mu_{k-\ell+1}=0$, $\mu_{k-\ell}\ge1$. 
This allows to take the residues at
$u_k=s/a_0$, $u_{k-1}=qs/a_0$, \ldots, $u_{k-\ell+1}=q^{\ell-1}s/a_0$,
(in this order), because for each variable $u_k,\ldots,u_{k-\ell+1}$
the corresponding point is the only pole inside the integration contour.
Rewriting 
\eqref{G_integral_formula} as 
\begin{multline*}
	\G_\mu(v_1,\ldots,v_N\md\ipbb,\SPB)=
	\frac{(\SP_0^{2};q)_{k}}{(1-q)^{k-\ell}}
	\oint\limits_{\contq {\ipbb\SPB}1}\frac{d u_1}{2\pi\i}
	\ldots
	\oint\limits_{\contq {\ipbb\SPB}k}\frac{d u_k}{2\pi\i}
	\prod_{1\le \aind<\bind\le k}\frac{u_\aind-u_\bind}{u_\aind-qu_\bind}
	\prod_{\substack{1\le i\le k\\1\le j\le N}}
	\frac{1-qu_iv_j}{1-u_iv_j}\\\times
	\prod_{i=1}^{k-\ell}
	\frac{u_i^{-1}\pow_{\mu_{i}}(u_{i}^{-1}\md\ipbb,\SPB)}{1-\SP_0\ip_0u_i}
	\prod_{i=k-\ell+1}^{k}
	\frac{1}{{(u_i-\SP_0/\ip_0)(1-\SP_0\ip_0u_i)}},
\end{multline*}
we consecutively obtain:
\begin{align*}
	&\Res_{u_k=\SP_0/\ip_0}
	\bigg(
	\frac{1}{(u_k-\SP_0/\ip_0)(1-\SP_0\ip_0u_k)}\prod_{i=1}^{k-1}
	\frac{u_i-u_k}{u_i-qu_k}
	\prod_{j=1}^{N}\frac{1-qu_kv_j}{1-u_kv_j}
	\bigg)\\&\hspace{180pt}=\frac{1}{1-\SP_0^{2}}
	\prod_{i=1}^{k-1}
	\frac{u_i-\SP_0/\ip_0}{u_i-q\SP_0/\ip_0}
	\prod_{j=1}^{N}\frac{1-q\SP_0v_j/\ip_0}{1-\SP_0v_j/\ip_0},
	\\&
	\Res_{u_{k-1}=q\SP_0/\ip_0}
	\bigg(
	\frac{1}{(u_{k-1}-q\SP_0/\ip_0)(1-\SP_0\ip_0u_{k-1})}\prod_{i=1}^{k-2}
	\frac{u_i-u_{k-1}}{u_i-qu_{k-1}}
	\prod_{j=1}^{N}\frac{1-qu_{k-1}v_j}{1-u_{k-1}v_j}
	\bigg)\\&\hspace{180pt}=\frac{1}{1-q\SP_0^{2}}
	\prod_{i=1}^{k-2}
	\frac{u_i-q\SP_0/\ip_0}{u_i-q^{2}\SP_0/\ip_0}
	\prod_{j=1}^{N}\frac{1-q^{2}\SP_0v_j/\ip_0}{1-q\SP_0v_j/\ip_0},
	\qquad
	\qquad
	\textnormal{etc.,}
\end{align*}
and thus the integral for 
$\G_\mu(v_1,\ldots,v_N\md\ipbb,\SPB)$ takes the form
\begin{multline*}
	\G_\mu(v_1,\ldots,v_N\md\ipbb,\SPB)=
	\frac{(q^{\ell}\SP_0^{2};q)_{k-\ell}}{(1-q)^{k-\ell}}
	\prod_{j=1}^{N}
	\frac{1-q^{\ell}\SP_0v_j/\ip_0}{1-\SP_0v_j/\ip_0}
	\oint\limits_{\contq {\ipbb\SPB}1}\frac{d u_1}{2\pi\i}
	\ldots
	\oint\limits_{\contq {\ipbb\SPB}{k-\ell}}\frac{d u_{k-\ell}}{2\pi\i}
	\prod_{1\le \aind<\bind\le k-\ell}\frac{u_\aind-u_\bind}{u_\aind-qu_\bind}
	\\\times\prod_{\substack{1\le i\le k-\ell\\1\le j\le N}}
	\frac{1-qu_iv_j}{1-u_iv_j}	
	\prod_{i=1}^{k-\ell}
	\bigg(
	\frac{u_i-\SP_0/\ip_0}{u_i-q^{\ell}\SP_0/\ip_0}
	\frac{u_i^{-1}\pow_{\mu_{i}}(u_{i}^{-1}\md\ipbb,\SPB)}{1-\SP_0\ip_0u_i}
	\bigg).
\end{multline*}
Now the integral over $u_1,\ldots,u_{k-\ell}$
has no poles outside the integration contours 
except $u_i=v_j^{-1}$ for some $j$. Thus, the 
integral can be evaluated by taking minus residues 
at these points. The same $v_j^{-1}$ cannot be used twice
because of the product 
$\prod_{1\le \aind<\bind\le k-\ell}({u_\aind-u_\bind})$
in the numerator. Therefore, all possible ways to choose the residues
at $v_j^{-1}$ are encoded by injective maps $\sigma\colon\{1,\ldots,k-\ell\}
\to\{1,\ldots,N\}$,
and we would need to sum over them.
We thus have
\begin{multline*}
	\G_\mu(v_1,\ldots,v_N\md\ipbb,\SPB)=
	(q^{\ell}\SP_0^{2};q)_{k-\ell}
	\prod_{j=1}^{N}
	\frac{1-\SP_0q^{\ell}v_j/\ip_0}{1-\SP_0v_j/\ip_0}
	\sum_{\sigma\colon\{1,\ldots,k-\ell\}
	\to\{1,\ldots,N\}}
	\prod_{1\le \aind<\bind\le k-\ell}\frac{v_{\sigma(\bind)}-v_{\sigma(\aind)}}
	{v_{\sigma(\bind)}-qv_{\sigma(\aind)}}
	\\\times
	\prod_{\substack{1\le i\le k-\ell\\	1\le j\le N\\j\ne \sigma(i)}}
	\frac{v_{\sigma(i)}-qv_j}{v_{\sigma(i)}-v_j}	
	\prod_{i=1}^{k-\ell}
	\bigg(
	\frac{1-\SP_0v_{\sigma(i)}/\ip_0}{1-\SP_0q^{\ell}v_{\sigma(i)}/\ip_0}
	\frac{\pow_{\mu_{i}}(v_{\sigma(i)}\md\ipbb,\SPB)}{1-\SP_0\ip_0/v_{\sigma(i)}}
	\bigg).
\end{multline*}
Letting $\IS$ denote the range of the map $\sigma$, we can rewrite the above formula as
\begin{multline*}
	\G_\mu(v_1,\ldots,v_N\md\ipbb,\SPB)=
	(q^{\ell}\SP_0^{2};q)_{k-\ell}
	\sum_{\IS\subseteq\{1,\ldots,N\}}
	\prod_{j\notin \IS}
	\frac{1-\SP_0q^{\ell}v_j/\ip_0}{1-\SP_0v_j/\ip_0}
	\prod_{\substack{i\in \IS\\j\notin \IS}}
	\frac{v_i-qv_j}{v_i-v_j}
	\\\times
	\sum_{\substack{\sigma\colon \{1,\ldots,k-\ell\}\to \IS
	\\\textnormal{bijection}}}\;
	\prod_{1\le \aind<\bind\le k-\ell}
	\frac
	{v_{\sigma(\aind)}-qv_{\sigma(\bind)}}
	{v_{\sigma(\aind)}-v_{\sigma(\bind)}}
	\prod_{i=1}^{k-\ell}
	\frac{\pow_{\mu_{i}}(v_{\sigma(i)}\md\ipbb,\SPB)}{1-s\ip_0/v_{\sigma(i)}}.
\end{multline*}
If one symmetrizes
the above expression
over $\{v_j\}_{j\notin\IS}$,
then the result will match
formula \eqref{G_symm_formula}
for $\G_\mu(v_1,\ldots,v_N\md\ipbb,\SPB)$.

This completes the derivation of the expression 
for $\G_\mu(v_1,\ldots,v_N\md\ipbb,\SPB)$
in Theorem \ref{thm:symmetrization}
for all (generic) $v_1,\ldots,v_N$, $\ipb$, $\SPB$, and $q$,
because both sides of 
that formula
are a~priori rational functions in all these parameters.



\section{$q$-correlation functions} 
\label{sec:observables_of_interacting_particle_systems}

In this section we compute $q$-correlation 
functions of the stochastic dynamics $\Xp_{\{u_t\};\RHO}$
of \S \ref{sub:interacting_particle_systems}
assuming it starts from the empty initial configuration.

\subsection{Computing observables via the Cauchy identity} 
\label{sub:computing_observables_via_cauchy_identity}

Let us first briefly 
explain main ideas behind our computations.
We are interested only in single-time observables
(i.e., the ones which depend 
on the state of $\Xp_{\{u_t\};\RHO}$
at a single time moment, say,~$t=n$),
and getting them is equivalent to 
computing expectations 
$\E_{\UU;\RHO}f(\nu)$
of certain functions
$f(\nu)$ of the configuration $\nu\in\signp n$
with respect to the probability measure
\begin{align}\nonumber
	\MM_{\UU;\RHO}(\nu\md\ipb,\SPB)&=
	\frac{1}{Z(\UU;\RHO\md\ipb,\SPB)}\F_\nu(u_1,\ldots,u_n\md\ipb,\SPB)\G_\nu^{\conj}(\RHO\md\ipbb,\SPB)\\&=
	\mathbf{1}_{\nu_n\ge1}
	\cdot (-\sh_1\SPB)^{\nu-1^{n}}\cdot\F^{\conj}_{\nu-1^{n}}(u_1,\ldots,u_n
	\md\sh_1\ipb,\sh_1\SPB),
	\label{MM_U_RHO_particular}
\end{align}
where $\UU=(u_1,\ldots,u_n)$,
and we use notation \eqref{sh_operation} and \eqref{SPB_la_notation}.
The measure $\MM_{\UU;\VV}$ \eqref{MM_measure}
takes the above form for $\VV=\RHO$ due to \eqref{Qp_RHO}.

The weights \eqref{MM_U_RHO_particular} are nonnegative if the parameters
satisfy \eqref{stochastic_weights_condition_qsxi}--\eqref{stochastic_weights_condition_u}.
To ensure that \eqref{MM_U_RHO_particular} defines a probability distribution
on the infinite set $\signp n$, we need to impose 
admissibility conditions (cf.~Definition~\ref{def:MM}).
The latter are implied by
\begin{align}\label{admissibility_RHO_conditions}
	\left|\SP_i\frac{\ip_iu_j-\SP_i}{1-\SP_i\ip_iu_j}\right|<1- \epsilon
	\qquad
	\textnormal{for some $\epsilon>0$ and all $i\in\Z_{\ge0}$ and $j=1,\ldots,n$.}
\end{align}
Indeed, these conditions ensure \eqref{admissible_sufficient} 
for very small $v$ (limit $v\to0$ is a part of the 
specialization~$\RHO$). Alternatively, 
interpret the probability weight
$(-\sh_1\SPB)^{\nu-1^{n}}\cdot\F^{\conj}_{\nu-1^{n}}(u_1,\ldots,u_n
\md\sh_1\ipb,\sh_1\SPB)$ in
\eqref{MM_U_RHO_particular}
as a partition function of path collections,
and fix $\nu$ with large $\nu_1$ (other parts can be large, too).
The only vertex weight
which enters 
the weight of a particular path collection 
a large number of times is
$\Lmatr_{\ip_xu_j,\SP_x}(0,1;0,1)=(-\SP_x\ip_xu_j+\SP_x^{2})/(1-\SP_x\ip_xu_j)$,
which is bounded in absolute value by \eqref{admissibility_RHO_conditions}.
One readily sees that conditions 
\eqref{stochastic_weights_condition_qsxi}--\eqref{stochastic_weights_condition_u} 
(which, in particular, require $u_i\ge0$)
automatically imply
\eqref{admissibility_RHO_conditions}.

Cauchy identity \eqref{usual_Cauchy} suggests a large family of 
observables of the measure 
$\MM_{\UU;\RHO}$ whose averages can be computed right away. Namely, let 
us fix variables $w_1,\ldots,w_k$, and set
\begin{align}\label{fnu_GRHO_www}
	f(\nu):=
	\frac{\G_\nu(\RHO,w_1,\ldots,w_k\md\ipbb,\SPB)}{\G_\nu(\RHO\md\ipbb,\SPB)},
\end{align}
where $(\RHO,w_1,\ldots,w_k)$ means that we add
$w_1,\ldots,w_k$ to the specialization $(\epsilon,q \epsilon,\ldots,q^{J-1}\epsilon)$, 
then set $q^{J}=\ip_0/(\SP_0 \epsilon)$, and finally send $\epsilon\to0$,
cf. \eqref{RHO_spec}.
Note that one can replace both $\G_\nu$ in \eqref{fnu_GRHO_www} by $\G_{\nu}^{\conj}$
without changing $f(\nu)$. The $\E_{\UU;\RHO}$ expectation of the function \eqref{fnu_GRHO_www} takes the form
\begin{multline}\label{EGoverG_computation}
	\E_{\UU;\RHO}f(\nu)=\sum_{\nu\in\signp n}
	\frac{1}{Z(\UU;\RHO\md\ipb,\SPB)}\F_\nu(u_1,\ldots,u_n\md\ipb,\SPB)\G_\nu^{\conj}(\RHO\md\ipbb,\SPB)
	\frac{\G_\nu^{\conj}(\RHO,w_1,\ldots,w_k\md\ipbb,\SPB)}{\G_\nu^{\conj}(\RHO\md\ipbb,\SPB)}
	\\=
	\frac{Z(\UU;\RHO,w_1,\ldots,w_k\md\ipb,\SPB)}{Z(\UU;\RHO\md\ipb,\SPB)}
	=\prod_{\substack{1\le i\le n\\1\le j\le k}}\frac{1-qu_iw_j}{1-u_iw_j},
\end{multline}
where the ratio of the partition functions $Z(\cdots)$
is computed via the corresponding $\RHO$ limit
of \eqref{Z_formula}. We will discuss admissibility conditions
(necessary for the convergence of the above sum)
in \S\ref{sub:_q_correlation_functions} below.

One now needs to understand 
the dependence of \eqref{fnu_GRHO_www} on $\nu$.
Using the integral formula 
\eqref{G_integral_formula}
for $\G_\nu$, we can compute for $k=1$:
\begin{align}\label{GoverG_as_sum_k1}
	\frac{\G_\nu(\RHO,w\md\ipbb,\SPB)}{\G_\nu(\RHO\md\ipbb,\SPB)}
	=
	q^{n}+
	\sum_{i=1}^{n}\frac{q^{i-1}}{(-\SPB)^{\nu_i}}
	\frac{1-\SP_0\ip_0^{-1}w}{1-\SP_0^{-1}\ip_0^{-1}w}
	\pow_{\nu_{i}}(w\md\ipbb,\SPB),\qquad
	\nu_n\ge1
\end{align}
(here and below in this section 
we are using notation similar to \eqref{SPB_la_notation}, so
$(-\SPB)^{\nu_i}=\prod_{j=0}^{\nu_i-1}\SP_j$).
A general result of this sort is given in Proposition \ref{prop:Gnu_spec_RHO_w} below.

Next, by a suitable contour integration in $w$
one can 
extract the term in \eqref{GoverG_as_sum_k1}
with $\nu_i=m$
for any fixed $m\ge1$.
Therefore, the same integration 
of the right-hand side of \eqref{EGoverG_computation} 
will yield a contour integral formula for 
\begin{align*}
	\E_{\UU;\RHO}\sum_{i=1}^{n}q^{i}\mathbf{1}_{\nu_i=m},
\end{align*}
which can be viewed as a \emph{$q$-analogue of the density function}
of the random configuration $\nu$.
Higher \emph{$q$-correlation functions} (defined in \S \ref{sub:_q_correlation_functions} below)
can be computed in a similar way by working with \eqref{fnu_GRHO_www} with 
general $k$. Therefore, for general $k$
the right-hand side of \eqref{EGoverG_computation} 
should be regarded as a generating function (in $w_1,\ldots,w_k$) for 
the $q$-correlation functions,
and the latter can be extracted by integrating over the $w_j$'s.


\subsection{Computation of $\G_\nu(\RHO,w_1,\ldots,w_k\md\ipbb,\SPB)$} 
\label{sub:computation_of_gnurhow}

In this subsection we 
fix $n\ge k\ge0$ and a signature $\nu\in\signp n$,
and compute the specialization $\G_\nu(\RHO,w_1,\ldots,w_k\md\ipbb,\SPB)$.
The result of this computation is
a general $k$ version of \eqref{GoverG_as_sum_k1},
and it is given in Proposition 
\ref{prop:Gnu_spec_RHO_w} below.

For the computation we will assume that $w_p$, $p=1,\ldots,k$,
are pairwise distinct
and are such that the points $w_p^{-1}$
are outside the integration contours $\contq {\ipbb\SPB}j$
of Definition \ref{def:orthogonality_contours}. 
Assume in addition that the parameters 
$q$, $\SPB$, and $\ipb$
satisfy
\eqref{stochastic_weights_condition_qsxi} and 
\eqref{assumptions_one_better}, so 
the integral formula \eqref{G_integral_formula} holds.
Thus,
we can readily take the $\RHO$
limit \eqref{RHO_spec} in \eqref{G_integral_formula}, and 
write
\begin{multline}
	\G_\nu(\RHO,w_1,\ldots,w_k\md\ipbb,\SPB)=
	\frac{(\SP_0^{2};q)_{n}}{(1-q)^{n}}
	\oint\limits_{\contq {\ipbb\SPB}1}\frac{d u_1}{2\pi\i}
	\ldots
	\oint\limits_{\contq {\ipbb\SPB}n}\frac{d u_n}{2\pi\i}
	\prod_{1\le \aind<\bind\le n}\frac{u_\aind-u_\bind}{u_\aind-qu_\bind}
	\\\times
	\prod_{i=1}^{n}\bigg(\frac{u_i^{-1}(1-\SP_0^{-1}\ip_0u_i)\pow_{\nu_{i}}(u_{i}^{-1}\md\ipbb,\SPB)}{1-\SP_0\ip_0u_i}\prod_{j=1}^{k}\frac{1-qu_iw_j}{1-u_iw_j}\bigg).
	\label{G_nu_RHO_w_integral1}
\end{multline}
We have
\begin{align*}
	\frac{u_i^{-1}(1-\ip_0\SP_0^{-1}u_i)
	\pow_{\nu_{i}}(u_i^{-1}\md\ipbb,\SPB)}{1-\SP_0\ip_0u_i}=
	(-\SP_0)^{-1}
	\frac{1-q}{u_i-\SP_{\nu_i}\ip_{\nu_i}^{-1}}
	\frac{\ip_0u_i-\SP_0}{1-\SP_0\ip_0u_i}
	\prod_{j=0}^{\nu_i-1}
	\frac{1-\SP_j\ip_j u_i}{\ip_ju_i-\SP_j}.
\end{align*}
If $\nu_n=0$, then the integral \eqref{G_nu_RHO_w_integral1} vanishes because
there are no $u_n$-poles inside $\contq {\ipbb\SPB}n$. 
We will thus assume that $\nu_n\ge1$
(so all $\nu_i\ge1$),
and explicitly compute this integral.
Denote it by
$\II^{w_1,\ldots,w_k}_{1,\ldots,n}$.

We aim to peel off the contours $\contq {\ipbb\SPB}1,\ldots,\contq {\ipbb\SPB}n$ (in this order),
and take minus residues at poles outside these contours.
Observe that the integrand in $u_1$ has only two types of simple
poles outside $\contq {\ipbb\SPB}1$,
namely, $u_1=\infty$ and $u_1=w_p^{-1}$ for $p=1,\ldots,k$ (there are no singularities at $\SP_j^{-1}\ip_j^{-1}$).
We have 
\begin{align*}
	-\Res_{u_1=\infty}\bigg(
	(-\SP_0)^{-1}
	\frac{1-q}{u_1-\SP_{\nu_1}\ip_{\nu_1}^{-1}}
	\prod_{j=1}^{\nu_1-1}
	\frac{1-\SP_j\ip_j u_1}{\ip_ju_1-\SP_j}
	\bigg)=(1-q)(-\SP_0)^{-2}(-\SPB)^{\nu_1}.
\end{align*}
Thus, the whole minus residue of \eqref{G_nu_RHO_w_integral1} at $u_1=\infty$ is equal to 
$(1-\SP_0^2q^{n-1})(-\SP_0)^{-2}
(-\SPB)^{\nu_1}q^{k}\cdot
\II^{w_1,\ldots,w_k}_{2,\ldots,n}$, where $q^{k}$ comes from the product involving the $w_p$'s.

Next, for any $p=1,\ldots,k$, the pole $w_p^{-1}$ yields
\begin{multline*}
	-\Res_{u_1=w_p^{-1}}=
	(1-\SP_0^{2}q^{n-1})
	(-\SP_0)^{-2}
	\frac{1-\SP_0\ip_0^{-1}w_p}{1-\SP_0^{-1}\ip_0^{-1}w_p}
	\underbrace{\frac{1-q}{1-\SP_{\nu_1}\ip_{\nu_1}^{-1}w_p}
	\prod_{j=0}^{\nu_1-1}
	\frac{\ip_j^{-1}w_p-\SP_j}{1-\SP_j\ip_j^{-1}w_p}}_{\pow_{\nu_1}(w_p\md\ipbb,\SPB)}
	\prod_{\substack{1\le j\le k\\j\ne p}}\frac{w_p-qw_j}{w_p-w_j}
	\\\times
	\frac{(\SP_0^{2};q)_{n-1}}{(1-q)^{n-1}}\oint\limits_{\contq {\ipbb\SPB}2}\frac{d u_2}{2\pi\i}
	\ldots
	\oint\limits_{\contq {\ipbb\SPB}n}\frac{d u_n}{2\pi\i}
	\prod_{2\le \aind<\bind\le n}\frac{u_\aind-u_\bind}{u_\aind-qu_\bind}
	\\\times
	\prod_{i=2}^{n}\bigg((-\SP_0)^{-1}
	\frac{1-q}{u_i-\SP_{\nu_i}\ip_{\nu_i}^{-1}}
	\prod_{j=1}^{\nu_i-1}
	\frac{1-\SP_j\ip_j u_i}{\ip_ju_i-\SP_j}
	\prod_{\substack{1\le j\le k\\j\ne p}}\frac{1-qu_iw_j}{1-u_iw_j}\bigg).
\end{multline*}
Therefore, taking the minus residue at $u_1=w_p^{-1}$
leads to
\begin{align*}
	(1-\SP_0^{2}q^{n-1})
	(-\SP_0)^{-2}
	\frac{1-\SP_0\ip_0^{-1}w_p}{1-\SP_0^{-1}\ip_0^{-1}w_p}
	\pow_{\nu_1}(w_p\md\ipbb,\SPB)
	\prod_{\substack{1\le j\le k\\j\ne p}}\frac{w_p-qw_j}{w_p-w_j}
	\cdot \II^{w_1,\ldots,w_{p-1},w_{p+1},\ldots,w_k}_{2,\ldots,n}.
\end{align*}
One can continue with similar computations for $\contq {\ipbb\SPB}2,\ldots,\contq {\ipbb\SPB}n$.
Let us write down the final integral
with the only remaining contour $\contq {\ipbb\SPB}n$:
\begin{multline*}
	\II^{w_1,\ldots,w_k}_{n}=
	\frac{1-\SP_0^{2}}{1-q}
	\oint\limits_{\contq {\ipbb\SPB}n}\frac{d u_n}{2\pi\i}
	(-\SP_0)^{-1}
	\frac{1-q}{u_n-\SP_{\nu_n}\ip_{\nu_n}^{-1}}
	\prod_{j=1}^{\nu_n-1}
	\frac{1-\SP_j\ip_j u_n}{\ip_ju_n-\SP_j}\prod_{j=1}^{k}\frac{1-qu_nw_j}{1-u_nw_j}
	\\=
	(1-\SP_0^{2})(-\SP_0)^{-2}q^{k}(-\SPB)^{\nu_n}
	+
	(1-\SP_0^{2})
	(-\SP_0)^{-2}\sum_{p=1}^{k}
	\frac{1-\SP_0\ip_0^{-1}w_p}{1-\SP_0^{-1}\ip_0^{-1}w_p}
	\pow_{\nu_n}(w_p\md\ipbb,\SPB)
	\prod_{\substack{1\le j\le k\\j\ne p}}\frac{w_p-qw_j}{w_p-w_j}.
\end{multline*}

In general, the integral in \eqref{G_nu_RHO_w_integral1}
is equal to a summation of the following sort. 
For every $\ell=0,\ldots,k$, choose two collections
of indices
$\IS=\{i_1<\ldots<i_\ell\}\subseteq\{1,\ldots,n\}$
and 
$\JS=(j_1,\ldots,j_\ell)\subseteq\{1,\ldots,k\}$
(note that the order of the $j_p$'s in $\JS$ matters, while the $i_p$'s
are assumed already ordered).
We will take residues at
$u_{i_p}=w_{j_p}^{-1}$, $1\le p\le \ell$,
and the remaining residues at $u_i=\infty$ for $i\notin\IS$.
Denote the summand corresponding to these residues by $\Res_{\IS,\JS}$.
All these summands have a common prefactor 
$(-\SP_0)^{-2n}(\SP_0^{2};q)_{n}$. The contribution to $\Res_{\IS,\JS}$
from residues at infinity is equal to
\begin{multline}
	q^{k(i_1-1)+(k-1)(i_2-i_1-1)+\ldots+(k-\ell+1)(i_{\ell}-i_{\ell-1}-1)
	+(k-\ell)(n-i_{\ell})}\prod_{i\notin\IS}(-\SPB)^{\nu_i}
	\\=
	q^{-\frac12\ell(2k+1-\ell)+n(k-\ell)}
	q^{i_1+\ldots+i_\ell}
	\prod_{i\notin\IS}(-\SPB)^{\nu_i}.
	\label{G_nu_RHO_infty}
\end{multline}
The residues at $u_{i_p}=w_{j_p}^{-1}$ contribute
\begin{align}
	\prod_{p=1}^{\ell}
	\bigg(
	\frac{1-\SP_0\ip_0^{-1}w_{j_p}}{1-\SP_0^{-1}\ip_0^{-1}w_{j_p}}
	\pow_{\nu_{i_p}}(w_{j_p}\md\ipbb,\SPB)
	\prod_{j\in\{1,\ldots,k\}\setminus\JS}
	\frac{w_{j_p}-qw_j}{w_{j_p}-w_j}
	\bigg)
	\prod_{1\le \aind<\bind\le \ell}
	\frac{w_{j_{\aind}}-qw_{j_{\bind}}}{w_{j_{\aind}}-w_{j_{\bind}}}
	.
	\label{G_nu_RHO_wp}
\end{align}
We see that \eqref{G_nu_RHO_infty} depends only on the choice of $\IS$,
while \eqref{G_nu_RHO_wp} depends on both $\IS$ and $\JS$.
Thus, for a fixed $\IS$, one can first sum 
$\Res_{\IS,\JS}$
over all 
subsets
$\JS=\{j_1<\ldots<j_\ell\}\subseteq\{1,\ldots,k\}$, 
and, for a fixed such $\JS$, over all its permutations.
This summation over $\JS$ is performed using the following lemma:
\begin{lemma}\label{lemma:partial_symmetrization}
	Let $f_i(\zeta)$ be arbitrary functions in $\zeta\in\C$. 
	For any $m\ge1$ and $\ell\le m$, we have
	\begin{multline*}
		\sum_{\sigma\in \Sym_m}\sigma
		\bigg(
		\prod_{i=1}^{\ell}f_i(\zeta_i)
		\prod_{1\le \aind<\bind\le m}\frac{\zeta_\aind-q \zeta_\bind}{\zeta_\aind- \zeta_\bind}
		\bigg)\\=
		\frac{(q;q)_{m-\ell}}{(1-q)^{m-\ell}}
		\sum_{\JS=\{j_1<\ldots<j_\ell\}\subseteq\{1,\ldots,m\}}
		\bigg(
		\prod_{\substack{j\in \JS\\r\notin \JS}}
		\frac{\zeta_{j}-q \zeta_r}{\zeta_{j}- \zeta_r}
		\sum_{\sigma'\in \Sym_\ell}
		\prod_{i=1}^{\ell}f_{i}(\zeta_{j_{\sigma'(i)}})
		\prod_{1\le \aind<\bind\le \ell}
		\frac{\zeta_{j_{\sigma'(\aind)}}-q \zeta_{j_{\sigma'(\bind)}}}{\zeta_{j_{\sigma'(\aind)}}-\zeta_{j_{\sigma'(\bind)}}}
		\bigg),
	\end{multline*}
	where the permutation $\sigma$ in the left-hand side
	acts on the variables $\zeta_j$.
\end{lemma}
\begin{proof}
	For each $\sigma\in \Sym_m$,
	let $\JS$ be the ordered list of elements of
	the set
	$\{\sigma(1),\ldots,\sigma(\ell)\}$.
	The left-hand side of the desired claim 
	contains
	\begin{multline*}
		\prod_{i=1}^{\ell}f_i(\zeta_{\sigma(i)})
		\prod_{1\le \aind<\bind\le m}\frac{\zeta_{\sigma(\aind)}-q \zeta_{\sigma(\bind)}}{\zeta_{\sigma(\aind)}-\zeta_{\sigma(\bind)}}
		\\=
		\prod_{i=1}^{\ell}f_i(\zeta_{j_i})
		\prod_{1\le \aind<\bind\le \ell}
		\frac{\zeta_{j_{\aind}}-q \zeta_{j_{\bind}}}{\zeta_{j_{\aind}}-\zeta_{j_{\bind}}}
		\underbrace{\prod_{\substack{j\in \JS\\r\notin \JS}}
		\frac{\zeta_{j}-q \zeta_r}{\zeta_{j}-\zeta_r}}_{\text{symmetric in $\JS$ and $\JSB$}}
		\prod_{\ell+1\le \aind<\bind\le m}
		\frac{\zeta_{\sigma(\aind)}-q \zeta_{\sigma(\bind)}}
		{\zeta_{\sigma(\aind)}-\zeta_{\sigma(\bind)}}.
	\end{multline*}
	Symmetrizing over $\sigma\in \Sym_m$ can be done in two steps: first, choose a subset $\JS$ 
	of $\{1,\ldots,m\}$ of size $\ell$, and then symmetrize 
	over indices inside and outside $\JS$. 
	For the symmetrization outside $\JS$ 
	we can use the symmetrization identity 
	of footnote$^{\ref{symm_footnote}}$.
	This yields the result.
\end{proof}
Applying this lemma to \eqref{G_nu_RHO_wp} 
with $m=k$, we 
arrive at the following formula for 
our specialization, which is the first step towards
$q$-correlation functions:
\begin{proposition}\label{prop:Gnu_spec_RHO_w}
	For $n\ge k\ge0$ and $\nu\in\signp n$, we have
	\begin{multline}\label{Gnu_spec_RHO_w_sum}
		\G_\nu(\RHO,w_1,\ldots,w_k\md\ipbb,\SPB)\\=
		\underbrace{\mathbf{1}_{\nu_n\ge1}
		(-\SPB)^{\nu}\frac{(\SP_0^{2};q)_{n}}{\SP_0^{2n}}}_{\G_\nu(\RHO\md\ipbb,\SPB)}
		\sum_{\ell=0}^{k}
		\frac{q^{-\frac12\ell(2k+1-\ell)+n(k-\ell)}(1-q)^{k-\ell}}{(q;q)_{k-\ell}}
		\sum_{\IS=\{i_1<\ldots<i_\ell\}\subseteq\{1,\ldots,n\}}
		q^{i_1+\ldots+i_\ell}
		\prod_{i\in\IS}\frac{1}{(-\SPB)^{\nu_i}}
		\\\times\sum_{\sigma\in\Sym_k}
		\sigma\bigg(
		\prod_{1\le \aind<\bind\le k}\frac{w_\aind-qw_\bind}{w_\aind-w_\bind}
		\prod_{p=1}^{\ell}
		\frac{1-\SP_0\ip_0^{-1}w_{p}}{1-\SP_0^{-1}\ip_0^{-1}w_{p}}
		\pow_{\nu_{i_p}}(w_{p}\md\ipbb,\SPB)
		\bigg),
	\end{multline}
	where the permutation $\sigma$ acts on $w_1,\ldots,w_k$.
\end{proposition}
\begin{proof}
	Identity \eqref{Gnu_spec_RHO_w_sum}
	is established above in this subsection
	under certain restrictions on the $w_p$'s
	and the parameters $q$, $\SPB$, and $\ipb$.
	Observe
	that both sides of the identity
	\eqref{Gnu_spec_RHO_w_sum}
	are a priori rational functions in all variables and parameters
	(and for fixed $\nu$ the number of parameters
	is finite).
	This is clear for the right-hand side,
	and the left-hand side of \eqref{Gnu_spec_RHO_w_sum} is also rational
	because using
	the branching rule of Proposition \ref{prop:branching}
	one can
	separate the specialization $\RHO$ and the variables $w_p$ (skew
	$\G$-functions in the $w_j$'s are rational by the very definition), 
	and then evaluate the specialization $\RHO$ by Proposition~\ref{prop:RHO_spec}.
	Thus,
	we can drop any restrictions,
	and so \eqref{Gnu_spec_RHO_w_sum} holds for generic values of the variables and parameters.
\end{proof}
Note that in the particular case $k=0$ 
the above proposition reduces to Proposition \ref{prop:RHO_spec}.


\subsection{Extracting terms by integrating over $w_i$} 
\label{sub:extracting_terms_by_integrating_over_w_i_}

Observe now that when $k=\ell$, the summation over $\sigma$
in \eqref{Gnu_spec_RHO_w_sum} above produces 
$(-\SP_0)^{k}\F_{(\nu_{i_1}-1,\ldots,\nu_{i_{\ell}}-1)}(w_1,\ldots,w_k\md \sh_1\ipbb,\sh_1\SPB)$.
Indeed, this is because
\begin{align*}
	\frac{1-\SP_0\ip_0^{-1}w_{p}}{1-\SP_0^{-1}\ip_0^{-1}w_{p}}
	\pow_{\nu_{i_p}}(w_{p}\md\ipbb,\SPB)
	=(-\SP_0)\pow_{\nu_{i_p}-1}(w_{p}\md\sh_1\ipbb,\sh_1\SPB),
\end{align*}
since all $\nu_{i_p}\ge1$ (here, as before, 
$\sh_1$ means the shift \eqref{sh_operation}).
The function 
$\F_{(\nu_{i_1}-1,\ldots,\nu_{i_{\ell}}-1)}$
then arises due to
\eqref{F_symm_formula}.

This observation
motivates our next step in computation of the $q$-correlation functions:
we will utilize orthogonality of the 
functions $\F_\mu$ (similar to Theorem~\ref{thm:orthogonality_F}),
and integrate \eqref{Gnu_spec_RHO_w_sum} over the $w_i$'s
to extract certain terms.
We will need the following nested integration contours:

\begin{definition}\label{def:bar_cont}
	For any $k\ge1$,
	let $\contqi {\ipb\SPBB}1,\ldots,\contqi {\ipb\SPBB}k$
	be 
	positively oriented closed contours
	such that
	\begin{itemize}
		\item 
		Each contour
		$\contqi {\ipb\SPBB}\aind$ encircles all the points of the set
		$\ipb\SPBB=\{\ip_i\SP_i^{-1}\}_{i\ge0}$, 
		while leaving outside all the points of
		$\ipb\SPB=\{\ip_i\SP_i\}_{i\ge0}$.
		\item For any $\bind>\aind$, the interior of $\contqi {\ipb\SPBB}\bind$ contains 
		the contour $q^{-1}\contqi {\ipb\SPBB}\aind$.
		\item The contour $\contqi {\ipb\SPBB}1$ is sufficiently small so that it does 
		not intersect with $q^{-1}\contqi {\ipb\SPBB}1$.
	\end{itemize}
	See Fig.~\ref{fig:bar_contours}. 
	The superscript ``$-$'' refers to the property that the 
	contours are $q^{-1}$-nested.
\end{definition}

\begin{figure}[htbp]
	\begin{center}
	\begin{tikzpicture}
		[scale=2.9]
		\def\pt{0.02}
		\def\q{.7}
		\def\ss{.56}
		\draw[->, thick] (-4.8,0) -- (.6,0);
	  	\draw[->, thick] (0,-.9) -- (0,.9);
	  	\draw[fill] (-.9,0) circle (\pt);
	  	\draw[fill] (-1,0) circle (\pt) node [below, xshift=7pt] {$\ip_i\SP_i\phantom{^{-1}}$};
	  	\draw[fill] (-1.05,0) circle (\pt);
	  	\draw[fill] (-.9/\ss,0) circle (\pt);
	  	\draw[fill] (-1/\ss,0) circle (\pt);
	  	\draw[fill] (-1.05/\ss,0) circle (\pt) node [below, xshift=12pt] {$\ip_i\SP_i^{-1}$};
	  	\draw[fill] (0,0) circle (\pt) node [below left] {$0$};
	  	\draw (-2.86,0) ellipse (1.47 and .95) node [below,xshift=-40,yshift=95] {$\contqie{3}$};
	  	\draw (-2.23,0) ellipse (.78 and .65) node [below,xshift=-20,yshift=71] {$\contqie{2}$};
	  	\draw (-1.78,0) circle (.29) node [below,xshift=-10,yshift=41] {$\contqie{1}$};
	  	\draw[dotted] (-1.78/\q,0) circle (.29/\q);
	  	\draw[dotted] (-1.78/\q/\q,0) circle (.29/\q/\q);
	\end{tikzpicture}
	\end{center}
  	\caption{A possible choice of nested integration contours 
  	$\contqie i=\contqi {\ipb\SPBB}i$, $i=1,2,3$ 
  	(Definition~\ref{def:bar_cont}).
  	Contours $q^{-1}\contqi {\ipb\SPBB}1$ and $q^{-2}\contqi {\ipb\SPBB}1$ are shown dotted.
	}
  	\label{fig:bar_contours}
\end{figure}
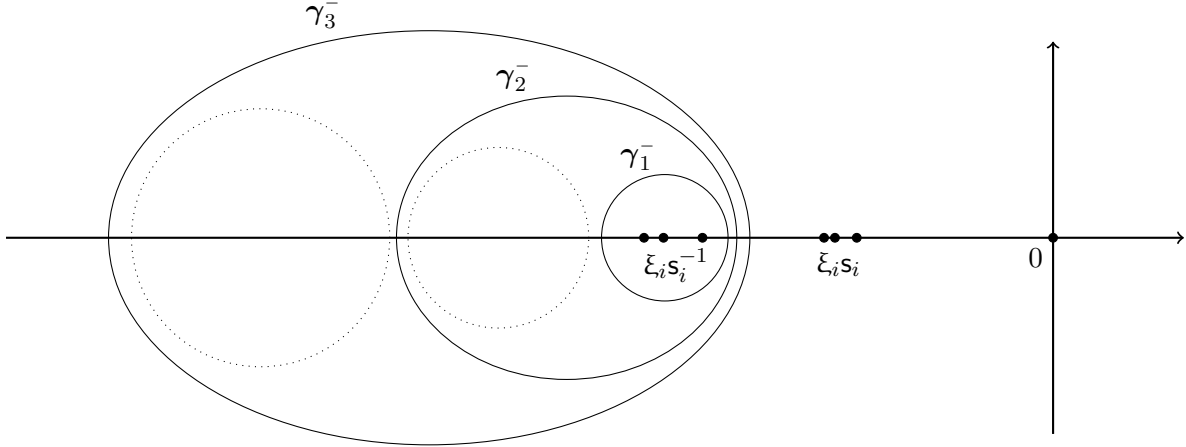

Conditions 
\eqref{stochastic_weights_condition_qsxi} and 
\eqref{assumptions_one_better} 
on $q$, $\SPB$, and $\ipb$
readily imply that the contours $\contqi {\ipb\SPBB}j$ exist. 
Throughout this subsection
we will assume that
these conditions hold.
\begin{remark}\label{rmk:drag_through_infinity}
	The integration contours $\contqi {\ipb\SPBB}j$
	can be obtained from the contours
	$\contq {\ipb\SPB}j$ of Definition 
	\ref{def:orthogonality_contours}\footnote{Note 
	the swapping $\ip_x\leftrightarrow\ip_x^{-1}$ in $\contq {\ipb\SPB}j$
	as compared to Definition \ref{def:orthogonality_contours}.
	Moreover, for the existence of the latter contours we must assume
	that $m_{\ipb|\SPB|}>q M_{\ipb|\SPB|}$, which is not equivalent
	to the first condition in \eqref{assumptions_one_better}.}
	by dragging 
	$\contq {\ipb\SPB}1,\ldots,\contq {\ipb\SPB}k$ (in this order)
	through
	infinity,
	if this operation is allowed for a particular integrand 
	(i.e., it must have no residues at infinity). 
\end{remark}

We will use the following integral transform:
\begin{definition}\label{def:int_transform}
	For $k\ge1$, let $R(w_1,\ldots,w_k)$ be a symmetric rational function 
	with singularities only occurring
	when some of the $w_j$'s belong to 
	$\sh_1(\ipb\SPB\cup\ipb\SPBB)=\{\ip_i\SP_i^{\pm1}\}_{i\in\Z_{\ge1}}$. 
	Let $\mm\in\signp k$. Define
	\begin{multline*}
		\big(\Pli k R\big)(\mm):=
		\frac{(-1)^{k}\conj_{\SPB}(\mm)}{(1-q)^{k}}
		\oint\limits_{\contqi {\ipb\SPBB}1}\frac{d w_1}{2\pi\i}
		\ldots
		\oint\limits_{\contqi {\ipb\SPBB}k}\frac{d w_k}{2\pi\i}
		\prod_{1\le \aind<\bind\le k}\frac{w_\aind-w_\bind}{w_\aind-qw_\bind}
		\\\times
		R(w_1,\ldots,w_k)
		\prod_{i=1}^{k}w_i^{-1}\pow_{\mm_{i}}(w_{i}^{-1}\md\sh_1\ipb,\sh_1\SPB),
	\end{multline*}
	where the integration contours are described in Definition \ref{def:bar_cont}.
\end{definition}

Let us denote for any $\la\in\signp \ell$:
\begin{align*}
	R^{(\ell)}_{\la}(w_1,\ldots,w_k):=
	\mathbf{1}_{\la_{\ell}\ge1}
	\sum_{\sigma\in\Sym_k}
	\sigma\bigg(
	\prod_{1\le \aind<\bind\le k}\frac{w_\aind-qw_\bind}{w_\aind-w_\bind}
	\prod_{p=1}^{\ell}
	\frac{1-\SP_0\ip_0^{-1}w_{p}}{1-\SP_0^{-1}\ip_0^{-1}w_{p}}
	\pow_{\la_p}(w_{p}\md\ipbb,\SPB)
	\bigg);
\end{align*}
these are the $w_j$-dependent summands in \eqref{Gnu_spec_RHO_w_sum}.
As mentioned above, for $\ell=k$,
\begin{align*}
	R^{(k)}_{\la}(w_1,\ldots,w_k)=(-\SP_0)^{k}\F_{\la-1^{k}}(w_1,\ldots,w_k\md\sh_1\ipbb,\sh_1\SPB).
\end{align*}
Moreover, the action of the transform $\Pli k$
on the $R^{(\ell)}_{\la}$'s for any $\ell$ and $\la$
turns out to be very simple:
\begin{lemma}\label{lemma:JkRl}
	For any $\ell=0,\ldots,k$ and any $\la\in\signp \ell$ we have
	\begin{align*}
		\big(\Pli k R^{(\ell)}_{\la}\big)(\mm)=
		(-\SP_0)^k\mathbf{1}_{\ell=k}\mathbf{1}_{\mm=\la-1^{k}},
		\qquad \mm\in\signp k.
	\end{align*}
\end{lemma}
\begin{proof}
	We may assume that $\la_{\ell}\ge1$.
	Let us complement $\la$ by zeros, $\la=(\la_1,\ldots,\la_\ell,0,\ldots,0)$,
	so that it has length $k$.
	We have
	\begin{multline}
		\big(\Pli k R^{(\ell)}_{\la}\big)(\mm)=
		\frac{(-1)^{k}\conj_{\SPB}(\mm)}{(1-q)^{k}}
		\sum_{\sigma\in\Sym_k}\,
		\oint\limits_{\contqi {\ipb\SPBB}1}\frac{d w_1}{2\pi\i}
		\ldots
		\oint\limits_{\contqi {\ipb\SPBB}k}\frac{d w_k}{2\pi\i}
		\prod_{1\le \aind<\bind\le k}\frac{w_\aind-w_\bind}{w_\aind-qw_\bind}
		\frac{w_{\sigma(\aind)}-qw_{\sigma(\bind)}}{w_{\sigma(\aind)}-w_{\sigma(\bind)}}
		\\\times
		(-\SP_0)^{\ell}
		\prod_{p=1}^{\ell}
		\underbrace{\pow_{\la_{p}-1}
		(w_{\sigma(p)}\md\sh_1\ipbb,\sh_1\SPB)}_{g_{\la_p}(w_{\sigma(p)})}\cdot
		\prod_{i=1}^{k}\underbrace{w_i^{-1}\pow_{\mm_{i}}(w_{i}^{-1}\md\sh_1\ipb,\sh_1\SPB)}_{f_{\mm_i+1 }(w_i)}.
		\label{JkRl_computation_proof}
	\end{multline}
	We now wish to apply Lemma \ref{lemma:general_lemma} with 
	\begin{align*}
		f_m(w)=
		\begin{cases}
			w^{-1}\pow_{m-1}(w^{-1}\md\sh_1\ipb,\sh_1\SPB),&m\ge1;\\
			1,&m=0,
		\end{cases}
		\qquad \qquad
		g_l(w)=\begin{cases}
			\pow_{l-1}(w\md\sh_1\ipbb,\sh_1\SPB),&l\ge1;
			\\
			1,&l=0,
		\end{cases}
	\end{align*}
	and $P_1'=\sh_1(\ipb\SPB)\cup\{\infty\}$,
	$P_2'=\sh_1(\ipb\SPBB)$ (we use notation $P_{1,2}'$ to distinguish from $P_{1,2}$
	in Definition \ref{def:orthogonality_contours}). 
	Indeed, since in our integral we always have $m\ge1$, 
	all singularities of $f_m(w)g_l(w)$ are in $P_1'\cup P_2'$.
	Moreover, for $m<l$, all poles of this product are in $P_2'$,
	and for $m>l$ (which may include $l=0$)
	all poles are in $P_1'$. Finally, observe that 
	we can deform the integration contours
	as in the hypothesis of Lemma \ref{lemma:general_lemma}.

	Therefore, since $\mm_i+1\ge1$ for all $i$
	in our integral, we can apply Lemma \ref{lemma:general_lemma},
	and conclude that it must be that $\la_i=\mm_i+1$ for all $i=1,\ldots,k$,
	in order for the integral to be nonzero. In particular, 
	the integral can be nonzero only for
	$\ell=k$.

	When $\ell=k$, the desired claim for $\mm=\la-1^{k}$
	follows by analogy with the last computation in the proof of Theorem~\ref{thm:orthogonality_F} 
	(with swapped parameters $\ip_x\leftrightarrow\ip_x^{-1}$).
	Namely, we first sum over $\sigma$
	(the integral vanishes unless $\sigma$
	permutes within clusters of $\mm$),
	and then
	compute the resulting smaller integrals by taking residues at 
	$w_1=\ip_{\mm_j-1}\SP_{\mm_j-1}^{-1}$,
	$w_2=q^{-1}\ip_{\mm_j-1}\SP_{\mm_j-1}^{-1}$, etc. This leads to the desired result.
\end{proof}
\begin{remark}
	Note that if 
	$m_{\ipb|\SPB|}>q M_{\ipb|\SPB|}$, then for $\ell=k$
	in the above proof we could simply
	drag the integration contours
	$\contqi {\ipb\SPBB}j$ 
	through infinity 
	to the negatively oriented $\contq {\ipb\SPB}j$
	(cf. Remark \ref{rmk:drag_through_infinity}).
	Indeed, this is because for $\ell=k$ the integrand 
	in \eqref{JkRl_computation_proof} 
	is regular at $w_i=\infty$ for all $i$.
	The passage to the contours 
	$\contq {\ipb\SPB}j$
	eliminates the sign $(-1)^{k}$,
	and the desired claim for $\mm=\la-1^{k}$
	directly follows from Theorem~\ref{thm:orthogonality_F} with $\ip_x\leftrightarrow\ip_x^{-1}$.
	In other words, under these additional assumptions on $q$, $\ipb$, and $\SPB$
	the transform $\Pli k$ 
	acts essentially as the inverse Plancherel transform
	(Definition \ref{def:Plancherel_transforms}). 
	It is however crucial that 
	the former is defined using the contours $\contqi {\ipb\SPBB}j$ and not $\contq {\ipb\SPB}j$
	because of nontrivial residues 
	at infinity in
	\eqref{JkRl_computation_proof} for $\ell<k$.
\end{remark}

Therefore, applying the transform $\Pli k$ to
the rational function
$\G_\nu(\RHO,w_1,\ldots,w_k\md\ipbb,\SPB)$
and using Proposition \ref{prop:Gnu_spec_RHO_w}, we 
arrive at the following statement summarizing the second step of the computation
of the $q$-correlation functions:
\begin{proposition}\label{prop:JkG}
	Under assumptions 
	\eqref{stochastic_weights_condition_qsxi} and \eqref{assumptions_one_better}
	on
	$q$, $\SPB$, and $\ipb$,
	for any $n\ge k\ge0$, $\nu\in\signp n$, and $\mm\in\signp k$, we have
	\begin{align}
		\big(\Pli k \G_\nu(\RHO,\bullet\md\ipbb,\SPB)\big)(\mm)
		=
		\G_\nu(\RHO\md\ipbb,\SPB)\,
		\frac{q^{-\frac12k(k+1)}}{(-\sh_1\SPB)^{\mm}}
		\sum_{\substack{\IS=\{i_1<\ldots<i_k\}\subseteq\{1,\ldots,n\}\\
		\nu_{i_1}=\mm_1+1,\ldots,\nu_{i_k}=\mm_k+1}}
		q^{i_1+\ldots+i_k},
		\label{computed_q_def_correlator}
	\end{align}
	where ``$\bullet$'' stands for the variables $w_1,\ldots,w_k$
	in which the transform $\Pli k$ is applied,
	and we are using the usual notation
	$(-\sh_1\SPB)^{\mm}=\prod_{i=1}^{k}(-\SP_1)\ldots(-\SP_{\mm_i})$.
	Note that for \eqref{computed_q_def_correlator} to be nontrivial
	one must have
	$\nu_n\ge1$.
\end{proposition}


\subsection{$q$-correlation functions} 
\label{sub:_q_correlation_functions}

The structure of formula \eqref{computed_q_def_correlator}
suggests the following definition. 
For any $n\ge k\ge0$, 
any 
$\mm=(\mm_1\ge\mm_2\ge \ldots\ge\mm_k\ge0)\in\signp k$
and any $\nu\in\signp n$, set
\begin{align}\label{q_def_correlation_thing}
	\corr\mm\nu q:=
	\sum_{\substack{\IS=\{i_1<\ldots<i_k\}\subseteq\{1,\ldots,n\}
	\\\nu_{i_1}=\mm_1,
	\ldots,
	\nu_{i_k}=\mm_k}}q^{i_1+\ldots+i_k}.
\end{align}
If $n<k$, then the above expression is zero, by agreement.
In this subsection we will employ Proposition~\ref{prop:JkG} to 
compute the expectations
\begin{align}\label{q_def_correlation}
	\E_{\UU;\RHO}\big(\corre\mm q\big)=
	\sum_{\nu\in\signp n}\MM_{\UU;\RHO}(\nu\md\ipb,\SPB)\corr\mm\nu q.
\end{align}
Note that by \eqref{RHO_spec}, 
the above summation
ranges only over $\nu\in\signp n$ with $\nu_n\ge1$.
The above sum converges if 
the parameters $u_i$ satisfy \eqref{admissibility_RHO_conditions}.
\begin{remark}
	Expectations \eqref{q_def_correlation}
	can be viewed as \emph{$q$-analogues of the correlation functions} 
	of $\MM_{\UU;\RHO}$.
	Indeed, when $q=1$ and all $\mm_i$'s are distinct, 
	\eqref{q_def_correlation_thing} turns into
	\begin{align*}
		\corr\mm\nu 1\Big\vert_{q=1}=\prod_{i=1}^{k}\#\{j\colon 1\le j\le n\textnormal{ and } \nu_j=\mm_i\}.
	\end{align*}
	When all the $\nu_j$'s are also pairwise distinct, 
	let us interpret them as coordinates of distinct particles 
	on $\Z_{\ge0}$. In this case
	the above expression further simplifies to
	\begin{align}\label{simple_correlation_thing}
		\corr\mm\nu 1\Big\vert_{q=1}=\mathbf{1}_{\{\textnormal{there is a particle of the configuration $\nu$
			at each of the locations $\mm_1,\ldots,\mm_k$}\}}.
	\end{align}
	A probability distribution on
	configurations $\nu$
	of $n$ distinct particles on $\Z_{\ge0}$
	is often referred to as the ($n$-)\emph{point process} on $\Z_{\ge0}$.
	An expectation of \eqref{simple_correlation_thing} with respect to 
	a point process on $\Z_{\ge0}$ is known as the 
	($k$-th) \emph{correlation function} of this point process.
\end{remark}

For the purpose of analytic continuation
in the parameter space, 
it is useful to establish that
our $q$-correlation functions are a priori rational:
\begin{lemma}\label{lemma:qcorr_rational}
	Fix $n\in\Z_{\ge0}$, and let \eqref{admissibility_RHO_conditions} hold.
	Then for any fixed $k=0,\ldots,n$ and $\mm\in\signp k$,
	the expectation
	$\E_{\UU;\RHO}\big(\corre\mm q\big)$
	is a rational function in $u_1,\ldots,u_n$ 
	and the 
	parameters $q$, $\ipb$, and $\SPB$.
\end{lemma}
\begin{proof}
	Write 
	\begin{align*}
		\E_{\UU;\RHO}\big(\corre\mm q\big)=
		\sum_{\IS=\{i_1<\ldots<i_k\}\subseteq\{1,\ldots,n\}}
		q^{i_1+\ldots+i_k}
		\sum_{\substack{\nu\in\signp n\\
		\nu_{i_1}=\mm_1,
		\ldots,
		\nu_{i_k}=\mm_k}}
		\MM_{\UU;\RHO}(\nu\md\ipb,\SPB)
		,
	\end{align*}
	and observe that only the second sum is infinite.
	Therefore, we may fix $\IS$
	and consider only the summation over $\nu$.
	By \eqref{MM_U_RHO_particular} and \eqref{Qp_RHO},
	the sum over $\nu$ 
	is the same as 
	the sum of products of stochastic vertex weights 
	$\Lmatr_{\ip_x u_y,\SP_x}$
	over certain collections of $n$
	paths in 
	$\{0,1,2,\ldots\}\times \{1,2,\ldots,n\}$,
	as in Definition \ref{def:F}.
	Namely, these paths start with $n$ horizontal edges $(-1,t)\to(0,t)$, $t=1,\ldots,n$,
	end with $n$ vertical edges 
	$(\nu_i,n)\to(\nu_i,n+1)$ (note that $\nu_n\ge1$),
	and the end edges are partially fixed
	by the condition
	$\nu_{i_1}=\mm_1,\ldots,\nu_{i_k}=\mm_k$. Therefore, only the 
	coordinates
	$\nu_1,\ldots,\nu_{i_1-1}$
	belong to the infinite range $\{\mm_1+1,\mm_1+2,\ldots\}$.
	
	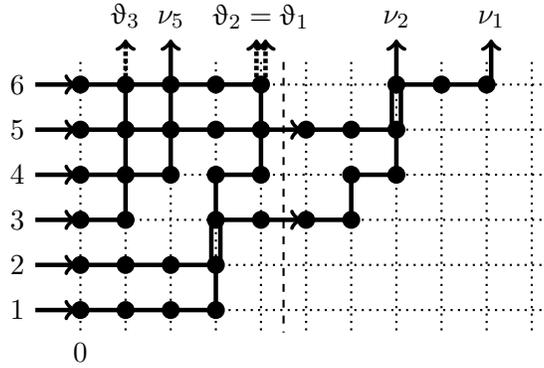
\begin{figure}[htbp]
			\begin{tikzpicture}
				[scale=.6,thick]
				\foreach \xxx in {0,1,2,3,4,5,6,7,8,9,10}
				{
					\draw[dotted] (\xxx,6.5)--++(0,-6);
				}
				\node [below] at (0,.5) {0};
				\foreach \xxx in {1,2,3,4,5,6}
				{
					\draw[dotted] (0,\xxx)--++(10.5,0);
					\node[left] at (-1,\xxx) {$\xxx$};
					\draw[->, line width=1.7pt] (-1,\xxx)--++(.9,0);
				}
				\draw[->, densely dotted, line width=1.7pt] (1,6)--++(0,1) node[above] {$\mm_3$};
				\draw[->, densely dotted, line width=1.7pt] (3.9,6)--++(0,1);
				\draw[->, densely dotted, line width=1.7pt] (4.1,6)--++(0,1);
				\draw[line width=1.7pt] (0,6)--(3.9,6);
				\draw[->, line width=1.7pt] (0,5)--(4.9,5);
				\draw[line width=1.7pt] (0,4)--++(1,0)--++(0,2);
				\draw[->, line width=1.7pt] (0,3)--++(1,0)--++(0,1)--++(1,0)--++(0,3)
				node [above] {$\nu_5$};
				\draw[line width=1.7pt] (0,2)--++(2.9,0)--++(0,1)--++(.1,0)--++(0,1)
				--++(1,0)--++(0,2);
				\draw[line width=1.7pt] (0,1)--++(3,0)--++(0,1);
				\draw[->, line width=1.7pt] (5,5)--++(1.9,0)--++(0,1)--++(.1,0)--++(0,1)
				node [above] {$\nu_2$};
				\draw[->, line width=1.7pt] (5,3)--++(1,0)--++(0,1)--++(1,0)--++(0,1)
				--++(.1,0)--++(0,1)--++(2,0)--++(0,1)
				node [above] {$\nu_1$};
				\draw[->, line width=1.7pt] (3.1,2)--++(0,1)--++(1.8,0);
				\node [above] at (4,7) {$\mm_2=\mm_1$};
				\draw[dashed] (4.5,6.5)--++(0,-6); 
				\foreach \ptt in {
				(0,1),(0,2),(0,3),(0,4),(0,5),(0,6),
				(1,1),(1,2),(1,3),(1,4),(1,5),(1,6),
				(2,1),(2,2),(2,4),(2,5),(2,6),
				(3,1),(3,2),(3,3),(3,4),(3,5),(3,6),
				(4,3),(4,4),(4,5),(4,6),
				(5,3),(5,5),
				(6,3),(6,4),(6,5),
				(7,4),(7,5),(7,6),
				(8,6),(9,6)
				}
				{
					\draw[fill] \ptt circle(5pt);
				}
			\end{tikzpicture}
		\caption{Splitting of summation over path collections 
		in the proof of Lemma \ref{lemma:qcorr_rational}
		for $n=6$, $k=3$, and $r=2$.
		The dotted arrows on the top correspond to fixed
		vertical edges prescribed by $\mm$,
		and here $\nu_3=\mm_1$, $\nu_4=\mm_2$, $\nu_6=\mm_3$,
		and $\JS=\{3,5\}$.}
		\label{fig:path_splitting}
	\end{figure}

	Assume that 
	$r\le i_1-1$
	out of our $n$ paths
	go strictly to the right of $\mm_1$ 
	(i.e., we have $\nu_r>\mm_1$ and $\nu_{r+1}=\ldots=\nu_{i_1}=\mm_1$).
	Let these paths contain edges 
	$(\mm_1,j_i)\to(\mm_1+1,j_i)$ 
	for some $\JS=\{j_1<\ldots<j_r\}\subseteq\{1,\ldots,n\}$.
	Fixing $r\le i_1-1$ and such $\JS$ (there are only finitely many ways to 
	choose this data), we may now
	split the summation over our $n$
	paths to paths in 
	$\{0,1,\ldots,\mm_1\}\times\{1,\ldots,n\}$
	and in 
	$\{\mm_1+1,\mm_1+2,\ldots\}\times \{1,\ldots,n\}$.
	The first sum over paths is also finite.
	See Fig.~\ref{fig:path_splitting} for an illustration of this splitting of paths.

	Since finite sums clearly produce rational functions,
	it now suffices to fix $r$ and $\JS$ as above,
	and consider the corresponding
	sum over collections of $r$ paths in 
	$\{\mm_1+1,\mm_1+2,\ldots\}\times \{1,\ldots,n\}$
	starting with 
	horizontal edges $(\mm_1,j_i)\to(\mm_1+1,j_i)$
	and ending with vertical edges
	$(\nu_i,n)\to(\nu_i,n+1)$, $i=1,\ldots,r$.
	Because this final infinite sum 
	involves stochastic vertex weights and 
	is over all unrestricted path collections, it is simply equal to $1$ 
	(recall that the $u_i$'s satisfy \eqref{admissibility_RHO_conditions}, so this sum converges).
	This completes the proof of the lemma.
\end{proof}

We are now in a position to compute the 
$q$-correlation functions \eqref{q_def_correlation}.
First, we will obtain a nested contour integration formula
when the points $u_i^{-1}$, $i=1,\ldots,n$, are inside 
the integration contour
$\contqi {\ipb\SPBB}1$ of Definition \ref{def:bar_cont}.
Note that this requires $\Re(u_i)<0$ for all $i$, which is  
incompatible with \eqref{stochastic_weights_condition_u}.
However, the correlation functions \eqref{q_def_correlation} 
are clearly well-defined as sums of possibly negative terms, 
and we have the following formula for them:
\begin{proposition}\label{prop:prelim_qcorr}
	Assume that 
	$q$, $\SPB$, and $\ipb$
	satisfy
	\eqref{stochastic_weights_condition_qsxi} and \eqref{assumptions_one_better}.
	Fix $n\ge k\ge0$,
	and let $u_1,\ldots,u_n$ satisfy \eqref{admissibility_RHO_conditions}
	and be such that 
	the 
	points $u_i^{-1}$ are inside the integration contour $\contqi {\ipb\SPBB}1$.
	Then for any $\mm=(\mm_1\ge\mm_2\ge \ldots\ge\mm_k\ge0)\in\signp k$ we have
	\begin{multline}\label{qcorr_prelim}
		\E_{\UU;\RHO}\big(\corre{\mm+1^{k}} q\big)=
		(-\sh_1\SPB)^{\mm}\conj_{\SPB}(\mm)\,
		\frac{(-1)^{k}q^{\frac{k(k+1)}{2}}}{(1-q)^{k}}
		\oint\limits_{\contqi {\ipb\SPBB}1}\frac{d w_1}{2\pi\i}
		\ldots
		\oint\limits_{\contqi {\ipb\SPBB}k}\frac{d w_k}{2\pi\i}
		\prod_{1\le \aind<\bind\le k}\frac{w_\aind-w_\bind}{w_\aind-qw_\bind}
		\\\times
		\prod_{i=1}^{k}\bigg(w_i^{-1}\pow_{\mm_{i}}(w_{i}^{-1}\md\sh_1\ipb,\sh_1\SPB)
		\prod_{j=1}^{n}\frac{1-qu_jw_i}{1-u_jw_i}
		\bigg),
	\end{multline}
	where $\conj_{\SPB}(\mm)$ is defined by \eqref{conj_definition}.
\end{proposition}
After proving this proposition, 
we will relax the conditions on the $u_i$'s
(to include $u_i\ge0$)
by suitably deforming the integration contours.
\begin{proof}
	The desired identity \eqref{qcorr_prelim}
	formally follows 
	from Proposition \ref{prop:JkG}
	combined with the Cauchy identity summation
	\eqref{EGoverG_computation}.	
	However, one needs to justify that this 
	summation can be performed inside the integral.
	
	Observe that
	if \eqref{admissibility_RHO_conditions} holds, 
	then the left-hand side of
	\eqref{qcorr_prelim}
	is a rational function in the $u_i$'s. 
	Therefore, the desired identity 
	would follow by
	analytic continuation of rational functions 
	(cf. footnote$^{\ref{footnote:rational_continuation}}$)
	if we can show that for certain restricted 
	values of $q$, $\SPB$, $\ipb$, and $\{u_i\}$
	and on certain deformed contours
	we have $\adm {u_i}{w_j}$ for all $i,j$, so that the summation
	can be performed inside the integral.
	
	This can be done similarly to the proof of Theorem \ref{thm:spec_orthogonality}. 
	Namely, it suffices to show that
	\begin{align*}
		\left|
		\frac{u_i^{-1}-\ip_x\SP_x^{-1}}{u_i^{-1}-\ip_x\SP_x}\right|
		<r,\qquad
		\left|
		\frac{w_j-\ip_x\SP_x}{w_j-\ip_x\SP_x^{-1}}
		\right|
		<R,
		\qquad
		\textnormal{for some $R>1$, $0<r<1$, with $rR<1$, and all $x$},
	\end{align*}
	which implies admissibility \eqref{admissible_sufficient}
	(note also that the first of these inequalities implies \eqref{admissibility_RHO_conditions}).
	That is, the points $u_i^{-1}$ should be closer to 
	$\ip_x\SP_x^{-1}$ than to $\ip_x\SP_x$, and the opposite for 
	the $w_j$'s. Considering the discs
	$\tilde B_x^{(r)}:=\{z\in\C\colon|z-\ip_x\SP_x^{-1}|<r|z-\ip_x\SP_x|\}$,
	one can check that for 
	restricted values of the parameters
	$\ipb$ and $\SPB$ similarly to \eqref{spec_bio_admiss2}
	and for $q$ sufficiently small, we have
	\begin{enumerate}[\bf1.]
		\item $\bigcap_{x=0}^{\infty}\tilde B_x^{(r)}$ is nonempty;
		\item $\bigcup_{x=0}^{\infty}\tilde B_x^{(1/R)}$ does not contain any 
		of the points $\ip_x\SP_x$;
		\item $\bigcup_{x=0}^{\infty}\tilde B_x^{(1/R)}$
		does not intersect with 
		$\bigcup_{x=0}^{\infty}q^{-1}\tilde B_x^{(1/R)}$,
	\end{enumerate}
	and so the deformed contours $\contqi {\ipb\SPBB}1',\ldots,\contqi {\ipb\SPBB}k'$
	exist.
	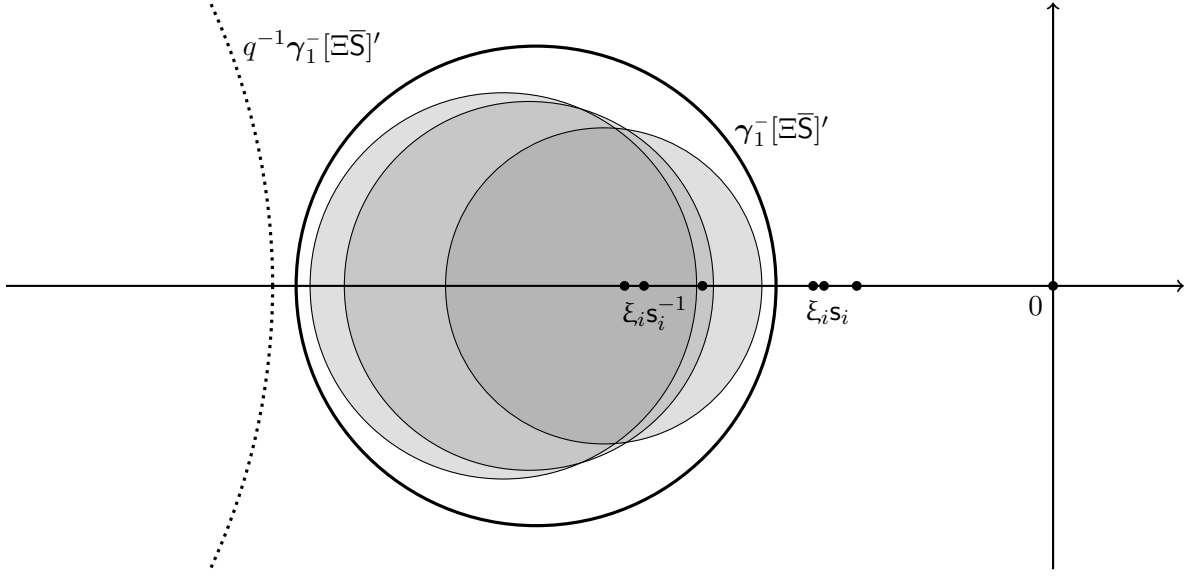
\begin{figure}[htbp]
		\begin{center}
		\begin{tikzpicture}
			[scale=2.9]
			\def\pt{0.02}
			\def\q{.355}
			\def\ss{.56}
			\def\ccc{-2.06}
			\def\rrr{.725}
			\def\cccc{-2.52}
			\def\rrrr{.886}
			\def\ccccc{-2.403}
			\def\rrrrr{.846}
			\draw[fill, color=gray, opacity=.25] (\ccc,0) circle (\rrr);
			\draw[fill, color=gray, opacity=.25] (\cccc,0) circle (\rrrr);
			\draw[fill, color=gray, opacity=.25] (\ccccc,0) circle (\rrrrr);
			\draw (\ccc,0) circle (\rrr);
			\draw (\cccc,0) circle (\rrrr);
			\draw (\ccccc,0) circle (\rrrrr);
			\draw[->, thick] (-4.8,0) -- (.6,0);
		  	\draw[->, thick] (0,-1.3) -- (0,1.3);
		  	\draw[fill] (-.9,0) circle (\pt);
		  	\draw[fill] (-1.1,0) circle (\pt);
		  	\draw[fill] (-1.05,0) circle (\pt) node [below, xshift=7pt] {$\ip_i\SP_i\phantom{^{-1}}$};
		  	\draw[fill] (-.9/\ss,0) circle (\pt);
		  	\draw[fill] (-1.1/\ss,0) circle (\pt);
		  	\draw[fill] (-1.05/\ss,0) circle (\pt) node [below, xshift=4pt] {$\ip_i\SP_i^{-1}$};
		  	\draw[fill] (0,0) circle (\pt) node [below left] {$0$};
		  	\draw[very thick] (-2.37,0) circle (1.1) node 
		  	[below,xshift=93,yshift=70] 
		  	{$\contqi {\ipb\SPBB}{1}'$};
		  	\draw[dotted, very thick] (-2.37/\q+1.1/\q,0) arc (0:25:1.1/\q);
		  	\draw[dotted, very thick] (-2.37/\q+1.1/\q,0) arc (0:-25:1.1/\q);
		  	\node at (-2.37/\q+1.1/\q+.18,1.1) {$q^{-1}\contqi {\ipb\SPBB}{1}'$};
		\end{tikzpicture}
		\end{center}
	  	\caption{Discs $\tilde B_x^{(r)}$ (shaded)
	  	for the proof of Proposition \ref{prop:prelim_qcorr},
	  	and a possible choice of the deformed contour $\contqi {\ipb\SPBB}1'$. 
	  	A part of the contour $q^{-1}\contqi {\ipb\SPBB}1'$
	  	is shown dotted, explaining why $q$ should be small for the contours
	  	$\contqi {\ipb\SPBB}j'$ to exist.}
	  	\label{fig:Appolonian}
	\end{figure}

	Therefore, for restricted values of parameters
	we can deform the contours $\contqi {\ipb\SPBB}j$
	to $\contqi {\ipb\SPBB}j'$, and the points $u_i^{-1}$ will be inside the contour
	$\contqi {\ipb\SPBB}1'$. On the deformed contours the summation \eqref{EGoverG_computation} 
	can be performed inside the integral,
	yielding the identity \eqref{qcorr_prelim} between
	rational functions for restricted values of parameters.
	We conclude that 
	\eqref{qcorr_prelim} then holds for all values of parameters as
	described in the claim, because 
	for them the contour integral represents the same rational function.
\end{proof}

To state our final result for the $q$-correlation
functions, we 
need the following integration contour:
\begin{definition}\label{def:cont_gat_U}
	Let $u_1,\ldots,u_n>0$, and assume that $u_i\ne q u_j$
	for any $i,j$. Define the 
	contour $\contn{\bar\UU}$ to be a union of sufficiently
	small positively oriented circles around all the points 
	$\{u_i^{-1}\}$, such that the interior of 
	$\contn{\bar\UU}$ does not intersect 
	with $q^{\pm1}\contn{\bar\UU}$, and
	the points $\sh_1(\ipb\SPB)=\{\ip_x\SP_x\}_{x\in\Z_{\ge1}}$ are 
	outside the contour $\contn{\bar\UU}$.
\end{definition}

For $u_i>0$ and
$q\cdot \max_{i}u_i<\min_i u_i$,
let the $q^{-1}$-nested contours $\contqi{\bar\UU}j$, $j=1,\ldots,k$, be 
defined analogously to 	
$\contqi{\ipb\SPBB}j$ of Definition \ref{def:bar_cont}
(but the $\contqi{\bar\UU}j$'s encircle the points $u_i^{-1}$).
In this case the contour $\contqi{\bar\UU}1$
can also play the role of
$\contn{\bar\UU}$ of Definition
\ref{def:cont_gat_U}.

With these contours we can now formulate the final result of
the computation in 
\S \ref{sub:computation_of_gnurhow}--\S \ref{sub:_q_correlation_functions}:
\begin{theorem}\label{thm:qcorr}
	The nested contour integral formula 
	\eqref{qcorr_prelim} 
	for the $q$-correlation functions of the 
	dynamics $\Xp_{\{u_t\};\RHO}$ at $\text{time}=n$
	holds in each of the following three cases:
	\begin{enumerate}[\bf1.]
		\item Let $q$, $\SPB$, and $\ipb$ satisfy
		\eqref{stochastic_weights_condition_qsxi}
		and \eqref{assumptions_one_better}, the points
		$u_i^{-1}$ be inside the integration contour $\contqi{\ipb\SPBB}1$,
		and the $u_i$'s satisfy \eqref{admissibility_RHO_conditions}.
		Then \eqref{qcorr_prelim} holds with the 
		integration contours $w_j\in\contqi{\ipb\SPBB}j$
		of Definition \ref{def:bar_cont}.
		\item 
		Under
		\eqref{stochastic_weights_condition_qsxi}
		and \eqref{assumptions_one_better},
		let $u_i>0$ for all $i$ and $u_i\ne qu_j$ for any $i,j$.
		Then 
		\eqref{qcorr_prelim} holds when all the integration
		contours are the same, $w_j\in\contn{\bar\UU}$ (described in 
		Definition \ref{def:cont_gat_U}).
		In this case we can symmetrize the integrand
		similarly to the proof of Corollary \ref{cor:spatial_biorth}, and 
		the formula takes the form
		\begin{multline}\label{qcorr_symmetrized}
			\E_{\UU;\RHO}\big(\corre{\mm+1^{k}} q\big)=
			\frac{(-1)^{k}q^{\frac{k(k+1)}{2}}}{(1-q)^{k}k!}
			\oint\limits_{\contn{\bar\UU}}\frac{d w_1}{2\pi\i}
			\ldots
			\oint\limits_{\contn{\bar\UU}}\frac{d w_k}{2\pi\i}
			\prod_{1\le \aind\ne \bind\le k}\frac{w_\aind-w_\bind}{w_\aind-qw_\bind}
			\\\times
			(-\sh_1\SPB)^{\mm}\frac{\F^{\conj}_{\mm}
			(w_1^{-1},\ldots,w_k^{-1}\md\sh_1\ipb,\sh_1\SPB)}{w_1 \ldots w_k}
			\prod_{i=1}^{k}
			\prod_{j=1}^{n}\frac{1-qu_jw_i}{1-u_jw_i}.
		\end{multline}
		\item 
		Under 
		\eqref{stochastic_weights_condition_qsxi}--\eqref{stochastic_weights_condition_u}
		and \eqref{assumptions_one_better},
		let $\UU$ have the form 
		\begin{multline}
			\UU=(u_1,qu_1,\ldots,q^{J-1}u_1,u_2,qu_2,\ldots,q^{J-1}u_2,
			\ldots,u_{n'},qu_{n'},\ldots,q^{J-1}u_{n'}),\\
			\textnormal{
			where
			$n=Jn'$ with some $J\in\Z_{\ge1}$, \quad
			$u_i>0$,\quad and \quad
			$q\cdot \max_{i}u_i<\min_i u_i$.}
			\label{U_fused_for_contours}
		\end{multline}
		Then 
		\eqref{qcorr_prelim} holds with the integration
		contours
		$w_j\in\contqi{\bar\UU}j$.
		In this case the double product in the integrand
	 	takes the form
		$\displaystyle\prod\nolimits_{i=1}^{k}\prod\nolimits_{j=1}^{n'}\frac{1-q^{J}u_jw_i}{1-u_jw_i}$.\footnote{Since the definition of the contours $\contqi{\bar\UU}j$ does not depend
		on $J\in\Z_{\ge1}$, we can analytically continue
		the nested contour integral formula
		in $q^{J}$ (similarly to the 
		discussion in \S \ref{sub:general_j_dynamics_and_q_hahn_degeneration}).
		We will employ this continuation in 
		\S \ref{sub:moments_of_q_hahn_and_q_boson_systems}
		below.}
	\end{enumerate}
\end{theorem}
\begin{proof}
	{\bf1.\/} This is Proposition \ref{prop:prelim_qcorr}. 
	
	\smallskip
	\noindent
	{\bf2.\/}
	To prove the second case, 
	start with 
	\eqref{qcorr_prelim} 
	with contours $w_j\in\contqi{\ipb\SPBB}j$ and
	$q$, $\SPB$, $\ipb$, and $\{u_i\}$ fixed, and observe the following effect.
	The integrand 
	\begin{align}
		\prod_{1\le \aind<\bind\le k}\frac{w_\aind-w_\bind}{w_\aind-qw_\bind}
		\prod_{i=1}^{k}\bigg(w_i^{-1}\pow_{\mm_{i}}(w_{i}^{-1}\md\sh_1\ipb,\sh_1\SPB)
		\prod_{j=1}^{n}\frac{1-qu_jw_i}{1-u_jw_i}\bigg)
		\label{qcorr_proof_integrand}
	\end{align}
	has 
	only the poles $w_1=u_j^{-1}$, $j=1,\ldots,n$,
	inside the contour $\contqi{\ipb\SPBB}1$, because the other poles $\infty$
	and $\{\ip_i\SP_i\}_{i\in\Z_{\ge1}}$ are outside 
	$\contqi{\ipb\SPBB}1$.
	Deform the integration contour 
	$\contqi{\ipb\SPBB}2$ so that it becomes the same as $\contqi{\ipb\SPBB}1$,
	thus picking the residue at $w_2=q^{-1}w_1$. We see that 
	\begin{align*}
		\Res_{w_2=q^{-1}w_1}
		\frac{w_1-w_2}{w_1-qw_2}
		\bigg(\prod_{j=1}^{n}\frac{1-qu_jw_1}{1-u_jw_1}
		\frac{1-qu_jw_2}{1-u_jw_2}\bigg)
		=(1-q)q^{-2}w_1
		\prod_{j=1}^{n}\frac{q(1-qu_jw_1)}{q-u_jw_1}.
	\end{align*}
	Hence, the residue at 
	$w_2=q^{-1}w_1$ is regular in $w_1$ on the contour 
	$\contqi{\ipb\SPBB}1$, and thus vanishes after the $w_1$
	integration. 
	Continuing this argument in a similar way, we may deform all integration contours to be $\contqi{\ipb\SPBB}1$.
	In other words, we see that the integral \eqref{qcorr_prelim}
	can be computed by taking only the residues at points
	$w_i=u_{j_i}^{-1}$, where $i=1,\ldots,k$ and 
	$\{j_1,\ldots,j_k\}\subseteq\{1,\ldots,n\}$.
	
	Next, observe that the points $\ip_x\SP_x^{-1}$ are not
	poles of the integrand \eqref{qcorr_proof_integrand},
	and so the requirement that the contour
	$\contqi{\ipb\SPBB}1$ encircles these points can be dropped.
	Thus, we may take $u_i>0$, and the 
	integration contours to be $\contn{\bar\UU}$ instead of $\contqi{\ipb\SPBB}1$.
	Symmetrizing the integration variables finishes the second case.

	\smallskip
	\noindent
	{\bf3.\/}
	This case can be obtained as a limit of the second case. 
	Namely, 
	under assumptions of case 2 and also assuming 
	$q\cdot \max_{i}u_i<\min_i u_i$,
	let us first pass to the $q^{-1}$--nested contours
	$\contqi{\bar\UU}j$,
	which start from $\contn{\bar\UU}=\contqi{\bar\UU}1$.
	This can be done following the above argument in case 2,
	because the integration in both cases
	$\contn{\bar\UU}$ and $\contqi{\bar\UU}j$ involves the same residues.

	Next, if $u_2\ne q u_1$, 
	then the integrand in \eqref{qcorr_prelim}
	has nonzero residues at both
	$(w_1,w_2)=(u_1^{-1},u_2^{-1})$
	and
	$(w_1,w_2)=(u_2^{-1},u_1^{-1})$,
	and so 
	both $u_{1,2}^{-1}$ must be inside
	$\contqi{\bar\UU}1$ to produce the correct rational function.
	However, if $u_2=qu_1$, then the second residue vanishes due to the presence
	of the factor $u_2-qu_1$. 
	One can readily check that the same effect occurs 
	when we first move $u_2^{-1}$
	outside the contour $\contqi{\bar\UU}1$
	(but still inside $\contqi{\bar\UU}2$),
	and then set $u_2=qu_1$. 
	This agrees with the 
	presence of the factors 
	$\frac{1-q^{2}u_1w_i}{1-u_1w_i}$ in the integrand
	after setting $u_2=qu_1$,
	which do not have poles at $u_2^{-1}=(qu_1)^{-1}$.
	One also sees that after taking residue at $w_1=u_1^{-1}$,
	the pole at $w_2=u_1^{-1}$ disappears, 
	but there is a new pole at $w_2=(qu_1)^{-1}$.
	
	Continuing on,
	we see that for the 
	contours $\contqi{\bar\UU}j$
	we can specialize $\UU$ to \eqref{U_fused_for_contours},
	and the contour integration will still yield the 
	correct rational function
	(i.e., the corresponding specialization of the left-hand
	side of \eqref{qcorr_prelim}).
	This establishes the
	third case.
\end{proof}


\subsection{Remark. From observables to duality, and back} 
\label{sub:duality_from_observables}

Formula \eqref{qcorr_symmetrized} for the $q$-correlation functions
readily suggests a certain \emph{self-duality} 
relation associated with the inhomogeneous stochastic higher spin six vertex model.
Denote
$H(\nu;\mm):=\corre{\mm+1^{k}} q(\nu)$.
Then \eqref{qcorr_symmetrized}
implies
\begin{align}\label{duality_from_observables}
	\sum_{\eta\in\signp k}T(\mm\to\eta)\E_{\UU\cup u;\RHO}H(\bullet;\eta)=
	\E_{\UU;\RHO}H(\bullet;\mm),
\end{align}
where
$T(\mm\to\eta):=
q^{-k}\frac{(-\sh_1\SPB)^{\mm}}{(-\sh_1\SPB)^{\eta}}\G_{\eta/\mm}\big((qu)^{-1}\md\sh_1\ipbb,\sh_1\SPB\big)$.
In \eqref{duality_from_observables} by ``$\bullet$'' we mean the variables in which the expectation is applied.
To see \eqref{duality_from_observables},
apply the operator $T$ inside the integral, and note that 
\eqref{Pieri1} is equivalent to 
\begin{align}\label{Pieri1_equivalent}
	\sum_{\eta\in\signp{k}}
	T(\mm\to\eta)
	(-\sh_1\SPB)^{\eta}\F^{\conj}_{\eta}(w_1^{-1},\ldots,w_k^{-1}\md\sh_1\ipb,\sh_1\SPB)
	=\prod_{i=1}^{k}
	\frac{1-u w_i}{1-quw_i}
	(-\sh_1\SPB)^{\mm}\F^{\conj}_{\mm}(w_1^{-1},\ldots,w_k^{-1}\md\sh_1\ipb,\sh_1\SPB).
\end{align}
	
On the other hand, adding the new parameter $u$ to the specialization $\UU$
in \eqref{duality_from_observables} corresponds to time evolution, i.e., 
to the application of the operator $\Qp_{u;\RHO}$ \eqref{Qp_RHO}. That is, 
the left-hand side of \eqref{duality_from_observables}
can be written as
\begin{multline*}
	\sum_{\la\in\signp {n+1}}\sum_{\eta\in\signp k}T(\mm\to\eta)\MM_{\UU\cup u;\RHO}(\la\md\ipb,\SPB)H(\la;\eta)
	\\=
	\sum_{\mu\in\signp n}
	\MM_{\UU;\RHO}(\mu\md\ipb,\SPB)
	\sum_{\eta\in\signp k}
	T(\mm\to\eta)
	\sum_{\la\in\signp {n+1}}
	\Qp_{u;\RHO}(\mu\to\la)
	H(\la;\eta).
\end{multline*}
Since the right-hand side of \eqref{duality_from_observables}
involves the expectation with respect to the 
same measure $\MM_{\UU;\RHO}$ and since 
identity \eqref{duality_from_observables} holds
for arbitrary $u_i$'s, this suggests the following \emph{duality relation}:
\begin{align}\label{duality_relation}
	\Qp_{u;\RHO}HT^{\textnormal{transpose}}=H,
\end{align}
where the operators
$\Qp_{u;\RHO}$
and 
$T^{\textnormal{transpose}}$ 
are applied in the first and the second variable in $H$, respectively.
Similar duality relations can be written down by considering $q$-moments which are 
computed in Theorem \ref{thm:multi_moments} below.

\medskip

It is worth noting that (self-)dualities like \eqref{duality_relation} can sometimes be independently proven from the very definition of the 
dynamics, and then utilized to 
produce nested contour integral formulas for the 
observables of these dynamics. This can be thought of as 
an alternative way to 
proving results like Theorem \ref{thm:qcorr}. Let us outline 
this argument.
Applying $(\Qp_{u;\RHO})^{n}$ to \eqref{duality_relation} gives
\begin{align*}
	(\Qp_{u;\RHO})^{n+1}HT^{\textnormal{transpose}}
	=(\Qp_{u;\RHO})^{n}H.
\end{align*}
Taking the expectation in both sides above, we arrive back at
our starting point \eqref{duality_from_observables}:
\begin{align}\label{duality_relation_2}
	(\E_{(u^{n+1});\RHO}H)T^{\textnormal{transpose}}
	=\E_{(u^{n});\RHO}H,\qquad
	(u^{m}):=(\underbrace{u,\ldots,u}_{\textnormal{$m$ times}}),
\end{align}
where, as before, the expectation of $H=H(\nu;\mm)$ is taken with 
respect to the probability distribution 
in $\nu$, and the operator 
$T^{\textnormal{transpose}}$ acts on $\mm$.
Thus, knowing \eqref{duality_relation} and passing to 
\eqref{duality_relation_2}, one gets a closed system 
of linear equations for the observables
$\E_{u^{n};\RHO}H(\bullet;\mm)$,
where $n$ runs over $\Z_{\ge0}$,
and $\mm$ --- over $\signp k$.
This system can sometimes be reduced to 
a simpler system of free evolution equations
subject to certain two-body boundary conditions, 
and the latter can be solved explicitly in terms of 
nested contour integrals.

This alternative route towards explicit formulas for 
averaging of observables was taken
(for various degenerations of the higher spin six vertex model)
in \cite{BorodinCorwinSasamoto2012}, \cite{BorodinCorwin2013discrete},
\cite{Corwin2014qmunu}.
Duality for the (homogeneous) higher spin six vertex model
started from infinitely
many particles at the leftmost location
was considered in~\cite{CorwinPetrov2015}.

\begin{remark}
	An advantage of this alternative 
	route starting from duality \eqref{duality_relation} is that 
	it implies
	equations
	\eqref{duality_relation_2} for 
	arbitrary (sufficiently nice) initial conditions,
	because one can take an arbitrary expectation
	in the last step leading to 
	\eqref{duality_relation_2}.\footnote{This is 
	in contrast with 
	\eqref{duality_from_observables} which
	is implied by 
	\eqref{qcorr_symmetrized}, and thus holds only 
	for the dynamics 
	$\Xp$
	started from the empty initial configuration.}
	This argument could lead to nested contour integral formulas
	for arbitrary initial conditions,
	similarly to what is done in
	\cite{BorodinCorwinPetrovSasamoto2013}
	and \cite{BCPS2014}.
	We will not discuss duality relations
	or formulas with
	arbitrary initial conditions here.
\end{remark}



\section{$q$-moments of the height function} 
\label{sec:_q_moments_of_the_height_function_of_interacting_particle_systems}

In this section we compute 
another type of observables
of the stochastic dynamics $\Xp_{\{u_t\};\RHO}$
started from the empty initial configuration --- the $q$-moments
of its height function.

\subsection{Height function and its $q$-moments} 
\label{sub:height_function_and_its_q_moments}

Let $\nu\in\signp n$. Define the \emph{height function}
corresponding to $\nu$ as follows:
\begin{align*}
	\nonumber
	\HT_\nu(x):=\#\{j\colon \nu_j\ge x\},\qquad x\in\Z.
\end{align*}
Clearly, $\HT_\nu(x)$ is a nonincreasing function of $x$,
$\HT_\nu(0)=n$, and $\HT_\nu(+\infty)=0$.
In this section we will compute the (multi-point) $q$-moments
\begin{align*}
	\E_{\UU;\RHO}\prod_{i=1}^{\ell}q^{\HT_{\nu}(x_i)}
	=
	\sum_{\nu\in\signp n}\MM_{\UU;\RHO}(\nu\md\ipb,\SPB)
	\prod_{i=1}^{\ell}q^{\HT_{\nu}(x_i)}
	\nonumber
\end{align*}
of the height function,
where $x_1\ge \ldots\ge x_{\ell}\ge1$ are arbitrary.
Note that the above summation ranges only over
signatures with
$\nu_n\ge1$.

\begin{lemma}\label{lemma:qmom_rational}
	Fix $n\in\Z_{\ge0}$ and $u_1,\ldots,u_n$ satisfying
	\eqref{admissibility_RHO_conditions}. Then for any 
	$\ell$ and $x_1\ge \ldots\ge x_{\ell}\ge1$,
	the $q$-moments
	$\E_{\UU;\RHO}\prod_{i=1}^{\ell}q^{\HT_{\nu}(x_i)}$
	are rational functions in the $u_i$'s
	and the parameters $q$, $\ipb$, and $\SPB$. 
\end{lemma}
\begin{proof}
	This is established similarly to Lemma \ref{lemma:qcorr_rational},
	because if $\HT_\nu(x_1)\in\{0,1,\ldots,n\}$ is fixed, then there is a 
	fixed number of the coordinates of $\nu$
	belonging to an infinite range, and the summation 
	over them produces a rational function.
\end{proof}

We will first use the $q$-correlation functions
discussed in \S \ref{sec:observables_of_interacting_particle_systems}
to compute one-point $q$-moments 
$\E_{\UU;\RHO}q^{\ell\,\HT_{\nu}(x)}$. 
The formula for these one-point $q$-moments
allows to formulate an analogous multi-point
statement, and we will then present its verification proof.
Thus, the one-point formula will be proven in two different ways.


\subsection{One-point $q$-moments from $q$-correlations } 
\label{sub:one_point_q_moments_from_q_correlations_}

Let us first establish an algebraic
identity connecting one-point
$q$-moments with $q$-correlation functions. 
In fact,
the identity holds even before taking the expectation:
\begin{lemma}\label{lemma:moments_from_qcorr}
	For any $x\ge 1$, $\ell\ge0$, and a signature
	$\nu$, we have
	\begin{align}\label{moments_from_qcorr}
		q^{\ell\,\HT_\nu(x)}
		=
		\sum_{k=0}^{\ell}
		(-q)^{-k}\binom{\ell}{k}_{q}
		(q;q)_{k}
		\sum_{\mm_1\ge \ldots\ge \mm_k\ge x}
		\corr{(\mm_1,\ldots,\mm_k)}\nu q.
	\end{align}
\end{lemma}
\begin{proof}
	Denote
	$\Delta\HT_\nu(x):=\HT_\nu(x)-\HT_\nu(x+1)$;
	this is the number of parts of $\nu$ that are equal to $x$.
	First, let us express the quantities 
	$\corr\mm\nu q$ through the height function. 
	We start with the case $\mm=(x^\ell)=(x,\ldots,x)$.
	We have
	\begin{align*}
		\corr{(x^{\ell})}\nu q&=
		\sum_{\substack{\IS=\{i_1<\ldots<i_\ell\}\subseteq\{1,\ldots,n\}
		\\\nu_{i_1}=\ldots=
		\nu_{i_\ell}=x}}q^{i_1+\ldots+i_\ell}
		\\&=q^{\frac{\ell(\ell+1)}{2}}
		q^{\ell\,\HT_\nu(x+1)}
		\binom{\Delta\HT_{\nu}(x)}{\ell}_{q}
		\\&=
		q^{\frac{\ell(\ell+1)}{2}}
		\frac{\big(q^{\Delta\HT_{\nu}(x)};q^{-1}\big)_{\ell}}{(q;q)_{\ell}}\,q^{\ell\,\HT_\nu(x+1)},
	\end{align*}
	where the second equality follows similarly
	to the computation of the partition function
	\eqref{Zjj_formula}.

	For general
	\begin{align*}
	 	\mm=(x_1^{\ell_1},\ldots,x_m^{\ell_m}):=
	 	(\underbrace{x_1,\ldots,x_1}_{\textnormal{$\ell_1$ times}},
	 	\ldots,\underbrace{x_m,\ldots,x_m}_{\textnormal{$\ell_m$ times}}),
	\end{align*}
	where $x_1>\ldots>x_m\ge0$ and $\boldsymbol\ell
	=(\ell_1,\ldots,\ell_m)\in\Z_{\ge0}^{m}$,
	the summation over $\IS$ in \eqref{q_def_correlation_thing}
	is clearly equal to the product of individual summations 
	corresponding to each $x_j$, $j=1,\ldots,m$. Therefore, 
	\begin{align}\label{corr_via_HT}
		\corr{(x_1^{\ell_1},\ldots,x_m^{\ell_m})}\nu q=
		\prod_{j=1}^{m}
		q^{\frac{\ell_j(\ell_j+1)}{2}}
		\frac{\big(q^{\Delta\HT_{\nu}(x_j)};q^{-1}\big)_{\ell_j}}{(q;q)_{\ell_j}}\,q^{\ell_j\HT_\nu(x_j+1)}.
	\end{align}

	Our next goal is to invert relation \eqref{corr_via_HT}.
	Let us write down certain abstract inversion formulas which will
	lead us to the desired statement.
	In these formulas, we will assume that 
	$A,B,A_0,A_1,A_2,\ldots$ are indeterminates.
	Let us also denote
	\begin{align*}
		T_i(A):=q^{\frac{i(i+1)}2}\frac{(A;q^{-1})_{i}}{(q;q)_i},
		\qquad
		R_i:=\frac{(-q)^{i}}{(q;q)_{i}}.
	\end{align*}
	Note that by the very definition,
	$T_i(A)=0$ for $i<0$, 
	$T_0(A)=1$,
	and $T_i(1)=\mathbf{1}_{i=0}$.
	Moreover, 
	$R_i=0$ for $i<0$, and $R_0=1$. 
	
	The first inversion formula is
	\begin{align}\label{basic_inversion_formula}
		A^n R_n=\sum_{k=0}^{n}T_k(A)R_{n-k}.
	\end{align}
	Indeed, multiply the above identity by $B^{n}$, and sum over $n\ge0$.
	The left-hand side gives, by the $q$-Binomial Theorem,
	\begin{align*}
		\sum_{n=0}^{\infty}(AB)^{n}\frac{(-q)^{n}}{(q;q)_{n}}=
		\frac{1}{(-ABq;q)_{\infty}},
	\end{align*}
	and in the right-hand side we first sum over $n\ge k$ and then over $k\ge0$, which yields
	\begin{multline*}
		\sum_{k=0}^{\infty}
		q^{\frac{k(k+1)}2}\frac{(A;q^{-1})_{k}}{(q;q)_k}
		B^{k}
		\sum_{n=k}^{\infty}
		\frac{(-q)^{n-k}}{(q;q)_{n-k}}
		B^{n-k}=
		\frac{1}{(-Bq;q)_{\infty}}
		\sum_{k=0}^{\infty}
		q^{\frac{k(k+1)}2}\frac{(A;q^{-1})_{k}}{(q;q)_k}
		B^{k}
		\\=
		\frac{1}{(-Bq;q)_{\infty}}
		\sum_{k=0}^{\infty}
		\frac{(A^{-1},q)_{k}}{(q;q)_k}
		(-q)^k A^k B^{k}
		=\frac{1}{(-Bq;q)_{\infty}}
		\frac{(-Bq;q)_{\infty}}{(-ABq;q)_{\infty}}=
		\frac{1}{{(-ABq;q)_{\infty}}},
	\end{multline*}
	where we have used the $q$-Binomial Theorem twice.
	This establishes \eqref{basic_inversion_formula},
	because the generating series of both its sides coincide.\footnote{In 
	the above manipulations with infinite series 
	we assume that $0\le q<1$ and that $A$ and $B$ are sufficiently small.
	Alternatively, it is enough to think that 
	we are working with formal power series.}

	Replace $A$ by $A_1$ in
	\eqref{basic_inversion_formula}, 
	multiply it by $A_2=A_2^{k}A_2^{n-k}$,
	and apply \eqref{basic_inversion_formula} to 
	$A_2^{n-k}R_{n-k}$ in the right-hand side. 
	Continuing this process with $A_3,\ldots,A_N$, we obtain for any
	$\ell\ge0$ and $N\ge1$:
	\begin{align}\label{inversion_formula_1}
		(A_1 \ldots A_N)^{\ell}=
		\sum_{\mathbf{k}\in\Z^{N}_{\ge0}}\frac{R_{\ell-|\mathbf{k}|}}{R_\ell}\prod_{j=1}^{N}T_{k_j}(A_j)
		\big(A_{j+1}A_{j+2}\ldots A_{N}\big)^{k_j},
	\end{align}
	where the sum is over all (unordered)
	nonnegative integer vectors $\mathbf{k}=(k_1,\ldots,k_N)$ of length $N$.
	Here and below $|\mathbf{k}|$ stands for 
	$k_1+\ldots+k_N$.
	Clearly, the sum ranges only over $\mathbf{k}$ with $|\mathbf{k}|\le\ell$.
	Note that if only finitely many of
	the indeterminates $A_j$
	differ from $1$, then one can send $N\to+\infty$ in \eqref{inversion_formula_1}
	and sum over integer vectors $\mathbf{k}$ of arbitrary length.
	
	If we set $A_j:=q^{\Delta\HT_\nu(x+j-1)}$ and send $N\to+\infty$, 
	the left-hand side of \eqref{inversion_formula_1} becomes
	$q^{\ell\,\HT_\nu(x)}$, and in the right-hand side we obtain
	\begin{align*}
		T_{k_j}(A_j)\big(A_{j+1}A_{j+2}\ldots \big)^{k_j}=
		q^{\frac{k_j(k_j+1)}2}\frac{(q^{\Delta\HT_{\nu}(x+j-1)};q^{-1})_{k_j}}{(q;q)_{k_j}}\,
		q^{k_j\HT_\nu(x+j)}.
	\end{align*}
	Therefore, the product of these quantities
	in \eqref{inversion_formula_1} matches formula
	\eqref{corr_via_HT} for 
	$\corr{\mm}\nu q$,
	where the point $x+j-1$ enters the signature
	$\mm$ with multiplicity $k_j\ge0$.
	This yields the desired formula.
\end{proof}
\begin{remark}
	Using a similar approach as in the above lemma, 
	one can write down more complicated formulas expressing
	$\prod_{i=1}^{\ell}q^{\HT_{\nu}(x_i)}$
	for any $x_1\ge \ldots\ge x_{\ell}\ge1$
	through the quantities $\corr{\mm}\nu q$. 
	However, except for the one-point case,
	these expressions do not seem to be convenient for
	computing the $q$-moments.
	Therefore, in \S \ref{sub:verification_for_multi_point_q_moments} 
	below we present a verification-style 
	proof for the multi-point $q$-moments.
\end{remark}

\begin{definition}\label{def:circular_contours_around_0}
	Fix $\ell\in\Z_{\ge1}$.
	Assume that \eqref{stochastic_weights_condition_qsxi} 
	and \eqref{assumptions_one_better} hold. 
	Let $u_i>0$ and $u_i\ne qu_j$ for any $i,j$.
	Then the integration contour $\contn{\bar\UU}$ 
	encircling all $u_i^{-1}$
	is well-defined (see
	Definition \ref{def:cont_gat_U}).
	Let also $c_0$ be a positively oriented circle around
	zero
	which is sufficiently small.
	Let $r>q^{-1}$ be such that 
	$q\contn{\bar\UU}$ does not intersect
	$r^{\ell}c_0$, and $r^{\ell}c_{0}$ does not encircle any of the 
	points $\{\ip_i\SP_i\}_{i\in\Z_{\ge1}}$. Denote
	$\contnz{\bar\UU}j:=\contn{\bar\UU}\cup r^{j}c_0$, where $j=1,\ldots,\ell$.
	See Fig.~\ref{fig:big_Gamma_contours}.
\end{definition}

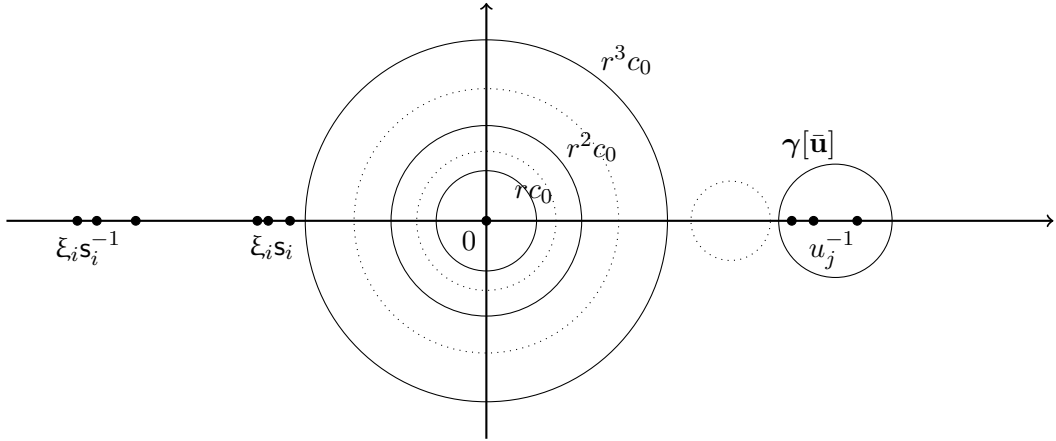
\begin{figure}[htbp]
	\begin{center}
	\begin{tikzpicture}
		[scale=2.9]
		\def\pt{0.02}
		\def\q{.7}
		\def\ss{.56}
		\draw[->, thick] (-2.2,0) -- (2.6,0);
	  	\draw[->, thick] (0,-1) -- (0,1);
	  	\draw[fill] (-.9,0) circle (\pt);
	  	\draw[fill] (-1,0) circle (\pt) node [below, xshift=7pt] 
	  	{$\ip_i\SP_i\phantom{^{-1}}$};
	  	\draw[fill] (-1.05,0) circle (\pt) ;
	  	\draw[fill] (-.9/\ss,0) circle (\pt);
	  	\draw[fill] (-1/\ss,0) circle (\pt);
	  	\draw[fill] (1.5,0) circle (\pt) node [below, xshift=7pt] {$u_j^{-1}$};
	  	\draw[fill] (1.7,0) circle (\pt) ;
	  	\draw[fill] (1.4,0) circle (\pt) ;
	  	\draw[fill] (-1.05/\ss,0) circle (\pt) node [below, xshift=4pt] {$\ip_i\SP_i^{-1}$};
	  	\draw[fill] (0,0) circle (\pt) node [below left] {$0$};
	  	\draw (0,0) circle (.23) node [xshift=18,yshift=10] {$rc_0$};
	  	\def\rrrrr{1.9}
	  	\draw (0,0) circle (.23*\rrrrr) node [xshift=40,yshift=28] {$r^{2}c_0$};
	  	\draw[dotted] (0,0) circle (.24*\rrrrr*\q);
	  	\draw (0,0) circle (.23*\rrrrr*\rrrrr) node [xshift=53,yshift=61] {$r^{3}c_0$};
	  	\draw[dotted] (0,0) circle (.24*\rrrrr*\rrrrr*\q);
	  	\draw (1.6,0) circle (.26) node [below,xshift=-10,yshift=38] {$\contn{\bar\UU}$};
	  	\draw[dotted] (1.6*\q,0) circle (.26*\q);
	\end{tikzpicture}
	\end{center}
  	\caption{A possible choice of integration contours 
  	$\contnz{\bar\UU}1=\contn{\bar\UU}\cup rc_0$, 
  	$\contnz{\bar\UU}2=\contn{\bar\UU}\cup r^{2}c_0$, 
  	and 
  	$\contnz{\bar\UU}3=\contn{\bar\UU}\cup r^{3}c_0$
  	for $\ell=3$ in 
  	Definition \ref{def:circular_contours_around_0}.
  	Contours $q\contn{\bar\UU}$ and $q\contnz{\bar\UU}2$, $q\contnz{\bar\UU}3$ are shown dotted.
	}
  	\label{fig:big_Gamma_contours}
\end{figure}

We are now in a position to compute the one-point $q$-moments:
\begin{proposition}\label{prop:one_moment}
	Assume that \eqref{stochastic_weights_condition_qsxi} 
	and \eqref{assumptions_one_better} hold. 
	Let $u_i>0$ and $u_i\ne qu_j$ for any $i,j$.
	Then for any $\ell\in\Z_{\ge0}$ and $x\in\Z_{\ge1}$ 
	we have
	\begin{multline}\label{one_moment}
		\E_{\UU;\RHO}q^{\ell\,\HT_{\nu}(x)}=
		q^{\frac{\ell(\ell-1)}2}
		\oint\limits_{\contnz{\bar\UU}1}\frac{d w_1}{2\pi\i}
		\ldots
		\oint\limits_{\contnz{\bar\UU}\ell}\frac{d w_\ell}{2\pi\i}
		\prod_{1\le \aind<\bind\le \ell}\frac{w_\aind-w_\bind}{w_\aind-qw_\bind}
		\\\times
		\prod_{i=1}^{\ell}\bigg(
		w_i^{-1}
		\prod_{j=1}^{x-1}
		\frac{\ip_j-\SP_jw_i}{\ip_j-\SP_j^{-1}w_i}
		\prod_{j=1}^{n}\frac{1-qu_jw_i}{1-u_jw_i}
		\bigg).
	\end{multline}
\end{proposition}
\begin{proof}
	Taking the expectation with respect to 
	$\MM_{\UU;\RHO}$ in both sides of
	\eqref{moments_from_qcorr} and using \eqref{qcorr_symmetrized} 
	in the right-hand side,
	we obtain 
	\begin{multline*}
		\E_{\UU;\RHO}q^{\ell\,\HT_{\nu}(x)}=
		\sum_{k=0}^{\ell}
		\binom{\ell}{k}_{q}
		\frac{q^{\frac{k(k-1)}{2}}(q;q)_{k}}{(1-q)^{k}k!}
		\sum_{\mm_1\ge \ldots\ge \mm_k\ge x-1}\,
		\oint\limits_{\contn{\bar\UU}}\frac{d w_1}{2\pi\i}
		\ldots
		\oint\limits_{\contn{\bar\UU}}\frac{d w_k}{2\pi\i}
		\prod_{1\le \aind\ne \bind\le k}\frac{w_\aind-w_\bind}{w_\aind-qw_\bind}
		\\\times
		(-\sh_1\SPB)^{\mm}\frac{\F^{\conj}_{\mm}
		(w_1^{-1},\ldots,w_k^{-1}\md\sh_1\ipb,\sh_1\SPB)}{w_1 \ldots w_k}
		\prod_{i=1}^{k}
		\prod_{j=1}^{n}\frac{1-qu_jw_i}{1-u_jw_i}.
	\end{multline*}
	Because $\mm_k\ge x-1$, we can subtract $(x-1)$ from all parts of 
	$\mm$. We readily have
	\begin{multline*}
		(-\sh_1\SPB)^{\mm}
		\F^{\conj}_{\mm}
		(w_1^{-1},\ldots,w_k^{-1}\md\sh_1\ipb,\sh_1\SPB)
		=
		(-\SP_1)^{k}\ldots(-\SP_{x-1})^{k}
		\prod_{i=1}^{k}\prod_{j=1}^{x-1}
		\frac{\ip_jw_i^{-1}-\SP_j}{1-\SP_j\ip_jw_i^{-1}}
		\\\times
		\underbrace{
		(-\sh_{x}\SPB)^{\mm-(x-1)^{k}}\F^{\conj}_{\mm-(x-1)^{k}}
		(w_1^{-1},\ldots,w_k^{-1}\md\sh_{x}\ipb,
		\sh_{x}\SPB)
		}_
		{
		\textnormal{
		$\MM_{(w_1^{-1},\ldots,w_{k}^{-1});\RHO}
		\big(\mm-(x-2)^{k}\md\sh_{x-1}\ipb,\sh_{x-1}\SPB\big )$
		by \eqref{MM_U_RHO_particular}}},
	\end{multline*}
	where we have also used the fact that $\mm_k-(x-2)\ge1$.
	The probability weight 
	$\MM_{(w_1^{-1},\ldots,w_{k}^{-1});\RHO}$
	above is the only thing which now depends on $\mm$,
	and the summation over all $\mm$ of these weights
	gives $1$. 
	This summation can be performed 
	under the integral 
	because on the contour $\contn{\bar\UU}$ we have $\Re(w_j)>0$,
	and so conditions \eqref{admissibility_RHO_conditions}
	with $u_j$ replaced by $w_j^{-1}$
	hold for all our values of $\ipb$ and $\SPB$.
	We see that this summation over $\mm$ yields
	\begin{multline*}
		\E_{\UU;\RHO}q^{\ell\,\HT_{\nu}(x)}=
		\sum_{k=0}^{\ell}
		\binom{\ell}{k}_{q}
		q^{\frac{k(k-1)}{2}}
		\oint\limits_{\contn{\bar\UU}}\frac{d w_1}{2\pi\i}
		\ldots
		\oint\limits_{\contn{\bar\UU}}\frac{d w_k}{2\pi\i}
		\prod_{1\le \aind<\bind\le k}\frac{w_\aind-w_\bind}{w_\aind-qw_\bind}
		\\\times
		\frac{1}{w_1 \ldots w_k}
		\prod_{i=1}^{k}\prod_{j=1}^{x-1}
		\frac{\ip_j-\SP_jw_i}{\ip_j-\SP_j^{-1}w_i}
		\prod_{i=1}^{k}
		\prod_{j=1}^{n}\frac{1-qu_jw_i}{1-u_jw_i}.
	\end{multline*}
	Here we have applied the symmetrization formula (footnote$^{\ref{symm_footnote}}$)
	to rewrite $(q;q)_k/(1-q)^k$, which canceled
	the factor $1/k!$ and half of the product over $\aind\ne \bind$.

	Finally, the summation over $k$ in the above formula can be eliminated 
	by changing the integration contours
	with the
	help of \cite[Lemma\;4.21]{BorodinCorwinSasamoto2012} 
	(which we recall as Lemma \ref{lemma:four_twenty_one} below for convenience). 
	This completes the proof of the desired identity \eqref{one_moment}.
\end{proof}
\begin{lemma}[{\cite[Lemma\;4.21]{BorodinCorwinSasamoto2012}}]\label{lemma:four_twenty_one}
	Let $\ell\ge1$ and $f(w)$ with $f(0)=1$
	be a meromorphic function in $\C$
	having no poles in a disc around $0$.
	Then we have
	\begin{multline}
		\oint\limits_{\contnz{\bullet}1}\frac{dw_1}{2\pi\i}
		\ldots
		\oint\limits_{\contnz{\bullet}\ell}\frac{dw_\ell}{2\pi\i}
		\prod_{1\le \aind<\bind\le \ell}\frac{w_\aind-w_\bind}{w_\aind-qw_\bind}
		\prod_{i=1}^{\ell}\frac{f(w_i)}{w_i}
		\\=\sum_{k=0}^{\ell}
		\binom{\ell}{k}_{q}
		q^{\frac12k(k-1)-\frac12\ell(\ell-1)}
		\oint\limits_{\contn{\bullet}}\frac{dw_1}{2\pi\i}
		\ldots
		\oint\limits_{\contn{\bullet}}\frac{dw_k}{2\pi\i}
		\prod_{1\le \aind<\bind\le k}\frac{w_\aind-w_\bind}{w_\aind-qw_\bind}
		\prod_{i=1}^{k}\frac{f(w_i)}{w_i},
		\nonumber
	\end{multline}
	where 
	as $\contn{\bullet}$ we can take an arbitrary closed contour not encircling $0$,
	and all other contours and conditions on them are analogous to 
	Definition \ref{def:circular_contours_around_0}.
\end{lemma}

\begin{remark}[Fredholm determinants]\label{rmk:Fredholm}
	Using a general approach outlined in \cite{BorodinCorwinSasamoto2012},
	the
	one-point $q$-moment formula
	of Proposition \ref{prop:one_moment}
	(as well as its degenerations discussed in \S \ref{sec:degenerations_of_moment_formulas})
	can be employed to obtain
	Fredholm determinantal
	expressions for the $q$-Laplace transform 
	$\E_{\UU;\RHO}\left({1}/{(\zeta q^{\HT_{\nu}(x)};q)_{\infty}}\right)$
	of the 
	height function,
	which may be suitable for asymptotic analysis.
	We will not pursue this direction here.
\end{remark}


\subsection{Multi-point $q$-moment formula} 
\label{sub:verification_for_multi_point_q_moments}

By analogy with existing multi-point
$q$-moment formulas for
related systems\footnote{Namely, $q$-TASEPs
\cite{BorodinCorwinSasamoto2012},
\cite{BorodinCorwin2013discrete},
$q$-Hahn TASEP
\cite{Corwin2014qmunu}, and the
homogeneous stochastic higher spin six vertex model
\cite{CorwinPetrov2015}. Note that
however all these systems start with infinitely many particles at
the leftmost location, and in our system a new particle is always
added at location $1$, so that the corresponding 
degenerations of
Theorem \ref{thm:multi_moments} do not follow from those works.}
we can formulate a generalization 
of Proposition \ref{prop:one_moment}:
\begin{theorem}\label{thm:multi_moments}
	Assume that \eqref{stochastic_weights_condition_qsxi} 
	and \eqref{assumptions_one_better} hold. 
	Let $u_i>0$ and $u_i\ne qu_j$ for any $i,j=1,\ldots,n$.
	Then for any integers $x_1\ge \ldots\ge x_{\ell}\ge1$
	the corresponding $q$-moment of the
	dynamics $\Xp_{\{u_t\};\RHO}$ at $\text{time}=n$ is given by
	\begin{multline}\label{multi_moments}
		\E_{\UU;\RHO}\prod_{i=1}^{\ell}q^{\HT_{\nu}(x_i)}=
		q^{\frac{\ell(\ell-1)}2}
		\oint\limits_{\contnz{\bar\UU}1}\frac{d w_1}{2\pi\i}
		\ldots
		\oint\limits_{\contnz{\bar\UU}\ell}\frac{d w_\ell}{2\pi\i}
		\prod_{1\le \aind<\bind\le \ell}\frac{w_\aind-w_\bind}{w_\aind-qw_\bind}
		\\\times
		\prod_{i=1}^{\ell}\bigg(
		w_i^{-1}
		\prod_{j=1}^{x_i-1}
		\frac{\ip_j-\SP_jw_i}{\ip_j-\SP_j^{-1}w_i}
		\prod_{j=1}^{n}\frac{1-qu_jw_i}{1-u_jw_i}
		\bigg),
	\end{multline}
	where the integration contours
	are described in Definition \ref{def:circular_contours_around_0}.
\end{theorem}
\begin{corollary}\label{cor:multi_moments_fused}
	Under 
	\eqref{stochastic_weights_condition_qsxi}--\eqref{stochastic_weights_condition_u}
	and \eqref{assumptions_one_better},
	let the parameters $\UU$ have the form \eqref{U_fused_for_contours}. Then the 
	$q$-moments 
	$\E_{\UU;\RHO}\prod_{i=1}^{\ell}q^{\HT_{\nu}(x_i)}$
	are given by the same formula as \eqref{multi_moments},
	but with integration contours 
	$w_j\in\contqiz{\bar\UU}j:=\contqi{\bar\UU}j\cup r^{j}c_0$.
\end{corollary}
\begin{proof}[Proof of Corollary \ref{cor:multi_moments_fused}]
	Assume that Theorem \ref{thm:multi_moments} holds.
	We argue as in the proof of
	case 3 in 
	Theorem \ref{thm:qcorr}, by
	first taking $\UU$ with $u_i>0$
	and
	$q\cdot \max_{i}u_i<\min_i u_i$,
	which allows to immediately pass to 
	the nested contours $\contqiz{\bar\UU}j$
	in \eqref{multi_moments}. 
	Then we can move $u_2^{-1}$ outside 
	$\contqi{\bar\UU}1$
	but still inside
	$\contqi{\bar\UU}2$, set $u_2=qu_1$,
	and continue specializing the rest of 
	$\UU$ to \eqref{U_fused_for_contours} in a similar way.
	This specialization inside the integral will coincide
	with the same specialization of the 
	left-hand side of \eqref{multi_moments}, 
	and thus the corollary is established.
\end{proof}
\begin{remark}
	In Theorem \ref{thm:multi_moments}
	(as well as in case 2 of Theorem \ref{thm:qcorr})
	the 
	conditions \eqref{stochastic_weights_condition_qsxi}
	and \eqref{assumptions_one_better} can be partially dropped or 
	replaced by more general ones,
	because the integration contour 
	$\contn{\bar\UU}$ is defined only using the $u_i$'s. 
	Here we will not discuss more 
	general values of $q$, $\SPB$, or $\ipb$.
\end{remark}

The rest of this subsection is devoted to the proof of 
Theorem \ref{thm:multi_moments}. The proof is of verification type: we start with 
the nested contour integral in the  
right-hand side of \eqref{multi_moments}, and  
rewrite it as an expectation with respect to
$\MM_{\UU;\RHO}$.

\begin{lemma}\label{lemma:rid_of_around_0}
	Under the assumptions of Theorem \ref{thm:multi_moments}, 
	the collection of
	identities \eqref{multi_moments} 
	(for all $\ell\ge1$ and all $x_1\ge \ldots\ge x_\ell\ge1$)
	follows from a collection of identities of the following form:
	\begin{multline}\label{multi_moments_rid_of_around_0}
		\E_{\UU;\RHO}
		\prod_{i=1}^{\ell}
		\big(q^{i-1}-q^{\HT_\nu(x_i)}\big)
		=
		(-1)^{\ell}q^{\frac{\ell(\ell-1)}2}
		\oint\limits_{\contn{\bar\UU}}\frac{d w_1}{2\pi\i}
		\ldots
		\oint\limits_{\contn{\bar\UU}}\frac{d w_\ell}{2\pi\i}
		\prod_{1\le \aind<\bind\le \ell}\frac{w_\aind-w_\bind}{w_\aind-qw_\bind}
		\\\times
		\prod_{i=1}^{\ell}\bigg(
		w_i^{-1}
		\prod_{j=1}^{x_i-1}
		\frac{\ip_j-\SP_jw_i}{\ip_j-\SP_j^{-1}w_i}
		\prod_{j=1}^{n}\frac{1-qu_jw_i}{1-u_jw_i}
		\bigg).
	\end{multline}
	That is, removing the parts of the contours around $0$
	leads to a modification of 
	the left-hand side, as shown above.
\end{lemma}
This lemma should also hold in the opposite direction (that identities
\eqref{multi_moments} imply \eqref{multi_moments_rid_of_around_0}), 
but we do not need this statement.
\begin{proof}
	The right-hand side of \eqref{multi_moments} can be written as
	\begin{align}\label{rid_of_around_0_proof}
		q^{\frac{\ell(\ell-1)}2}\oint\limits_{\contnz{\bar\UU}1}\frac{dw_1}{2\pi\i}
		\ldots
		\oint\limits_{\contnz{\bar\UU}\ell}\frac{dw_\ell}{2\pi\i}
		\prod_{1\le \aind<\bind\le \ell}\frac{w_\aind-w_\bind}{w_\aind-qw_\bind}
		\prod_{i=1}^{\ell}\frac{f_{x_i}(w_i)}{w_i},
	\end{align}
	where each
	$f_x(w)$, $x\in\Z_{\ge1}$,
	is a meromorphic (in fact, rational) function without 
	poles in a disc around $0$,
	and $f_x(0)=1$.

	Split the integral in \eqref{rid_of_around_0_proof}
	into $2^{n}$ integrals indexed by subsets $\IS\subseteq\{1,\ldots,\ell\}$
	determining that $w_i$ for $i\notin \IS$
	are integrated around $0$, while other $w_i$'s
	are integrated over $\contn{\bar\UU}$. 
	Let $|\IS|=k$, $k=0,\ldots,\ell$, and also denote
	$\|\IS\|:=\sum_{i\in \IS}i$.
	Let 
	$\IS=\{i_1<\ldots<i_k\}$ and 
	$\{1,2,\ldots,\ell\}\setminus \IS=\{p_1<\ldots<p_{\ell-k}\}$.
	The contours around $0$ 
	(corresponding to $w_{p_{j}}$) 
	can be
	shrunk to $0$ in the order $p_1,\ldots,p_{\ell-k}$ 
	without crossing any other poles, 
	and each such 
	integration produces the factor $q^{-(\ell-p_j)}$ coming from the 
	cross-product over $\aind<\bind$.
	Thus, 
	\eqref{rid_of_around_0_proof} becomes 
	(after renaming $w_{i_j}=z_j$)
	\begin{align}\label{rid_of_around_0_proof1}
		\sum_{k=0}^{\ell}\,\sum_{{\IS=\{i_1<\ldots<i_k\}\subseteq\{1,\ldots,\ell\}}}
		q^{k\ell-
		\|\IS\|}
		\oint\limits_{\contn{\bar\UU}}\frac{dz_1}{2\pi\i}
		\ldots
		\oint\limits_{\contn{\bar\UU}}\frac{dz_{k}}{2\pi\i}
		\prod_{1\le \aind<\bind\le k}\frac{z_\aind-z_\bind}{z_\aind-qz_\bind}
		\prod_{j=1}^{k}\frac{f_{x_{i_j}}(z_j)}{z_j}.
	\end{align}
	The above summation now involves integrals as in the 
	right-hand side of \eqref{multi_moments_rid_of_around_0}. 
	If the latter identity holds, then we can 
	rewrite each such integral as a certain expectation as in the 
	left-hand side of \eqref{multi_moments_rid_of_around_0}. 
	Relation \eqref{multi_moments} 
	now follows from a formal identity 
	in indeterminates $\hat{q},X_1,\ldots,X_\ell$:
	\begin{align}\label{rid_of_around_0_proof2}
		\sum_{k=0}^{\ell}\,\sum_{{\IS=\{i_1<\ldots<i_k\}\subseteq\{1,\ldots,\ell\}}}
		\hat {q}^{\frac{(\ell-k)(\ell-k+1)}2}
		(X_{i_1}-\hat{q}^{i_1})(X_{i_2}-\hat{q}^{i_2-1})
		\ldots(X_{i_k}-\hat{q}^{i_k-k+1})
		=X_1 \ldots X_\ell,
	\end{align}
	where we have matched $\hat{q}$ to $q^{-1}$ in \eqref{rid_of_around_0_proof}
	and \eqref{rid_of_around_0_proof1}.
	To establish \eqref{rid_of_around_0_proof2}, 
	observe that both its sides are linear in $X_1$,
	and so it suffices to show that the identity holds 
	at two points, say, $X_1=\hat q$ and $X_1=\infty$. 
	Substituting each of these values into \eqref{rid_of_around_0_proof2}
	leads to an equivalent
	identity with $\ell$ replaced by $\ell-1$.
	Namely, for $X_1=\hat q$ we obtain
	\begin{align*}
		\sum_{k=0}^{\ell-1}\,\sum_{{\IS=\{1<i_1<\ldots<i_k\}\subseteq\{1,\ldots,\ell\}}}
		\hat {q}^{\frac{(\ell-1-k)(\ell-1-k+1)}2}
		(\hat q^{-1}X_{i_1}-\hat{q}^{i_1-1})
		\ldots(\hat q^{-1}X_{i_k}-\hat{q}^{i_k-k})
		=\hat q^{1-\ell}X_2 \ldots X_\ell,
	\end{align*}
	which becomes \eqref{rid_of_around_0_proof2} with $\ell-1$
	after setting $Y_i=\hat q^{-1}X_{i+1}$.
	Dividing \eqref{rid_of_around_0_proof2} 
	by $X_1$ and letting $X_1\to\infty$, we obtain
	\begin{align*}
		\sum_{k=1}^{\ell}\,\sum_{{\IS=\{1=i_1<\ldots<i_k\}\subseteq\{1,\ldots,\ell\}}}
		\hat {q}^{\frac{((\ell-1)-(k-1))((\ell-1)-(k-1)+1)}2}
		(X_{i_2}-\hat{q}^{i_2-1})
		\ldots(X_{i_k}-\hat{q}^{i_k-k+1})
		=X_2 \ldots X_\ell,
	\end{align*}
	which becomes \eqref{rid_of_around_0_proof2} 
	with $\ell-1$
	after
	setting $Y_i=X_{i+1}$.
	Thus,
	\eqref{rid_of_around_0_proof2} follows by induction,
	which implies the lemma.
\end{proof}

Below in this subsection we will assume that the $u_j$'s are pairwise distinct. 
If \eqref{multi_moments} and 
\eqref{multi_moments_rid_of_around_0} 
hold for distinct $u_j$'s, then
when some of the $u_j$'s
coincide the same formulas can be obtained by 
a simple substitution. Indeed, this is because both sides of each of the
identities are
a priori rational functions in
the $u_j$'s (cf. footnote$^{\ref{footnote:rational_continuation}}$).

Denote the right-hand side of 
\eqref{multi_moments_rid_of_around_0}
by $R(u_1,\ldots,u_n)$. First, 
let us show that there exists a decomposition of
$R(u_1,\ldots,u_n)$ into the functions $\F_\la^{\conj}$:
\begin{lemma}\label{lemma:expansion_r_la_exists}
	For certain restricted values of the parameters $\ipb$, $\SPB$, and $u_1,\ldots,u_n$, and for $q$ sufficiently close to $1$,
	the integral in the right-hand side of \eqref{multi_moments_rid_of_around_0} can
	be written as
	\begin{align}\label{expancoeff}
		R(u_1,\ldots,u_n)=
		\sum_{\la\in\signp n}
		\expancoeff{\la} \F_\la^{\conj}(u_1,\ldots,u_n\md\sh_1\ipb,\sh_1\SPB),
	\end{align}
	where the sum over $\la$ converges uniformly in 
	$u_j\in\contq{\ipbb\SPB}j$, $j=1,\ldots,n$.
\end{lemma}
\begin{proof}
	Write the product 
	in \eqref{multi_moments_rid_of_around_0} 
	as a sum over $\mu\in\signp n$
	using the 
	Cauchy identity (Corollary \ref{cor:usual_Cauchy}):
	\begin{align}
		\prod_{i=1}^{\ell}\prod_{j=1}^{n}\frac{1-qu_jw_i}{1-u_jw_i}=
		\frac1{(\SP_0^{2};q)_{n}}
		\prod_{i=1}^{n}\frac{1-\ip_0\SP_0 u_i}{1-\ip_0\SP_0^{-1}u_i}
		\sum_{\mu\in\signp n\colon \mu_n\ge1}
		\F^{\conj}_{\mu}(u_1,\ldots,u_n\md\ipb,\SPB)
		\G_{\mu}(\RHO,w_1,\ldots,w_\ell\md\ipbb,\SPB),
		\label{expancoeff_proof}
	\end{align}
	where we also used the specialization $\RHO$
	\eqref{RHO_spec} and
	Proposition \ref{prop:Gnu_spec_RHO_w}.
	This is possible if $\adm{u_i}{w_j}$ for all $i,j$,
	and the $u_i$'s satisfy \eqref{admissibility_RHO_conditions}.
	These conditions 
	can be achieved by restricting the parameters 
	$u_i$, $\ipb$, and $\SPB$,
	and deforming the contour $\contn{\bar\UU}$
	similarly to the proofs of 
	Theorem \ref{thm:spec_orthogonality} and
	Proposition \ref{prop:prelim_qcorr} (all
	the $\ip_j\SP_j^{-1}$'s should be close together, and the $u_i^{-1}$'s
	should be close to these points). The integrand in \eqref{multi_moments_rid_of_around_0}
	is regular at each $\ip_x\SP_x^{-1}$, so the $w_j$-contour
	$\contn{\bar\UU}$ can be deformed to $\contn{\ipb\SPBB}$. Note that we do not need to restrict the
	parameter $q$ yet, because the $w_j$'s lie on the same contour.
	
	Now for the restricted $\{u_i\}$, $\ipb$, and $\SPB$,
	the sum over $\mu$ in \eqref{expancoeff_proof} 
	converges uniformly in $w_i$ belonging to the deformed contours $\contn{\ipb\SPBB}$. 
	Thus, the integration
	in the $w_j$'s can be performed for each $\mu$ separately.
	These integrals 
	involving $\G_{\mu}(\RHO,w_1,\ldots,w_\ell\md\ipbb,\SPB)$
	obviously do not introduce any new 
	dependence on the $u_i$'s. 
	Therefore, the right-hand side of \eqref{expancoeff_proof}
	depends on the $u_i$'s only through $\F^{\conj}_{\mu-1^{n}}(u_1,\ldots,u_n\md\sh_1\ipb,\sh_1\SPB)$
	(cf.~\eqref{F_shifts}), which yields expansion \eqref{expancoeff}.
	
	It remains to show the uniform convergence of \eqref{expancoeff}
	in $u_j\in\contq{\ipbb\SPB}j$.
	We see that the coefficients in \eqref{expancoeff} have the form
	\begin{multline*}
		\expancoeff{\mu-1^{n}}=
		\frac{(-1)^{\ell}(-\SP_0)^{n}q^{\frac{\ell(\ell-1)}2}}{(\SP_0^{2};q)_{n}}
		\oint\limits_{\contn{\ipb\SPBB}}\frac{d w_1}{2\pi\i}
		\ldots
		\oint\limits_{\contn{\ipb\SPBB}}\frac{d w_\ell}{2\pi\i}
		\prod_{1\le \aind<\bind\le \ell}\frac{w_\aind-w_\bind}{w_\aind-qw_\bind}
		\\\times
		\G_{\mu}(\RHO,w_1,\ldots,w_\ell\md\ipbb,\SPB)
		\prod_{i=1}^{\ell}
		\bigg(w_i^{-1}
		\prod_{j=1}^{x_i-1}
		\frac{\ip_j-\SP_jw_i}{\ip_j-\SP_j^{-1}w_i}\bigg).
	\end{multline*}
	One can see that these coefficients grow in $\mu$
	not faster than of order $\exp\{c|\mu|\}$ for some constant $c>0$.
	To ensure uniform convergence one thus needs estimates of the form
	$\left|\frac{u_i-\ip_x^{-1}\SP_x}{u_i-\ip_x^{-1}\SP_x^{-1}}\right|<c_1<1$,
	that is, the points $u_i$ must be close to the $\ip_x^{-1}\SP_x$'s
	and on the contours $\contq{\ipbb\SPB}i$ at the same time. 
	This can be achieved by restricting the parameters $\ipb$ and $\SPB$
	further if necessary, and also by taking $q$ sufficiently close to $1$, 
	so that the contours $\contq{\ipbb\SPB}i$
	are sufficiently close to the $\ip_x^{-1}\SP_x$'s.
\end{proof}
The integral formula for the coefficients $\expancoeff\la$ in the proof of the above lemma 
does not seem to be convenient for their direct computation. 
We will instead rewrite 
$R(u_1,\ldots,u_n)$
by integrating over the $w_i$'s in \eqref{multi_moments_rid_of_around_0},
and then employ orthogonality of the functions $\F_\la$ 
to extract the $\expancoeff\la$'s. This will imply that 
$R(u_1,\ldots,u_n)$ is equal to the left-hand side of \eqref{multi_moments_rid_of_around_0},
yielding Theorem \ref{thm:multi_moments}.

\medskip

The 
integral in \eqref{multi_moments_rid_of_around_0}
can be computed by taking residues
at $w_i=u_{\sigma(i)}^{-1}$ for all $i=1,\ldots,\ell$,
where $\sigma$ runs over all maps $\{1,\ldots,\ell\}\to\{1,\ldots,n\}$
(we will see below that other residues do not participate).
Denote the residue corresponding to $\sigma$ by $\Res_{\sigma}$,
and also denote
$\JS:=\sigma(\{1,\ldots,\ell\})$. Because of the factors $w_{\aind}-w_{\bind}$,
the same $u_{\sigma(i)}^{-1}$ cannot participate twice, so $\sigma$ must be injective.
Thus,
in contrast with 
\eqref{multi_moments},
the integral in the right-hand side of 
\eqref{multi_moments_rid_of_around_0} vanishes if $\ell>n$.
Note however that since $\HT_\nu(x)\le n$ for all $x\in\Z_{\ge1}$, 
the product in the left-hand side also vanishes for $\ell>n$,
as it should be. Therefore, it suffices to consider only the case~$\ell\le n$.

The integral in \eqref{multi_moments_rid_of_around_0} can be
written in the form
\begin{align*}
	(-1)^{\ell}q^{\frac{\ell(\ell-1)}2}
	\oint\limits_{\contn{\bar\UU}}\frac{d w_1}{2\pi\i}
	\ldots
	\oint\limits_{\contn{\bar\UU}}\frac{d w_\ell}{2\pi\i}
	\prod_{1\le \aind<\bind\le \ell}\frac{w_\aind-w_\bind}{w_\aind-qw_\bind}
	\prod_{i=1}^{\ell}\bigg(
	f_{x_i}(w_i;\sigma)
	\prod_{j=1}^{\ell}\frac{1-qu_{\sigma(j)}w_i}{1-u_{\sigma(j)}w_i}
	\bigg),
\end{align*}
with 
\begin{align*}
	f_x(w;\sigma):=w^{-1}
	\prod\limits_{j=1}^{x-1}
	\frac{\ip_j-\SP_jw}{\ip_j-\SP_j^{-1}w}\prod\limits_{j\notin
	\JS}\frac{1-qu_jw}{1-u_jw}.
\end{align*}
Taking residues at $w_1=u_{\sigma(1)}^{-1},\ldots,
w_\ell=u_{\sigma(\ell)}^{-1}$ (in this order),
we see that
\begin{align}
	\Res\nolimits_{\sigma}&=q^{\frac{\ell(\ell-1)}2}
	f_{x_1}(u_{\sigma(1)}^{-1};\sigma)
	\frac{1-q}{u_{\sigma(1)}}(-1)^{\ell-1}
	\prod_{j=2}^{\ell}
	\frac{u_{\sigma(1)}-qu_{\sigma(j)}}{u_{\sigma(1)}-u_{\sigma(j)}}
	\nonumber\\&\hspace{70pt}\times
	\Res_{w_\ell=u_{\sigma(\ell)}^{-1}}
	\ldots
	\Res_{w_2=u_{\sigma(2)}^{-1}}\,
	\prod_{2\le \aind<\bind\le \ell}\frac{w_\aind-w_\bind}{w_\aind-qw_\bind}
	\prod_{i=2}^{\ell}\bigg(
	f_{x_i}(w_i;\sigma)
	\prod_{j=2}^{\ell}\frac{1-qu_{\sigma(j)}w_i}{1-u_{\sigma(j)}w_i}
	\bigg)
	\nonumber\\&=q^{\frac{\ell(\ell-1)}2}
	f_{x_1}(u_{\sigma(1)}^{-1};\sigma)
	f_{x_2}(u_{\sigma(2)}^{-1};\sigma)
	\frac{(1-q)^{2}}{u_{\sigma(1)}u_{\sigma(2)}}
	(-1)^{\ell-2}
	\prod_{j=2}^{\ell}
	\frac{u_{\sigma(1)}-qu_{\sigma(j)}}{u_{\sigma(1)}-u_{\sigma(j)}}
	\prod_{j=3}^{\ell}
	\frac{u_{\sigma(2)}-qu_{\sigma(j)}}{u_{\sigma(2)}-u_{\sigma(j)}}
	\nonumber\\&\hspace{70pt}\times
	\Res_{w_\ell=u_{\sigma(\ell)}^{-1}}
	\ldots
	\Res_{w_3=u_{\sigma(3)}^{-1}}\,
	\prod_{3\le \aind<\bind\le \ell}\frac{w_\aind-w_\bind}{w_\aind-qw_\bind}
	\prod_{i=3}^{\ell}\bigg(
	f_{x_i}(w_i;\sigma)
	\prod_{j=3}^{\ell}\frac{1-qu_{\sigma(j)}w_i}{1-u_{\sigma(j)}w_i}
	\bigg)
	\nonumber\\&=\textnormal{etc.}
	\nonumber\\&=
	(1-q)^{\ell}q^{\frac{\ell(\ell-1)}2}\prod_{i=1}^{\ell}
	\frac{f_{x_i}(u_{\sigma(i)}^{-1};\sigma)}{u_{\sigma(i)}}
	\prod_{1\le \aind<\bind\le \ell}
	\frac{u_{\sigma(\aind)}-qu_{\sigma(\bind)}}{u_{\sigma(\aind)}-u_{\sigma(\bind)}}
	\nonumber\\&=
	(1-q)^{\ell}q^{\frac{\ell(\ell-1)}2}
	\prod_{i=1}^{\ell}
	\prod\limits_{j=1}^{x_i-1}
	\frac{\ip_ju_{\sigma(i)}-\SP_j}{\ip_ju_{\sigma(i)}-\SP_j^{-1}}
	\prod_{\aind\in \JS,\;
	\bind\notin \JS}
	\frac{u_{\aind}-qu_{\bind}}
	{u_{\aind}-u_{\bind}}
	\prod_{1\le \aind<\bind\le \ell}
	\frac{u_{\sigma(\aind)}-qu_{\sigma(\bind)}}{u_{\sigma(\aind)}-u_{\sigma(\bind)}}.
	\nonumber
\end{align}
In particular, we see that each step does not introduce 
any new poles inside
the integration contours besides $u_{j}^{-1}$.
Therefore,
\begin{align}\label{R_function_definition}
	R(u_1,\ldots,u_n)=\sum_{\sigma\colon
	\{1,\ldots,\ell\}\to\{1,\ldots,n\}}
	\Res\nolimits_\sigma(u_1,\ldots,u_n),
\end{align}
with 
$\Res\nolimits_\sigma$ as above.
This identity clearly holds for generic complex $u_1,\ldots,u_n$
(and not only for $u_i>0$)
because both
sides are rational functions in the $u_i$'s.

\medskip

Let us now apply the inverse Plancherel
transform $\Pltransi n$ of
Definition \ref{def:Plancherel_transforms}
(but without the factor $\conj_{\SPB}(\la)$ and with the shifted parameters
$\sh_1\ipb$, $\sh_1\SPB$)
to $R(u_1,\ldots,u_n)$ to recover the coefficients
$\expancoeff\la$ in \eqref{expancoeff}. This is possible
under the restrictions of Lemma \ref{lemma:expansion_r_la_exists}
because the series in the right-hand side 
of \eqref{expancoeff} converges uniformly
in the $u_i$'s on the contours involved
in $\Pltransi n$.
The application of this slightly modified transform to 
$R(u_1,\ldots,u_n)$
written as \eqref{R_function_definition}
can be 
performed separately for each $\sigma$, and 
the result is the following:
\begin{lemma}\label{lemma:computing_multi_moment_integral}
	Under
	\eqref{stochastic_weights_condition_qsxi} and 
	\eqref{assumptions_one_better},
	for any $\sigma\colon
	\{1,\ldots,\ell\}\to\{1,\ldots,n\}$
	and $\la\in\signp n$, we have
	\begin{multline}\label{computing_multi_moment_integral}
		(1-q)^{-n}
		\oint\limits_{\contq {\ipbb\SPB}1}\frac{d u_1}{2\pi\i}
		\ldots
		\oint\limits_{\contq {\ipbb\SPB}n}\frac{d u_n}{2\pi\i}
		\prod_{1\le \aind<\bind\le n}\frac{u_\aind-u_\bind}{u_\aind-qu_\bind}
		\Res\nolimits_\sigma(u_1,\ldots,u_n)
		\prod_{i=1}^{n}u_i^{-1}\pow_{\la_{i}}(u_{i}^{-1}\md\sh_1\ipbb,\sh_1\SPB)
		\\=\mathbf{1}_{\{\textnormal{$\la_{\sigma(i)}\ge x_i-1$ 
		for all $i$}\}}
		\cdot (-\sh_1\SPB)^{\la}
		q^{\inv(\sigma)}q^{\sigma(1)+\ldots+\sigma(\ell)}q^{-\ell}(1-q)^{\ell}.
	\end{multline}
	Here the integration contours are described
	in Definition \ref{def:orthogonality_contours},
	and $\inv(\sigma)$ is the number of 
	inversions in $\sigma$, i.e., the 
	number of pairs $(i,j)$ with 
	$i<j$ and $\sigma(i)>\sigma(j)$.
\end{lemma}
\begin{proof}
	We need to compute
	\begin{multline*}
		\oint\limits_{\contq {\ipbb\SPB}1}\frac{d u_1}{2\pi\i}
		\ldots
		\oint\limits_{\contq {\ipbb\SPB}n}\frac{d u_n}{2\pi\i}
		\prod_{1\le \aind<\bind\le n}\frac{u_\aind-u_\bind}{u_\aind-qu_\bind}
		\prod_{\aind\in \JS,\;
		\bind\notin \JS}
		\frac{u_{\aind}-qu_{\bind}}
		{u_{\aind}-u_{\bind}}
		\prod_{1\le \aind<\bind\le \ell}
		\frac{u_{\sigma(\aind)}-qu_{\sigma(\bind)}}{u_{\sigma(\aind)}-u_{\sigma(\bind)}}\\\times
		\prod_{i=1}^{\ell}\bigg(
		u_{\sigma(i)}^{-1}\pow_{\la_{\sigma(i)}}(u_{\sigma(i)}^{-1}\md\sh_1\ipbb,\sh_1\SPB)
		\prod\limits_{j=1}^{x_i-1}
		\frac{\ip_ju_{\sigma(i)}-\SP_j}{\ip_ju_{\sigma(i)}-\SP_j^{-1}}
		\bigg)
		\prod_{\bind\notin\JS}u_\bind^{-1}\pow_{\la_{\bind}}(u_{\bind}^{-1}\md\sh_1\ipbb,\sh_1\SPB).
	\end{multline*}
	Observe the following:
	\begin{itemize}
		\item The product $\prod\limits_{j=1}^{x_i-1}
		\frac{\ip_ju_{\sigma(i)}-\SP_j}{\ip_ju_{\sigma(i)}-\SP_j^{-1}}$, $x_i\ge1$,
		has zeros at $u_{\sigma(i)}=\SP_j\ip_j^{-1}$ for $j=1,\ldots,x_i-1$,
		and poles at
		$u_{\sigma(i)}=\SP_j^{-1}\ip_j^{-1}$ for $j=1,\ldots,x_i-1$.
		\item 
		The term $u_{\sigma(i)}^{-1}\pow_{\la_{\sigma(i)}}(u_{\sigma(i)}^{-1}\md\sh_1\ipbb,\sh_1\SPB)$,
		$\la_{\sigma(i)}\ge0$,
		has zeros at
		$u_{\sigma(i)}=\SP_j^{-1}\ip_j^{-1}$ for $j=1,\ldots,\la_{\sigma(i)}$,
		and poles at
		$u_{\sigma(i)}=\SP_j\ip_j^{-1}$ for $j=1,\ldots,\la_{\sigma(i)}+1$.
	\end{itemize}
	This implies that the following cancellations of poles:
	\begin{itemize}
		\item If $x_i\ge\la_{\sigma(i)}+2$, then the integrand does not have poles 
		$\SP_j\ip_j^{-1}$
		inside the contour $\contq {\ipbb\SPB}{\sigma(i)}$. In this case, if we can shrink this contour
		without picking residues at $u_{\sigma(i)}=qu_\bind$ for any $\bind>\sigma(i)$,
		then the whole integral vanishes.

		\item If $x_i\le\la_{\sigma(i)}+1$, then the integrand does not have poles
		$\SP_j^{-1}\ip_j^{-1}$ outside the contour $\contq {\ipbb\SPB}{\sigma(i)}$. 
		Note also that for $\bind\notin\JS$,
		the integrand also 
		does not have poles
		$\SP_j^{-1}\ip_j^{-1}$ outside the contour $\contq {\ipbb\SPB}{\bind}$.
		The integrand, however, has simple poles at each
		$u_{i}=\infty$. 
	\end{itemize}
	If $\sigma(i)>\max\big(\sigma(1),\ldots,\sigma(i-1)\big)$ is a running
	maximum, then the contour $\contq {\ipbb\SPB}{\sigma(i)}$ can be shrunk without picking
	residues at $u_{\sigma(i)}=qu_\bind$. Indeed,
	the factors $u_{\sigma(i)}-qu_\bind$ in the denominator with $\bind\notin\JS$ are
	canceled out by the product over $\aind\in\JS$ and $\bind\notin\JS$,
	and all the factors 
	$u_{\sigma(i)}-qu_{\sigma(j)}$ with $\sigma(i)<\sigma(j)$
	are present in the other product
	over $1\le\aind<\bind\le \ell$.
	Therefore, the whole integral vanishes unless
	$x_i\le\la_{\sigma(i)}+1$ for each such running maximum $\sigma(i)$.

	Next, if the latter condition holds, then we also have 
	$x_j\le\la_{\sigma(j)}+1$ for all $j=1,\ldots,\ell$. Indeed, if $\sigma(j)$
	is not a running maximum, then there exists $i<j$ with $\sigma(i)>\sigma(j)$ (as $\sigma(i)$
	we can take the previous running maximum),
	and it remains to recall that both the $\la_p$'s and the $x_p$'s are ordered:
	\begin{align*}
		x_j\le x_i\le \la_{\sigma(i)}+1\le \la_{\sigma(j)}+1.
	\end{align*}
		
	Assuming now that $x_j\le\la_{\sigma(j)}+1$ for all $j$, we can expand 
	the contours $\contq {\ipbb\SPB}1,\ldots,\contq {\ipbb\SPB}n$ (in this order) to infinity,
	and evaluate the integral by taking minus residues at that point.
	The single products over $i=1,\ldots,\ell$ and $\bind\notin\JS$
	produce
	the factor $(1-q)^{n}(-\sh_1\SPB)^{\la}$.
	Let the ordered sequence of 
	elements of $\JS$ be
	$\JS=\{j_1<\ldots<j_\ell\}$.
	One can readily see that
	the 
	three remaining cross-products lead to the factor
	\begin{align*}
		q^{\inv(\sigma)}q^{\ell(j_1-1)+(\ell-1)(j_2-j_1-1)+\ldots+(j_\ell-j_{\ell-1}-1)}=
		q^{\inv(\sigma)}q^{j_1+\ldots+j_\ell}q^{-\frac{\ell(\ell+1)}{2}}.
	\end{align*}
	This completes the proof.
\end{proof}

The coefficient $\expancoeff{\la}$ in \eqref{expancoeff} 
is thus equal to the sum of
the right-hand sides of \eqref{computing_multi_moment_integral}
over all $\sigma\colon\{1,\ldots,\ell\}\to\{1,\ldots,n\}$. 
This sum can be computed using the following lemma:

\begin{lemma}\label{lemma:about_inversions_summation}	
	Let $X_1,X_2,\ldots$ be indeterminates, $k\in\Z_{\ge1}$,
	and $n_1,\ldots,n_k\in\Z_{\ge1}$ be arbitrary. 
	We have the following identity:
	\begin{multline*}
		\sum_{\substack{1\le i_1 \le n_1,\ldots,1\le i_k\le n_k\\
		\textnormal{$(i_1,\ldots,i_k)$ pairwise distinct}}}
		X_{i_1}X_{i_2+\inv_{\le 2}}\ldots X_{i_k+\inv_{\le k}}
		\\=(X_1+\ldots+X_{n_1})
		(X_2+\ldots+X_{n_2})
		\ldots
		(X_k+\ldots+X_{n_k}),
	\end{multline*}
	where $\inv_{\le p}:=\#\{j< p\colon i_j>i_p\}$.
	By agreement, the right-hand side is 
	zero if one of the sums is empty.
\end{lemma}
\begin{proof}
	It suffices to show that the map
	\begin{align*}
		(i_1,\ldots,i_k)\mapsto(i_1,i_2+\inv_{\le2},\ldots,i_k+\inv_{\le k})
	\end{align*}
	is a bijection between the sets
	\begin{multline*}
		\{1\le i_1 \le n_1,\ldots,1\le i_k\le n_k,
		\; \textnormal{$(i_1,\ldots,i_k)$ pairwise distinct}
		\}
		\\\textnormal{and}\quad
		\{1\le j_1\le n_1,\,2\le j_2\le n_2,\ldots,k\le j_k\le n_k\}.
	\end{multline*}
	By induction, this statement will follow if 
	we show that for any pairwise distinct
	$1\le i_1,\ldots,i_{k-1}\le n_{k-1}$,
	the map
	$i_k\mapsto i_k+\inv_{\le k}$ is a bijection between
	\begin{align*}
		\{1\le j\le n_k\colon j\notin\{i_1,\ldots,i_{k-1}\}\}
		\quad\textnormal{and}\quad
		\{k,k+1,\ldots,n_k\}.
	\end{align*}
	But the latter fact is evident from Fig.~\ref{fig:bijection},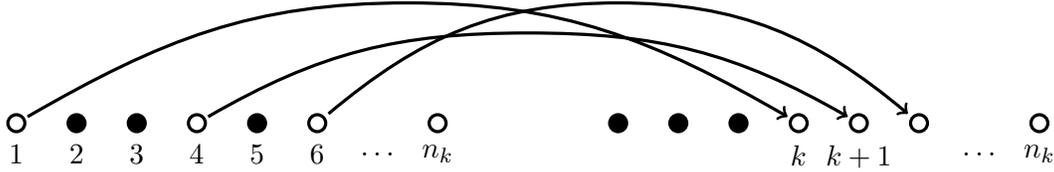
\begin{figure}[htpb]
		\begin{tikzpicture}
			[scale=.8,very thick]
			\draw (0,0) circle(4pt) node[below, yshift=-4] {$1$};
			\draw[fill] (1,0) circle(4pt) node[below, yshift=-4] {$2$};
			\draw[fill] (2,0) circle(4pt) node[below, yshift=-4] {$3$};
			\draw (3,0) circle(4pt) node[below, yshift=-4]  {$4$};
			\draw[fill] (4,0) circle(4pt) node[below, yshift=-4] {$5$};
			\draw (5,0) circle(4pt) node[below, yshift=-4] {$6$};
			\node [below, yshift=-7] at (6,0) {$\ldots$};
			\draw (7,0) circle(4pt) node[below, yshift=-4] {$n_k$};
			\draw[fill] (10,0) circle(4pt) node[below, yshift=-4] {};
			\draw[fill] (11,0) circle(4pt) node[below, yshift=-4] {};
			\draw[fill] (12,0) circle(4pt) node[below, yshift=-4] {};
			\draw (13,0) circle(4pt) node[below, yshift=-4] {$k$};
			\draw (14,0) circle(4pt) node[below, yshift=-4] {$k+1$};
			\draw (15,0) circle(4pt) node[below, yshift=-4] {};
			\node [below, yshift=-7] at (16,0) {$\ldots$};
			\draw (17,0) circle(4pt) node[below, yshift=-4] {$n_k$};
			\node (i1) at (0,0) {};
			\node (i2) at (3,0) {};
			\node (i3) at (5,0) {};
			\node (j1) at (13,0) {};
			\node (j2) at (14,0) {};
			\node (j3) at (15,0) {};
			\draw[->] (i1) [out=30,in=180] to (6.5,2) [out=0,in=150] to (j1);
			\draw[->] (i2) [out=30,in=180] to (8.5,1.5) [out=0,in=150] to (j2);
			\draw[->] (i3) [out=40,in=180] to (10,2) [out=0,in=140] to (j3);
		\end{tikzpicture}
		\caption{A bijection used in the proof of Lemma \ref{lemma:about_inversions_summation}.}
		\label{fig:bijection}
	\end{figure}
	as the map $i_k\mapsto i_k+\inv_{\le k}$
	simply corresponds to stacking together the elements of
	$\{1\le j\le n_k\colon j\notin\{i_1,\ldots,i_{k-1}\}\}$.
\end{proof}

By Lemma \ref{lemma:computing_multi_moment_integral},
the right-hand side of \eqref{computing_multi_moment_integral}
takes the form
\begin{multline*}
	\expancoeff{\la}=q^{-\ell}(1-q)^{\ell}(-\sh_1\SPB)^{\la}
	\sum_{\substack{\sigma(1),\ldots,\sigma(\ell)\in\{1,\ldots,n\}\\
	\textnormal{pairwise distinct}}}
	q^{\sigma(1)+\ldots+\sigma(\ell)+\inv(\sigma)}\mathbf{1}_{\textnormal{$\la_{\sigma(i)}\ge x_i-1$ for all $i$}}
	\\=q^{-\ell}(1-q)^{\ell}(-\sh_1\SPB)^{\la}
	\sum_{\substack{1\le \sigma(1)\le \HT_{\la}(x_1-1),\ldots,1\le \sigma(\ell)\le \HT_{\la}(x_\ell-1)
	\\\textnormal{pairwise distinct}}}q^{\sigma(1)+\ldots+\sigma(\ell)+\inv(\sigma)}.
\end{multline*}
where we have recalled that the height function
is defined as $\HT_{\la}(x-1)=\max\{j\colon \la_j\ge x-1\}$.
We can now apply Lemma \ref{lemma:about_inversions_summation}
with $X_i=q^{i}$, and conclude that the above sum 
factorizes as
\begin{align}\label{expancoeff_computed}
	\expancoeff{\la}=q^{-\ell}(1-q)^{\ell}(-\sh_1\SPB)^{\la}\prod_{i=1}^{\ell}
	(q^{i}+q^{2}+\ldots+q^{\HT_\la(x_i-1)})=
	(-\sh_1\SPB)^{\la}\prod_{i=1}^{\ell}
	(q^{i-1}-q^{\HT_\la(x_i-1)}).
\end{align}
Therefore, we have finally computed the right-hand side of \eqref{multi_moments_rid_of_around_0},
and it is equal to 	
\begin{align}
	\sum_{\la\in\signp n}
	(-\sh_1\SPB)^{\la}
	\F_\la^{\conj}(u_1,\ldots,u_n\md
	\sh_1\ipb,\sh_1\SPB)
	\prod_{i=1}^{\ell}
	(q^{i-1}-q^{\HT_\la(x_i-1)})
	=
	\sum_{\la\in\signp n}
	\prod_{i=1}^{\ell}
	(q^{i-1}-q^{\HT_\la(x_i)})
	\MM_{\UU;\RHO}(\la\md\ipb,\SPB)\label{expancoeff_computed2}
\end{align}
(because $\MM_{\UU;\RHO}$ is given by \eqref{MM_U_RHO_particular}),
which is the same as the left-hand side of \eqref{multi_moments_rid_of_around_0} by the very definition. 
Identity \eqref{multi_moments_rid_of_around_0} is thus established
under the restrictions of Lemma \ref{lemma:expansion_r_la_exists}.
However, as both sides of this identity are rational functions
in all the variables and parameters (cf. Lemma \ref{lemma:qmom_rational} for the left-hand side and
formula \eqref{R_function_definition} for the right-hand side),
we conclude that these restrictions can be dropped as long as the sum over $\la$ 
in \eqref{expancoeff_computed2} converges. 
This implies Theorem \ref{thm:multi_moments}.



\section{Degenerations of moment formulas} 
\label{sec:degenerations_of_moment_formulas}

Here we apply $q$-moment formulas from 
\S \ref{sec:_q_moments_of_the_height_function_of_interacting_particle_systems}
to rederive 
$q$-moment formulas for the
stochastic six vertex model, the ASEP, 
$q$-Hahn, and $q$-Boson systems 
obtained earlier in the
literature (see references below).
In some cases we also present their inhomogeneous
generalizations.

\subsection{Moment formulas for the stochastic six vertex model and the ASEP} 
\label{sub:moments_of_six_vertex_model_and_asep}

Recall the stochastic six vertex model described in \S \ref{sub:asep_degeneration}. 
That is, we take the parameters $\SP_x\equiv q^{-\frac12}$ for all $x\in\Z_{\ge0}$, and consider the 
dynamics $\Xp_{\{u_t\};\RHO}$ 
in which at each discrete time step, a new particle 
is born at location $1$ (\S \ref{sub:interacting_particle_systems}). 
For this dynamics to be an honest Markov process (i.e., with nonnegative transition probabilities), 
we require that all other parameters satisfy
\begin{align}\label{six_vertex_nontrivial_nonnegativity}
	\textnormal{$0<q<1$,\quad $\ip_j>0$,\quad $u_i>0$,\quad and\quad $\ip_ju_i> q^{-\frac12}$ for all $i,j$,}
\end{align}
and that the $\ip_j$'s are uniformly bounded away from $0$ and $\infty$.
(Another range with $q>1$ will lead to a trivial limit shape
for the stochastic six vertex model, see the discussion in \S \ref{sub:asep_degeneration}.)
\begin{corollary}\label{cor:six_vertex_moments}
	Assume that \eqref{six_vertex_nontrivial_nonnegativity} holds, and that 
	the $\ip_j$'s satisfy $q\cdot M_{\ipb}<m_{\ipb}$ (we use notation 
	\eqref{min_max_notation}). Moreover, let $u_i\ne q u_j$ for any $i,j=1,\ldots,n$. 
	Then the $q$-moments of the height function of the inhomogeneous stochastic
	six vertex model 
	$\Xp_{\{u_t\};\RHO}$ at $\text{time}=n$ are given by
	\begin{multline}\label{six_vertex_moments}
		\E_{\UU;\RHO}^{\textnormal{six vertex}}
		\prod_{i=1}^{\ell}q^{\HT_{\nu}(x_i)}=
		q^{\frac{\ell(\ell-1)}2}
		\oint\limits_{\contnz {\bar\UU}1}\frac{d w_1}{2\pi\i}
		\ldots
		\oint\limits_{\contnz {\bar\UU}\ell}\frac{d w_\ell}{2\pi\i}
		\prod_{1\le \aind<\bind\le \ell}\frac{w_\aind-w_\bind}{w_\aind-qw_\bind}
		\\\times
		\prod_{i=1}^{\ell}\bigg(
		w_i^{-1}
		\prod_{j=1}^{x_i-1}
		\frac{\ip_j-q^{-\frac12}w_i}{\ip_j-q^{\frac12}w_i}
		\prod_{j=1}^{n}\frac{1-qu_jw_i}{1-u_jw_i}
		\bigg)
	\end{multline}
	for any $\ell\in\Z_{\ge1}$ and $x_1\ge \ldots\ge x_\ell\ge1$.
	The integration contours above are as in Definition \ref{def:circular_contours_around_0}.
\end{corollary}
In the homogeneous case $u_i\equiv u$, $\ip_j\equiv 1$, this formula essentially
reduces to \cite[Thm.\;4.12]{BCG6V} which was proven by a different method.
\begin{proof}
	The claim follows
	from Theorem \ref{thm:multi_moments} because
	both sides of the identity \eqref{multi_moments} 
	are rational functions in all parameters,
	and, moreover, 
	the integrations in the right-hand sides of \eqref{multi_moments} and \eqref{six_vertex_moments}
	are sums over the same sets of residues. Indeed, our conditions
	\eqref{six_vertex_nontrivial_nonnegativity} imply that 
	$u_i^{-1}<\SP^{-1}\ip_j<\SP\ip_j$ (where $\SP=q^{-\frac12}$), and 
	so the contours 
	$\contnz {\bar\UU}j$ exist and yield the same residues. Note also that even though now $\SP_j= q^{-1/2}>1$
	instead of belonging to $(-1,0)$,
	conditions \eqref{admissibility_RHO_conditions} 
	(ensuring the existence of the measure $\MM_{\UU;\RHO}$)
	readily follow from \eqref{six_vertex_nontrivial_nonnegativity}.
\end{proof}

Let us now consider the continuous time limit of the stochastic six vertex model
to the ASEP. In \S \ref{sub:asep_degeneration} we have 
described this limit in the case of a fixed number of particles, 
but the dynamics $\Xp_{\{u_t\};\RHO}$ (in which at each time step 
a new particle is born at location $1$) in this limit
also produces a meaningful initial condition for the ASEP.
Indeed, setting $\ip_j= 1$, $u_i= q^{-\frac12}+(1-q)q^{-\frac12}\epsilon$,
we see that for $\epsilon=0$ the configuration $\la\in\signp n$
of the stochastic six vertex model at $\text{time}=n$ is simply
$\la=(n,n-1,\ldots,1)$. For small $\epsilon>0$, at times $n=\lfloor t\epsilon^{-1}\rfloor$, the configuration of the 
stochastic six vertex model will be a finite perturbation of 
$(n,n-1,\ldots,1)$ near the diagonal. Thus, shifting the lattice coordinate 
as $\la_i=n+1+y_i$ with $y_i\in\Z$, we see that in the limit $\epsilon\searrow0$ 
the initial condition for the ASEP 
becomes $y_1(0)=-1$, $y_2(0)=-2,\ldots$, which is known as the \emph{step initial configuration}.

\begin{corollary}\label{cor:ASEP_moments}
	For $0<q<1$, any $\ell\in\Z_{\ge1}$, and $(x_1\ge \ldots\ge x_\ell)\in\Z^{\ell}$,
	the $q$-moments of the height function of the ASEP
	$\HT_{\textnormal{ASEP}}(x):=\#\{i\colon y_i\ge x\}$
	started from the step initial configuration
	$y_i(0)=-i$
	are given by
	\begin{multline}\label{multi_ASEP_moments}
		\E^{\textnormal{ASEP, step}}\prod_{i=1}^{\ell}
		q^{\HT_{\textnormal{ASEP}}(x_i)}
		=
		q^{\frac{\ell(\ell-1)}2}
		\oint\limits_{\contnz {\sqrt q}1}\frac{d w_1}{2\pi\i}
		\ldots
		\oint\limits_{\contnz {\sqrt q}\ell}\frac{d w_\ell}{2\pi\i}
		\prod_{1\le \aind<\bind\le \ell}\frac{w_\aind-w_\bind}{w_\aind-qw_\bind}
		\\\times
		\prod_{i=1}^{\ell}\bigg(
		w_i^{-1}
		\bigg(\frac{1-q^{-\frac12}w_i}{1-q^{\frac12}w_i}\bigg)^{x_i}
		\exp\Big\{-\frac{(1-q)^{2}}{(1-q^{\frac12}w_i)(1-q^{\frac12}/w_i)}\,t\Big\}
		\bigg),
	\end{multline}
	for any time $t\ge0$. 
	Each integration contour
	$\contnz {\sqrt q}j$ consists of 
	two positively oriented circles --- one is a small circle around $\sqrt q$, and 
	another one is a circle $r^{j}c_0$ around zero as in Definition \ref{def:circular_contours_around_0}.
\end{corollary}
\begin{proof}
	Start with the moment formula 
	\eqref{six_vertex_moments} for 
	$\ip_j= 1$, $u_i= q^{-\frac12}+(1-q)q^{-\frac12}\epsilon$,
	$x_i=n+1+x_i'$, 
	and $n=\lfloor t \epsilon^{-1}\rfloor$. Here $x_i'\in\Z$ are the shifted labels of moments.
	As $\epsilon\searrow0$,
	the contours $\contnz{\bar\UU}j$
	will still contain the same parts around zero, and the 
	part $\contn{\bar\UU}$ encircling the $u_i^{-1}$'s will turn into a small circle 
	$\contn{\sqrt{q}}$
	around~$\sqrt{q}$
	(because $\sqrt{q}$ is the limit of $u^{-1}$ as $\epsilon\searrow0$).
	Let us now look at the integrand. We have
	\begin{align*}
		\bigg(\frac{1-q^{-\frac12}w}{1-q^{\frac12}w}\bigg)^{x_i-1}\bigg(\frac{1-quw}{1-uw}\bigg)^{n}
		&=
		\bigg(\frac{1-q^{-\frac12}w}{1-q^{\frac12}w}\bigg)^{n+x_i'}
		\bigg(
		\frac{1-q^{\frac12}w}{1-q^{-\frac12}w}-
		\epsilon\,\frac{(1-q)^{2}}{(1-q^{-\frac12}w)(1-q^{\frac12}/w)}+O(\epsilon^{2})
		\bigg)^{n}
		\\&=
		\bigg(\frac{1-q^{-\frac12}w}{1-q^{\frac12}w}\bigg)^{x_i'}
		\bigg(
		1-\epsilon\,\frac{(1-q)^{2}}{(1-q^{\frac12}w)(1-q^{\frac12}/w)}+O(\epsilon^{2})
		\bigg)^{\lfloor t \epsilon^{-1}\rfloor}.
	\end{align*}
	In the limit as $\epsilon\searrow0$, the second factor turns into the 
	exponential. 
	Thus, renaming $x_i'$ back to $x_i$, we arrive at the desired claim.
\end{proof}

When $x_1=\ldots=x_\ell$, formula \eqref{multi_ASEP_moments} essentially coincides with 
the one obtained in \cite[Thm.\;4.20]{BorodinCorwinSasamoto2012} using 
duality. The multi-point generalization \eqref{multi_ASEP_moments}
of that formula seems to be new.

The paper \cite{BorodinCorwinSasamoto2012}
also deals with other multi-point observables. Namely, denote 
\begin{align*}
	\tilde Q_{x}:=
	q^{\HT_{\textnormal{ASEP}}(x+1)}
	\mathbf{1}_{\textnormal{there is a particle at location $x$}}.
\end{align*}
The expectations 
$\E^{\textnormal{ASEP}}\big(\tilde Q_{x_1}\ldots \tilde Q_{x_k}\big)$,
$x_1> \ldots>x_k$
(for the step, and in fact also for the step-Bernoulli initial conditions),
were computed in \cite[Cor.\;4.14]{BorodinCorwinSasamoto2012}.
The duality statement pertaining to these observables
dates back to \cite{schutz1997dualityASEP}.
Then the expectation of $q^{\ell\,\HT_{\textnormal{ASEP}}(x)}$ was recovered
from these multi-point observables \cite[Thm.\;4.20]{BorodinCorwinSasamoto2012}.

Note that for the ASEP, expectations of 
$\tilde Q_{x_1}\ldots \tilde Q_{x_k}$ are essentially the same as the 
$q$-correlation functions (\S \ref{sub:_q_correlation_functions}). 
In fact, our proof
of Proposition \ref{prop:one_moment} 
(recovering one-point $q$-moments from the $q$-correlation functions)
somewhat mimics the ASEP approach mentioned above, 
but dealing with a higher spin system introduces the need for 
the more complicated observables \eqref{q_def_correlation_thing}.


\subsection{Starting from infinitely many particles at location $1$} 
\label{sub:starting_from_infinitely_many_particles_at_zero}

Let us now return to generic values of $\SP_j$, and consider the limit transition
from $\Xp_{\{u_t\};\RHO}$
to the dynamics $\Xii_{\{u_t\}}$ 
described in Remark \ref{rmk:infinitely_at_one}.
Recall that to pass to the dynamics $\Xii_{\{u_t\}}$,
one should take
$\ip_1=\SP_1$, and let 
the $\UU$ parameters be 
$(1,q,q^{2},\ldots,q^{K-1},u_1,\ldots,u_n)$. 
For the process to have nonnegative transition probabilities in the $K\to+\infty$
limit (and be nontrivial), we should take $\SP_1^{2}<0$ and
$u_i>0$, while all other parameters
$\SP_x$, $\ip_x$ ($x\ge2$), 
and $q$ should satisfy \eqref{stochastic_weights_condition_qsxi}.
Under these assumptions, in the $K\to+\infty$
limit we obtain Markov dynamics
$\Xii_{\{u_t\}}$ which starts from the configuration
$1^{\infty}2^{0}3^{0}\ldots$. 
Moreover, we also need to 
assume that 
the nested integration contours (in the corollary below) 
are well-defined,
which requires
$m_{\ipb|\SPB|}>q M_{\ipb|\SPB|}$.

\begin{corollary}\label{cor:Xinfinity_moments}
	Under the assumptions described above,
	the $q$-moments of $\Xii_{\{u_t\}}$ at $\text{time}=n$
	have the form
	\begin{multline}\label{Xinfinity_moments}
		\E^{\Xii}_{\UU}\prod_{i=1}^{\ell}q^{\HT_{\nu}(x_i)}=
		(-1)^{\ell}q^{\frac{\ell(\ell-1)}2}
		\oint\limits_{\contq{\ipb\SPB}1}\frac{d w_1}{2\pi\i}
		\ldots
		\oint\limits_{\contq{\ipb\SPB}\ell}\frac{d w_\ell}{2\pi\i}
		\prod_{1\le \aind<\bind\le \ell}\frac{w_\aind-w_\bind}{w_\aind-qw_\bind}
		\\\times
		\prod_{i=1}^{\ell}\bigg(
		\frac{1}{w_i(1-w_i)}
		\prod_{j=1}^{x_i-1}
		\frac{\ip_j-\SP_jw_i}{\ip_j-\SP_j^{-1}w_i}
		\prod_{j=1}^{n}\frac{1-qu_jw_i}{1-u_jw_i}
		\bigg),
	\end{multline}
	where $\text{time}=n$, $\ell\in\Z_{\ge1}$, and 
	$x_1\ge \ldots\ge x_\ell\ge1$ are arbitrary.
	The integration contours $\contq{\ipb\SPB}j$
	are $q$-nested (as in Definition \ref{def:orthogonality_contours}),
	encircle $\{\SP_1^{2},\ip_2\SP_2,\ldots\}$,
	and leave outside $0$, $1$, and $u_i^{-1}$ for all $i$.
	We also assume that $\ip_1=\SP_1$  in \eqref{Xinfinity_moments}.
\end{corollary}
Note that if $x_\ell=1$, then the integrand has no $w_\ell$-poles
inside the smallest contour $\contq{\ipb\SPB}\ell$,
and thus vanishes, as it should be for the left-hand side of \eqref{Xinfinity_moments} because $\HT_\nu(1)=+\infty$.
\begin{proof}
	Since before the limit the parameters $\UU$ have the form 
	\eqref{U_fused_for_contours}, we must use Corollary \ref{cor:multi_moments_fused}
	instead of Theorem \ref{thm:multi_moments}, and 
	take the integration contours to be 
	$\contqiz{\bar\UU}j$, $j=1,\ldots,\ell$.
	After setting $\ip_1=\SP_1$ with $\SP_1^{2}<0$,
	the integration contours $\contqiz{\bar\UU}j$
	are $q^{-1}$-nested around $\{1,u_1^{-1},\ldots,u_n^{-1}\}$,
	leave $\{\SP_1^{2},\ip_2\SP_2,\ldots\}$ outside, and contain parts $r^{j}c_0$
	around zero (cf. Definition \ref{def:circular_contours_around_0}).
	This readily implies
	that the 
	substitution $\ip_1=\SP_1$ is allowed because the resulting 
	integral is the sum of the same residues
	as before the substitution.

	Now, since the integration contours do not change in the 
	limit as $K\to+\infty$, let us take it in the integrand. 
	Then
	the product over $i=1,\ldots,\ell$ becomes
	\begin{align*}
		\prod_{i=1}^{\ell}\bigg(
		\frac{1}{w_i(1-w_i)}
		\prod_{j=1}^{x_i-1}
		\frac{\ip_j-\SP_jw_i}{\ip_j-\SP_j^{-1}w_i}
		\prod_{j=1}^{n}\frac{1-qu_jw_i}{1-u_jw_i}
		\bigg).
	\end{align*}
	We see that 
	the integrand for $K=+\infty$ is regular at infinity,
	so we can 
	drag the integration contours 
	$\contqiz{\bar\UU}\ell,\ldots,\contqiz{\bar\UU}1$
	(in this order)
	through infinity, and they turn into $q$-nested and negatively oriented
	contours 
	$\contq{\ipb\SPB}j$
	around $\{\SP_1^{2},\ip_2\SP_2,\ldots\}$, which leave 
	$\{u_1^{-1},\ldots,u_n^{-1}\}$, $0$ and $1$ outside.
	Note that the first group of points lies in the
	left half-plane, while $u_i^{-1}>0$.
	
	In the proof we have assumed \eqref{assumptions_one_better},
	so that the contours $\contqiz{\bar\UU}j$ are well-defined.
	However, this assumption can be dropped in \eqref{Xinfinity_moments}
	because both sides of \eqref{Xinfinity_moments} are a priori rational
	functions in all parameters.
	This implies the desired claim.
\end{proof}

In the homogeneous case $\ip_j\equiv 1$ and $\SP_j\equiv \SP$,
the result of Corollary \ref{cor:Xinfinity_moments}
was obtained in 
\cite[Thm. 4.1]{CorwinPetrov2015} using duality. More precisely, to obtain 
the system considered in that paper, one needs to take $\SP_1^{2}=\SP$ and $\SP_j\equiv \SP$ for $j\ge2$,
so that the probabilities with which particles leave location~$1$
are in agreement with what is going on at all further locations.


\subsection{Moment formulas for $q$-Hahn and $q$-Boson systems} 
\label{sub:moments_of_q_hahn_and_q_boson_systems}

As explained in \S \ref{ssub:_q_hahn_particle_system},
the $q$-Hahn particle system $\Xih$ depending on $J\in\Z_{\ge1}$ is obtained from the process 
$\Xii_{\{u_t\}}$ by fusion
\begin{align*}
	\UU=(1,q,\ldots,q^{J-1},\ldots,1,q,\ldots,q^{J-1})\qquad
	\textnormal{($n$ groups)}
\end{align*}
and by setting $\SP_j=\ip_j$ for all $j\in\Z_{\ge2}$, with $\SP_j^{2}<0$. 
The process $\Xih$ also starts with infinitely many particles at $1$.
\begin{corollary}\label{cor:qHahn_moments}
	Let $J\in\Z_{\ge1}$ and $\SP_j^{2}<0$ for all $j$.
	Moreover, let us assume that $m_{|\SPB^{2}|}>qM_{|\SPB^{2}|}$,
	so the integration contours below exist. Then
	for any $\ell\in\Z_{\ge1}$ and $x_1\ge \ldots\ge x_\ell\ge1$,
	the moments of $\Xih$ at $\text{time}=n$ have the form:
	\begin{multline}\label{qHahn_moments}
		\E^{\textnormal{$q$-Hahn}}\prod_{i=1}^{\ell}q^{\HT_{\nu}(x_i)}=
		(-1)^{\ell}q^{\frac{\ell(\ell-1)}2}
		\oint\limits_{\contq{\SPB^{2}}1}\frac{d w_1}{2\pi\i}
		\ldots
		\oint\limits_{\contq{\SPB^{2}}\ell}\frac{d w_\ell}{2\pi\i}
		\prod_{1\le \aind<\bind\le \ell}\frac{w_\aind-w_\bind}{w_\aind-qw_\bind}
		\\\times
		\prod_{i=1}^{\ell}\left(
		\bigg(\frac{1-q^{J}w_i}{1-w_i}\bigg)^{n}
		\frac{1}{w_i(1-w_i)}
		\prod_{j=1}^{x_i-1}
		\frac{1-w_i}{1-\SP_j^{-2}w_i}
		\right),
	\end{multline}
	where the integration contours 
	$\contq{\SPB^{2}}j$ are $q$-nested around $\{\SP_j^{2}\}_{j\in\Z_{\ge1}}$
	and leave $0$ and $1$ outside.
\end{corollary}
\begin{proof}
	Immediately follows from Corollary \ref{cor:Xinfinity_moments}.
\end{proof}
The $q$-moment formula \eqref{qHahn_moments} 
holds when $J\in\Z_{\ge1}$ and $\SP_x^{2}<0$ for all $x$ (case 2 in
\eqref{qHahn_weights_condition}), but can be also analytically continued
to other values of parameters. For example, \eqref{qHahn_moments} also holds
when $0<q<1$, $q^{J}$ is regarded as an independent parameter,
$\SP_x^{2}<q^{J}\SP_x^{2}<1$ and $q^{J}\SP_x^{2}>0$ for all $x$ 
(case 1 in
\eqref{qHahn_weights_condition}), and, moreover, 
$m_{|\SPB^{2}|}>qM_{|\SPB^{2}|}$ holds.
Since $0\ne \SP_x^{2}<1$, we see that the contours $\contq{\SPB^{2}}j$ 
leaving $0$ and $1$ outside
also make sense in this case.

In the homogeneous case $\SP_j\equiv \SP$, $q$-moments 
\eqref{qHahn_moments} 
of the $q$-Hahn process
were computed in \cite{Corwin2014qmunu} using duality (cf. the discussion
in \S \ref{sub:duality_from_observables}). 
An inhomogeneous generalization of this duality 
(which differs from the inhomogeneity considered in \eqref{qHahn_moments})
was also written down in that paper. Namely, returning to the notation of
Remark \ref{rmk:analityc_continuation_in_Xe}, consider the
$q$-Hahn TASEP
in which each particle $x_j$ jumps 
according to the distribution 
$\phi_{q,\smu_j,\snu_j}(\cdot\md\gap_j)$.
When the parameters $\snu_j\equiv \snu$
are homogeneous
and the $\smu_j$'s are arbitrary,
duality relations for this process were obtained in
\cite{Corwin2014qmunu}.
However, 
the corresponding evolution equations
were solved (yielding contour integral
formulas for observables)
in \cite{Corwin2014qmunu}
only when the parameters 
$\smu_j\equiv \smu$ are also homogeneous.
The remaining case 
when the $\snu_j$'s are homogeneous and the $\smu_j$'s are not
does not seem to fall under our framework. 
Corollary \ref{cor:qHahn_moments} provides another
``solvable'' case of the inhomogeneous $q$-Hahn TASEP, 
when both the $\snu_j$'s and the $\smu_j$'s are inhomogeneous,
but $\smu_j/\snu_j\equiv q^{J}=\textnormal{const}$.

\medskip

Let us now turn to the stochastic $q$-Boson system which is obtained from the
$q$-Hahn process by setting 
$J=1$ and $\SP_j^{2}=-\epsilon a_j$ with $a_j>0$, and speeding 
up the time by a factor of $\epsilon^{-1}$
(see \S \ref{ssub:qBoson}). For the nested contours in the corollary below
to make sense, we must also require that 
$\min\{a_i\}<q\cdot\max\{a_i\}$.
\begin{corollary}\label{cor:qBoson_moments}
	Under the above assumptions, the $q$-moments of the height function 
	of the $q$-Boson process (started with infinitely many particles at $1$) 
	have the form
	\begin{multline}\label{qBoson_moments}
		\E^{\textnormal{$q$-Boson}}\prod_{i=1}^{\ell}q^{\HT_{\nu}(x_i)}=
		(-1)^{\ell}q^{\frac{\ell(\ell-1)}2}
		\oint\limits_{\contq{-\mathbf{a}}1}\frac{d w_1}{2\pi\i}
		\ldots
		\oint\limits_{\contq{-\mathbf{a}}\ell}\frac{d w_\ell}{2\pi\i}
		\prod_{1\le \aind<\bind\le \ell}\frac{w_\aind-w_\bind}{w_\aind-qw_\bind}
		\\\times\prod_{i=1}^{\ell}\bigg(
		\frac{e^{(1-q)tw_i}}{w_i}
		\prod_{j=1}^{x_i-1}
		\frac{a_j}{a_j+w_i}
		\bigg),
	\end{multline}
	where $t\ge0$ is the time and
	$\ell\in\Z_{\ge1}$ and $x_1\ge x_2\ge \ldots\ge x_\ell\ge1$ are arbitrary.
	The integration contours $\contq{-\mathbf{a}}j$
	are $q$-nested around the points $\{-a_i\}$,
	and do not contain $0$.
\end{corollary}
\begin{proof}
	Set $J=1$, $\SP_j^{2}=-\epsilon a_j$, and $n=\lfloor t \epsilon^{-1} \rfloor$ in 
	\eqref{qHahn_moments}, and change the variables as $w_i=\epsilon w'_i$. Then the integral in 
	\eqref{qHahn_moments} becomes
	\begin{align*}
		\oint\limits_{\contq{-\mathbf{a}}1}\frac{d w'_1}{2\pi\i}
		\ldots
		\oint\limits_{\contq{-\mathbf{a}}\ell}\frac{d w'_\ell}{2\pi\i}
		\prod_{1\le \aind<\bind\le \ell}\frac{w'_\aind-w'_\bind}{w'_\aind-qw'_\bind}
		\prod_{i=1}^{\ell}\left(
		\bigg(\frac{1-q \epsilon w'_i}{1- \epsilon w'_i}\bigg)^{\lfloor t \epsilon^{-1} \rfloor}
		\frac{1}{w'_i(1- \epsilon w'_i)}
		\prod_{j=1}^{x_i-1}
		\frac{1- \epsilon w'_i}{1+w'_i/a_j}
		\right).
	\end{align*}
	Sending $\epsilon\searrow0$ and renaming $w_i'$
	back to $w_i$, we 
	arrive at the desired formula.
\end{proof}
The continuous time moment formula \eqref{qBoson_moments} 
first appeared
in \cite{BorodinCorwin2011Macdonald} and \cite{BorodinCorwinSasamoto2012}. 
Analogous formulas for discrete time $q$-TASEPs 
(corresponding to discrete time $q$-Boson systems as 
in \S \ref{sub:general_j_dynamics_and_q_hahn_degeneration})
were obtained in
\cite{BorodinCorwin2013discrete},
and they can also be derived from the $q$-Hahn moment formula \eqref{qHahn_moments}.

It is worth noting that a recent paper \cite{Wang2015inhomqTASEP} contains a 
formula for transition probabilities of the inhomogeneous $q$-Boson system,
which is essentially equivalent to the 
$q$-Boson version of the 
first Plancherel isomorphism
of Theorem \ref{thm:Plancherel}.




\providecommand{\bysame}{\leavevmode\hbox to3em{\hrulefill}\thinspace}
\providecommand{\MR}{\relax\ifhmode\unskip\space\fi MR }
\providecommand{\MRhref}[2]{%
  \href{http://www.ams.org/mathscinet-getitem?mr=#1}{#2}
}
\providecommand{\href}[2]{#2}

\end{document}